\documentclass[11pt,a4paper]{amsart}
\usepackage[T1]{fontenc}
\usepackage{lmodern}
\usepackage[margin=1in]{geometry}
\usepackage{amsmath,amssymb,amsthm,mathtools,dsfont,mathrsfs,centernot}
\usepackage{microtype,graphicx,xcolor,booktabs,longtable,array,enumitem}
\usepackage{tikz-cd}
\usepackage{color}
\usepackage[colorlinks=true,linkcolor=blue,citecolor=blue,urlcolor=blue]{hyperref}
\usepackage[nameinlink,capitalise]{cleveref}
\usepackage{soul}
\makeatletter
\@namedef{subjclassname@2020}{%
  \textup{2020} Mathematics Subject Classification}
\makeatother

\allowdisplaybreaks

\theoremstyle{plain}
\newtheorem{theorem}{Theorem}[section]
\newtheorem{lemma}[theorem]{Lemma}

\newtheorem{proposition}[theorem]{Proposition}

\newtheorem{corollary}[theorem]{Corollary}

\newtheorem{introtheorem}{Theorem}

\newtheorem{introcorollary}{Corollary}[introtheorem]

\theoremstyle{definition}
\newtheorem{definition}[theorem]{Definition}
\newtheorem{example}[theorem]{Example}

\theoremstyle{remark}
\newtheorem{remark}[theorem]{Remark}

\numberwithin{equation}{section}

\DeclareMathOperator{\rank}{rank}

\DeclareMathOperator{\End}{End}

\newcommand{\C}{\mathcal{C}}

\newcommand{\D}{\mathcal{D}}
\newcommand{\E}{\mathcal{E}}

\newcommand{\cL}{\mathcal{L}}
\newcommand{\M}{\mathcal{M}}

\newcommand{\bC}{\mathbb{C}}
\newcommand{\bZ}{\mathbb{Z}}

\newcommand{\kk}{\Bbbk}

\newcommand{\Hom}{\mathrm{Hom}}
\newcommand{\unit}{\mathds{1}}

\definecolor{darkgreen}{RGB}{55,138,0}
\definecolor{burntorange}{RGB}{180,85,0}
\definecolor{navyblue}{RGB}{18,40,180}
\definecolor{cyan(process)}{rgb}{0.0, 0.6, 1.0}

\newcommand{\semanticTag}[1]{\stepcounter{equation}\tag{#1}}

\DeclareMathOperator{\FPdim}{FPdim}
\newcommand{\one}{\mathbf1}

\hypersetup{hypertexnames=false,pdftitle={Haagerup--Izumi categories of Q-system type},pdfauthor={Terry Gannon, Andrew Schopieray, Harshit Yadav}}
\title[On Haagerup--Izumi fusion categories]{On Haagerup--Izumi fusion categories}
\author{Terry Gannon}
\address{Department of Mathematics, University of Alberta, Edmonton, Alberta, Canada}
\email{tjgannon@ualberta.ca}

\author{Andrew Schopieray}
\address{Department of Mathematics, SUNY-Plattsburgh, Plattsburgh, NY 12901, US}
\email{ascho013@plattsburgh.edu}

\author{Harshit Yadav}
\address{Max-Planck-Institut f\"ur Mathematik, Vivatsgasse 7, 53111 Bonn, Germany}
\email{yadav@mpim-bonn.mpg.de}

\date{September 15, 2026}

\begin{document}
\raggedbottom
\begin{abstract}
For every odd $N$, we construct a 
Haagerup--Izumi (HI) category  for $\mathbb Z/N$,
using Barnes double-sine functions.
Equivariantization gives near-group categories of type
$((\mathbb Z/N)^2,N^2)$ for every odd $N$.
 We classify connected separable algebras up to Morita equivalence in
such HI categories and
describe their dual categories. These duals are
often HI for noncyclic groups with nontrivial pointed associators. We prove the Evans--Gannon conjecture that the modular data of the center is always a smashed sum involving a finite metric group of
order $N^2+4$ and recover this group and its quadratic form directly from the data  of the HI category. 
We also give the HI equations and reconstruction in arbitrary
characteristic and prove a Frobenius symmetry in characteristic two.
Using this, we obtain the complete list of (untwisted, subfactor type) HI categories over $\mathbb{C}$  for odd $N\le 43$.
Other applications include finding three pairwise non-Morita-equivalent
$\mathbb Z/15$ categories whose centers have the same modular data.
We show that the center of a $\mathbb Z/5$ HI category cannot be defined (i.e., admits no split ribbon form) over a cyclotomic field, disproving a conjecture of
Davidovich--Hagge--Wang.
\end{abstract}
\maketitle
{\small
\setcounter{tocdepth}{1}
\tableofcontents}


\section{Introduction}
\label{sec:introduction}

Let $G$ be a finite abelian group. The Haagerup--Izumi (HI) fusion
rules have simple objects $\alpha_g$ and $\alpha_g\rho$, with
\begin{equation*}
 \alpha_g\alpha_h\cong\alpha_{g+h},\qquad
 \alpha_g\rho\cong\rho\alpha_{-g},\qquad\text{ and }\qquad \rho^2\cong\mathbf1\oplus\bigoplus_{g\in G}\alpha_g\rho.
\end{equation*}
We construct categorifications of these fusion rules for every cyclic group of
odd order. We also characterize and describe separable algebras, Morita duals and centers under the hypotheses stated below. The primary framework of our exposition is the endomorphism presentation developed by M.\ Izumi for Cuntz algebras
\cite[Sections~8 and 9]{IzumiLR2}
and extended by D.\ Evans and T.\ Gannon to nonunitary categories using Leavitt
algebras \cite[Theorems~1--2]{evans2017non}. In this presentation, the 
category is reconstructed from a matrix indexed by $G$ satisfying the
polynomial HI equations. We solve these equations uniformly for cyclic
$G$.

The importance of the HI categories is that they (and their centers) are \textit{exotic}, i.e.\ can't be constructed yet from standard methods applied to standard tensor categories. Studying them can challenge or sharpen our conjectures, and hopefully lead to new constructions and new classes of tensor categories. The HI categories seem to be the most accessible class of exotic fusion categories.

Many of our results restrict to HI categories with the specified algebra structure on $Q=\mathbf1\oplus\rho$ associated with the boundary condition in
Definition~\ref{def:q-system-type}. We call these categories
of \emph{Q-system type}; for unitary HI categories this is equivalent to there being a subfactor with canonical endomorphism $Q$. Over $\mathbb C$, the algebra $Q$ is
connected and separable and its module category gives a Morita dual; this condition does not require unitarity. We call an HI category
\emph{untwisted} when the associator on its pointed subcategory is trivial.

\subsection{The main theorem}

Let $N\geq3$ be odd,
\begin{equation*}
 G=\mathbb Z/N\mathbb Z,\qquad
 \mu=N^2+4,\qquad\text{ and }\qquad
 d_\pm=\frac{1}{2}(N\pm\sqrt{\mu}).
\end{equation*}
Write $\tau(N)$ for the number of positive divisors of $N$.  For a subgroup
$H\leq G$, let $R_H=\bigoplus_{h\in H}\alpha_h$ with its inherited
subgroup-algebra multiplication.  We use
$\mathcal{MD}_\pm(G;H',q)$ for the smashed sum modular data discussed in
Appendix~\ref{ce:compact-MD}: $H'$ is the metric group labeling the
remaining center primaries, $q$ is its quadratic form, and $\pm$ fixes
the spherical dimension.  In this notation, $S$ includes its
normalizing factor $1/N$. Let $\sigma\in\{1,-1\}$.

\begin{introtheorem}[Cyclic HI categories of Q-system type over $\mathbb C$]
\label{thm:intro-main}
\;
\begin{enumerate}[label=\textup{(\arabic*)},ref=\arabic*]
\item\label{item:intro-existence}
\textbf{Uniform existence.} 
There is an untwisted spherical HI category of Q-system type
for $G$, with $\dim(\rho)=d_-$, which we give in closed form. There is also a pseudo-unitary untwisted HI category of 
Q-system type for $G$, with $\dim(\rho)=d_+$, obtained by Galois
conjugation.

\item\label{item:intro-centers}
\textbf{Modular data.}
Let $\mathcal C$ be a complex spherical untwisted HI category 
with $\dim(\rho)=d_\sigma$. 
The modular data of $Z(\mathcal C)$ are
$\mathcal{MD}_\sigma(G;H',q)$ for a nondegenerate metric group
$(H',q)$ of order $\mu$, with
\begin{equation*}
 \sum_{h\in H'}e^{2\pi i q(h)}=-\sigma\sqrt\mu.
\end{equation*}
The isometry class of $(H',q)$ is determined by explicit scalar
invariants of the HI matrix, given in Proposition~\ref{ce:metric-recognition}.
If $\mu$ is square-free, then $H'\cong\mathbb Z/\mu$ and
$q(h)=s_\sigma h^2/\mu\pmod{\mathbb Z}$, where $s_\sigma\in(\mathbb Z/\mu)^\times$ has Jacobi symbol $(s_\sigma/\mu)=-\sigma$. When $\mu$ is prime,
$\sigma$ determines this form up to isometry.

\item\label{item:intro-classification}
\textbf{Separable algebras and their duals.}
Let $\mathcal C$ be a  unitary untwisted HI category of Q-system
type for $G$.
Then every connected separable algebra in $\mathcal C$ is Morita equivalent to exactly one of
\begin{equation*}
 \{R_H:H\leq G\}\ \sqcup\ \{Q\}.
\end{equation*}
Thus there are $\tau(N)+1$ equivalence classes of indecomposable module categories.
The dual associated with $R_H$ has HI fusion rules with invertible group
$(G/H)\times\widehat H$ and pointed associator given by the carry cocycle
\eqref{sa:carry-cocycle}; this class is trivial precisely when
$\gcd(|H|,[G:H])=1$.  The dual associated with $Q$ has $N+1$ simple
objects, of Frobenius--Perron dimensions
\begin{equation*}
 1,\quad d_+,\quad
 \underbrace{d_+-1,\ldots,d_+-1}_{(N-1)/2},\quad
 \underbrace{d_++1,\ldots,d_++1}_{(N-1)/2},
\end{equation*}
and is tensor-generated by the object of dimension $d_+$, with fusion
matrix \eqref{sa:Q-dual-matrix}.
\end{enumerate}
\end{introtheorem}

\subsection{The construction}

We observed that for some numerical solutions in Table~\ref{tab:cyclic-negative}, $A_{1,g}$ was close to $\pm g$ for $g\ll N$. In these solutions, the entries $A_{g,h}$ approximate binomial coefficients up to a sign. This insinuates ratios of gamma values, since $\Gamma(n+1)=n!$. Proposition~\ref{ds:fixed-index-integrality} proves
the fixed-$g$ limit that motivated this line of reasoning. But the ordinary gamma function does not supply the shift
relations needed for the quadratic and cubic identities. Through its connection to extensions of real quadratic fields, we were led to the Barnes double-sine, which has two compatible shift directions.
Its nonzero values $s_g$ give the matrix
\begin{equation*}
 A_{g,h}=\frac{s_gs_{h-g}}{s_h}+\delta_{h,0}
 \qquad(g,h\in G).
\end{equation*}
Theorem~\ref{ds:uniform-existence} shows that this matrix satisfies the
full HI equations for $d_-$ and $\omega=1$.

\subsection{Unitarity}
Given any automorphism of $\overline{\mathbb Q}$ sending $d_-\mapsto d_+$, the conjugate category has spherical dimensions equal to its
Frobenius--Perron dimensions, hence is pseudo-unitary; this is also a special
case of Thornton's general result for generalized near-group categories
\cite[Theorem~IV.3.6]{thornton2012generalized}. By \cite{IzumiLR2} and \cite{evans2017non}, this will be unitary if $A_{g,h}=\overline{A_{h,g}}$ 
for all $0<g<h<N$.  We expect all untwisted cyclic HI categories with dimension $d_+$ to be unitary, and have verified this for all odd $N\le43$ in Table~\ref{tab:cyclic-census}.

\subsection{Positive characteristic} The HI equations and reconstruction are treated over algebraically closed
fields of arbitrary characteristic. The HI matrices will then have entries in finite fields. Moreover, by the Frobenius Density Theorem, for a positive density of primes $p$ the HI categories will be defined over $\mathbb Z/p$. 
Characteristic 2 is especially useful: every HI-matrix satisfies
\begin{equation*}
 A_{2g,2h}=A_{g,h}^4.
\end{equation*}
This gives the finite-field bound used for enumeration. Conversely,
Theorem~\ref{thm:HI-integrality-at-two} proves that  full
HI-matrices are integral above $2$. The proof uses a Fourier identity to
rule out a nonzero reduction of a nonintegral matrix.
Proposition~\ref{prop:HI-good-reduction-at-two} then relates the
characteristic-zero categories to their reductions by lifting.
Corollary~\ref{cor:HI-categorical-completeness} gives completeness up to
tensor equivalence for Q-system HI categories, provided the
finite-field list is exhaustive. 
Using this we obtain the completeness of Tables~\ref{tab:cyclic-census} and~\ref{tab:cyclic-negative}.

\subsection{Modular data and categorical applications}

Expressions for the modular data of the centers of HI categories are given in \cite{evans2017non}, building on \cite{IzumiLR2}, and much of it is very simple, although one block was very complicated. An important observation of \cite{evans2011exoticness} was that in examples, that mysterious block is equally simple, and is given by the smashed sum  $\mathcal{MD}_\pm(G;H',q)$. Part~\ref{item:intro-centers} proves this is always the case. To identify that block, we study a subalgebra  it defines of the full fusion algebra. Its product rules
recover addition in a finite abelian group up to sign. Galois symmetry of the modular $S$ matrix and
Frobenius at $2$ force these rules; $S=S^{\top}$  and the
modular relation $(ST)^3=C$ then identify the pairing and ribbon twists. 
Section~\ref{ce:n15-application} explains the three $\mathbb Z/15$ examples in Table~\ref{tab:cyclic-census} are pairwise
non-Morita-equivalent by part~\ref{item:intro-classification}, while their
centers have the same modular data by part~\ref{item:intro-centers}.
 These HI categories are Galois conjugates of each other. The first examples of inequivalent modular  categories with identical modular data were found in \cite{mignard2021modular}.

In the context of Theorem~\ref{thm:intro-main}(3), the Morita dual ${}_A\C_A$ need not be cyclic or untwisted. For example, for $G=\mathbb Z/9$ and $H=3\mathbb Z/9$, the  dual has invertible
group $\mathbb Z/3\times\mathbb Z/3$ and nontrivial pointed associator
(Example~\ref{sa:n9-dual}). 
Compare \cite{evans2011exoticness} where it was shown there are no untwisted solutions to (E1)-(E4) for $G=\mathbb Z/3\times\mathbb Z/3$.

\subsection{An infinite family of near-group categories}

We apply an algebraic form of Izumi's HI-to-near-group
construction \cite[Theorem~12.9 and Example~12.15]{IzumiNearGroupCuntz}
to the categories constructed above.
Conjugation by the pointed objects gives a $G$-action on an HI category
$\mathcal C$. Its equivariantization is the relative center with respect
to the pointed subcategory. When this subcategory is untwisted, the
result is a $G\times\widehat G$ near-group category with multiplicity $N^2$
(Theorem~\ref{ng:near-group}). The existence theorem therefore gives:

\begin{introcorollary}[Uniform near-group existence]
\label{cor:intro-near-group}
For every odd $N$, there is a near-group fusion category over
$\mathbb C$ of type
\begin{equation*}
 \bigl(\mathbb Z/N\times\mathbb Z/N,N^2\bigr).
\end{equation*}
In particular, these categories form an infinite family with noncyclic
invertible groups.
\end{introcorollary}

For the pseudo-unitary conjugates, Corollary~\ref{ng:modular-realization}
also identifies the center data with the Evans--Gannon and Grossman--Izumi
near-group formulas. The metric group of order $N^2+4$ is the same one
that occurs in the HI center.

Equivariantization gives
a central functor $\operatorname{Rep}(G)\to Z(\mathcal C^G)$;
de-equivariantization by its regular algebra recovers $\mathcal C$.
For a nontrivial pointed associator, the invertible group can instead be
a nontrivial extension of $G$ by $\widehat G$. Section~\ref{sec:near-group}
gives this extension and the data required for de-equivariantization. This construction suggests that near-groups are more fundamental than HI categories.

\subsection{A counterexample to cyclotomic fields of definition}

In \cite[Section 2]{etingof2005fusion}, Etingof, Nikshych and Ostrik explain that the usual examples of fusion categories are all defined over cyclotomic fields, and asked whether all are. Morrison and Snyder \cite{morrison2012non} showed that the unique unitary Q-system $\mathbb Z/3$ HI category is not cyclotomic. They also showed that its center, regarded as a fusion category, was cyclotomic.
Davidovich, Hagge and Wang conjectured that every modular category over
$\mathbb C$ can be defined over a cyclotomic  field
\cite[Conjecture~4.5]{DavidovichHaggeWang2013}. They require
the endomorphism algebra of each simple object to be the defining field
\cite[Sections~2.1--2.2]{DavidovichHaggeWang2013}, so the question concerns
split ribbon forms.

Let $\mathcal C_5$ be the unique unitary Q-system $\mathbb Z/5$ HI category
first found in  \cite[Example~7.2]{IzumiLR2}. 
Using the explicit HI matrix for $\mathcal C_5$, we prove the following in Theorem~\ref{sf:counterexample}.

\begin{introtheorem}[Noncyclotomic modular category]
\label{thm:intro-cyclotomic}
Every normal number field $J\subset\overline{\mathbb Q}$ supporting a split
ribbon form of $Z(\mathcal C_5)$ contains both $d_+$ and $\eta=(d_++\sqrt{d_+ +9})/{2}$.  Since the normal closure of $\mathbb Q(d_+,\eta)$ has dihedral Galois group of order $8$, no cyclotomic field can support it.
\end{introtheorem}

\subsection{Foundations}

For an odd-order abelian group, the HI categories are constructed using 
a matrix $A\in M_{|G|}(\kk)$ satisfying the HI system of equations \eqref{eq:HImatrix1}--\eqref{eq:HImatrix4}; these 
depend on the choice of $d,\omega\in\kk$ where $d$ satisfies \eqref{eq:HI-scalars} and $\omega$ is a third root of unity. 
We show that symmetry \eqref{eq:HImatrix1} and the cubic identity \eqref{eq:HI-cubic} are equivalent to the full HI system
(Theorem~\ref{thm:cubic-criterion}), in any characteristic. The Q-system boundary gives the
algebra $Q$, with the exceptional positive-characteristic cases
stated explicitly. If the group order is invertible in the field, $\omega$ must be $1$.

\subsection{Organization}
Section~\ref{sec:foundations} gives the foundations, with detailed proofs in
Appendices~\ref{app:equations}, \ref{app:reconstruction}, and \ref{app:q-system}.
Section~\ref{sec:double-sine}, together with the analytic details in Appendix~\ref{app:double-sine}, gives the construction and proof of Theorem~\ref{thm:intro-main}(1). 
Section~\ref{sec:separable} and Appendix~\ref{app:separable} give the separable 
algebra classification and duals. Section~\ref{sec:centers} proves the smashed-sum
formula for the center and recovers its metric group by scalar
contractions of the HI matrix;
the arithmetic recognition proof is in Appendix~\ref{ce:recognition-proof}. Section~\ref{sec:near-group} treats
equivariantization. Section~\ref{sec:split-field} gives the noncyclotomic example,
and Section~\ref{sec:characteristic-two} gives the characteristic-two results, integrality above $2$, and the categorical reduction criterion. Appendix~\ref{app:finite-cyclic-data} contains the explicit cyclic
data over $\mathbb C$ for $N\le 43$.

\textbf{Note added}: While completing this paper we learned that Tzu-Chen Huang was independently working on similar questions. In particular, in \cite{huang2026cyclic} he established like us existence of an HI category for each odd cyclic $G$. He constructed the same category as we did, though his proof was quite different. He did not address the other things we have included.

\textbf{Acknowledgements.} TG would like to thank Paul Budinsky, who in his 2021 MSc thesis obtained some of the data in Table~\ref{tab:cyclic-census} for $N\le 29$, and Pavel Etingof for encouraging us to study these categories in positive characteristic. The research of TG is supported, and HY was supported during much of this work, by NSERC Canada. HY is grateful to MPIM, Bonn for its hospitality and financial support.

\textbf{AI use disclosure.} We used ChatGPT during the development and 
drafting of this paper to explore questions, possible formulations, and proof
strategies; to assist with computations and literature searches; and
to review arguments and improve the exposition. This assistance
included suggestions for intermediate claims, proof steps, and
corrections. All mathematical statements, proofs, computational
results, and references were independently checked by the authors,
who take sole responsibility for the contents of the paper.


\section{Q-system-type categories and the HI equations}
\label{sec:foundations}\label{sec:matrixtoLrealization}\label{sec:HI-fixed-boundary-general}
We give the HI equations and the reconstruction theorem, and identify the
boundary condition defining Q-system type. The cubic criterion reduces
verification of the HI equations to \eqref{eq:HImatrix1} and the cubic
identity; we use it in the double-sine construction.

Let $\kk$ be an algebraically closed field of arbitrary characteristic and let $G$ be a finite abelian
group of odd order $N$ (not necessarily cyclic). A category of type $\mathfrak{HI}_G$ has simple objects
$\alpha_g$ and $\alpha_g\rho$, for $g\in G$, with fusion rules
\begin{equation}\label{HIfusrul}
 \alpha_g\alpha_h=\alpha_{g+h},\qquad
 \alpha_g\rho=\rho\alpha_{-g},\qquad
 \rho^2=\mathbf1\oplus\bigoplus_{g\in G}\alpha_g\rho.
\end{equation}
We call the category \emph{untwisted} when its pointed subcategory is
$\mathrm{Vec}_G$ with trivial associator. 

\subsection{The equations and the cubic criterion}
Choose scalars $d,\omega\in \kk$ satisfying
\begin{equation}\label{eq:HI-scalars}
 d^2=Nd+1,\qquad \omega^3=1,
\end{equation}
and put
\begin{equation}\label{eq:HI-a}
 a=d^{-1},\qquad a^2+Na=1.
\end{equation}
The scalar equation gives $d\ne0$. In these equations, an integer denotes
its image in $\kk$; the group itself remains the abstract finite group $G$.
No division by $N$, $2$, $3$, or $N^2+4$ is used in this setup.
\begin{definition}[HI-matrix]\label{def:HI-matrix}
For fixed $(d,\omega)$ satisfying \eqref{eq:HI-scalars}, a \emph{$(d,\omega)$-HI-matrix} is a matrix
$A=(A_{g,h})_{g,h\in G}\in M_{|G|}(\kk)$ satisfying the following equations:
\begin{align}
 A_{g,h}
 &=\omega A_{-h,g-h}
  =\omega^{-1}A_{h-g,-g},
 . \semanticTag{E1}\label{eq:HImatrix1}\\
 \sum_{r\in G}A_{r,0}
 &=-\omega^{-1}a, \semanticTag{E2}\label{eq:HImatrix2}\\
 \sum_{r\in G}A_{h+r,k}A_{k,r}
 &=\delta_{h,0}-a\delta_{k,0}, \semanticTag{E3}\label{eq:HImatrix3}\\
 \sum_{l,m\in G}
 A_{l,m}A_{l+g,h}A_{h+m,l+i}A_{i,k+m}
 &=A_{h-g,i-g}\delta_{k,g}
   -\omega^{-1}a\delta_{h,0}A_{i,k}
   -\omega a A_{g,h}\delta_{i,0},. 
   \semanticTag{E4}\label{eq:HImatrix4}
\end{align}
for all $g,h,i,k\in G$. 
Equations \eqref{eq:HImatrix1}--\eqref{eq:HImatrix4} are the polynomial form of the complex equations in
\cite[Theorem~1, equations~(4.7)--(4.10)]{evans2017non},
with the complex-conjugate  $\overline\omega$ replaced by  $\omega^{-1}$.
\end{definition}

The first change of indices in \eqref{eq:HImatrix1} has order three, so a
cube root $\omega$ is allowed at this stage. It can be shown, directly from the Leavitt presentation, that $\omega$ corresponds to the third Frobenius--Schur indicator of $\rho$. We will show in 
Theorem~\ref{thm:HI-omega-trivial}, when $|G|\neq 0$ in $\kk$, the remaining equations force $\omega$ to be $1$.

\begin{remark}\label{rem:HI-trivial-group}
If $G=\{0\}$, then \eqref{eq:HImatrix2} makes $A_{0,0}$ nonzero, while
\eqref{eq:HImatrix1} gives $A_{0,0}=\omega A_{0,0}$.  Thus every HI-matrix
for the trivial group has $\omega=1$.
\end{remark}

In \cite[Proposition~2]{evans2017non}, \eqref{eq:HImatrix1}--\eqref{eq:HImatrix4} were shown to imply the following cubic relation:
\begin{equation}
 \omega^{-1}\sum_{m\in G}
 A_{m,g+h}A_{g,m+k}A_{h,m+l}
 =A_{g+l,k}A_{h+k,l}-a\delta_{g,0}\delta_{h,0},
 \qquad g,h,k,l\in G. \semanticTag{C}\label{eq:HI-cubic}
\end{equation}
Moreover, it was asked whether the cubic implies the complicated 
relation \eqref{eq:HImatrix4}. Our next results proves something stronger:

\begin{theorem}[Cubic criterion]\label{thm:cubic-criterion}
If $G$ is nontrivial, a matrix satisfies
\eqref{eq:HImatrix1}--\eqref{eq:HImatrix4} if and only if it satisfies
\eqref{eq:HImatrix1} and \eqref{eq:HI-cubic}.
\end{theorem}
\begin{proof}
The forward implication follows from
Theorem~\ref{thm:HI-quartic-implies-cubic}, since $1-Na=a^2$ is a unit.
For the reverse implication,
Theorem~\ref{thm:HI-cubic-implies-quadratic} gives the sum \eqref{eq:HImatrix2} and quadratic \eqref{eq:HImatrix3}
identities, and Theorem~\ref{thm:HI-cubic-implies-quartic} gives the quartic
identity \eqref{eq:HImatrix4}. These results, including their proofs over the stated fields or
rings, are in Appendix~\ref{app:equations}.
\end{proof}
Thus, after the symmetry~\eqref{eq:HImatrix1} is checked, the cubic identity is
the only equation needed for the reconstruction. The precise ring hypotheses
for the two implications are recorded in Appendix~\ref{app:equations}.

\begin{theorem}[Trivial cube-root]\label{thm:HI-omega-trivial}
\label{thm:scalar-obstruction}
If $N\ne0$ in $\kk$, every $(d,\omega)$-HI-matrix has $\omega=1$.
\end{theorem}
This is proved in Appendix~\ref{app:omega}. The hypothesis always holds in characteristic zero and in characteristic two, and in characteristic three $\omega^3=1$ itself implies $\omega=1$.

\subsection{From a matrix to a spherical category}
The reconstruction uses a Leavitt realization. The explicit endomorphism
formulas and their verification in arbitrary characteristic are given in
Appendix~\ref{app:reconstruction}; the main construction uses their
$\omega=1$ specialization. The tensor product in the category of algebra
endomorphisms is composition, and an intertwiner $x:f\to g$ satisfies
$xf(y)=g(y)x$ for every $y$ in the Leavitt algebra.

\begin{theorem}[Reconstruction]\label{thm:reconstruction}
\label{thm:HImatrix-to-Leavitt}
Every $(d,\omega)$-HI-matrix $A$ over an algebraically closed field gives an
untwisted spherical fusion category $\mathcal C(A)$ of type $\mathfrak{HI}_G$.
Its noninvertible simple objects have dimension $d$, and its pivotal
structure is unique. In the main setting $N\ne0$ in $\kk$, $\omega=1$, and
$b^2=a$, the category is the additive idempotent completion of the
endomorphism category given in Appendix~\ref{app:reconstruction}.
\end{theorem}
\begin{proof}
The general assignments use $b^2=\omega^{-1}a$ and are displayed in
Appendix~\ref{app:reconstruction}. Theorem~\ref{thm:HImatrix-to-Leavitt-general}
verifies their relations and their intertwiner spaces; setting $\omega=1$
gives the specialization used in the main construction.
Theorem~\ref{thm:LRtoHIcat} then gives the spherical fusion category.
The construction is strict on the pointed subcategory, so its pointed
associator is trivial.
\end{proof}
The detailed proof includes the cases in which the characteristic divides $N$, the discriminant $\mu$, or the global dimension. The global dimension of the resulting spherical category is $N(1+d^2)$. 
For every finite abelian $G$ and in every characteristic, simultaneous
relabeling a solution $A$ by $\alpha\in\operatorname{Aut}(G)$ to get $\alpha(A)$ (that is, $\alpha(A)_{g,h}=A_{\alpha g,\alpha h}$) gives tensor-equivalent reconstructed categories.

When $G$ is cyclic and $\kk=\mathbb C$, \cite[Theorem~2(b)]{evans2017non} says the categories $\mathcal C(A)$ and $\mathcal C(A')$ are tensor equivalent iff $A'=\alpha(A)$ for some $\alpha\in\mathrm{Aut}(G)$. 
When $G$ is arbitrary but both $\mathcal C(A)$ and $\mathcal C(A')$ are unitary, \cite[Theorem~5.9]{izumi2018classification} says they are tensor equivalent iff $A$ and $A'$ lie in the same $H^2(G,\mathbb C^\times)\rtimes\mathrm{Aut}(G)$ orbit.

\begin{remark}[Algebraicity of solutions]\label{algebraicic} 
By Ocneanu rigidity \cite[Theorem~2.28]{etingof2005fusion} (together with the equivalence criterion for HI-matrices \cite[Theorem~2(b)]{evans2017non}), the number of solutions $A$ to equations (E1)--(E4) must be finite, at least in characteristic 0; in characteristic two the same conclusion follows from Corollary~\ref{cor:HI-char-two-field-bound}. In these characteristics, the entries $A_{g,h}\in\kk$ are algebraic over any subfield of $\kk$. In Section~\ref{sec:double-sine} we therefore find an infinite family of special values of the double-sine function which are algebraic. \end{remark}

\subsection{The Q-system condition}
The object $Q=\mathbf1\oplus\rho$ connects the matrix construction to the
module categories. For $a\ne1$, define
\begin{equation*}
 \theta=-\frac{a}{1-a}=-\frac1{d-1}.
\end{equation*}
\begin{definition}[Q-system condition]\label{def:HI-Q-system-condition}
\label{def:q-system-type}
A $(d,1)$-HI-matrix, with $a\ne1$, satisfies the \emph{Q-system
condition} if
\begin{equation}\label{eq:display-0271}
 A_{g,0}=\delta_{g,0}+\theta\qquad(g\in G).
\end{equation}
\end{definition}
We use \emph{Q-system type} in this algebraic sense. Theorem~\ref{thm:algebraic-q-system} below identifies the corresponding
algebra on $Q$. Over $\mathbb C$ this
includes the negative-dimension family constructed below. A
genuine $C^*$-Q-system (equivalently, a corresponding subfactor) requires unitarity of the HI category.

The double-sine construction uses reciprocal samples to impose this
boundary condition. 

\begin{theorem}[Q-system interpretation]\label{thm:algebraic-q-system}
If $a\ne1$, the Q-system condition is equivalent to the existence of a unital
associative multiplication on $Q=\mathbf1\oplus\rho$ that is nonzero on
$\rho\otimes\rho$. If $a\ne\pm1$, it is equivalent to the existence of a
connected separable algebra structure on $Q$.
If $a=1$, no such separable algebra exists; when $a=-1\ne1$, the indicated
nonzero multiplication is not separable.
\end{theorem}
\begin{proof}
The multiplication and separability section are computed explicitly in
Theorem~\ref{thm:HI-fixed-boundary-algebra}, proved in
Appendix~\ref{app:q-system}.
\end{proof}
When $N$ is invertible in $\kk$, the scalar equation excludes $a=\pm1$. Hence the
algebraic boundary used in the complex construction produces the connected
separable algebra needed later.

\begin{remark}[Exhaustiveness for subfactor-type categorifications]
\label{rem:subfactor-parametrization}
Here \emph{subfactor type} means that the category is unitary $Q=\mathbf1\oplus\rho$ carries a genuine $C^*$-Q-system. All such categorifications of HI fusion rules \eqref{HIfusrul} are tensor equivalent to $\mathcal C(A)$ for some 
Hermitian matrix $A$ with $d=d_+$ that satisfies
\eqref{eq:HImatrix1}, \eqref{eq:HI-cubic}, and the Q-system boundary
\eqref{eq:display-0271}, with $\omega=1$. $G$ need not be cyclic. 

The exhaustiveness comes from starting with an arbitrary subfactor-type
category. By \cite[Lemma~7.1]{izumi2018classification} it will be untwisted. The Q-system--subfactor correspondence realizes it as the even
part of the associated $3^G$ subfactor
\cite[Sections~6.1 and~7]{izumi2018classification}, and 
\cite[Theorem~7.3]{izumi2018classification} then  gives an HI-matrix with the stated
boundary.
 Conversely, these equations reconstruct a
category by Theorem~\ref{thm:reconstruction}; $d=d_+$ and
Hermitian condition give unitarity
\cite[Theorem~2\textup{(c)}]{evans2017non}.
The boundary gives a connected separable algebra by
Theorem~\ref{thm:algebraic-q-system}.  

When we drop unitarity and/or the Q-system condition, then we have to allow for non-trivial associator on the pointed part. We  expect though that the remaining unitary Q-system type HI categories are obtained via Fourier transform/Morita dual of the unitary ones. When $|G|$ is even, the data and equations are much more involved  \cite[Theorem~5.7]{izumi2018classification}.
\end{remark}


\section{The uniform double-sine construction}\label{sec:double-sine}
We construct a $(d_-,1)$-HI-matrix of Q-system type for every odd
$N\ge3$ over $\mathbb C$. Its entries are determined by a single sequence
$s_g$ on $\mathbb Z/N\mathbb Z$, defined using values of the double sine
at integers. The proof reduces to two finite summation identities.

We first define the sequence and record its basic identities. We then
state the two finite summation identities and use them to verify the matrix
equations. Their analytic proofs are given in Appendix~\ref{app:double-sine}.
Finally, we derive the asymptotic and categorical consequences of the
construction.

\subsection{Construction}\label{ds:sec:construction}
Put $ N=2r+1\ge3$ and let $G=\mathbb Z/N\mathbb Z$. In this section, we write $d=d_-=-d_+^{-1}$. 
Throughout the proof, $\omega=1$.
We also introduce the following constants which will be used for defining the double sine function:
\begin{equation}\label{ds:analytic-parameters}
 \varpi_1=1+d_+^{-1}=1-d,\quad
 \varpi_2=1+d_+=1-d^{-1},\quad
 \Omega=\varpi_1+\varpi_2=\varpi_1\varpi_2.
\end{equation}
They satisfy $\varpi_2-\varpi_1=N$, $\varpi_1^{-1}+\varpi_2^{-1}=1$ and $\varpi_1d_+=\varpi_2$; 
in particular, $\varpi_1$ is irrational.
Recall $\theta=(1-d)^{-1}=\varpi_1^{-1}=-a/(1-a)$ and put $\kappa = d\theta^2 = -1/(\varpi_1\varpi_2)$. 
All square roots below are positive. We will use
\begin{equation}\label{ds:matrix-scalar-identities}
\begin{gathered}
 \kappa=\theta(\theta-1)=d\theta^2,\qquad
 N\kappa=1-2\theta,\qquad \varpi_2=-\theta/\kappa,\\
 N\theta+1=N\theta^2+2\theta=-d^{-1}.
\end{gathered}
\end{equation}

We recall the definition of the double sine to fix its normalization. It is defined 
using Barnes' double gamma function \cite{Barnes1901} and was used in the works 
of Shintani \cite{Shintani1977}. We follow the convention of
\cite[equations~(2.1)--(2.2)]{KK}. For the positive constants defined above, the 
double zeta function is initially defined by
\[ 
\zeta_2(s,w\mid \varpi_1, \varpi_2) = \sum_{m,n\geq 0} (w + m\varpi_1 + n\varpi_2)^{-s}, \qquad \Re s>2, \;\; \Re w >0,
\]
with the principal branch of the power. Its meromorphic continuation in $s$ is regular at $s=0$. Define the normalized double 
Gamma function and the \emph{double sine} by
\[
\begin{gathered}
 \Gamma_2(w\mid\varpi_1,\varpi_2)
 =\exp\!\left(\left.\frac{\partial}{\partial s}
       \zeta_2(s,w\mid\varpi_1,\varpi_2)\right|_{s=0}\right),\\
 S(w)=S(w\mid\varpi_1,\varpi_2)
 =\frac{\Gamma_2(\Omega-w\mid\varpi_1,\varpi_2)}
        {\Gamma_2(w\mid\varpi_1,\varpi_2)}.
\end{gathered}
\]
The quotient is initially defined for $0 < \Re w < \Omega$ and extends meromorphically to $\mathbb C$. Thus,
$S$ is the function denoted $S_2$ in the references; its reciprocal will be denoted by $\Gamma$ in Appendix~\ref{ds:sec:hypGamma}.
The double-gamma satisfies shift equations which give shift equations for the double sine \cite[Theorem~2.1(a)]{KK}; 
reflection and the value at the midpoint
follow from the defining quotient:
\begin{equation}\label{ds:short-S-convention}
\begin{gathered}
 S(w+\varpi_1)=\frac{S(w)}{2\sin(\pi w/\varpi_2)},\qquad
 S(w+\varpi_2)=\frac{S(w)}{2\sin(\pi w/\varpi_1)},\\
 S(w)S(\Omega-w)=1,\qquad S(\Omega/2)=1.
\end{gathered}
\end{equation}

Write $F(z)=S(\varpi_1+z)$ and choose $\lambda$ with
\begin{equation}\label{ds:lambda}
 \lambda^N=(-1)^{N(N+1)/2+1}.
\end{equation}
Define
\begin{equation}\label{ds:samples}
 s_{[g]}=\frac{1}{\sqrt{\Omega}}\lambda^g(-1)^{g(g+1)/2}F(g)\quad(0\le g<N),\qquad
 s_{[g+N]}=s_{[g]}.
\end{equation}
All group indices below lie in $G$; we omit brackets when no integer lift is involved. We write $\mathbf1_{\{P\}}$ for the indicator of a condition $P$.

\begin{theorem}\label{ds:uniform-existence}
For every odd $N\geq3$, the entries
\begin{equation}
 A_{g,h}=\frac{s_{[g]}s_{[h-g]}}{s_{[h]}}+\delta_{h,0},
 \qquad g,h\in G,
 \label{ds:matrix-formula}
\end{equation}
are defined and form a $(d,1)$-HI-matrix.  The matrix is independent of the
choice of $\lambda$ satisfying~\eqref{ds:lambda}.  Its boundary is
\begin{equation*}
 A_{g,0}=\delta_{g,0}+\theta,
\end{equation*}
so the solution is of Q-system type.  The construction does not classify other
$(d,1)$-HI-matrices.
\end{theorem}

\subsection{Basic identities}\label{ds:sec:basic-identities}
The function $F$ has neither a zero nor a pole at an integer, and
$S(w)>0$ for $0<w<\Omega$; see Appendix~\ref{ds:sec:hypGamma}.
Hence all values used in the sequence are finite and nonzero.
Reflection and the two shifts give
\begin{equation}\label{ds:short-eq:analytic:F-properties}
\begin{gathered}
 F(z)F(N-z)=1,\qquad
 \frac{F(z+N)}{F(z)}=\frac{\sin(\pi z/\varpi_2)}{\sin(\pi z/\varpi_1)},\\
 F(0)=\sqrt{d_+},\qquad
 \frac{F(k+N)}{F(k)}=
 \begin{cases}(-1)^{k+1},&k\in\mathbb Z\setminus\{0\},\\ d_+^{-1},&k=0.
 \end{cases}
\end{gathered}
\end{equation}
Here $F(0)^2=d_+$ follows by taking $z=0$ in the first two identities, and positivity selects the square root.

Extend the formula for $s$ in \eqref{ds:samples} to every integer and call it $\bar s_k$. Equation \eqref{ds:short-eq:analytic:F-properties} and the quadratic phase give
\begin{equation}\label{ds:short-eq:matrix:seam}
 \bar s_0=\theta,\quad \bar s_k\bar s_{-k}=\kappa\ (k\ne0),\quad
 \bar s_{k+N}=\bar s_k\ (k\ne0),\quad \bar s_N=d\bar s_0.
\end{equation}
Thus the only exceptional shift is from $0$ to $N$. Consequently $\bar s_k=s_{[k]}$ except at positive multiples of $N$, where $\bar s_k=d s_0$. To check the product at opposite indices, the phases contribute
$(-1)^{k^2}$ and $F(k)F(-k)=(-1)^{k+1}$. For the shift,
$\lambda^N$ cancels the sign at every nonzero integer. These calculations
also verify \eqref{ds:short-eq:matrix:seam} at nonzero multiples of $N$.
Replacing $\lambda$ by another allowed root multiplies $s$ by a character of $G$, which cancels from $A$.

\subsection{Cyclic sums and the matrix equations}\label{ds:short-sec:matrix}

Define
\[
 K(x,y)=\sum_m s_m s_{m+x+y}s_{-m-x}s_{-m-y},\qquad R_0(x)=s_xs_{-x}=\kappa+\theta\mathbf1_{\{x=0\}}.
\]
\begin{lemma}[Four-factor sum]\label{ds:lem:four-factor-sum}
For every $x,y\in G$,
\begin{equation}\label{ds:short-eq:matrix:K-all}
 K(x,y)=\kappa\bigl(\mathbf1_{\{x=0\}}+\mathbf1_{\{y=0\}}\bigr)+\theta\mathbf1_{\{x=0\}}\mathbf1_{\{y=0\}}.
\end{equation}
\end{lemma}

Call $p,q\in G^3$ \emph{balanced} if $\sum_i p_i=\sum_j q_j$ in $G$, that is, modulo $N$.
For balanced triples write
\[
 M(p,q)=\sum_m\prod_i s_{m+p_i}\prod_j s_{-m-q_j},\qquad
 T(p,q)=\prod_{i,j}s_{p_i-q_j}.
\]
\begin{lemma}[Six-factor sum]\label{ds:lem:six-factor-sum}
For balanced triples with no common entry, we have:
\begin{equation}\label{ds:short-eq:matrix:six-factor}
 \kappa^2 M(p,q)=T(p,q).
\end{equation}
\end{lemma}

The proofs of Lemmas~\ref{ds:lem:four-factor-sum}
and~\ref{ds:lem:six-factor-sum} are given in
Appendix~\ref{ds:app:cyclic-sums}. They use the double-sine shift identities
to express the finite sums as contour integrals; the six-factor calculation
also uses the hyperbolic Saalsch\"utz integral.

\begin{lemma}[Cancelling common entries]\label{ds:short-lem:matrix:multisets}
Equation \eqref{ds:short-eq:matrix:six-factor} holds whenever the two balanced triples differ after ignoring order but counting repetitions. If they are equal and have three distinct entries, then
\begin{equation}\label{ds:short-eq:matrix:multiset-defect}
 \kappa^2M(p,p)-T(p,p)=-\theta\kappa^3.
\end{equation}
\end{lemma}
\begin{proof}
If two such unequal triples share an entry $c$, remove one copy from each, obtaining pairs $p',q'$. These pairs have no entry in common: a second common entry would force the last entries to agree. Their sum of products of four factors is zero by \eqref{ds:short-eq:matrix:K-all}. Expanding the removed pair as
$s_{m+c}s_{-m-c}=\kappa+\theta\mathbf1_{\{m=-c\}}$ gives $M=\theta B$, where
$B=\prod_{i=1}^2s_{p'_i-c}\prod_{j=1}^2s_{c-q'_j}$. Repeated entries contribute repeated factors.
The product on the right-hand side is $T=\theta B\kappa^2$: its remaining four differences form two nonzero opposite pairs. This proves \eqref{ds:short-eq:matrix:six-factor}. For three distinct common entries,
$M=N\kappa^3+3\theta\kappa^2$ and $T=\theta^3\kappa^3$; \eqref{ds:matrix-scalar-identities} gives \eqref{ds:short-eq:matrix:multiset-defect}.
\end{proof}

To verify the matrix equations, we write the correction at the origin
separately:
\begin{equation}\label{ds:short-eq:matrix:symmetric-origin}
 A_{g,h}=\widetilde A_{g,h}+\varpi_2\mathbf1_{\{g=0\}}\mathbf1_{\{h=0\}},\qquad
 \widetilde A_{g,h}=\kappa^{-1}s_gs_{h-g}s_{-h},\qquad
 A_{g,0}=A_{0,g}=A_{g,g}=\theta+\mathbf1_{\{g=0\}}.
\end{equation}
This follows from $R_0(x)$ and $\varpi_2=-\theta/\kappa$. The three factors give the symmetry \eqref{eq:HImatrix1}, and summing the
boundary entries gives \eqref{eq:HImatrix2}. For $k\ne0$, substitution in \eqref{ds:short-eq:matrix:symmetric-origin} gives
$\sum_rA_{h+r,k}A_{k,r}=K(h,k)/\kappa$.
For $k=0$, the boundary formula gives $N\theta^2+2\theta+\mathbf1_{\{h=0\}}$. Thus \eqref{eq:HImatrix3} follows from \eqref{ds:short-eq:matrix:K-all} and \eqref{ds:matrix-scalar-identities}.

It remains to prove \eqref{eq:HI-cubic}. Let $\mathcal R(g,h,k,l)$ be its left side minus its right side. First suppose $g,h,g+h$ are nonzero, and set
\begin{equation}\label{ds:short-eq:matrix:pq}
 p=(0,k-g,l-h),\quad q=(-g-h,k,l),\quad
 c=\frac{s_gs_hs_{-g-h}}{\kappa^3}.
\end{equation}
All three factors in the cubic sum equal their $\widetilde A$-expressions, hence its value is $cM(p,q)$. Grouping the nine factors in $T(p,q)$ gives
\[
 T(p,q)=\kappa^2s_{g+h}s_{-g}s_{-h}\,\widetilde A_{g+l,k}\widetilde A_{h+k,l},\qquad
 cT(p,q)/\kappa^2=\widetilde A_{g+l,k}\widetilde A_{h+k,l}.
\]
If $p,q$ differ after ignoring order but counting repetitions, Lemma~\ref{ds:short-lem:matrix:multisets} proves \eqref{eq:HI-cubic}: a correction at the origin in a matrix entry on the right-hand side would require $(k,l)=(0,-g)$ or $(-h,0)$, and either makes the triples equal after ignoring order and counting repetitions. 

Conversely, equality of the triples forces one of these two alternatives. Indeed $0\in p$ must occur in $q$, so $k=0$ or $l=0$; canceling this entry and using $g,h\ne0$ gives the alternatives. In each, the common triple has three distinct entries. Equation \eqref{ds:short-eq:matrix:multiset-defect} says the cubic sum is
$\widetilde A_{g+l,k}\widetilde A_{h+k,l}-\theta\kappa c$.
Exactly one matrix entry on the right-hand side is at the origin, and the other equals $\kappa^2c$, so \eqref{ds:short-eq:matrix:symmetric-origin} restores precisely $\varpi_2\kappa^2c=-\theta\kappa c$. This proves the cubic identity for all $k,l$ under the standing assumption that $g,h,g+h$ are nonzero.

All zero-row cases follow from one identity. Put $\epsilon_x=\mathbf1_{\{x=0\}}$ and
\[
 Q(u,v)=R_0(u)R_0(v)-\theta\bigl(\epsilon_uR_0(v)+\epsilon_vR_0(u)\bigr)
       =\kappa^2-\theta^2\epsilon_u\epsilon_v.
\]
Expanding \eqref{ds:short-eq:matrix:symmetric-origin}, with $\varpi_2=-\theta/\kappa$, gives
\[
 \begin{aligned}
 D_{h;k,l}&:=A_{-k,h}A_{h,l-k}-A_{l,k}A_{h+k,l}\\
 &=\frac{\theta}{\kappa^2}
   \bigl(\epsilon_hQ(k,l-k)-\epsilon_lQ(k,h+k)\bigr)
 =\theta(\epsilon_h-\epsilon_l).
 \end{aligned}
\]
The two quadratic origin-correction terms cancel. After substituting the formula for $Q$, the remaining triple-indicator terms are both supported at $h=k=l=0$ and cancel as well.
For $g=0$, the boundary formula and \eqref{eq:HImatrix3} give
\[
 \mathcal R(0,h,k,l)
 =\theta(\epsilon_l-a\epsilon_h)+D_{h;k,l}+a\epsilon_h
 =\theta(\epsilon_l-\epsilon_h)+D_{h;k,l}=0,
\]
where $(1-\theta)a=-\theta$. This includes $h=0$.
Finally \eqref{eq:HImatrix1}, with the change of variable $m'=m+k-g$, shows that the residual is invariant under
\[
 (g,h,k,l)\mapsto(h,g,l,k),\qquad
 (g,h,k,l)\mapsto(-g-h,h,-h-k,g-k+l).
\]
These give the cases $h=0$ and $g+h=0$. We have proved \eqref{eq:HI-cubic} for every tuple. Theorem~\ref{thm:cubic-criterion} now proves Theorem~\ref{ds:uniform-existence}.

\subsection{Consequences of the construction}\label{ds:sec:consequences}

\begin{proposition}[Asymptotic integrality at fixed indices]
\label{ds:fixed-index-integrality}
Let $A^{(N)}$ be the matrix of Theorem~\ref{ds:uniform-existence}.
As $N\to\infty$ through odd integers, for each fixed pair $1\leq u<v$,
\begin{equation*}
 A^{(N)}_{u,v}\longrightarrow(-1)^{u(v-u)}\binom vu,
 \qquad A^{(N)}_{v,u}\longrightarrow0,
 \qquad A^{(N)}_{u,u}\longrightarrow1.
\end{equation*}
The indices are represented by the indicated integers once $N>v$.
\end{proposition}

\begin{proof}
By homogeneity of the double sine \cite[Theorem~2.1(e)]{KK} and
$\varpi_2/\varpi_1=d_+$,
\begin{equation*}
 F(k)=S(\varpi_1+k\mid\varpi_1,\varpi_2)
 =S\!\left(1+\frac{k}{\varpi_1}\,\middle|\,1,d_+\right).
\end{equation*}
As $N\to\infty$, $d_+\to\infty$ and $\varpi_1\to1$.
We use the standard degeneration of the double sine to the Euler gamma
function \cite[Proposition~III.6]{ruijsenaars1997first}. In our reciprocal-sine
normalization,
\begin{equation}
 S(x\mid1,L)
 =\sqrt{2\pi}
 \left(\frac{2\pi}{L}\right)^{1/2-x}
 \Gamma_{\mathrm E}(x)^{-1}(1+o(1)),
 \qquad L\to\infty,
 \label{ds:double-sine-gamma-limit}
\end{equation}
locally uniformly in $x$.

For fixed $j\geq1$, put
\begin{equation*}
 \beta_j=\frac{F(1)}{F(0)}\frac{F(j-1)}{F(j)}.
\end{equation*}
Applying \eqref{ds:double-sine-gamma-limit} to the four factors gives
\begin{align*}
 \beta_j
 =\frac{S(1+\varpi_1^{-1}\mid1,d_+)}
         {S(1\mid1,d_+)}
   \frac{S(1+(j-1)\varpi_1^{-1}\mid1,d_+)}
         {S(1+j\varpi_1^{-1}\mid1,d_+)}
         =
         \frac{\Gamma_{\mathrm E}(1+j\varpi_1^{-1})}
 {\Gamma_{\mathrm E}(1+\varpi_1^{-1})
  \Gamma_{\mathrm E}(1+(j-1)\varpi_1^{-1})}(1+o(1)).
\end{align*}
The powers of $2\pi/d_+$ cancel. Therefore
\begin{equation}
 \beta_j\longrightarrow
 \frac{\Gamma_{\mathrm E}(j+1)}
 {\Gamma_{\mathrm E}(2)\Gamma_{\mathrm E}(j)}=j.
 \label{ds:beta-limit}
\end{equation}

Put $v_j=(-1)^j\beta_j$, $P_1=1$, and
$P_g=\prod_{j=2}^g v_j$ for $g\geq2$.
By \eqref{ds:beta-limit}, for each fixed $g$,
\begin{equation*}
 P_g\longrightarrow(-1)^{g(g+1)/2-1}g!.
\end{equation*}

Using the consecutive-ratio formula from \eqref{ds:samples}, the matrix
formula \eqref{ds:matrix-formula}, and the product identity
$s_{[g]}s_{[-g]}=\kappa=d\theta^2$ for $g\neq0$
from \eqref{ds:short-eq:matrix:seam}, we obtain
\begin{equation*}
 (d-1)A^{(N)}_{u,v}=\frac{P_v}{P_uP_{v-u}},\qquad
 (d-1)A^{(N)}_{v,u}=d\frac{P_uP_{v-u}}{P_v}.
\end{equation*}
Hence
\begin{equation*}
 \frac{P_v}{P_uP_{v-u}}
 \longrightarrow(-1)^{u(v-u)+1}\binom vu.
\end{equation*}
Since $d\to0$ and $d-1\to-1$, it follows that
\begin{equation*}
 A^{(N)}_{u,v}\longrightarrow(-1)^{u(v-u)}\binom vu,
 \qquad
 A^{(N)}_{v,u}\longrightarrow0.
\end{equation*}
Finally, \eqref{ds:matrix-formula} and $s_{[0]}=\theta$ give 
$A^{(N)}_{u,u}=\theta=\frac1{1-d}\longrightarrow1$.
\end{proof}

\begin{remark}
For the rescaled matrix $B^{(N)}=(d_--1)A^{(N)}$, the proposition gives
$B^{(N)}_{1,2}\to2$, whereas $A^{(N)}_{1,2}\to-2$ in our normalization.
\end{remark}

\begin{corollary}[Categorical Q-system existence]\label{ds:categorical-existence}
Assume the field and scalar hypotheses of Theorem~\ref{thm:reconstruction}.
Then the matrix in~\eqref{ds:matrix-formula} reconstructs a spherical
HI fusion category of Q-system type.
\end{corollary}

\begin{proof}
Apply Theorem~\ref{thm:reconstruction} to the matrix in
\eqref{ds:matrix-formula}. Its boundary is
$A_{g,0}=\delta_{g,0}+\theta$, and the complex parameters here satisfy
$a\ne\pm1$. Hence Theorem~\ref{thm:algebraic-q-system} gives a connected
separable algebra structure on $Q=\mathbf1\oplus\rho$, and the reconstruction
produces the stated spherical category. This is an algebraic conclusion; no
positivity or unitary structure is claimed.
\end{proof}

\begin{remark}[Pseudo-unitarity and the remaining unitarity question]
\label{ds:unitarity-expectation}
Every generalized near-group category is $\phi$-pseudounitary for a suitable
field automorphism $\phi$ \cite[Theorem~IV.3.6]{thornton2012generalized}.
In particular, for every odd $N>1$, the positive Galois conjugates of the
double-sine construction are pseudo-unitary untwisted HI categories of
Q-system type for $\mathbb Z/N$: their invertible simples have spherical
dimension $1$, and their noninvertible simples have spherical dimension
$d_+$, equal to their Frobenius--Perron dimension.

We expect every positive Galois conjugate of the double-sine family to be
unitary. Exact Hermiticity checks verify this for the computed positive
conjugates at every odd $3\leq N\leq31$.
The uniform assertion remains open: Galois conjugation need not preserve
a positive $*$-structure. In the present HI normalization, a
positive-branch category is unitary exactly when its matrix is Hermitian,
$A_{g,h}=\overline{A_{h,g}}$
\cite[Theorem~2(c)]{evans2017non}. Thus a Hermitian positive conjugate
would give unitary existence, as would a separate construction of a
unitary structure. 
\end{remark}


\section{Connected separable algebras and Morita duals}
\label{sec:separable}

We classify connected separable algebras up to Morita equivalence in a
fixed unitary odd cyclic HI category of Q-system type.  The representatives are
the subgroup algebras $R_H$ and the algebra $Q=\mathbf1\oplus\rho$.
The subgroup duals again have HI fusion rules, possibly with a noncyclic
invertible group or a nontrivial pointed associator.  The $Q$-dual has
$N+1$ simple objects and is not an HI category; rather, it generalizes the commutative $M$-$M$ system familiar from the Haagerup subfactor.

Throughout this section, let $G=\mathbb Z/N$ for $N\geq 3$ odd and let $\mathcal C$ be a fixed complex unitary cyclic HI
category of Q-system type with trivial pointed associator. Theorem~\ref{thm:algebraic-q-system} then gives a separable algebra structure on $Q=\mathbf 1\oplus\rho$. Alternatively, by \cite[Lemma~7.1]{izumi2018classification}, we can assume the separable algebra structure on $Q$, and derive from this the trivial pointed associator.
Recall that separability means that multiplication admits a bimodule
section.  In a fusion category this implies semisimplicity of the module
category \cite[Definition~7.8.29 and Proposition~7.8.30]{EGNO}.

If $B$ is a connected separable algebra in a fusion category, write
\(\operatorname{Mod}_{\mathcal C}(B)\) for its category of right modules.
We use reverse composition for module endofunctors, so that the opposite
appearing in the usual bimodule description is absorbed into the tensor
product convention.  Thus the dual category is
\begin{equation*}
 \mathcal C^*_{\operatorname{Mod}_{\mathcal C}(B)}
 \simeq {}_B\mathcal C_B
\end{equation*}
with the relative tensor product over $B$
\cite[Proposition~7.11.1 and Remark~7.12.5]{EGNO}.
Two algebras in $\mathcal C$ are Morita equivalent when their right-module
categories are equivalent as \(\mathcal C\)-module categories.  The
classification is therefore one of indecomposable $\mathcal C$-module
categories, not of algebra structures on a fixed object.

\subsection{Subgroup algebras}

The pointed part of $\mathcal C$ is the strict pointed category on the
objects $\alpha_g$, $g\in G$.  For a subgroup $H\leq G$, put
\begin{equation*}
 R_H=\bigoplus_{h\in H}\alpha_h.
\end{equation*}
The usual group-algebra multiplication makes $R_H$ a connected separable
algebra.  In the normalized strict model, a bimodule section of its
multiplication is given on the $\alpha_g$ summand by
\begin{equation*}
 \alpha_g\longmapsto
 \frac{1}{|H|}\sum_{h\in H}
 \alpha_h\otimes\alpha_{g-h}.
\end{equation*}
The connectedness follows from the single copy of the tensor unit.  In this untwisted pointed category, the cohomological classification of
algebra structures \cite[Remark 3.9(i)]{fuchs2004tft} shows that every connected separable structure on the
same underlying object is algebra-isomorphic to this one: since $H$ is
cyclic, $H^2(H,\mathbb C^\times)=0$.  

The simple right $R_H$-modules are indexed by the quotient $G/H$ and by
which of the two HI families they come from:
\begin{equation}
 U_{\bar g}=\alpha_g\otimes R_H,\qquad
 V_{\bar g}=\alpha_g\rho\otimes R_H,
 \qquad \bar g\in G/H.
 \label{sa:RH-modules}
\end{equation}
Consequently,
\begin{equation*}
 \operatorname{rank}\operatorname{Mod}_{\mathcal C}(R_H)=2[G:H].
\end{equation*}
Free-module adjunction and the HI fusion rules show that the objects in
\eqref{sa:RH-modules} are simple and pairwise nonisomorphic
\cite[Lemma~7.8.12 and Proposition~7.8.30]{EGNO}.
Separability makes every module a summand of a free module, so this list is
exhaustive.

\subsection{The subgroup dual and the carry class}
\label{sa:subgroup-dual}

Let $L_H=(G/H)\times\widehat H$ and $\widehat H=\operatorname{Hom}(H,\mathbb C^\times)$. The pointed simple bimodules are indexed by $(P,u)\in L_H$.  Their
underlying objects are sums over a quotient coset, and $u$ records the right
$H$-action.  Write $U_{P,u}$ for the invertible bimodule and $V_{P,u}$ for
the noninvertible one.  These account for all simple objects of
\({}_{R_H}\mathcal C_{R_H}\).  Writing
\begin{equation*}
 \sigma=V_{0,1},\qquad W_\ell=U_\ell\otimes_{R_H}\sigma,
 \quad \ell\in L_H,
\end{equation*}
their fusion rules are
\begin{align*}
 U_\ell\otimes W_r&\cong W_{\ell+r},
 &W_\ell\otimes U_r&\cong W_{\ell-r},
 \\
 W_\ell\otimes W_r&\cong U_{\ell-r}\oplus
 \bigoplus_{t\in L_H}W_t,
 &U_\ell\otimes U_r&\cong U_{\ell+r}.
\end{align*}
Thus the dual has HI fusion rules with pointed group $L_H$.  This
identifies the fusion rules; a strict HI-matrix presentation also requires a
trivial pointed associator, and that condition can fail. Choose a section $s:G/H\to G$ with $s(0)=0$, and define its carry by
\begin{equation*}
 c(P,Q)=s(P)+s(Q)-s(P+Q)\in H.
\end{equation*}
The associator on the pointed bimodules is represented by
\begin{equation}
 \omega_H((P,u),(Q,v),(R,w))=u\bigl(c(Q,R)\bigr).
 \label{sa:carry-cocycle}
\end{equation}
Changing the section changes \(\omega_H\) by a coboundary. Hence its
cohomology class, the \emph{carry class}, is intrinsic to the subgroup dual.

\begin{proposition}[Carry criterion]
\label{sa:carry-criterion}
Put $m=|H|$ and $q=[G:H]$.  The carry class of the pointed subcategory of
\({}_{R_H}\mathcal C_{R_H}\) is trivial if and only if $\gcd(m,q)=1$. Equivalently, the associator of the pointed subcategory of \({}_{R_H}\mathcal C_{R_H}\) is trivial if and only if $H$ is a Hall subgroup of $G$.
\end{proposition}

\begin{proof}
Identify the pointed group with $C_q\times C_m$ and use the standard
section of the quotient $C_N\to C_q$.  If
\(\zeta_m\) is a primitive $m$th root of unity, the cocycle in
\eqref{sa:carry-cocycle} has the form
\begin{equation*}
 \omega_H((i,r),(j,s),(k,t))
 =\zeta_m^{r\lfloor(j+k)/q\rfloor},
\end{equation*}
where first-coordinate entries are represented in
$\{0,\ldots,q-1\}$.  Let $f=\gcd(m,q)$ and
$L=\operatorname{lcm}(m,q)$.  The element $(1,1)$ has order $L$.
In the restriction to the cyclic subgroup it generates, $(j,j)$ means
$(j\bmod q,j\bmod m)$.  Hence the carry
$\lfloor((j\bmod q)+1)/q\rfloor$ is $1$ exactly $L/q$ times, and
\begin{equation*}
 \prod_{j=0}^{L-1}\omega_H((1,1),(j,j),(1,1))
 =\zeta_m^{L/q}=\exp(2\pi i/f).
\end{equation*}
For a $3$-coboundary the corresponding product telescopes to $1$.  Hence
the class is nontrivial when $f>1$.  When $f=1$, the extension
\(0\to C_m\to C_N\to C_q\to0\) splits, and a homomorphic section makes the carry zero.  This proves the criterion.
\end{proof}

For Hall subgroups, the subgroup dual admits an explicit matrix
description. The following result does not require unitarity or
the Q-system condition.

\begin{theorem}[Subgroup Fourier transform and Morita duality]
\label{sa:fourier-morita}
Let $A$ be a complex $(d_A,1)$-HI-matrix for
$G=\mathbb Z/N\mathbb Z$, where $N\geq3$ is odd.
Let $H\leq G$ be a Hall subgroup, write $G=H\oplus K$, and put
$L_H=K\times\widehat H$.
Define
\begin{equation*}
 (\mathscr P_H A)_{(P,u),(Q,v)}
 =\frac1{|H|}\sum_{x,y\in H}
 A_{P+x,Q+y}\,\overline{u(y/2)}\,v(x/2),
 \qquad P,Q\in K,\quad u,v\in\widehat H,
\end{equation*}
where division by $2$ is taken in $H$.
Then $\mathscr P_H A$ is a $(d_A,1)$-HI-matrix for the cyclic
group $L_H$. For the subgroup algebra
$R_H=\bigoplus_{h\in H}\alpha_h$ in $\mathcal C(A)$, there is
a tensor equivalence
\begin{equation*}
 \mathcal C(\mathscr P_H A)
 \simeq {}_{R_H}\mathcal C(A)_{R_H}.
\end{equation*}
\end{theorem}

\begin{proof}
The matrix assertion is
Theorem~\ref{sap:thm:partial-fourier-HI}.
The tensor equivalence and its simple-object correspondence are
Theorem~\ref{sap:thm:subgroup-fourier-morita}.
Both are proved in Appendix~\ref{app:separable}. With the bimodule labels of
Theorem~\ref{sap:thm:subgroup-fourier-morita}, this equivalence
sends the standard invertible and noninvertible objects labelled
$(P,u)$ to $U_{P,u^{-1}}$ and $V_{P,u}$, respectively.
\end{proof}

For the normalized Hermitian matrix $A$ of the fixed category
$\mathcal C$, the transformed matrix has $d=d_+$ and is normalized Hermitian by 
Corollary~\ref{sap:cor:partial-fourier-hermiticity}.

\begin{example}[A noncyclic dual at order nine]
\label{sa:n9-dual}
Take $G=\mathbb Z/9\mathbb Z$ and
\(H=3\mathbb Z/9\mathbb Z\cong C_3\).  Then
\begin{equation*}
 L_H\cong (G/H)\times\widehat H\cong C_3\times C_3,
 \qquad \gcd(|H|,[G:H])=3.
\end{equation*}
For the standard section $s(0)=0$, $s(1)=1$, $s(2)=2$, and a primitive
third root $\zeta_3$, the restriction of
\eqref{sa:carry-cocycle} to the generator $(1,1)$ has
\begin{equation*}
 \prod_{j=0}^{2}
 \omega_H((1,1),(j,j),(1,1))=\zeta_3.
\end{equation*}
This direct calculation detects a nontrivial pointed associator class.
Consequently the pointed subcategory of \({}_{R_H}\mathcal C_{R_H}\) is
\begin{equation*}
 \operatorname{Vec}_{C_3\times C_3}^{\,\omega_H}
\end{equation*}
with nontrivial associator.  Thus the dual ${}_{R_H}\mathcal C_{R_H}$ is twisted HI for the  noncyclic group $C_3\times C_3$.
\end{example}

Obviously there are several other examples. For example take as input either $N=25$ category in Table~\ref{tab:cyclic-census} and subgroup $H = 5 \mathbb Z/25$; then the dual is twisted HI for $C_5\times C_5$.

\subsection{The algebra on \texorpdfstring{$\mathbf 1\oplus\rho$}{1 plus rho}}
\label{sa:Q-dual}

The Q-system algebra $Q$ has a different dual.  Its right-module category
has $N+1$ simple objects.  There are $N$ modules in one regular pointed
orbit and one fixed module.  The corresponding dual category ${}_Q\mathcal C_Q$ has simple objects
\begin{equation*}
 \mathbf 1,\quad \eta,\quad
 \mu_1,\ldots,\mu_{(N-1)/2},\quad
 \nu_1,\ldots,\nu_{(N-1)/2}.
\end{equation*}
The object $\nu_i$ is attached to the pair \(\{g_i,-g_i\}\), while the
ordering of the $\mu_i$ is immaterial.  Their Frobenius--Perron dimensions
are
\begin{equation*}
 \operatorname{FPdim}(\eta)=d,\qquad
 \operatorname{FPdim}(\mu_i)=d-1,\qquad
 \operatorname{FPdim}(\nu_i)=d+1.
\end{equation*}
The self-dual object $\eta$ tensor-generates ${}_Q\mathcal C_Q$.  With columns
recording multiplication by $\eta$, its fusion matrix is
\begin{equation}
 N_\eta=
 \begin{pmatrix}
 0&1&0&0\\
 1&1&\boldsymbol 1_r^{\top}&\boldsymbol 1_r^{\top}\\
 0&\boldsymbol 1_r&J_r-I_r&J_r\\
 0&\boldsymbol 1_r&J_r&J_r+I_r
 \end{pmatrix},
 \qquad r=\frac{N-1}{2}.
 \label{sa:Q-dual-matrix}
\end{equation}
Here, $\mathbf 1_r=(1,\ldots,1)$, and $I_r$ resp.\ $J_r$ are the identity resp.\ all ones $r\times r$ matrices.
In particular, $\mathbf 1$ is the only invertible simple object, and
\({}_Q\mathcal C_Q\) does not have HI fusion rules. For $N=3$ this recovers of course the commutative $M$-$M$ system of the Haagerup subfactor.

For $U_g=\alpha_g\otimes Q$ and $W_g=\alpha_g\rho\otimes Q$, free-module
adjunction and semisimplicity give
$W_g\cong U_g\oplus X_g$, where all $X_g$ are isomorphic to one simple
module $X$; separability makes this list exhaustive.  The invertibles
translate the $U_g$ and fix $X$, giving the two pointed orbits.  The full
identification of the dual, including the tensor product and the matrix
\eqref{sa:Q-dual-matrix}, is proved in Appendix~\ref{app:separable} using
the subfactor restriction calculation and its multiplicity-one statement
\cite[Section~7, especially Proposition~7.4 and Theorem~7.5]{izumi2018classification}.

\subsection{Classification and the Morita consequence}

The subgroup algebras and $Q$ exhaust the Morita classes.

\begin{theorem}[Connected separable algebras]
\label{sa:classification}
Let $\mathcal C$ be a complex unitary cyclic untwisted HI category of Q-system type.  Every connected
separable algebra in $\mathcal C$ is Morita equivalent to exactly one of
\begin{equation*}
 \{R_H:H\leq G\}\ \sqcup\ \{Q\}.
\end{equation*}
There are $\tau(N)+1$ indecomposable Morita classes.  Their module-category
ranks are
\begin{equation}
 \operatorname{rank}\operatorname{Mod}_{\mathcal C}(R_H)=2[G:H],
 \qquad
 \operatorname{rank}\operatorname{Mod}_{\mathcal C}(Q)=N+1.
 \label{sa:classification-ranks}
\end{equation}
\end{theorem}

We prove the theorem in Appendix~\ref{app:separable}.  The invertible
objects have exactly two orbits on the simple objects of an indecomposable
module category.  Passing to a subgroup dual reduces an internal End to
$\mathbf1$ or $\mathbf1\oplus\sigma$.  The first case gives a subgroup
algebra.  In the second, the Hall-subgroup obstruction forces the subgroup
to be trivial, leaving the class of $Q$.  Finally, the ranks in
\eqref{sa:classification-ranks} distinguish all these classes.

\begin{proposition}[Morita rigidity for the specified Q-system algebras]
\label{sa:morita-rigidity}
Let $\mathcal C_1$ and $\mathcal C_2$ be  
HI categories satisfying the hypotheses of Theorem~\ref{sa:classification}. 
If $\mathcal C_1$ and $\mathcal C_2$ are
categorically Morita equivalent, then they are tensor equivalent.
\end{proposition}

\begin{proof}
A categorical Morita equivalence is represented by an indecomposable
module category over $\mathcal C_1$.  By Theorem~\ref{sa:classification}, it is
equivalent to the module category of $R_H$ for some $H\leq G$, or of $Q$.
Equivalent module categories have tensor-equivalent dual categories, so
$\mathcal C_2$ must be one of these duals.

The $Q$-dual has $N+1$ simple objects and only one invertible simple,
whereas a cyclic HI category has $2N$ simples and $N$ invertibles.
Since $N\geq3$, this excludes the $Q$-dual.
For a subgroup algebra, Theorem~\ref{sa:carry-criterion} shows that a non-Hall
subgroup gives a nontrivial pointed associator.  A tensor equivalence
restricts to an equivalence of pointed subcategories, so this case cannot
give $\mathcal C_2$.

It remains to consider Hall subgroups.  The Hall obstruction in
Appendix~\ref{app:separable} proves that, for every nontrivial Hall $H$,
including $H=G$, the dual has no connected separable algebra on
$\mathbf1\oplus W_\ell$ for any noninvertible simple $W_\ell$.
A tensor equivalence would transport the specified algebra of
$\mathcal C_2$ to just such an algebra.  Thus $H=\{0\}$, whose subgroup
algebra is the tensor unit and whose dual is $\mathcal C_1$ itself.
Hence $\mathcal C_1$ and $\mathcal C_2$ are tensor equivalent.
\end{proof}

In Section~\ref{ce:n15-application}, we apply
Proposition~\ref{sa:morita-rigidity} to three tensor-inequivalent order-fifteen
categories whose centers have the same modular data.


\section{Modular data of the center}
\label{sec:centers}

In this section we determine the modular data $S,T$ of the center of any HI category $\mathcal C$ (not necessarily unitary) over $\mathbb C$. This is computed in \cite[Section~6.3]{evans2017non}, building on \cite[Theorem~8.4]{IzumiLR2}. Some of these entries are manifestly nice (we call them the \emph{tractable} ones), and we recall them next subsection. The expressions for the remainder are much more complicated. However it was observed in \cite{evans2011exoticness} that the `intractable' entries are also quite simple, and in \cite[Conjecture 1]{evans2017non} suggested that this modular data always takes the form of a smashed sum of metric groups. We prove this conjecture here.

\par We use the normalization $S_{0_+,0_+}=1/\mathrm{dim}(\mathcal C)$, where $0_+$ (see below) is the tensor unit of $Z(\mathcal C)$. For brevity, set
\begin{equation*}
 \mu=N^2+4,\qquad m=\frac{1}{2}(\mu-1),\qquad
 r_\sigma=\sigma\sqrt{\mu},\qquad
 d=\frac{1}{2}(N+r_\sigma),\qquad\text{ and }\qquad\sigma\in\{+1,-1\}.
\end{equation*}
Thus $d-d^{-1}=N$, and global dimension $\operatorname{dim}(\mathcal C)=N(1+d^2)=Ndr_\sigma$.

\subsection{Tractable center data}
\label{ce:known-center-data}

Define
\begin{equation*}
 \Phi_G=\bigl((G\times\widehat G)\setminus\{(0,1)\}\bigr)/
       \bigl((g,\chi)\sim(-g,\chi^{-1})\bigr)\qquad\text{ and }\qquad
 R=\bigoplus_{g\in G}\alpha_g\rho .
\end{equation*}
The simple objects of $\mathcal{Z}(\mathcal{C})$ are $0_+$, $0_-$,
$\mathfrak a_{[g,\chi]}$ for $[g,\chi]\in\Phi_G$, and
$\mathfrak d_i$ for $1\leq i\leq(\mu-1)/2$.  Their restrictions under the
forgetful functor, dimensions, and tractable ribbon twists are
\begin{center}
\begin{tabular}{llll}
\toprule
Simple & Restriction & Dimension & Twist\\
\midrule
$0_+$ & $\mathbf1$ & $1$ & $1$ \\
$0_-$ & $\mathbf1\oplus R$ & $d^2$ & $1$\\
$\mathfrak a_{[g,\chi]}$ & $\alpha_g\oplus\alpha_{-g}\oplus R$ & $1+d^2$ & $\chi(g)$ \\
$\mathfrak d_i$ & $R$ & $Nd$ & ? \\
\bottomrule
\end{tabular}
\end{center}
The tractable portion of the $S$-matrix is $S_{\mathfrak a_{[g,\chi]},\mathfrak d_i}=0$, with nonzero entries
\begin{align*}
 S_{0_+,0_+}=S_{0_-,0_-}&=\frac{d^{-1}}{Nr_\sigma},
 &S_{0_+,0_-}&=\frac d{Nr_\sigma},\\
 S_{0_\pm,\mathfrak a_{[g,\chi]}}&=\frac1N,
 &S_{\mathfrak a_{[g,\chi]},\mathfrak a_{[h,\psi]}}
   &=\frac{\chi(h)\psi(g)+(\chi(h)\psi(g))^{-1}}N,\\
 S_{0_+,\mathfrak d_i}&=\frac1{r_\sigma},
 &S_{0_-,\mathfrak d_i}&=-\frac1{r_\sigma},
\end{align*}
together with the entries coming from $S=S^{\top}$. This data was determined in \cite[equations~(6.5) and (6.24)]
{evans2017non}. 
The remaining (less tractable) entries are \begin{equation*}T_{\mathfrak d_i}=:w_i\qquad \mathrm{and}\qquad S_{\mathfrak d_i,\mathfrak d_j}=:K_{i,j}=K_{j,i}.\end{equation*}
We see that $0_\pm$ and $\mathfrak a_{[g,\chi]}$ are all self-dual; write $(\mathfrak d_i)^\vee=\mathfrak d_{i^\vee}$. All fusion coefficients except $N_{\mathfrak d_i,\mathfrak d_j}^{\mathfrak d_k}=:N_{i,j}^k$ can now be computed from Verlinde: e.g.\ letting $\Sigma$ denote the sum of all simples with multiplicity 1, we obtain $0_-0_-=\Sigma$ and $0_-\mathfrak d_i=-0_+-\mathfrak d_i+\Sigma$.

\subsection{The modular-data theorem}\label{secthm}

For an abelian group $H$, 
a quadratic form $q_H:H\to\mathbb Q/\mathbb Z$ below is homogeneous:
$q_H(nh)=n^2q_H(h)$, and its associated bilinear form
$B_H(h,k)=q_H(h+k)-q_H(h)-q_H(k)$ is biadditive.
It is nondegenerate if $h\mapsto B_H(h,-)$ identifies $H$ with 
$\widehat{H}$, viewed additively in $\mathbb Q/\mathbb Z$. A metric group is a pair $(H,q_H)$ with nondegenerate $q_H$.

\begin{theorem}\label{ce:metric-data}
Recall the simple-objects of Section~\ref{ce:known-center-data}. Then there exists a metric group $(H,q_H)$ of order $\mu=N^2+4$ and an indexing of the simples $\mathfrak{d}_i$ by $[h]\in(H\setminus\{0\})/\{\pm1\}$ such that
\begin{align}\label{ce:general-metric}
 T_{\mathfrak d_{[h]}}=e^{2\pi i q_H(h)},\qquad
 S_{\mathfrak d_{[h]},\mathfrak d_{[k]}}
 =-\frac{1}{r_\sigma}(e^{2\pi iB_H(h,k)}+e^{-2\pi iB_H(h,k)}),
\end{align}
and $\sum_{h\in H}e^{2\pi i q_H(h)}=-r_\sigma$.
In other words, $S$ and $T$ form the smashed-sum $\mathcal{MD}_\sigma(G;H,q_H)$ \cite[Section~6.3]{evans2017non}, recalled in Appendix~\ref{ce:compact-MD}.
\end{theorem}

We  recover addition in $H$ up to sign from a fusion subalgebra coming from
 the $\mathfrak d_i$. Symmetry $S=S^{\top}$ then identifies its character pairing, and $(ST)^3=C$ determines the quadratic
form.

Denote $I=\{1,\ldots,m\}$, the indexing set of the $\mathfrak d_i$. Write $S$ in block form, using the entries given in the previous subsection. Then the modular relations $S^2=C=(ST)^3$, with $C$ the duality permutation, give
\begin{equation}\label{ce:block-identities}
 \mathbf 1 K=r_\sigma^{-1}\mathbf 1 ,\qquad
 K^2=C_D-2J/\mu,\qquad 1+2\sum_iT_{\mathfrak d_i}=-r_\sigma,
\end{equation}
where $C_D$ is the restriction of $C$ to $I$, $\mathbf 1=(1,\ldots,1)$  and
$J=\mathbf 1^\top\mathbf 1$.
In particular, from the $K^2$ expression we see that $K$ is invertible.

The fusion rules of $Z(\mathcal C)$ are mostly 1's. But assuming the smashed sum formulas, they can be simplified: define $c_{[g]}=0_++{0_-}-[g]$, and we find $c_{[g]}c_{[h]}=c_{[g+h]}+c_{[g-h]}$. Motivated by this, in the fusion ring of $Z(\mathcal C)$ define $c_i=0_++0_--\mathfrak d_i$ and write 1 for the identity $0_+$. Then from the fusion rules given last section, we obtain
\begin{equation}\label{ce:integral-products}
 c_ic_j=2\delta_{i,j^\vee}+\sum_k a_{ij}^kc_k,\qquad
 a_{ij}^k:=1-N_{ij}^k\leq1,\qquad
 \sum_k a_{ij}^k=2-\delta_{i,j^\vee}.
\end{equation}

\begin{lemma}\label{ce:integral-algebra}
Every $\mathfrak d_i$ is self-dual. There is a permutation $\pi$ of $I$
which acts as an algebra automorphism on the functions $c_i$, with
\begin{equation}\label{ce:product-rules}
 c_i^2=2+c_{\pi(i)},\qquad c_ic_j=c_p+c_q\quad(i\ne j),
\end{equation}
where $p,q\in I$ are distinct.
\end{lemma}
\begin{proof}
We can interpret the $c_i$ as functions $x_i:I\to\mathbb C$ by \begin{equation}x_i(j)=1+\frac{S_{0_-,\mathfrak d_j}}{S_{0_+,\mathfrak d_j}}-\frac{S_{\mathfrak d_i,\mathfrak d_j}}{S_{0_+,\mathfrak d_j}}.\end{equation}
Invertibility of $K$ gives linear independence of these $x_i$. We see from  \eqref{ce:block-identities} that $\sum_ix_i=-1$. Verlinde's formula implies we can read off $x_ix_j$ from \eqref{ce:integral-products}, and in particular the $x_i$ span a unital integral algebra $\mathscr O$.
Its rational algebra $\mathscr O_{\mathbb Q}$ is finite \'etale (i.e.\ a product of number fields), with
trace pairing
\begin{equation*}
 \operatorname{Tr}_{\mathscr O_{\mathbb Q}}(x_ix_j)
 =\mu\delta_{i,j^\vee}-2. 
\end{equation*}
In particular, this pairing is invertible modulo $2$, which implies prime ideals above 2 are unramified.

The $S$ entries for any modular category lie in a cyclotomic field. Choosing any automorphism $\gamma$ of that field,  Galois symmetry gives
\begin{equation}\label{ce:Galois}
 \gamma(S_{X,Y})=\epsilon_\gamma(X)S_{\widehat\gamma(X),Y}
 =\epsilon_\gamma(Y)S_{X,\widehat\gamma(Y)}
\end{equation}
with $\epsilon_\gamma(X)\in\{\pm1\}$ \cite{coste1994},
\cite[Theorem~10.1]{etingof2005fusion}.
Applying it to $S_{0_+,\mathfrak d_i}$, we see $\gamma(\mathfrak d_i)=\mathfrak d_{\hat{\gamma}(i)}$ and $\epsilon_\gamma(\mathfrak d_i)=r_\sigma/\gamma(r_\sigma)$, and hence 
$\gamma(x_i(j))=x_{\widehat\gamma(i)}(j)$. 

Choose a prime $\mathfrak p$ above $2$ in the cyclotomic field  and a decomposition-group
lift $\gamma_2$ of residue Frobenius
\cite[Corollary~7.59 and Proposition~8.10]{milneANT}.
This lift exists even when $2$ ramifies. Each $x_i(h)$ is integral,
being an evaluation of the finite unital algebra $\mathscr O$.
If $\pi=\widehat\gamma_2|_I$, then
$z=x_i^2-x_{\pi(i)}$ vanishes at every $h$ modulo $\mathfrak p$.
Consequently $\operatorname{Tr}_{\mathscr O_{\mathbb Q}}(zx_j)$ is
an even integer for every $j$. The trace pairing gives
$x_i^2\equiv x_{\pi(i)}\pmod{2\mathscr O}$.

Put $\eta_i=\delta_{i,i^\vee}$. In the basis $x_k$, the coefficients
of $x_i^2$ are $b_k=a_{ii}^k-2\eta_i$; thus
$b_k\leq1-2\eta_i$ and $\sum_kb_k=2-(2m+1)\eta_i$.
Exactly the coefficient at $\pi(i)$ is odd. If $\eta_i=0$, the even
coefficients are nonpositive and the odd coefficient is at most $1$,
contradicting their sum $2$. Thus $\eta_i=1$. Now the even coefficients
are at most $-2$ and the odd one at most $-1$; their sum $1-2m$
forces equality throughout.
The permutation $\pi$ acts by a Galois automorphism.
For $i\ne j$, comparing coefficients of $1$ in the basis
$1,c_1,\ldots,c_m$ in $(c_ic_j)^2=c_i^2c_j^2$ gives
$2\sum_k(a_{ij}^k)^2=4$.
Together with \eqref{ce:integral-products}, this says
$\sum_k(a_{ij}^k)^2=\sum_ka_{ij}^k=2$. Exactly two coefficients are
$1$ and the rest are zero, proving \eqref{ce:product-rules}.
\end{proof}

\begin{lemma}\label{ce:kummer-reconstruction}
Let $\star$ be a new label. There is an abelian group $H$ of order
$\mu$ and an identification
$\{\star\}\sqcup I\cong H/\{\pm1\}$ taking $\star$ to $[0]$, for which
$c_{[h]}\mapsto e_h+e_{-h}$ identifies the algebra spanned by $1,c_i$
with $\mathbb C[H]^{\{\pm1\}}$.
\end{lemma}

\begin{proof}
Put $c_\star=2$ and define a two-valued product by
$a*b=\{p,q\}$ when $c_ac_b=c_p+c_q$. Linear independence of the
$c_a$ makes the output multiset unique. Associativity of the algebra
identifies the four-element multisets obtained from the two bracketings of
every triple product. Thus $*$ is associative in the sense of
\cite[Definition~1.1]{buchstaber2022classification}. It is commutative,
$\star$ is a strong identity, and the square and distinct-product rules
show that a product contains $\star$ precisely when its two inputs agree.
Every label is therefore its own unique inverse, so the two-valued group is
involutive in the sense of
\cite[Definition~1.4]{buchstaber2022classification}. It is finitely
generated because it is finite.

There is no nonidentity strong involution. Indeed, if $a\ne\star$, then
$a*a=\{\star,\delta(a)\}$ with $\delta(a)\ne\star$. The classification
\cite[Theorem~1.10]{buchstaber2022classification} leaves the principal,
unipotent, and special families, together with the Boolean factors described
there. Every nonprincipal case has a nonidentity strong involution. In the
unipotent model $(V\times V)/\iota$, where
$\iota(u,v)=(u,u+v)$ and $V\ne0$ is Boolean, the class of $(0,v)$ for
$v\ne0$ squares to the identity twice. In the special model the
distinguished extra element $s$ satisfies $s*s=\{\star,\star\}$. A
nontrivial Boolean direct factor also gives such an element. Hence all
nonprincipal cases are excluded.

We therefore have a two-valued isomorphism
$\{\star\}\sqcup I\cong H/\{\pm1\}$ for an abelian group $H$. The quotient is finite
and its fibres have size at most two, so $H$ is finite. Any nonzero element
of $H[2]$ would give a nonidentity strong involution. Thus $H$ has odd
order, and counting sign orbits gives $|H|=2m+1=\mu$.

Finally,
\begin{equation*}
 (e_h+e_{-h})(e_k+e_{-k})
 =(e_{h+k}+e_{-h-k})+(e_{h-k}+e_{-h+k}).
\end{equation*}
This agrees with the prescribed products, including $c_\star=2$, and the
displayed images form a basis of $\mathbb C[H]^{\{\pm1\}}$.
\end{proof}

\begin{lemma}[Prescribed lift, unique up to a global sign]
\label{ce:prescribed-lift}
Let $H_1,H_2$ be finite abelian groups of odd order. Every two-valued
isomorphism
\begin{equation*}
 f:H_1/\{\pm1\}\longrightarrow H_2/\{\pm1\}
\end{equation*}
is induced by a group isomorphism $F:H_1\to H_2$. If $H_1$ is nontrivial,
the two lifts are $F$ and $-F$.
\end{lemma}

\begin{proof}
The trivial case is immediate. Choose $0\ne g\in H_1$ and a representative
$0\ne b\in H_2$ of $f([g])$. For each $x\in H_1$, there is a unique
representative $F(x)\in f([x])$ such that
\begin{equation}\label{ce:lift-orientation}
 f([x+g])=[F(x)+b].
\end{equation}
Existence follows by applying $f$ to the product of $[x]$ and $[g]$:
changing a representative $u$ to $-u$ exchanges $[u+b]$ and $[u-b]$.
For uniqueness, $[u+b]=[-u+b]$ implies either $2u=0$ or $2b=0$.
The latter is impossible, and the former gives $u=0$, when there was only
one representative.

Preservation of products and cancellation in two-element multisets now give,
by induction in both directions,
\begin{equation}\label{ce:lift-sequence}
 f([x+ng])=[F(x)+nb]\qquad(n\in\mathbb Z).
\end{equation}
For the induction step, use
$[x+ng]*[g]=\{[x+(n+1)g],[x+(n-1)g]\}$.

Apply $f$ to the product of $[x+ng]$ and $[y+mg]$ and use
\eqref{ce:lift-sequence}. For all integers $n,m$ this gives
\begin{align*}
 &\{[F(x)+F(y)+(n+m)b],[F(x)-F(y)+(n-m)b]\}\\
 &\hspace{12mm}=
 \{[F(x+y)+(n+m)b],[F(x-y)+(n-m)b]\}.
\end{align*}
Let $t$ be the order of $b$. It is odd and at least three, so $n+m$ and
$n-m$ vary independently modulo $t$. Fix $u=n+m$. If the first entries
on the two sides were different, equality of multisets would force
$[F(x)-F(y)+vb]$ to be constant for all $v=n-m$ modulo $t$. This is
impossible: these are the images of $t\ge3$ distinct elements, whereas a
sign orbit contains at most two elements. Consequently
\begin{equation*}
 [F(x)+F(y)+ub]=[F(x+y)+ub]\qquad(u\in\mathbb Z).
\end{equation*}
Taking $u=0$ and $u=1$, and using the uniqueness argument from
\eqref{ce:lift-orientation}, proves $F(x+y)=F(x)+F(y)$.

The homomorphism $F$ realizes $f$. Its kernel is trivial because $f$
preserves the identity orbit, and
$|H_1|=2|H_1/\{\pm1\}|-1=|H_2|$, so it is an isomorphism. Any lift sends
$g$ to $b$ or $-b$, and \eqref{ce:lift-orientation} then determines it
everywhere. The only lifts are therefore $F$ and $-F$.
\end{proof}

We now have all the ingredients needed for proving Theorem~\ref{ce:metric-data}.

\begin{proof}[Proof of Theorem~\ref{ce:metric-data}]
\emph{The pairing in the prescribed labeling.}
The characters of $\mathbb C[H]^{\{\pm1\}}$ are the restrictions of
the characters of $H$, modulo inversion. Each evaluation label $[k]$
therefore gives a unique $[\chi_k]\in\widehat H/\{\pm1\}$ with
$c_{[h]}([k])=\chi_k(h)+\chi_k(h)^{-1}$.
This is a bijection of the quotient sets taking $[0]$ to $[1]$.
Symmetry of $K$, followed by
the product rule, gives
\begin{equation*}
 (\chi_k(h)+\chi_k(h)^{-1})(\chi_l(h)+\chi_l(h)^{-1})
 =c_{[k+l]}([h])+c_{[k-l]}([h]).
\end{equation*}
Independence of characters, with multiplicities, shows that
$[k]\mapsto[\chi_k]$ preserves the sum-and-difference operation.
Lemma~\ref{ce:prescribed-lift} therefore lifts this specified quotient map
to an isomorphism $\Phi:H\to\widehat H$ in the original evaluation labels,
not merely to an abstract isomorphism of the two groups.
Write the resulting nondegenerate bicharacter as $\beta(h,k)$.
Then $c_{[h]}([k])=\beta(h,k)+\beta(h,k)^{-1}$.

By symmetry and independence of characters, for each $h$ either
$\beta(h,-)=\beta(-,h)$ or $\beta(h,-)=\beta(-,h)^{-1}$.
The sets of $h$ satisfying the two alternatives are subgroups covering
$H$; one must be all of $H$. In the skew case oddness gives
$\beta(h,h)=1$, so every diagonal entry of $K$ is $-2/r_\sigma$.
Character orthogonality gives
\[
\sum_{[g,\chi]\in\Phi_G}
S_{\mathfrak a_{[g,\chi]},\mathfrak a_{[g,\chi]}}
=\frac{N-1}{N}.
\]
Consequently,
\[
\operatorname{Tr}(S)
=\frac{2}{Nr_\sigma d}+\frac{N-1}{N}-\frac{2m}{r_\sigma}
=1-r_\sigma,
\]
which is irrational. But Lemma~\ref{ce:integral-algebra} gives
$S^2=1$, so $\operatorname{Tr}(S)\in\mathbb Z$.
This contradiction proves that $\beta$ is symmetric.

\smallskip\noindent\emph{The ribbon twists.}
Write $w_h=T_{\mathfrak d_{[h]}}$, $w_0=1$, and $v_h=w_h^{-1}$.
Ribbon duality makes these functions even. By Vafa's theorem
\cite{vafa1988toward}, in the modular-category formulation
\cite[Theorem~3.1.19]{bakalov2001lectures}, ribbon twists have finite order;
hence these functions have unit modulus. The entries of
$STS=T^{-1}ST^{-1}$ give
\begin{equation}\label{ce:quadratic-relations}
 \begin{gathered}
 \sum_h w_h=-r_\sigma,\qquad
 \sum_h w_h\beta(h,k)=-r_\sigma v_k,\\
 v_{h+k}+v_{h-k}=v_hv_k\bigl(\beta(h,k)+\beta(h,k)^{-1}\bigr).
 \end{gathered}
\end{equation}
For the last identity, expand the two character sums in the remaining
block. The tractable-column contribution is $2/\mu$, exactly the contribution
of $h=0$ after the sums over sign pairs are replaced by sums over $H$.

The sum $\beta(h,k)+\beta(h,k)^{-1}$ is nonzero since its factors have
odd order. Two unit-modulus numbers with nonzero sum $z$ have product
$z/\overline z$. Applying this to both sides of the last identity gives
$v_{h+k}v_{h-k}=v_h^2v_k^2$. Induction gives $v_{nh}=v_h^{n^2}$.
If $h$ has odd order $t$, compare $n=(t+1)/2$ and $(t-1)/2$, using
evenness, to obtain $v_h^t=1$.

Let $M$ be the exponent of $H$ and write $v_h=e^{2\pi iQ(h)/M}$.
The function $Q:H\to\mathbb Z/M$ satisfies the parallelogram identity.
Since $2$ is invertible, its polarization is
\begin{equation*}
 P(h,k)=Q(h+k)-Q(h)-Q(k)=\frac{Q(h+k)-Q(h-k)}2.
\end{equation*}
It is symmetric and satisfies
$P(h+l,k)+P(h-l,k)=2P(h,k)$ and $P(2h,k)=2P(h,k)$.
Substituting $h=(a+b)/2$, $l=(a-b)/2$ proves additivity.
Thus $b_v(h,k)=v_{h+k}/(v_hv_k)$ is a symmetric bicharacter.
The last identity in \eqref{ce:quadratic-relations} now gives
$b_v+b_v^{-1}=\beta+\beta^{-1}$. Character independence and the same
two-subgroup argument imply $b_v=\beta$ or $\beta^{-1}$ globally.
Consequently the polarization of $w$ is nondegenerate and has the
same symmetrized kernel as $\beta$. Taking $q_H(h)=-Q(h)/M$ proves
\eqref{ce:general-metric} and its Gauss-sum normalization.
\end{proof}

\subsection{Square-free discriminant}

\begin{corollary}\label{ce:squarefree-discriminant}
If $\mu$ is square-free, then $H\cong\mathbb Z/\mu$ and
$q_H(h)=s h^2/\mu\pmod{\mathbb Z}$ for a unit $s$ satisfying
$\left(\frac{s}{\mu}\right)=-\sigma$ (the Jacobi symbol).
The coefficient $s$ is determined up to multiplication by a unit square.
\end{corollary}
\begin{proof}
An abelian group of square-free order is cyclic, and a nondegenerate
homogeneous form on an odd cyclic group has the stated unit coefficient. 
Every prime $p\mid\mu$ is $1$ modulo $4$, since $(N/2)^2\equiv-1\pmod p$. 
The Chinese remainder theorem and the prime Gauss-sum formula
\cite[Theorem~1.1]{murtyPathakGauss} give
$\sum_h e^{2\pi i s h^2/\mu}=(s/\mu)\sqrt\mu$;
the cross-factors cancel by quadratic reciprocity. Apply
Theorem~\ref{ce:metric-data}. Automorphisms of the cyclic group change
$s$ by a unit square.
\end{proof}

As is well-known, $8/\pi^2\approx 81\%$ of all odd numbers are square-free.

\begin{corollary}\label{ce:prime-discriminant}
If $\ell=\mu$ is prime, the spherical sign determines $(H,q_H)$ up to
isometry. With $\zeta=e^{2\pi i/\ell}$ and $(\tau/\ell)=-\sigma$,
\begin{align}
 T_{\mathfrak d_{[x]}}&=\zeta^{\tau x^2},\label{eq:ce-prime-twist}\\
 S_{\mathfrak d_{[x]},\mathfrak d_{[y]}}
 &=-r_\sigma^{-1}(\zeta^{2\tau xy}+\zeta^{-2\tau xy}).\label{eq:ce-prime-S}
\end{align}
\end{corollary}
Indeed, the Legendre symbol determines the square class of $\tau$.
For square-free $\mu$ with $t$ prime factors, the sign condition instead
permits $2^{t-1}$ isometry classes; it does not assert that all are realized.

\subsection{Recovering the metric group from the HI matrix}\label{secrecover}

The metric group can be determined without computing the remaining
$S$-block. For the defining HI matrix $A$, consider the operator $D_A$ on the
basis $C_\rho,F_\rho,X_{g,h}$ ($g,h\in G$) given by
\begin{align*}
 D_AC_\rho&=dF_\rho,&
 D_AF_\rho&=d^{-2}C_\rho+d^{-1}\sum_uX_{u,0},\\
 D_AX_{g,h}&=\delta_{g,0}d^{-1}C_\rho+
                 \sum_uA_{h+u,2g}X_{u,g}.
\end{align*}
It is left multiplication by $dF_\rho$ in the $\rho$ tube corner
\cite[equation~(6.2)]{evans2017non}.

Write $u_k=[C_\rho]D_A^kC_\rho$ for the coefficient of $C_\rho$ and
$\Theta_k(A)=\mathcal G_A(k)/r_\sigma$ for the normalized Gauss sum below.

\begin{proposition}\label{ce:metric-recognition}\label{ce:return-coefficient}
For every $k\geq0$,
\begin{equation}\label{ce:twist-trace}
 \mathcal G_A(k):=\sum_{h\in H}e^{2\pi i kq_H(h)}
 =2\operatorname{Tr}(D_A^k)-N|G[k]|,
 \qquad \mathcal G_A(1)=-r_\sigma,
\end{equation}
\begin{equation}\label{ce:return-identity}
 \Theta_k(A)=2d u_k-|G[k]|.
\end{equation}
The values $\Theta_{p^t}(A)$, for $p\mid\mu$ and
$1\leq t\leq v_p(\mu)$, determine $(H,q_H)$ up to isometry.
If $5\mid\mu$, then
\begin{equation}\label{ce:fifth-invariant}
 \Theta_5(A)=\Xi(A):=
 2\sum_{g,h\in G}A_{g,2h}A_{h,2g}
       -4\sum_{g\in G}A_{g,0}-1.
\end{equation}
\end{proposition}
\begin{proof}
By the forgetful functor table, the multiplicity of $\rho$ is zero in $0_+$ and
one in every other simple in $Z(\mathcal C)$. Thus the corner is commutative
semisimple, and its regular representation counts each nonunit label once.
For the normalized primitive idempotent
$e_Z=t_Z(C_\rho+d\alpha_ZF_\rho+d\sum\beta_{g,h}X_{g,h})$,
the actual twist is $\alpha_Z/d+\sum_h\beta_{0,h}$
\cite[equations~(6.4) and~(6.22)]{evans2017non}.
The coefficient of $C_\rho$ in $dF_\rho e_Z$ is therefore $t_ZT_Z$,
so the eigenvalue of $D_A$ is $T_Z$, not its inverse.
Since $\sum_{g,\chi}\chi(g)^k=N|G[k]|$, the $\mathfrak a$ sign orbits
contribute $(N|G[k]|-1)/2$ to the corner trace. The label $0_-$ contributes
$1$, and the $\mathfrak d$ sign orbits contribute
$(\mathcal G_A(k)-1)/2$. Their sum proves \eqref{ce:twist-trace}.

The coefficients of $C_\rho$ in the primitive idempotents for
$0_-$, $\mathfrak a$ and $\mathfrak d$ are, respectively,
$1/(Nr_\sigma)$, $1/(Nd)$ and $1/(dr_\sigma)$
\cite[equations~(6.7), (6.9), (6.11) and (6.13)]{evans2017non}.
The same decomposition therefore gives
\begin{equation*}
 u_k=\frac1{Nr_\sigma}
 +\frac{N|G[k]|-1}{2Nd}
 +\frac{\mathcal G_A(k)-1}{2dr_\sigma}.
\end{equation*}
Since $2d-r_\sigma=N$, this is \eqref{ce:return-identity}.
Directly iterating the displayed operator through degree five gives
\begin{equation*}
 u_5=2d^{-2}+d^{-1}\sum_{g,h}A_{g,2h}A_{h,2g}.
\end{equation*}
Now $5\mid\mu$ implies $5\nmid N$, hence $|G[5]|=1$.
Substitute $d^{-1}=-\sum_gA_{g,0}$ from (E2).

The arithmetic recovery from the prime-power Gauss sums is proved in
Appendix~\ref{ce:recognition-proof}.
\end{proof}

Appendix~\ref{ce:return-recurrence} computes all the coefficients $u_k$
by contractions of arrays indexed by $G\times G$.
Their arithmetic meaning follows from the Gauss-sum norm identity:
\begin{equation*}
 |\Theta_{p^t}(A)|^2=|H[p^t]|\qquad(p\mid\mu).
\end{equation*}
Thus their absolute values determine the primary decomposition of $H$;
their signs determine the quadratic form, as detailed in
Appendix~\ref{ce:recognition-proof}.

When $\mu$ is square-free and $q_H(x)=s x^2/\mu$, the local coefficients
are determined by
\begin{equation}\label{ce:squarefree-coefficient}
 \left(\frac{s}{p}\right)=-\frac{\Theta_p(A)}{\sqrt p}
 \qquad(p\mid\mu).
\end{equation}
Indeed, the $p$-primary generator has coefficient $s\mu/p$.
The ratio of its Gauss sums at $p$ and $1$ is
$\sqrt p(s\mu/p\mid p)$; the other primary factors contribute
$(\mu/p\mid p)$ by quadratic reciprocity. Since
$\mathcal G_A(1)=-r_\sigma$, their product gives the formula.
For $N=41$, we have $\mu=5\cdot337$. By \eqref{ce:squarefree-coefficient}, the quadratic expression
$\Xi(A)$ determines $(s/5)$, and
$(s/337)=-\sigma(s/5)$ follows from the global Gauss-sum sign.
This identifies the full modular data from one contraction of $A$.

\begin{remark}[Galois conjugation]\label{ce:galois-metric-group}
Recall that $A$ has algebraic entries. Galois conjugation preserves the
underlying group $H$, although it can change $q_H$ and the spherical sign.
More precisely, extend a coefficient-field embedding $\gamma$ to a field
containing the $\mu$th roots of unity, and write
$\gamma(\zeta_\mu)=\zeta_\mu^v$.
The operator formula gives $D_{\gamma(A)}=\gamma(D_A)$, so its twists
are the conjugates of those for $A$. Their remaining Gauss sums are
those of $(H,vq_H)$. Proposition~\ref{ce:metric-recognition} therefore
identifies the metric group of $\gamma(A)$ with $(H,vq_H)$.
In particular, solutions with nonisomorphic groups $H$ cannot be Galois
conjugate. 
\end{remark}

\subsection{The order-fifteen application}
\label{ce:n15-application}

The completeness statement in Appendix~\ref{app:finite-cyclic-data} and
Table~\ref{tab:cyclic-census} show that, in the 
Q-system setting, there are exactly three untwisted, unitary HI categories for
$\mathbb Z/15$ up to tensor equivalence, represented by
\begin{equation*}
 \mathcal C_{j^{(1)}(15)},\qquad
 \mathcal C_{j^{(2)}(15)},\qquad
 \mathcal C_{j^{(3)}(15)}.
\end{equation*}

\begin{corollary}[Non-Morita-equivalent categories with the same modular data]
\label{sa:n15-morita}
The three categories
\(\mathcal C_{j^{(1)}(15)},\mathcal C_{j^{(2)}(15)},\mathcal C_{j^{(3)}(15)}\)
are pairwise not categorically Morita
equivalent, while
their centers have the same modular data.
\end{corollary}

\begin{proof}
By Proposition~\ref{sa:morita-rigidity}, two of these categories can be
categorically Morita equivalent only if they are tensor equivalent. Since, using Table~\ref{tab:cyclic-census}, they are not tensor equivalent, we are done.

Since $15^2+4=229$ is prime, and these categories all have dim$(\rho)=d_+$, Corollary~\ref{ce:prime-discriminant} uniquely determines their modular data.
\end{proof}

Thus equal center modular data need not imply Morita equivalence of the
input categories. This was first observed in \cite{mignard2021modular}. Inspecting Tables~\ref{tab:cyclic-census} and \ref{tab:cyclic-negative}, we see this is far from a rare occurrence. In all these examples though, the categories are Galois associates.

\begin{remark}\label{rem:borromean}
The three order-fifteen centers in
Corollary~\ref{sa:n15-morita} can be distinguished by the
Borromean tensor $B$ introduced by Kulkarni, Mignard, and
Schauenburg \cite[Section~4]{kulkarni2018topological}. 
In each center, let $\beta$ denote
the unique simple object of dimension $d_+^2$.
An exact computer-assisted calculation shows that the three
values of $B_{\beta,\beta,\beta}$ are pairwise distinct.
Thus no relabeling identifies their modular data and
Borromean tensors simultaneously.

Further examples of non-Morita-equivalent categories with the
same modular data occur in Table~\ref{tab:cyclic-census} at
$N=21,27,33,35$.
Numerical calculations also show that
$B_{\beta,\beta,\beta}$ distinguishes all the centers
with identical modular data in these larger-order examples.
\end{remark}


\section{Equivariantization and near-group categories}
\label{sec:near-group}

Izumi constructed the passage from HI categories to
near-group categories in the $C^*$ setting
\cite[Theorem~12.9]{IzumiNearGroupCuntz}.  We give an algebraic formulation
over an algebraically closed field in which $|G|$ is invertible, without
assuming unitarity. A unitary Q-system example gives a unitary near-group
category.

Let $\kk$ be an algebraically closed field, let $G$ be a finite abelian group
of odd order $N$, and assume that $N$ is nonzero in $\kk$.  Let
$\mathcal C$ be a (possibly twisted) fusion category of type $\mathfrak{HI}_G$ over $\kk$. 
Write $\mathcal P$ for the pointed subcategory.  After choosing the labels,
\begin{equation*}
 \mathcal P\simeq\operatorname{Vec}_{G}^{\Omega}
\end{equation*}
for a normalized $3$-cocycle $\Omega$.  

\subsection{Relative centers and conjugation}

For a pointed tensor subcategory $\mathcal Q\subseteq\mathcal D$, the
relative center $Z_{\mathcal Q}(\mathcal D)$ consists of objects $X$ with
natural isomorphisms
\begin{equation*}
 c_V:V\otimes X\longrightarrow X\otimes V
 \qquad(V\in\mathcal Q)
\end{equation*}
satisfying the half-braiding axiom for objects of $\mathcal Q$.  It is
monoidal, but it need not be braided because the half-braiding is only
specified against $\mathcal Q$.

Choose representatives $\alpha_g$ and normalized multiplication maps
\(\mu_{g,h}:\alpha_g\otimes\alpha_h\to\alpha_{g+h}\).  Conjugation gives
functors
\begin{equation*}
 F_g(X)=\alpha_g\otimes X\otimes\alpha_g^\vee.
\end{equation*}
The multiplication maps and their duals give coherent tensorators
\(F_gF_h\Rightarrow F_{g+h}\).  The associator scalar in the pointed
subcategory occurs once on the left and once with its inverse on the dual
right, so it cancels in the action pentagon.

\begin{proposition}[Relative center as equivariantization]
\label{ng:relative-center}
For any fusion category $\mathcal D$ over $\kk$ and pointed tensor
subcategory \(\mathcal Q\subseteq\mathcal D\), the conjugation functors
define a tensor action of the group $\Gamma$ of isomorphism classes of
simple objects of \(\mathcal Q\), and
\begin{equation}
 \mathcal D^{\Gamma}\simeq Z_{\mathcal Q}(\mathcal D).
 \label{ng:relative-center-equivalence}
\end{equation}
The equivalence restricts to
\begin{equation*}
 \mathcal Q^{\Gamma}\simeq Z(\mathcal Q).
\end{equation*}
If $|\Gamma|$ is invertible in $\kk$, the equivariantization is again a
fusion category.  De-equivariantization by the canonical central functor
\(\operatorname{Rep}(\Gamma)\to Z(\mathcal D^\Gamma)\) then recovers
\(\mathcal D\).
\end{proposition}

The graded version of this relative-center/equivariantization
correspondence is developed in
\cite[Section~3A, equations~(24)--(28), and Theorem~3.5]{GelakiNaiduNikshych2009}.
That reference works in characteristic zero and in the faithfully graded
setting.  The proposition here is proved directly, and the only field
restriction needed for its fusion and de-equivariantization assertions is
that $|\Gamma|$ be invertible in $\kk$.

\begin{proof}
There are two parts.  First, an equivariant structure on $X$ consists of
maps \(u_g:F_g(X)\to X\) satisfying the action coherence.  Evaluation and
coevaluation for $\alpha_g$ turn these maps into
\begin{equation*}
 \alpha_g\otimes X\longrightarrow X\otimes\alpha_g.
\end{equation*}
These maps give the half-braiding on the simple pointed objects, and the
half-braiding axiom is exactly the equivariant coherence for products of
such objects.  Conversely, tensor a half-braiding with
$\alpha_g^\vee$ and contract the adjacent invertible pair.  The zig-zag
identities show that the constructions are inverse on objects, morphisms,
and tensor products.  This proves \eqref{ng:relative-center-equivalence};
restricting to \(\mathcal Q\) gives the second equivalence.

Second, suppose that $|\Gamma|$ is invertible in $\kk$.  Averaging a splitting
of an equivariant epimorphism makes the equivariantization semisimple.  The
canonical central functor sends a representation to the corresponding
equivariant copy of the tensor unit.  The regular algebra
\(\operatorname{Fun}(\Gamma)\) in this central copy is separable.  Its
module category is identified with $\mathcal D$ by taking the identity
component of the associated $\Gamma$-graded module, which is the inverse
de-equivariantization.  In characteristic zero this is the stated
2-equivalence
\cite[Theorem~4.18 and Proposition~4.19]{DGNO2010}.
Over a field in which $|\Gamma|$ is invertible, the exact-sequence and
fusion statement is recorded in
\cite[Section~1]{BruguieresNatale2014}.
\end{proof}

The construction in Theorem~\ref{ng:relative-center} uses no HI fusion rule.
We now use \eqref{HIfusrul} to identify the simple objects and the
invertible group of the relative center.

\subsection{The pointed center and the metric extension}

For $g\in G$, the half-braiding of an object of degree $g$ is a module over
the twisted group algebra with factor set
\begin{equation*}
 \vartheta_g(x,y)=
 \frac{\Omega(x,g,y)}{\Omega(g,x,y)\Omega(x,y,g)}.
\end{equation*}
Its commutator is the alternating tricharacter
\begin{equation}
 \Lambda_\Omega(g,x,y)=
 \frac{\vartheta_g(x,y)}{\vartheta_g(y,x)}.
 \label{ng:slant-commutator}
\end{equation}
The twisted group algebra in degree $g$ is commutative exactly when
\(\Lambda_\Omega(g,-,-)=1\).  Since $N$ is invertible in $\kk$, it is then
semisimple and all its simple modules are one-dimensional.  Thus
\(Z(\operatorname{Vec}_{G}^{\Omega})\) is pointed exactly when
\(\Lambda_\Omega=1\).

The setup preceding
\cite[Lemma~12.8]{IzumiNearGroupCuntz} assumes the corresponding
pointedness condition.  Here it follows from the HI mixed rule and the
oddness of $|G|$.
The mixed rule in \eqref{HIfusrul} forces inversion
\(g\mapsto-g\) to preserve the cohomology class $[\Omega]$.  Applying
inversion to \eqref{ng:slant-commutator} gives
\begin{equation*}
 \Lambda_\Omega(g,x,y)=\Lambda_\Omega(-g,-x,-y)
 =\Lambda_\Omega(g,x,y)^{-1}.
\end{equation*}
Since $G$ has odd order and $\Lambda_\Omega$ is trilinear, every value is
both of odd order and of order dividing two.  Hence it is $1$.

\begin{proposition}[Pointed center of an HI pointed part]
\label{ng:pointed-center}
For every  HI category satisfying \eqref{HIfusrul}, the center
of its pointed subcategory is pointed.  Its invertible objects form an
abelian group $E_\Omega$ of order $N^2$ in an exact sequence
\begin{equation*}
 0\longrightarrow\widehat G\longrightarrow E_\Omega
 \overset{\deg}{\longrightarrow}G\longrightarrow0.
\end{equation*}
The braiding gives a nondegenerate quadratic form $q_\Omega$ on
\(E_\Omega\), and the degree-zero subgroup
\(H_\Omega=\ker(\deg)\cong\widehat G\) is Lagrangian.
\end{proposition}

The extension and its quadratic form describe the pointed braided category
\(Z(\operatorname{Vec}_{G}^{\Omega})\).  They do not assert a braiding on
the full relative center.

\subsection{The near-group category}

We can now count the simples of the relative center.  On pointed objects,
conjugation is trivial up to the equivariant structure.  On a
noninvertible object,
\begin{equation}
 F_g(\alpha_x\rho)
 \cong \alpha_g\alpha_x\rho\alpha_{-g}
 \cong \alpha_{x+2g}\rho.
 \label{ng:conjugation-orbit}
\end{equation}
Because multiplication by $2$ is bijective on $G$, these noninvertible
objects form one free orbit.

In the $C^*$ setting, the following construction and its
multiplicity are due to Izumi
\cite[Theorem~12.9]{IzumiNearGroupCuntz}.
\begin{theorem}[The relative center is near-group]
\label{ng:near-group}
Let $\mathcal C$ satisfy \eqref{HIfusrul}, with $|G|$ invertible in
$\kk$.  Then
\begin{equation*}
 \mathcal N:=Z_{\mathcal P}(\mathcal C)\simeq\mathcal C^G
\end{equation*}
is a near-group fusion category of type $(E_\Omega,N^2)$.  If $R$ is its
unique noninvertible simple object, then
\begin{equation}
 R^2\cong\bigoplus_{e\in E_\Omega}e\oplus N^2R.
 \label{ng:near-group-fusion}
\end{equation}
\end{theorem}

\begin{proof}
The equivariant objects whose underlying objects lie in $\mathcal P$ form
\(Z(\mathcal P)\), which is pointed by Theorem~\ref{ng:pointed-center}.  They
therefore give the $N^2$ invertible objects $E_\Omega$.  The free orbit in
\eqref{ng:conjugation-orbit} gives one noninvertible equivariant simple,
which we call $R$.

Every $e\in E_\Omega$ satisfies $eR\cong R\cong Re$, and $R$ is self-dual.
Frobenius reciprocity gives each $e$ once in $R^2$.  It remains to determine
the coefficient of $R$.  In the Grothendieck ring of $\mathcal C$, set
\begin{equation*}
 A_0=\sum_{x\in G}[\alpha_x],\qquad
 S=\sum_{x\in G}[\alpha_x\rho].
\end{equation*}
The HI rules give
\begin{equation*}
 S^2=NA_0+N^2S.
\end{equation*}
Forgetting equivariant structures sends $R$ to $S$ and
\(\bigoplus_e e\) to $NA_0$.  Comparing the coefficient of any
noninvertible simple yields the coefficient $N^2$ in
\eqref{ng:near-group-fusion}.
\end{proof}

The inverse operation retains more than the near-group fusion ring.  The
equivariantization gives a central functor
\begin{equation*}
 \operatorname{Rep}(G)\longrightarrow Z(\mathcal N),
\end{equation*}
and the regular algebra \(\operatorname{Fun}(G)\) in this central copy has
underlying object
\begin{equation*}
 \bigoplus_{h\in H_\Omega}h,
 \qquad H_\Omega\cong\widehat G\leq E_\Omega.
\end{equation*}
De-equivariantizing by this specified central algebra recovers
\(\mathcal C\).  The subgroup $H_\Omega$ alone does not determine its
multiplication or central embedding, so the near-group fusion rules do not
classify the possible de-equivariantizations.

\begin{remark}[Stretched fusion rules]
\label{rem:ng-stretched}
The multiplicity formula $r|G|^2$ already occurs in
\cite[Theorem~12.9]{IzumiNearGroupCuntz}; the argument here gives it under
the algebraic hypotheses of this section.
The same construction relates stretched HI and near-group fusion rules.
Retain the field and group hypotheses of this section, and suppose that
$\mathcal C$ has the same simple objects and mixed rules as in
\eqref{HIfusrul}, but with
\begin{equation*}
 \rho^2\cong\mathbf 1\oplus r\bigoplus_{g\in G}\alpha_g\rho,
 \qquad r\in\mathbb Z_{\geq1}.
\end{equation*}
The pointed-center and free-orbit arguments are unchanged.  With $A_0$
and $S$ as in the proof of Theorem~\ref{ng:near-group}, the Grothendieck-ring
calculation becomes
\begin{equation*}
 S^2=NA_0+rN^2S.
\end{equation*}
Thus $Z_{\mathcal P}(\mathcal C)$ is near-group of type
$(E_\Omega,rN^2)=(E_\Omega,r|E_\Omega|)$, and de-equivariantization by
its specified central $\operatorname{Rep}(G)$ recovers the stretched HI
category $\mathcal C$.  The integer $r$ is a fusion multiplicity; it need
not be invertible in $\kk$.

This statement starts from an existing stretched HI category.  It does
not assert that an arbitrary near-group category of type $(E,r|E|)$
admits such a de-equivariantization.  That converse requires the
appropriate central algebra, the multiplicity-one splitting of the
noninvertible object, and the inversion mixed rule.  When the quotient
has this HI form, comparison of the same Grothendieck-ring coefficients
forces its square multiplicity to be $r$.
\end{remark}

\begin{corollary}[The untwisted cyclic family]
\label{ng:cyclic-near-group}
If the pointed associator is trivial, then
\begin{equation*}
 E_0\cong G\times\widehat G,\qquad q_0(x,\chi)=\chi(x).
\end{equation*}
Consequently every untwisted cyclic HI category of
Q-system type, has a near-group equivariantization of type
\begin{equation*}
 \bigl(G\times\widehat G,N^2\bigr).
\end{equation*}
\end{corollary}

Izumi applied this construction to cyclic HI categories of
orders $5$, $7$, and $9$ \cite[Example~12.15]{IzumiNearGroupCuntz}.
Theorem~\ref{ds:uniform-existence} gives HI inputs for every odd
$N\geq3$, and hence gives the family in
Corollary~\ref{ng:cyclic-near-group}.  This conclusion does not require
unitarity.

For $G=\mathbb Z/N$ with $N>1$, the group
\(G\times\widehat G\) is noncyclic.  Thus cyclic HI categories provide
canonical noncyclic near-group categories. Over $\mathbb C$, the relative
center is unitarizable when the input category is unitarizable.

\subsection{Modular data of the near-group centers}
\label{ng:center-data}

Work now over $\mathbb C$. Let $G=\mathbb Z/N\mathbb Z$, with $N\geq3$
odd, and let $\mathcal C$ be an untwisted pseudo-unitary HI category in
our constructed family. Put $\mathcal D=\mathcal C^G$ and
$E=G\times\widehat G$. Write
\begin{equation*}
 q_E(x,\chi)=\chi(x),\qquad
 b_E((x,\chi),(y,\psi))=\chi(y)\psi(x),\qquad
 \beta(a,b)=b_E(a,b/2).
\end{equation*}
Thus $\beta$ is a symmetric bicharacter with $\beta(a,a)=q_E(a)$.
Let $(H,q_H)$ be the metric group of Theorem~\ref{ce:metric-data} for
$Z(\mathcal C)$, where $|H|=N^2+4$, and put
$Q_H(h)=e^{2\pi i q_H(h)}$.

\begin{corollary}\label{ng:modular-realization}
There is a ribbon equivalence
\begin{equation}\label{ng:center-factorization}
 Z(\mathcal D)\simeq Z(\mathcal C)\boxtimes Z(\operatorname{Vec}_G).
\end{equation}
Its modular data are the Evans--Gannon near-group data for
$(E,q_E)$ and $(H,Q_H)$, and equivalently the Grossman--Izumi data for
\begin{align*}
 &(E\times E,q_1,\theta_1),&
 q_1(a,b)&=\beta(a,b),&\theta_1(a,b)&=(b,a),\\
 &(E\times H,q_2,\theta_2),&
 q_2(a,h)&=q_E(a)Q_H(h),&\theta_2(a,h)&=(a,-h).
\end{align*}
Thus the modular-data formulas of
\cite[Proposition~7 and Conjecture~2]{EvansGannonNearGroup}
and \cite[Definition~3.1 and Conjecture~3.7]{GrossmanIzumiPotential}
are realized by the constructed near-group categories of type
$((\mathbb Z/N)^2,N^2)$ for every odd $N\geq3$.
\end{corollary}

\begin{proof}
The equivariantization is Morita equivalent to the crossed product
$\mathcal C\rtimes G$. Since the action is conjugation and the pointed
associator is trivial, the functor
\begin{equation*}
 (X,g)\longmapsto (X\otimes\alpha_g)\boxtimes g
\end{equation*}
is a tensor equivalence from $\mathcal C\rtimes G$ to
$\mathcal C\boxtimes\operatorname{Vec}_G$: its tensorator cancels
$\alpha_{-g}\alpha_{g+h}$ to $\alpha_h$. Coherence follows from the
chosen associative multiplication on the $\alpha_g$.
Morita invariance of the center gives \eqref{ng:center-factorization};
this factorization was already observed in
\cite[Remark~12.10]{IzumiNearGroupCuntz}.
The categories here are pseudo-unitary, so the canonical spherical
structures make the equivalence ribbon.

For clarity, use the convention of
\cite[Definition~3.1]{GrossmanIzumiPotential}, in which the pointed
$S$-matrix uses the conjugate polarization. The HI $S$-matrix is real,
so its displayed entries are unchanged. For $a,b\in E$ and
$X,Y\in\operatorname{Irr}Z(\mathcal C)$, the complete data are therefore
\begin{equation}\label{ng:product-modular-data}
 S^{Z(\mathcal D)}_{(a,X),(b,Y)}
 =\frac{\overline{b_E(a,b)}}N S^{Z(\mathcal C)}_{X,Y},
 \qquad
 T^{Z(\mathcal D)}_{(a,X)}=q_E(a)T^{Z(\mathcal C)}_X.
\end{equation}
To identify the proposed data, put $J(x,\chi)=(x,\chi^{-1})$.
The change of variables
$(z,u)\mapsto(z+Ju,z-Ju)$ identifies $(E\times E,q_1,\theta_1)$
with $(E,q_E,1)\times(E,q_E,-1)$, since
$\beta(z+Ju,z-Ju)=q_E(z)q_E(u)$.
The second triple is already
$(E,q_E,1)\times(H,Q_H,-1)$. Their common fixed metric subgroup is
$(E,q_E)$; their orders are $N^4$ and $N^2(N^2+4)$, and their
normalized Gauss sums are $1$ and $-1$.
Grossman--Izumi's factorization lemma
\cite[Lemma~3.2]{GrossmanIzumiPotential} now splits off $(E,q_E)$.
The remaining arrays are precisely the HI arrays of
Theorem~\ref{ce:metric-data}, proving \eqref{ng:product-modular-data}
in their notation. For odd groups this is the Evans--Gannon formula;
explicitly its labels $a_z,b_z,c_{z+Ju,z-Ju},d_{z,h}$ correspond to
$(z,0_+),(z,0_-),(z,\mathfrak a_{[u]}),(z,\mathfrak d_{[h]})$.
\end{proof}

These realizations are algebraic; uniform unitarity remains open.

\section{A non-cyclotomic modular category}
\label{sec:split-field}

\begin{definition}[Split ribbon form]
A split ribbon $J$-form of a modular category $\mathcal B$, for a number field
$J\subset\overline{\mathbb Q}$, is a semisimple ribbon category over $J$
whose scalar extension to $\overline{\mathbb Q}$ is ribbon equivalent to $\mathcal B$, and for which every
geometric simple is the scalar extension of an absolutely simple object
defined over $J$.
\end{definition}

We use an explicit HI matrix for a $\mathbb Z/5$ HI category
to show that the Galois closure of every split ribbon field
of definition for its center contains a non-cyclotomic quartic field.

\subsection{The specified order-five example}

In \cite[Example~7.2]{IzumiLR2}, Izumi constructed the unitary $\mathbb Z/5$ HI category with HI matrix
\begin{equation}
 A_5=
 \frac{1}{d-1}\begin{pmatrix}
 d-2&-1&-1&-1&-1\\
 -1&-1&x&y&x\\
 -1&\overline{x}&-1&y&y\\
 -1&\overline{y}&\overline{y}&-1&x\\
 -1&\overline{x}&\overline{y}&\overline{x}&-1
 \end{pmatrix},
  \label{sf:order-five-matrix}
\end{equation}
where $d=d_+$, $\eta=\frac{d+\sqrt{d+9}}{2}$, $\eta'=\frac{d-\sqrt{d+9}}{2}$, $x=\frac{\eta+i\sqrt{4d-\eta^2}}{2}$ and $y=\frac{\eta'+i\sqrt{4d-\eta^{\prime2}}}{2}$.   In the notation of
\cite[Section~6.6]{GrossmanSnyder2012}, this is their
$\mathcal I_2$ fusion category; there are two other categories $\mathcal I_1,\mathcal I_3$ in that Morita class.
Then
\begin{equation*}
\mathcal B=Z(\mathcal I_2)=Z(\mathcal C(A_5))
\end{equation*}
is the modular category of interest here. We will prove:

\begin{theorem}[Noncyclotomic split-ribbon field]
\label{sf:counterexample}
If $J$ supports a split ribbon form of the center $\mathcal B$,  then the Galois closure of $J$ must contain  $\eta$.
In particular, no cyclotomic field supports a split ribbon form of 
this modular category.
\end{theorem}

Write $d^\ast=5-d$. Since $(d+9)(d^\ast+9)=125$, the normal closure of $E=\mathbb Q(d,\sqrt{d+9})$ is $\widetilde E=\mathbb Q(d,\sqrt{d+9},\sqrt5)$; it has degree eight and dihedral Galois group. A normal field containing $E$ contains $\widetilde E$. An abelian cyclotomic extension cannot contain this nonabelian Galois extension.

\subsection{Proof that \texorpdfstring{$\eta$}{eta} lies in the Galois closure of \texorpdfstring{$J$}{J}}

Define 
\begin{equation} F=\mathbb Q[d],\qquad E=\mathbb Q[F,\eta,\eta'],\qquad L=E[x,y].\end{equation}
Let $\tau$ be the unique nontrivial automorphism in Gal$(E/F)$ and consider any lift  to $\overline{\mathbb Q}$. This lift must permute $x,y,\overline{x},\overline{y}$. Since $\tau$ exchanges the polynomials $X^2-\eta X+d$ and
$X^2-\eta'X+d$, it sends $\{x,\overline{x}\}$ to
$\{y,\overline{y}\}$.

Note that the category $\mathcal I_2$ defined from \eqref{sf:order-five-matrix} is manifestly defined over $L$. Hitting $A_5$ with $\tau$ will give another HI matrix for $d_+$, and hence define another unitary HI category for $\mathbb Z/5$ of Q-system type. But $\mathcal I_2$ is the unique such category, up to tensor equivalence. Therefore, by \cite[Theorem 2(b)]{evans2017non},  there must exist an $\ell\in\mathrm{Aut}(\mathbb Z/5)\cong(\mathbb Z/5)^\times$ such that $\tau(A_5)_{g,h}=(A_5)_{\ell g,\ell h}$. By considering $\tau(A_5)_{1,2}$, and replacing this lift if necessary by its inverse, we can take $\ell=2$, i.e.\ $\tau(A_5)_{g,h}=(A_5)_{2 g,2 h}$.

In a split $J$-form, each simple object has endomorphism algebra $J$.
All ribbon twists of simples are therefore scalars in $J$. Corollary~\ref{ce:prime-discriminant} then tells us $J$ must contain a primitive 29th root of 1, and hence $d\in J$ using Gauss sums.

Suppose for contradiction that $\eta$ does not lie in the Galois closure of $J$.  Hence $E\cap J=F$ and
$\operatorname{Gal}(EJ/J)\cong\operatorname{Gal}(E/F)$.  Thus $\tau$
extends to an automorphism of $EJ$ fixing $J$, and then to one of
$\overline{\mathbb Q}$ fixing $J$. As explained earlier, we can assume without loss of generality that $\tau(A_5)_{g,h}=(A_5)_{2 g,2 h}$.

Comparing the half-braidings in
\cite[equation~(6.11)]{evans2017non}, the relabeling gives a
braided equivalence $T:\mathcal B^\tau\to\mathcal B$ exchanging
$\mathfrak a_{[1,1]}$ and $\mathfrak a_{[2,1]}$, where the second
coordinate denotes the trivial character. 
On the other hand, base change of the split $J$-form gives a braided
descent equivalence $D:\mathcal B^\tau\to\mathcal B$ fixing every
geometric simple-object class: each such simple is already defined over
$J$, which $\tau$ fixes.  Then
$T\circ D^{-1}$ is a braided autoequivalence of $\mathcal B$, and it must be nontrivial since its
permutation of simples is precisely that of $T$.

We can identify the group $\operatorname{EqBr}(Z(\mathcal I_2))$ of braided autoequivalences, using its isomorphism \cite[Theorem~1.1]{EtingofNikshychOstrik2010} \begin{equation*}
 \operatorname{BrPic}(\mathcal C)\cong
 \operatorname{EqBr}(Z(\mathcal C))
\end{equation*} 
with the Brauer-Picard group of
 invertible exact \(\mathcal{C}\)-bimodule categories.  First, the outer tensor
autoequivalences of $\mathcal I_2$ follow from \cite[Theorem~2.4]{GrossmanIzumiSnyder2022}: because the stabilizer of Aut$(\mathbb Z/5)$ on  \eqref{sf:order-five-matrix} is trivial, so is the outer tensor
autoequivalence group. \cite[Theorem~6.11]{GrossmanSnyder2012}
 states that $\mathcal I_2$ has exactly three
indecomposable module categories, with exactly one having dual category
$\mathcal I_2$; the other two have duals $\mathcal I_1$ and $\mathcal I_3$.   
Since the outer tensor autoequivalence group is trivial,
the unique self-dual module category gives $\operatorname{BrPic}(\mathcal I_2)=1$, contradicting the existence of  $T\circ D^{-1}$ constructed above. This concludes the proof of Theorem \ref{sf:counterexample}.

Incidentally, the same argument applied to $\mathbb Z/3$ HI categories produces a braided autoequivalence with a trivial permutation of simples, so does not yield a contradiction.


\section{Characteristic two}
\label{sec:characteristic-two}\label{sec:HI-char-two-abelian}\label{sec:HI-char-two}
In characteristic two, we prove a Frobenius relation that bounds the
field of HI-matrix entries and makes fixed-order enumeration finite.
The reconstructed categories lift to characteristic zero. Conversely,
we prove integrality above $2$ for characteristic-zero HI-matrices
and compare Q-system categories with their reductions. This gives
an orbit-count criterion for completeness of the lists in
Appendix~\ref{app:finite-cyclic-data}.

\subsection{A squared-entry identity}
\label{sec:HI-diagonal-square}

Recall the Fourier transforms $X_t,Y_t$ from \eqref{eq:display-0263}.
We combine their factorization in Lemma~\ref{lem:HI-Fourier-E3}
with a squared-entry identity obtained from the cubic equation.
The latter holds in every characteristic when $\omega=1$. 
In characteristic two, these transforms are Frobenius squares
after reindexing the characters; in this characteristic,
Theorem~\ref{thm:HI-omega-trivial} gives $\omega=1$ for every HI-matrix.

\begin{theorem}[Squared-entry identity]
\label{thm:HI-diagonal-square}
With $\omega=1$ in the scalar setup \eqref{eq:HI-scalars}--\eqref{eq:HI-a}, if $A$ satisfies
\eqref{eq:HImatrix1} and the cubic identity, then
\begin{equation}
 \label{eq:HI-diagonal-square}
 \sum_{r\in G}A_{r,2t}A_{t,r+k}^2
 =A_{-k,t}^2-a\delta_{t,0}
 \qquad(t,k\in G).
\end{equation}
\end{theorem}

\begin{proof}
Set $g=h=t$ and $l=k$ in \eqref{eq:HI-cubic}.  Its two sides become
\begin{align*}
 \sum_{r\in G}A_{r,2t}A_{t,r+k}^2
 =A_{t+k,k}^2-a\delta_{t,0}
 =A_{-k,t}^2-a\delta_{t,0},
\end{align*}
where the second equality is \eqref{eq:HImatrix1}.
\end{proof}

Suppose now that $G$ is finite abelian and
$\operatorname{char}(\kk)\nmid N=|G|$.  Write
$\widehat G:=\Hom(G,\kk^\times)$.  Using the column and row transforms from
\eqref{eq:display-0263}, define the entrywise-square transforms
\begin{equation*}
 X_t^{[2]}(\chi):=\sum_{r\in G}A_{r,t}^2\chi(r),
 \qquad
 Y_t^{[2]}(\chi):=\sum_{r\in G}A_{t,r}^2\chi(r).
\end{equation*}
The superscript $[2]$ means that the entries are squared before taking the
Fourier transform; it does not mean $X_t(\chi)^2$.

\begin{corollary}[Fourier form of the squared-entry identity]
\label{cor:HI-Fourier-diagonal-square}
Under the assumptions of Theorem~\ref{thm:HI-diagonal-square}, suppose also
that $G$ is finite abelian and $\operatorname{char}(\kk)\nmid|G|$.  Then
\begin{equation}
 \label{eq:HI-Fourier-diagonal-square}
 X_{2t}(\chi^{-1})Y_t^{[2]}(\chi)
 =X_t^{[2]}(\chi^{-1})
  -N a\delta_{t,0}\delta_{\chi,1}.
\end{equation}
\end{corollary}

\begin{proof}
Multiply \eqref{eq:HI-diagonal-square} by $\chi(k)$ and sum over $k$.
For the left-hand side, put $s=r+k$.  Since
$\chi(k)=\chi(s)\chi(r)^{-1}$, the double sum factors as
\begin{equation*}
 \left(\sum_rA_{r,2t}\chi(r)^{-1}\right)
 \left(\sum_sA_{t,s}^2\chi(s)\right)
 =X_{2t}(\chi^{-1})Y_t^{[2]}(\chi).
\end{equation*}
For the first term on the right, put $u=-k$.  Its sum is
$X_t^{[2]}(\chi^{-1})$.  The last term uses $\sum_{k\in G}\chi(k)=N\delta_{\chi,1}$. 
This proves \eqref{eq:HI-Fourier-diagonal-square}.
\end{proof}

We now specialize to characteristic two. Since $G$ has odd order,
$N=1$ in $\kk$, and the scalar equation becomes $a^2+a=1$.
\begin{theorem}[Fourth-power relation for odd abelian groups]
\label{thm:HI-char-two-abelian-fourth-power}
Suppose that $G$ is nontrivial, take the scalar setup
\eqref{eq:HI-scalars}--\eqref{eq:HI-a} in characteristic $2$, and let $A$
satisfy \eqref{eq:HImatrix1} and \eqref{eq:HI-cubic}.  Then
\begin{equation}
 \label{eq:HI-char-two-fourth-power}
 A_{2g,2h}=A_{g,h}^4
 \qquad(g,h\in G).
\end{equation}
\end{theorem}

\begin{proof}
Theorem~\ref{thm:cubic-criterion} shows that $A$ is an HI-matrix, so
\eqref{eq:HImatrix3} holds. We divide the proof into three parts. First we
express the two identities through Fourier coefficients. Then we cancel a
nonzero Fourier coefficient and recover the matrix entries by Fourier
inversion.

\emph{Fourier coefficients and Frobenius.}
Since the odd integer $|G|$ equals $1$ in $\kk$,
Lemma~\ref{lem:HI-Fourier-E3} gives
\begin{equation}
 \label{eq:HI-char-two-abelian-Fourier-inverse}
 X_t(\chi)Y_t(\chi^{-1})
 =1-a\delta_{t,0}\delta_{\chi,1}.
\end{equation}

Because $G$ has odd order, squaring is an automorphism of $\widehat G$.
Write $\chi^{1/2}$ for the unique character whose square is $\chi$.
Frobenius additivity then gives
\begin{equation*}
 X_t^{[2]}(\chi)=X_t(\chi^{1/2})^2,
 \qquad
 Y_t^{[2]}(\chi)=Y_t(\chi^{1/2})^2.
\end{equation*}

\emph{Cancellation of the nonzero Fourier coefficient.}
Substituting these identities in
\eqref{eq:HI-Fourier-diagonal-square} gives
\begin{equation}
 X_{2t}(\chi^{-1})Y_t(\chi^{1/2})^2
 =X_t((\chi^{-1})^{1/2})^2
  +a\delta_{t,0}\delta_{\chi,1}.
 \label{eq:display-0436}
\end{equation}
Suppose first that $(t,\chi)\neq(0,1)$, and put
$\xi=(\chi^{-1})^{1/2}$.  Then $\xi^{-1}=\chi^{1/2}$, and
\eqref{eq:HI-char-two-abelian-Fourier-inverse}, applied to $\xi$, gives
\begin{equation}
 X_t(\xi)Y_t(\xi^{-1})=1.
 \label{eq:display-0437}
\end{equation}
Both factors in \eqref{eq:display-0437} are nonzero.  Divide
\eqref{eq:display-0436} by $Y_t(\xi^{-1})^2$ and use
\eqref{eq:display-0437}.  This gives
\begin{equation}
 X_{2t}(\xi^2)=X_t(\xi)^4.
 \label{eq:display-0438}
\end{equation}
At the remaining point $(t,\chi)=(0,1)$,
\eqref{eq:HImatrix2} gives $X_0(1)=-a=a$.  The
scalar relation is $a^2+a=1$, so $a^4=a$.  Hence
\eqref{eq:display-0438} also holds at this point.

\emph{Fourier inversion.}
Since multiplication by $2$ permutes $G$,
\begin{equation*}
 X_{2t}(\xi^2)
 =\sum_{g\in G}A_{2g,2t}\xi(g)^4,
 \qquad
 X_t(\xi)^4
 =\sum_{g\in G}A_{g,t}^4\xi(g)^4.
\end{equation*}
As $\xi$ varies, so does the character $\xi^4$, because $G$ has odd order.
Thus every Fourier coefficient of
$g\mapsto A_{2g,2t}-A_{g,t}^4$ is zero.  The odd integer $|G|$ is invertible
in $\kk$, so Fourier inversion gives
$A_{2g,2t}=A_{g,t}^4$ for every $g,t\in G$.
\end{proof}

This identity holds for the original matrix $A$ with its given group labels.

\begin{remark}[Finite fields and scalar extension]
\label{rem:HI-base-change}
Suppose that the entries of $A$ lie in a finite field $F$. Let $E$ be the
field generated over $F$ by the required scalars $d$ and $b$.
Then $E/F$ is finite, and the presentation and reduced-word basis give
\begin{equation*}
 \cL_{N+1}(F)\otimes_FE\cong\cL_{N+1}(E).
\end{equation*}
The matrix identities and the formulas for $\rho$ and $\alpha_g$ commute
with scalar extension. Thus the reconstruction theorem applies over $E$ and
over its algebraic closure.
\end{remark}

\begin{corollary}[Finite-field bound and lifting]
\label{cor:HI-char-two-field-bound}
Suppose $G$ is nontrivial, and let $A$ satisfy \eqref{eq:HImatrix1} and the
cubic identity \eqref{eq:HI-cubic} in the characteristic-$2$ setup of
Theorem~\ref{thm:HI-char-two-abelian-fourth-power}.  Put
\begin{equation*}
 r:=\min\{s\geq1:2^s g=g\text{ for every }g\in G\}.
\end{equation*}
Equivalently, if $e_G:=\exp(G)$, then
$r=\operatorname{ord}_{e_G}(2)$.
Every entry of $A$ belongs to $\mathbb F_{2^{2r}}\subset\kk$.
The matrix $A$ is an HI-matrix and reconstructs a spherical
$\mathfrak{HI}_G$-category of global dimension $d\neq0$. It admits
a unique lifting to $W(\kk)$. Extending scalars to an algebraic closure of
the fraction field of $W(\kk)$ gives a characteristic-zero categorification.
\end{corollary}

\begin{proof}
\emph{Field of definition.}
Iteration of \eqref{eq:HI-char-two-fourth-power} gives
\begin{equation*}
 A_{2^rg,2^rh}=A_{g,h}^{4^r}.
\end{equation*}
By the definition of $r$, every entry is fixed by the $4^r$-power Frobenius.
It therefore belongs to $\mathbb F_{4^r}=\mathbb F_{2^{2r}}$.

\emph{Lifting.}
By Theorem~\ref{thm:cubic-criterion} and
Theorems~\ref{thm:HImatrix-to-Leavitt} and \ref{thm:LRtoHIcat}, let
$\mathcal C$ be the resulting spherical $\mathfrak{HI}_G$-category.  In
characteristic $2$, the odd integer $|G|$ is $1$ in $\kk$, and $d^2=d+1$.
Remark~\ref{rem:LR-global-dimension} gives
\begin{equation*}
 \operatorname{Dim}(\mathcal C)=|G|(1+d^2)=d\neq0.
\end{equation*}
Thus $\mathcal C$ is nondegenerate in the positive-characteristic sense of
\cite[Definition~9.1, p.~634]{etingof2005fusion}; the lifting assertion is
\cite[Theorem~9.3, p.~635]{etingof2005fusion}.
\end{proof}

\subsection{Reduction from characteristic zero}
\label{sec:HI-reduction-at-two}

We now work in characteristic zero. Let $G$ be a nontrivial finite abelian
group of odd order $N$, let $A$ be a full HI-matrix,
and put $a=d^{-1}$. Theorem~\ref{thm:HI-omega-trivial} gives $\omega=1$.
We first prove that the entries are integral at every place above $2$.
The proof uses an exact Fourier identity before reduction; it does not
require the Q-system condition.

\begin{lemma}[An exact Fourier identity]\label{lem:HI-exact-fourier-lift}
Choose any total order on $G$, and define
\begin{equation}
 \mathcal T_{g,\chi}(A):=
 \sum_{k<l}\chi(k+l)^{-1}
 \bigl(A_{g+l,k}A_{g+k,l}-A_{g+k,k}A_{g+l,l}\bigr).
 \label{eq:HI-fourier-correction}
\end{equation}
For the Fourier transforms in \eqref{eq:display-0263}, one has
\begin{equation}
 X_{2g}(\chi^2)=X_g(\chi)^4+
 2X_g(\chi)^2\mathcal T_{g,\chi}(A)
 \qquad ((g,\chi)\ne(0,1)).
 \label{eq:HI-exact-fourier-lift}
\end{equation}
\end{lemma}

\begin{proof}
Set $h=g$ in the cubic identity \eqref{eq:HI-cubic}, multiply by
$\chi(k+l)^{-1}$, and sum over $k,l$. The left side becomes
$X_{2g}(\chi^2)Y_g(\chi^{-1})^2$. On the right, the product term is
symmetric in $k,l$. Separate its diagonal terms and its unordered pairs.
Since $A_{g+k,k}=A_{-k,g}$ by \eqref{eq:HImatrix1}, this gives
\begin{equation*}
 X_{2g}(\chi^2)Y_g(\chi^{-1})^2
 =X_g(\chi)^2+2\mathcal T_{g,\chi}(A)
   -aN^2\delta_{g,0}\delta_{\chi,1}.
\end{equation*}
At a nonexceptional pair, Lemma~\ref{lem:HI-Fourier-E3} gives
$X_g(\chi)Y_g(\chi^{-1})=1$. Multiply by $X_g(\chi)^2$ to obtain
\eqref{eq:HI-exact-fourier-lift}.
\end{proof}

\begin{theorem}[Integrality above two]\label{thm:HI-integrality-at-two}
Every full HI-matrix for a nontrivial finite abelian group of
odd order is integral at every place above $2$.
\end{theorem}

\begin{proof}
Fix such a place, and choose a number field containing the entries of $A$
and all character values. Write $v$ for its valuation, $R$ for its valuation
ring, and $F$ for its residue field. The scalars $a,d$ are algebraic units.
The character values and $N$ are also units in $R$, since $N$ is odd.

Suppose some entry is not integral, and put
$\mu=\min_{r,g}v(A_{r,g})<0$. Choose $c$ with $v(c)=-\mu$, set $C=cA$,
and let $B$ be its reduction. Every entry of $C$ lies in $R$, and an entry
of valuation $\mu$ in $A$ becomes a unit in $C$. Thus $B\ne0$.
We will show that all Fourier coefficients of $B$ vanish.

The transforms are linear in the entries and
$\mathcal T_{g,\chi}$ is homogeneous quadratic. Multiplying
\eqref{eq:HI-exact-fourier-lift} by $c^4$ therefore gives
\begin{equation}
 c^3X_{2g}(\chi^2;C)
 =X_g(\chi;C)^4+
 2X_g(\chi;C)^2\mathcal T_{g,\chi}(C).
 \label{eq:HI-scaled-fourier-lift}
\end{equation}
All terms are integral. Since $c$ and $2$ have positive valuation,
reduction gives $X_g(\bar\chi;B)^4=0$ at every nonexceptional pair.
The residue field has no nonzero nilpotents, so these coefficients vanish.
At the exceptional pair, \eqref{eq:HImatrix2} gives
$X_0(1;C)=-ca$, whose reduction is also zero.

Reduction identifies the odd-order characters with the characters over
$F$: roots of unity of odd order remain distinct modulo the chosen place.
Since $N$ is invertible in $F$, Fourier inversion now gives $B=0$,
contrary to its construction. Hence all entries of $A$ are integral.
\end{proof}

Fix a place of $\overline{\mathbb Q}$ above $2$ and identify its residue
field with $\overline{\mathbb F}_2$. Write $\bar A$ for reduction at this
place and $\mathcal C(A)$ for the reconstructed category. The following
comparison concerns tensor-equivalence classes, not labelled matrices.

\begin{proposition}[Good reduction and categorical comparison]
\label{prop:HI-good-reduction-at-two}
Every characteristic-zero Q-system HI-matrix $A$ has a split
integral categorical model whose special fibre is $\mathcal C(\bar A)$.
After extension to algebraically closed fields, $\mathcal C(A)$ is
equivalent to the characteristic-zero lift of $\mathcal C(\bar A)$.
For any two such matrices,
\begin{equation}
 \mathcal C(A)\simeq_\otimes\mathcal C(A')
 \quad\Longleftrightarrow\quad
 \mathcal C(\bar A)\simeq_\otimes\mathcal C(\bar A').
 \label{eq:HI-categorical-reduction-comparison}
\end{equation}
\end{proposition}

\begin{proof}
We construct the integral model and identify its two fibres before applying
lifting. Complete the coefficient field at the chosen place and enlarge it
finitely to contain $b$ with $b^2=a$. Let $R$ be the resulting complete
discrete valuation ring, with fraction field $K$. Theorem~
\ref{thm:HI-integrality-at-two} puts every entry in $R$.
The scalars $a,b,d,d-1,1-a,1+a$ are units, since
$\bar a^2+\bar a+1=0$. In particular, reduction preserves the Q-system
boundary condition.

\emph{The integral category.}
Let $L_R$ be the Leavitt algebra with the generators and relations used in
Appendix~\ref{app:reconstruction}. Its reduced-word basis is valid over
$R$: the reduction rules are monic. Thus $L_R$ is free over $R$, embeds
in $L_K$, and satisfies $L_R\cap K=R$.
The reconstruction formulas define $\rho,\alpha_g$ over $R$, since all
their coefficients lie in $R$. Their identities hold by the embedding
into $L_K$. For words $U,V$ in these endomorphisms, define their Hom module
by the intertwining equations in $L_R$. These equations give
\begin{equation*}
 \operatorname{Hom}_R(U,V)
 =L_R\cap\operatorname{Hom}_K(U,V).
\end{equation*}
Indeed, the equations extend from $R$ to $K$ and descend along the
injection $L_R\hookrightarrow L_K$.

For the labelled simples $S_i\in\{\alpha_g,\alpha_g\rho:g\in G\}$,
reconstruction over an algebraic closure of $K$ gives the generic Hom
spaces; the same intertwining equations descend them to $K$.
The intersection formula therefore gives
$\operatorname{Hom}_R(S_i,S_j)=R\delta_{i,j}$.
The integral splitting maps $s,t_g,s',t'_g$ decompose
$\rho^2\cong\mathbf1\oplus\bigoplus_g\alpha_g\rho$ and, recursively,
every word into these simples. Write $i_{U,j},p_{U,j}$ for the resulting
inclusions and projections. The composites $i_{V,l}p_{U,j}$ with equal
simple labels form an $R$-basis of $\operatorname{Hom}_R(U,V)$:
inserting the decompositions proves spanning, and composition with the
matching projection and inclusion extracts each coefficient.

These bases give the full generic Hom spaces after extending scalars.
After reduction they are the same splitting maps for $\bar A$, and
reconstruction in characteristic $2$ shows that they also give the full
special-fibre Hom spaces. Thus both fibre categories are the claimed
reconstructions. The additive category is idempotent complete, since
finite projective modules over the local ring $R$ are free. Duality is
integral: use evaluation $s'$ and coevaluation $a^{-1}s$ for $\rho$,
and inverse $\alpha_{-g}$ for $\alpha_g$.
Finally, the multiplication and separability element in
\eqref{eq:HI-fixed-boundary-multiplication} and
\eqref{eq:HI-fixed-boundary-separability-element} have coefficients in
$R$, because their denominators are units. Their identities hold in
$L_K$ and hence in $L_R$. This gives the connected separable algebra
$Q$ on both fibres.

\emph{Comparison by lifting.}
The special fibre has global dimension $N(1+\bar d^2)=\bar d\ne0$.
Apply \cite[Theorem~9.3, p.~635]{etingof2005fusion}.
To compare its Witt lift with the possibly ramified model just constructed,
enlarge the residue field to $\overline{\mathbb F}_2$ and base change
both models to a common complete discrete valuation ring.
The formal uniqueness argument applies modulo successive powers of a
uniformizer: at each square-zero step the degree-three deformation
cohomology vanishes, so a change of tensor bases identifies the two
associators while preserving the preceding identification. Completeness
gives a tensor equivalence over the ring and hence on the generic fibre.
Thus $\mathcal C(A)$ is the lift of $\mathcal C(\bar A)$.
Uniqueness of lifting proves the reverse implication in
\eqref{eq:HI-categorical-reduction-comparison}; faithfulness proves the
forward implication. For faithfulness we use \cite[Theorem~4.1 and Remark~4.3]{etingof2018faithfulness}. 
\end{proof}

\begin{corollary}[An orbit-count completeness criterion]
\label{cor:HI-categorical-completeness}
Fix $G$ and a characteristic-zero scalar root $d$, with reduction $\bar d$.
Suppose an exhaustive census of the full normalized characteristic-two
Q-system locus at $\bar d$ has exactly $m$ orbits under simultaneous
$\operatorname{Aut}(G)$-relabeling. Suppose that the displayed
characteristic-zero Q-system matrices at $d$ represent $m$ pairwise
tensor-inequivalent categories. Then they exhaust the characteristic-zero
tensor-equivalence classes at $d$. Reduction gives a bijection between the
characteristic-zero and characteristic-two categorical classes.
\end{corollary}

\begin{proof}
Matrices in the same $\operatorname{Aut}(G)$-orbit reconstruct
tensor-equivalent categories. Hence the characteristic-two locus has at
most $m$ categorical classes.

On the other hand, Proposition~\ref{prop:HI-good-reduction-at-two} shows
that reduction is injective on characteristic-zero categorical classes.
The reductions of the $m$ displayed pairwise inequivalent categories
therefore give $m$ pairwise inequivalent characteristic-two categories.
Thus the characteristic-two locus has exactly $m$ categorical classes,
and every one is represented by the reduction of a displayed matrix.

Finally, any characteristic-zero Q-system matrix reduces to one of these
characteristic-two classes. Equation~
\eqref{eq:HI-categorical-reduction-comparison} then shows that its category
is tensor equivalent to the corresponding displayed category.
\end{proof}

The criterion compares numbers of $\operatorname{Aut}(G)$-orbits and
categorical classes, not numbers of labelled matrices. It applies only
after the finite-field census has been proved exhaustive for the full
normalized Q-system locus.

\subsection{Recovering exact characteristic-zero matrices}
\label{sec:char2-algebraic-recovery}

We first enumerate cyclic Q-system solutions in characteristic two.
The Frobenius relation \eqref{eq:HI-char-two-fourth-power} determines
each orbit of entries under $(g,h)\mapsto(2g,2h)$ from one entry,
and Corollary~\ref{cor:HI-char-two-field-bound} bounds their field of
definition. Proposition~\ref{prop:HI-cyclic-Q-system-normal-form}
reduces the search to coordinates in a single row.

We lift the row coordinates $2$-adically, recognize their algebraic
presentations, and verify the reconstructed matrices exactly. This
calculation in matrix coordinates is separate from the categorical lifting
in Corollary~\ref{cor:HI-char-two-field-bound}.

For the cyclic Q-system matrices considered here, put
$B=(d-1)A$ and $v_g=B_{1,g}$.
Proposition~\ref{prop:HI-cyclic-Q-system-normal-form} shows that
$v_2,\ldots,v_{(N+1)/2}$ determine the matrix in both characteristics
zero and two. Write
\begin{equation*}
 Q_0=Q_1=1,\qquad Q_s=\prod_{j=2}^s v_j.
\end{equation*}
The reciprocal coordinates used in the characteristic-two census are
$u_g=(d-1)^{g-1}/Q_g$ and $t_g=u_g/u_{N-1}$ for $1\leq g<N$. 
The row symmetry gives $Q_gQ_{N-g}=Q_{N-1}$, and hence
$u_gu_{N-g}=u_{N-1}$. Consequently,
\begin{equation*}
 t_gt_{N-g}=t_1,\qquad
 A_{g,h}=\frac{t_gt_{h-g}}{t_1t_h}\quad(0<g<h<N).
\end{equation*}
The inverse change is $v_g=(d-1)t_{g-1}/t_g$ for $2\leq g<N$.
Both changes are defined over the same finite field. Thus the census in
these coordinates, followed by the remaining HI checks, covers the full
normalized cyclic Q-system locus.

\emph{Lifting the row equations.}
Let $(d_0,A_0)$ be a characteristic-two solution. The initial row is
\begin{equation*}
 v_g\equiv(d_0+1)(A_0)_{1,g}\pmod2.
\end{equation*}
Lift $d_0$ to a root of $d^2-Nd-1=0$ in the unramified $2$-adic
coefficient ring. We use the selected E3 equations
\begin{equation}\label{eq:char2-row-hensel}
 F_h(v):=v_h+v_{h+1}
 +\frac{1+d}{N}\sum_{r\in\mathbb Z/N}\frac{v_r}{v_{r+h}}=0,
 \qquad 1\leq h\leq\frac{N-1}{2}.
\end{equation}
The indices in the sum are cyclic. Since $N$ is odd and the row coordinates
are units, these equations are defined over that ring. When their Jacobian
is invertible modulo $2$, Newton--Hensel iteration gives a unique lift of
this row-system solution, successively modulo $2,2^2,2^4,2^8,\ldots$.
The reduction selects the branch, but solving
\eqref{eq:char2-row-hensel} does not establish the other HI identities.

\emph{Algebraic recognition and verification.}
Suppose that $\lambda_i=v_2^{(i)}$ distinguishes the collected lifted rows.
Form
\begin{equation*}
 P(X)=\prod_i(X-\lambda_i),\qquad
 R_j(X)=\sum_i v_j^{(i)}\frac{P(X)}{X-\lambda_i}.
\end{equation*}
Then $v_j^{(i)}=R_j(\lambda_i)/P'(\lambda_i)$. If this coordinate does not
separate the rows, one first chooses a separating linear combination.
We recognize the coefficients of $P$ and $R_j$ in $\mathbb Q(d)$ or a
specified auxiliary extension, obtaining a candidate presentation
\begin{equation*}
 P(\lambda)=0,\qquad v_j=\frac{R_j(\lambda)}{P'(\lambda)}.
\end{equation*}
An alternative is to interpolate the unit-relabeling orbit of a complex
numerical matrix; the product reconstruction \eqref{eq:char2-row-recovery}
is the same. Numerical or $2$-adic recognition proposes the coefficients.
We then check denominator nonvanishing and verify the remaining E3 and
cubic identities exactly, either in the coefficient field or by modular
calculations with rigorous coefficient or norm bounds. The cubic criterion
then gives the full HI system.

At $N=43$, this procedure used all eighty-four
characteristic-two solutions and $1024$ reliable $2$-adic bits. It recovered
a degree-$84$ polynomial over $\mathbb Q(d)$; the subsequent exact modular
checks with norm bounds verified all remaining E3 and cubic identities.

The ordinary fourth-power map is not a characteristic-zero field
endomorphism. On unramified lifted coefficients the corresponding operation
is the lift of Frobenius, rather than ordinary fourth powering. Thus the
characteristic-two identity \eqref{eq:HI-char-two-fourth-power} must not be
imposed unchanged on the recovered characteristic-zero matrix.

\emph{Comparison with the characteristic-zero data.}
Entrywise Frobenius gives an $\operatorname{Aut}(G)$-equivariant
bijection between the characteristic-two solution sets at $d$ and
$d^2$, so their orbit counts agree. 
For each odd $N\leq43$ and each scalar fibre, the exhaustive
characteristic-two census has the same number of
$\operatorname{Aut}(\mathbb Z/N)$-orbits as there are pairwise
tensor-inequivalent displayed characteristic-zero categories.
The characteristic-zero representatives are inequivalent by the cyclic
equivalence criterion recalled after Theorem~\ref{thm:reconstruction}.
Corollary~\ref{cor:HI-categorical-completeness} therefore makes reduction
a bijection on categorical classes and proves that the displayed list is
complete up to tensor equivalence.

\bibliographystyle{alpha}
\bibliography{references}
\label{page:main-end}
\clearpage
\appendix


\section{The cubic criterion and the scalar obstruction}
\label{app:equations}\label{subsec:HI-general-equation-comparison}
We prove the equation criterion in the general scalar setup
\eqref{eq:HI-scalars}--\eqref{eq:HI-a}, allowing the characteristic to divide
$N$. The two implications use different hypotheses, which we state separately.
We then impose $N\ne0$ for the scalar obstruction.

\subsection{From the cubic identity to the HI equations}
\begin{theorem}[The symmetry and cubic identities imply the sum and quadratic identities]
\label{thm:HI-cubic-implies-quadratic}
Let $R$ be a field, let $G$ be a nontrivial finite abelian group of
cardinality $N$, and choose $a,\omega\in R$ such that
\begin{equation*}
 a\neq0,
 \qquad
 \omega^3=1.
\end{equation*}
Suppose that $A=(A_{g,h})_{g,h\in G}\in M_G(R)$ satisfies
\eqref{eq:HImatrix1} and \eqref{eq:HI-cubic}, interpreted over $R$.
Then $A$ satisfies \eqref{eq:HImatrix2} and \eqref{eq:HImatrix3}.
\end{theorem}

\begin{proof}
Write $\delta_x:=\delta_{x,0}$ and define
\begin{equation*}
 P(x,y):=\sum_{r\in G}A_{x+r,y}A_{y,r},
 \qquad
 F(x,y):=P(x,y)-(\delta_x-a\delta_y).
\end{equation*}
Equation~\eqref{eq:HImatrix3} is equivalent to $F=0$.  We divide the
argument into three parts.  First we derive an identity for the error term
$F$ and prove that the row indices of the nonzero entries of $A$ generate
$G$.  We then prove that $F(x,h)=0$ for $x\neq0$.  Finally, we show that
$F(0,h)$ is independent of $h$, determine its value, and obtain
\eqref{eq:HImatrix2}.

\smallskip
\noindent\textbf{Part 1: the identity for \(F\) and the nonzero row indices.}
For all $x,h,u,v\in G$,
\begin{equation}
 A_{u,v}F(x,h)=A_{u,v+x}F(x,h+u).
 \label{eq:HI-cubic-R-residual}
\end{equation}
Moreover, if
\begin{equation*}
 \Sigma_A:=\{u\in G:A_{u,v}\neq0\text{ for some }v\in G\},
 \qquad
 H_A:=\langle\Sigma_A\rangle,
\end{equation*}
then $H_A=G$.

\smallskip
\noindent\emph{Derivation of Part 1.}
For $g,h,i,k\in G$, put
\begin{equation*}
 \Xi(g,h,i,k):=
 \sum_{r,m\in G}
 A_{r,m}A_{r+g,h}A_{h+m,r+i}A_{i,k+m}.
\end{equation*}
To derive \eqref{eq:HI-cubic-R-residual}, evaluate $\Xi$ in two ways and
equate the results.
In \eqref{eq:HI-cubic}, replace $(g,h,k,l)$ by
$(r,i,-h,k-h)$, replace the old summation variable by $q+h$, and then
rename $q$ as $m$.  This gives
\begin{equation}
 \sum_{m\in G} A_{h+m,r+i} A_{r,m}A_{i,k+m}
 =\omega A_{r+k-h,-h}A_{i-h,k-h}
  -\omega a\delta_r\delta_i.
 \label{eq:display-0073}
\end{equation}
Multiply \eqref{eq:display-0073} by $A_{r+g,h}$ and sum over $r$.
The resulting left-hand side is $\Xi(g,h,i,k)$. For the first term on the
right, \eqref{eq:HImatrix1} gives $A_{r+k-h,-h}=\omega A_{h,r+k}$. 
After the change of variable $s=r+k$, the corresponding sum is
\begin{equation*}
 \sum_{r\in G}A_{r+g,h}A_{h,r+k}
 =\sum_{s\in G}A_{g-k+s,h}A_{h,s}
 =P(g-k,h).
\end{equation*}
Since $\delta_r=0$ unless $r=0$, we have
\begin{equation*}
 \sum_{r\in G}A_{r+g,h}\delta_r=A_{g,h}.
\end{equation*}
Combining these equalities gives the first expression for $\Xi$:
\begin{equation}
 \Xi(g,h,i,k)
 =\omega^2A_{i-h,k-h}P(g-k,h)
  -\omega aA_{g,h}\delta_i.
 \label{eq:HI-cubic-Xi-first}
\end{equation}
For the second expression, \eqref{eq:HImatrix1} gives
$A_{r+g,h}=\omega A_{-h,r+g-h}$. 
Apply \eqref{eq:HI-cubic} with
$(g,h,k,l)=(-h,h+m,g-h,i)$ and multiply by $\omega$.  We obtain
\begin{equation*}
 \sum_{r\in G}A_{r,m}A_{-h,r+g-h}A_{h+m,r+i}
 =\omega A_{i-h,g-h}A_{m+g,i}
  -\omega a\delta_h\delta_{h+m}.
\end{equation*}
Multiply this identity by $\omega A_{i,k+m}$ and sum over $m$.  After the
change of variable $s=k+m$, the first resulting sum is
\begin{equation*}
 \sum_{m\in G}A_{m+g,i}A_{i,k+m}
 =\sum_{s\in G}A_{g-k+s,i}A_{i,s}
 =P(g-k,i).
\end{equation*}
If $\delta_h\neq0$, then $h=0$, and $\delta_{h+m}\neq0$ only for $m=0$.
Hence
\begin{equation*}
 \delta_h\sum_{m\in G}\delta_{h+m}A_{i,k+m}
 =\delta_hA_{i,k}.
\end{equation*}
This gives the second expression for $\Xi$:
\begin{equation}
 \Xi(g,h,i,k)
 =\omega^2A_{i-h,g-h}P(g-k,i)
  -\omega^2a\delta_hA_{i,k}.
 \label{eq:HI-cubic-Xi-second}
\end{equation}
Equate \eqref{eq:HI-cubic-Xi-first} and
\eqref{eq:HI-cubic-Xi-second}, divide by the unit $\omega^2$, and put
$x=g-k$.  We obtain
\begin{align}
 A_{i-h,k-h}P(x,h)-A_{i-h,k+x-h}P(x,i)
 =\omega^{-1}a\delta_iA_{k+x,h}
  -a\delta_hA_{i,k}.
 \label{eq:HI-cubic-Q-residual}
\end{align}
Evaluate the left-hand side of
\eqref{eq:HI-cubic-Q-residual} at the desired values
$P(x,h)=\delta_x-a\delta_h$ and
$P(x,i)=\delta_x-a\delta_i$.  It becomes
\begin{equation*}
 \delta_x\bigl(A_{i-h,k-h}-A_{i-h,k+x-h}\bigr)
 -a\delta_hA_{i-h,k-h}
 +a\delta_iA_{i-h,k+x-h}.
\end{equation*}
The first term is zero: if $\delta_x\neq0$, then $x=0$ and the two matrix
entries in parentheses are equal.  The other two terms give exactly the
right-hand side of \eqref{eq:HI-cubic-Q-residual}, because
\begin{equation*}
 \delta_hA_{i-h,k-h}=\delta_hA_{i,k},
 \qquad
 \delta_iA_{i-h,k+x-h}
 =\omega^{-1}\delta_iA_{k+x,h}.
\end{equation*}
The desired values of $P$ satisfy the same equation as $P$ itself.
Subtracting the two equations and setting
$u=i-h$, $v=k-h$ gives the identity in
\eqref{eq:HI-cubic-R-residual}.
The change of variables
\begin{equation*}
 x=g-k,\qquad u=i-h,\qquad v=k-h
\end{equation*}
is bijective, with inverse
\begin{equation*}
 g=x+v+h,\qquad i=u+h,\qquad k=v+h.
\end{equation*}
Hence $x,h,u,v$ in \eqref{eq:HI-cubic-R-residual} are independent
variables.

If $A_{u,v}\neq0$, then $u\in\Sigma_A\subseteq H_A$.  By
\eqref{eq:HImatrix1},
\begin{equation}
 A_{-v,u-v}=\omega^{-1}A_{u,v}\neq0,
 \label{eq:HI-cubic-row-index-transfer}
\end{equation}
so $-v\in\Sigma_A\subseteq H_A$, and hence $v\in H_A$.

We claim that $H_A=G$. Otherwise choose $z\in G\setminus H_A$ and take
$(g,h,k,l)=(0,0,z,z)$ in \eqref{eq:HI-cubic}.  The $m$th summand on its
left-hand side is $A_{m,0}A_{0,m+z}^2$.  If this were nonzero, then
$A_{m,0}\neq0$ would imply $m\in H_A$, and
$A_{0,m+z}\neq0$ would imply $m+z\in H_A$ by
\eqref{eq:HI-cubic-row-index-transfer}.  This would give $z\in H_A$, a
contradiction.  Thus the left-hand side is zero.  Also, $A_{z,z}=0$, since
otherwise the same implication would give $z\in H_A$.  The resulting identity
is therefore $0=-a$, which contradicts $a\neq0$.

\smallskip
\noindent\textbf{Part 2: the case \(x\neq0\).}
For every $x\neq0$ and $h\in G$,
\begin{equation}
 F(x,h)=0.
 \label{eq:display-0082}
\end{equation}

\smallskip
\noindent\emph{Proof of Part 2.}
Fix $x\neq0$ and suppose, for a contradiction, that
$F(x,h_0)\neq0$ for some $h_0\in G$.  Write
$f_h:=F(x,h)$.  We first show that this one nonzero value forces
$f_h\neq0$ for every $h$.  Two substitutions in the cubic identity will then
give the contradiction $a=0$.
If $u\in\Sigma_A$, choose $v$ with $A_{u,v}\neq0$.
Equation~\eqref{eq:HI-cubic-R-residual} shows that
\begin{equation*}
 f_h\neq0 \quad\Longrightarrow\quad f_{h+u}\neq0.
\end{equation*}
Indeed, the left-hand side of \eqref{eq:HI-cubic-R-residual} is nonzero.
Since $R$ is a field, the right-hand side is nonzero, and hence
$f_{h+u}\neq0$.
Starting with $h=h_0$ and iterating this implication gives
$f_{h_0+w}\neq0$ for every $w$ in the submonoid generated by $\Sigma_A$.
Since $G$ is finite, this submonoid is the subgroup
$\langle\Sigma_A\rangle=G$.  Hence $f_h\neq0$ for every $h\in G$.

Set $u=0$ in \eqref{eq:HI-cubic-R-residual}. Since $f_h\neq0$, cancellation
gives $A_{0,v+x}=A_{0,v}$.  Equation~\eqref{eq:HImatrix1} gives
$A_{q,0}=\omega A_{0,q}$, so the same equality also holds for the first
column.  Thus, for a fresh variable $q\in G$,
\begin{equation}
 A_{0,v+x}=A_{0,v},
 \qquad
 A_{q+x,0}=A_{q,0}.
 \label{eq:HI-cubic-boundary-translation}
\end{equation}

For $q,h\in G$, equation \eqref{eq:HI-cubic-R-residual}, first with
$(u,v)=(q,0)$ and then with $(u,v)=(-q,-q)$ and $h$ replaced by $h+q$,
gives
\begin{equation}
 A_{q,0}f_h=A_{q,x}f_{h+q},
 \qquad
 A_{-q,-q}f_{h+q}=A_{-q,x-q}f_h.
 \label{eq:HI-cubic-product-translation-prep}
\end{equation}
Since $f_h$ and $f_{h+q}$ are nonzero, multiply the two equalities in
\eqref{eq:HI-cubic-product-translation-prep} and cancel their product.
Equation~\eqref{eq:HImatrix1} gives
$A_{-q,-q}=\omega^{-1}A_{0,q}$ and
$A_{-q,x-q}=\omega^{-1}A_{x,q}$.  Hence
\begin{equation}
 A_{q,x}A_{x,q}=A_{q,0}A_{0,q}.
 \label{eq:HI-cubic-product-translation}
\end{equation}

Fix $l\in G$.  Substituting $(g,h,k,l)=(0,0,x,l)$ in
\eqref{eq:HI-cubic} gives
\begin{equation*}
 \omega^{-1}\sum_{m\in G}A_{m,0}A_{0,m+x}A_{0,m+l}
 =A_{l,x}A_{x,l}-a.
\end{equation*}
On the left, \eqref{eq:HI-cubic-boundary-translation} with $q=m$ replaces
$A_{0,m+x}$ by $A_{0,m}$.  On the right,
\eqref{eq:HI-cubic-product-translation} with $q=l$ replaces
$A_{l,x}A_{x,l}$ by $A_{l,0}A_{0,l}$.  Hence
\begin{equation}
 \omega^{-1}\sum_{m\in G}A_{m,0}A_{0,m}A_{0,m+l}
 =A_{l,0}A_{0,l}-a.
 \label{eq:display-0080}
\end{equation}

Substituting $(g,h,k,l)=(x,0,0,l)$ in
\eqref{eq:HI-cubic} gives
\begin{equation*}
 \omega^{-1}\sum_{m\in G}A_{m,x}A_{x,m}A_{0,m+l}
 =A_{x+l,0}A_{0,l}.
\end{equation*}
The delta term is zero because $x\neq0$.  On the left,
\eqref{eq:HI-cubic-product-translation} with $q=m$ replaces
$A_{m,x}A_{x,m}$ by $A_{m,0}A_{0,m}$.  On the right,
\eqref{eq:HI-cubic-boundary-translation} with $q=l$ replaces
$A_{x+l,0}$ by $A_{l,0}$.  Hence
\begin{equation}
 \omega^{-1}\sum_{m\in G}A_{m,0}A_{0,m}A_{0,m+l}
 =A_{l,0}A_{0,l}.
 \label{eq:display-0081}
\end{equation}
Comparing \eqref{eq:display-0080} and \eqref{eq:display-0081} gives
$a=0$, contradicting the assumption $a\neq0$.  Therefore there is no $h_0\in G$
for which $F(x,h_0)\neq0$.  Hence $F(x,h)=0$ for every $h\in G$.
Since $x\neq0$ was arbitrary, this proves \eqref{eq:display-0082}.

\smallskip
\noindent\textbf{Part 3: the case \(x=0\).}
For every $h\in G$, one has $F(0,h)=0$.  The matrix $A$ also satisfies
\eqref{eq:HImatrix2}.

\smallskip
\noindent\emph{Proof of Part 3.}
For $x=0$, first show that the error is independent of the second variable,
then determine this constant together with the row sum.
Let $h\in G$ be arbitrary.  For each $u\in\Sigma_A$, choose $v$ with
$A_{u,v}\neq0$.  Setting $x=0$ in
\eqref{eq:HI-cubic-R-residual} and cancelling $A_{u,v}$ gives
\begin{equation*}
 F(0,h)=F(0,h+u).
\end{equation*}
Since $\Sigma_A$ generates $G$, the value $F(0,h)$ is independent of $h$.
Write this constant as $\lambda-1$. Then the desired conclusion
$F(0,h)=0$ is equivalent to $\lambda=1$, and
\begin{equation*}
 F(0,h)=\lambda-1
 \qquad(h\in G).
\end{equation*}
By the definition of $F$, this 
\begin{equation*}
 P(0,h)=\lambda-a\delta_h
 \qquad(h\in G).
\end{equation*}
For $x\neq0$, Part 2 gives $F(x,h)=0$, so
$P(x,h)=\delta_x-a\delta_h=-a\delta_h$.  Combining the two cases gives
\begin{equation}
 P(x,h)=\lambda\delta_x-a\delta_h
 \qquad(x,h\in G).
 \label{eq:HI-cubic-Q-defect}
\end{equation}

It remains to determine $\lambda$. Put
\begin{equation*}
 c:=\omega^{-1}\sum_{y\in G}A_{0,y}.
\end{equation*}
Set $h=0$ in \eqref{eq:HI-cubic} and sum over $l$.  The sum of the product
term over $l$ is $P(g,k)$.  In the cubic sum, summing first over $l$ gives
$\omega c$.  Next use
$A_{m,g}=\omega A_{-g,m-g}$ and
$A_{g,m+k}=\omega^{-1}A_{m+k-g,-g}$, and put $r=m-g$.  The remaining sum
over $m$ is then $P(k,-g)$.  The constant term contributes
$N a\delta_g$.  Hence
\begin{equation}
 P(g,k)-cP(k,-g)=N a\delta_g.
 \label{eq:HI-cubic-Q-contraction}
\end{equation}
Choose a nonzero element of $G$.  Evaluating
\eqref{eq:HI-cubic-Q-contraction} first at $g\neq0$, $k=0$ and at
$g=0$, $k\neq0$ gives
\begin{equation}
 -a-c\lambda=0,
 \qquad
 \lambda+ac=N a.
 \label{eq:display-0084}
\end{equation}
The scalar $a$ is a unit because $R$ is a field and $a\neq0$.  The first
equality in \eqref{eq:display-0084} therefore implies $c\neq0$.  We record
this as $c\lambda=-a$. 
For one more relation, sum \eqref{eq:HI-cubic-Q-defect} over $x$ with
$h=0$.  Its left-hand side factors as
\begin{equation}
 \left(\sum_xA_{x,0}\right)\left(\sum_rA_{0,r}\right),
 \label{eq:display-0085}
\end{equation}
and its right-hand side is $\lambda-N a$.  By
\eqref{eq:HImatrix1} and the definition of $c$, the two factors in
\eqref{eq:display-0085} are $\omega^2c$ and $\omega c$.  Therefore $c^2=\lambda-N a$. 
The second equality in \eqref{eq:display-0084} now gives
 $c^2=-ac$. 
Since $c\neq0$, we obtain $c=-a$ and $\lambda=1$. 
Therefore \eqref{eq:HI-cubic-Q-defect} is exactly
\eqref{eq:HImatrix3}.  Finally,
\begin{equation*}
 \sum_{r\in G}A_{r,0}
 =\omega^2c=-\omega^{-1}a,
\end{equation*}
which is \eqref{eq:HImatrix2}.

Parts 2 and 3 show that $F=0$, so
\eqref{eq:HImatrix3} holds.  Part 3 also gives
\eqref{eq:HImatrix2}.  This proves the theorem.
\end{proof}

The question whether \eqref{eq:HI-cubic}, together with
\eqref{eq:HImatrix1}--\eqref{eq:HImatrix3}, implies the quartic equation
\eqref{eq:HImatrix4} is raised in
\cite[equation~(4.11), p.~14, and Proposition~2, p.~29]{evans2017non}.
The comparison uses one map. For the vectors $B$ and $C$ below, the cubic
identity is $B=C$, while \eqref{eq:HImatrix4} is
$\mathcal A_hB=\mathcal A_hC$.  One direction follows by applying
$\mathcal A_h$; the reverse direction follows when $\mathcal A_h$ is
injective.  The next lemma proves these statements.

\begin{lemma}[Convolution operator]
\label{lem:HI-convolution-operator}
Let $R$ be a commutative ring, let $G$ be a finite abelian group of
cardinality $N$, and choose $a,\omega\in R$ with $\omega^3=1$.  Suppose
that $A=(A_{g,h})_{g,h\in G}\in M_G(R)$ satisfies
\begin{equation}
 A_{g,h}=\omega A_{-h,g-h},
 \qquad g,h\in G,
 \label{eq:HI-order-three-ring}
\end{equation}
and \eqref{eq:HImatrix3}.  For $h\in G$, define
\begin{equation*}
 (\mathcal A_hX)_g:=\sum_{r\in G}A_{r+g,h}X_r
 \qquad(X\in R^G).
\end{equation*}
If $1-N a$ is a unit in $R$, then $\mathcal A_h$ is injective for every
$h\in G$.  The same conclusion holds if $A$ also satisfies
\eqref{eq:HImatrix2} and $a$ is not a zero-divisor in $R$.

For fixed $h,i,k\in G$, define $B,C\in R^G$ by
\begin{align*}
 B_r&:=\sum_{m\in G}
 A_{r,m}A_{h+m,r+i}A_{i,k+m},
\\
 C_r&:=\omega A_{r+k-h,-h}A_{i-h,k-h}
       -\omega a\delta_{r,0}\delta_{i,0}.
\end{align*}
Then the cubic identity is equivalent to $B=C$, and
\eqref{eq:HImatrix4} is equivalent to
$\mathcal A_hB=\mathcal A_hC$ for every $h,i,k\in G$.
\end{lemma}

\begin{proof}
Equations \eqref{eq:HI-order-three-ring} and \eqref{eq:HImatrix3} give
the convolution identity
\begin{align}
 \sum_{r\in G}A_{r+x,h}A_{r,-h}
 =\omega\sum_{u\in G}A_{u+x-h,h}A_{h,u}
 =\omega\bigl(\delta_{x,h}-a\delta_{h,0}\bigr).
 \label{eq:HI-cubic-convolution}
\end{align}
For the first equality, use
$A_{r,-h}=\omega A_{h,r+h}$ and put $u=r+h$.  The second equality is
\eqref{eq:HImatrix3} with its two fixed indices equal to $(x-h,h)$.

We next prove injectivity. Suppose that $\mathcal A_hX=0$.
For fixed $s\in G$, multiply $(\mathcal A_hX)_g=0$ by
$A_{g-s,-h}$ and sum over $g$.  Applying
\eqref{eq:HI-cubic-convolution} to the inner sum gives
\begin{equation}
 0=\omega\left(X_{h-s}-a\delta_{h,0}\sum_rX_r\right).
 \label{eq:display-0097}
\end{equation}
If $h\neq0$, varying $s$ shows that every coordinate of $X$ is zero.  If
$h=0$, summing \eqref{eq:display-0097} over $s$ gives
$(1-N a)\sum_rX_r=0$.  When $1-N a$ is a unit, this again gives
$X=0$.

For the second injectivity criterion, suppose that $A$ satisfies \eqref{eq:HImatrix2} and that $a$ is
not a zero-divisor.  The case $h\neq0$ is unchanged.  If $h=0$, then
\eqref{eq:display-0097} shows that all coordinates of $X$ are equal to a
scalar $c$.  Translation by $g$ permutes $G$.  Since $\mathcal A_0X=0$,
equation \eqref{eq:HImatrix2} therefore gives
\begin{equation*}
 0=(\mathcal A_0X)_g
   =c\sum_rA_{r+g,0}
   =-\omega^{-1}ac.
\end{equation*}
Thus $c=0$, and $\mathcal A_0$ is also injective.

The substitution
\begin{equation*}
 (g,h,k,l)=(r,i,-h,k-h),
 \qquad q=m+h,
\end{equation*}
shows that the cubic identity, multiplied by $\omega$, is
\begin{equation*}
 B_r=C_r.
\end{equation*}
The definition of $B$ gives
\begin{align*}
 (\mathcal A_hB)_g
 =\sum_{r,m\in G}
 A_{r+g,h}A_{r,m}A_{h+m,r+i}A_{i,k+m} 
 =A_{h-g,i-g}\delta_{k,g}
  -\omega^{-1}a\delta_{h,0}A_{i,k}
  -\omega aA_{g,h}\delta_{i,0}
\end{align*}
if and only if \eqref{eq:HImatrix4} holds.  For $C$, we have
\begin{align}
 (\mathcal A_hC)_g
 &=\omega A_{i-h,k-h}
   \sum_rA_{r+g,h}A_{r+k-h,-h}
   -\omega aA_{g,h}\delta_{i,0}\notag
\\
 &=\omega^2A_{i-h,k-h}
   \bigl(\delta_{g,k}-a\delta_{h,0}\bigr)
   -\omega aA_{g,h}\delta_{i,0}.
 \label{eq:display-0101}
\end{align}
Using \eqref{eq:HI-cubic-convolution} after the change of variable
$u=r+k-h$ gives
\begin{equation*}
 \sum_rA_{r+g,h}A_{r+k-h,-h}
 =\omega\bigl(\delta_{g,k}-a\delta_{h,0}\bigr).
\end{equation*}
When $g=k$, equation \eqref{eq:HI-order-three-ring} gives
$\omega^2A_{i-h,k-h}=A_{h-g,i-g}$.  When $h=0$, it gives
$A_{i-h,k-h}=A_{i,k}$.  Thus \eqref{eq:display-0101} is the right-hand side
of \eqref{eq:HImatrix4}.  Hence \eqref{eq:HImatrix4} is equivalent to
$\mathcal A_hB=\mathcal A_hC$.
\end{proof}

\begin{theorem}[Cubic implies quartic]
\label{thm:HI-cubic-implies-quartic}
Let $R$ be a commutative ring, let $G$ be a finite abelian group, and choose
$a\in R$ and $\omega\in R$ with $\omega^3=1$.  Suppose that
$A=(A_{g,h})_{g,h\in G}\in M_G(R)$ satisfies
\eqref{eq:HI-order-three-ring}, \eqref{eq:HImatrix3}, and
\eqref{eq:HI-cubic}, interpreted over $R$.  Then $A$ satisfies
\eqref{eq:HImatrix4}.
\end{theorem}

\begin{proof}
For the vectors in Lemma~\ref{lem:HI-convolution-operator}, the cubic
identity gives $B=C$. Apply $\mathcal A_h$ and use the last assertion of that
lemma to obtain \eqref{eq:HImatrix4}.
\end{proof}

\begin{theorem}[Quartic implies cubic]
\label{thm:HI-quartic-implies-cubic}
Let $R$ be a commutative ring, let $G$ be a finite abelian group of
cardinality $N$, and choose $a,\omega\in R$ such that
\begin{equation*}
 1-N a\in R^\times,
 \qquad
 \omega^3=1.
\end{equation*}
If $A=(A_{g,h})_{g,h\in G}\in M_G(R)$ satisfies
\eqref{eq:HI-order-three-ring}, \eqref{eq:HImatrix3}, and
\eqref{eq:HImatrix4}, interpreted over $R$, then it satisfies the cubic
identity \eqref{eq:HI-cubic}.
\end{theorem}

\begin{proof}
Lemma~\ref{lem:HI-convolution-operator} shows that every $\mathcal A_h$ is
injective.  Equation~\eqref{eq:HImatrix4} gives
$\mathcal A_hB=\mathcal A_hC$, so $B=C$.  By the same lemma, this is the
cubic identity.
\end{proof}

\begin{corollary}[General equivalence of the two HI equation systems]
\label{cor:HI-cubic-full-equivalence-general}
\label{cor:HI-cubic-criterion}
In the scalar setup \eqref{eq:HI-scalars}--\eqref{eq:HI-a}, suppose that
$G$ is nontrivial.  A matrix satisfies
\eqref{eq:HImatrix1}--\eqref{eq:HImatrix4} if and only if it satisfies
\eqref{eq:HImatrix1} and the cubic identity \eqref{eq:HI-cubic}.
\end{corollary}

\begin{proof}
For the implication from the symmetry and cubic identities to the full
system, Theorem~\ref{thm:HI-cubic-implies-quadratic} gives the sum and
quadratic identities, and Theorem~\ref{thm:HI-cubic-implies-quartic} gives the
quartic identity. Conversely, apply
Theorem~\ref{thm:HI-quartic-implies-cubic}, using $1-N a=a^2$, which is a
unit in the scalar field.
\end{proof}

\begin{remark}
Theorem~\ref{thm:HI-cubic-implies-quadratic} uses neither oddness nor
cyclicity of $G$ nor algebraic closedness of the coefficient field.  It only
requires $a\neq0$.  It does not divide by $2$, $3$, $N$, or $N^2+4$,
and hence applies when the characteristic divides $|G|$.
Theorem~\ref{thm:HI-cubic-implies-quartic} works over a commutative ring and
does not use any scalar relation involving $a$.  The converse only requires
$1-N a$ to be a unit.  Thus both directions remain valid in every
characteristic under their stated hypotheses.
\end{remark}

\subsection{The scalar cube-root}\label{app:omega}
For the rest of this appendix, the field is algebraically closed, $G$ has
odd order, and $N\ne0$ in the field. We prove
Theorem~\ref{thm:HI-omega-trivial}. The row and trace sums and the Fourier
identity below are the preliminary computations.

The determinant argument will also use the following two sums.

\begin{lemma}[Row and trace sums]
\label{lem:HI-row-trace-sums}
Every HI-matrix satisfies
\begin{align}
 \sum_{r\in G}A_{0,r}&=-\omega a,
 \label{eq:HI-row-sum}\\
 \sum_{r,k\in G}A_{r,k}A_{k,r}&=N-a.
 \label{eq:HI-trace-sum}
\end{align}
\end{lemma}

\begin{proof}
Putting $(g,h)=(0,r)$ and then $(g,h)=(-r,-r)$ in
\eqref{eq:HImatrix1} gives
$A_{0,r}=\omega^2A_{r,0}=\omega^{-1}A_{r,0}$.  Summing and using
\eqref{eq:HImatrix2} proves \eqref{eq:HI-row-sum}.  To prove
\eqref{eq:HI-trace-sum}, put $h=0$ in \eqref{eq:HImatrix3} and sum over
$k$.
\end{proof}

We next record the Fourier form of \eqref{eq:HImatrix3}. The sum in that
equation repeatedly translates a group element. A character turns each such
translation into multiplication by a scalar, so the sum factors. This form is
used in both the obstruction theorem and the
characteristic-$2$ arguments below.  For the next lemma, let $\kk$ be
algebraically closed, let $G$ be a finite
abelian group of order $N$, choose $a\in\kk$,
assume $\operatorname{char}(\kk)\nmid N$, and write
$\widehat G:=\Hom(G,\kk^\times)$.  For a matrix
$A=(A_{g,h})_{g,h\in G}$ and for $t\in G$, $\chi\in\widehat G$, define the
column and row transforms by
\begin{equation}
 X_t(\chi):=\sum_{r\in G}A_{r,t}\chi(r),
 \qquad
 Y_t(\chi):=\sum_{r\in G}A_{t,r}\chi(r).
 \label{eq:display-0263}
\end{equation}

\begin{lemma}[Fourier form of \texorpdfstring{\eqref{eq:HImatrix3}}{E3}]
\label{lem:HI-Fourier-E3}
If $A$ satisfies \eqref{eq:HImatrix3}, then
\begin{equation*}
 X_t(\chi)Y_t(\chi^{-1})
 =1-N a\delta_{t,0}\delta_{\chi,1}
 \qquad(t\in G,\ \chi\in\widehat G).
\end{equation*}
\end{lemma}

\begin{proof}
Set $k=t$ in \eqref{eq:HImatrix3}, multiply by $\chi(h)$, and sum over
$h\in G$.  On the left, put $u=h+r$.  Since
$\chi(h)=\chi(u)\chi(r)^{-1}$, the double sum is
$X_t(\chi)Y_t(\chi^{-1})$ by \eqref{eq:display-0263}. On the right, use
$\sum_h\chi(h)=N\delta_{\chi,1}$.
\end{proof}

The scalar $a$ may be either root of \eqref{eq:HI-a} when the two
roots are distinct. The proof computes the square of the product of all row Fourier coefficients
in two ways. Write this product as $\Pi$. The first calculation uses
\eqref{eq:HImatrix1} and \eqref{eq:HImatrix3}. For the second, the cubic
identity gives a trace Gram matrix, whose rank-one correction can be evaluated
explicitly. Comparing the two formulas for $\Pi^2$ then gives $\omega=1$.

\begin{proof}[Proof of Theorem~\ref{thm:HI-omega-trivial}]
If $G=\{0\}$, the conclusion is Remark~\ref{rem:HI-trivial-group}. We may
therefore assume that $G$ is nontrivial.  By definition, $A$ satisfies
\eqref{eq:HImatrix2} and \eqref{eq:HImatrix3}.
Theorem~\ref{thm:cubic-criterion} also gives
\eqref{eq:HI-cubic}.  As in the
section setup, $N$ in a scalar identity denotes its image in $\kk$, whereas
$N^2$ in an exponent denotes the integer $|G\times\widehat G|$.

\smallskip
\noindent\emph{Part 1. The first formula and the trace Gram determinant.}
We first obtain \eqref{eq:HI-omega-quadratic-product}, then construct a
matrix basis and express its trace Gram determinant in terms of the same
Fourier product.

For $g\in G$ and $\chi\in\widehat G$, put
\begin{equation*}
 \varepsilon(g,\chi):=1-N a\delta_{g,0}\delta_{\chi,1},
 \qquad
 \Pi:=\prod_{g\in G,\,\chi\in\widehat G}Y_g(\chi).
\end{equation*}
The scalar relation gives $\varepsilon(0,1)=1-N a=a^2$, so every value of
$\varepsilon$ is nonzero.  Equation~\eqref{eq:HImatrix1} gives
\begin{equation}
 X_k(\chi)=\omega\chi(k)Y_{-k}(\chi).
 \label{eq:HI-omega-column-transform}
\end{equation}
Combining this with Lemma~\ref{lem:HI-Fourier-E3} gives
\begin{equation}
 \omega\chi(k)Y_{-k}(\chi)Y_k(\chi^{-1})
 =\varepsilon(k,\chi).
 \label{eq:HI-omega-inverse-pair}
\end{equation}
In particular, every row and column Fourier coefficient is nonzero.

Multiply \eqref{eq:HI-omega-inverse-pair} over
$G\times\widehat G$.  The substitutions $k\mapsto-k$ and
$\chi\mapsto\chi^{-1}$ show separately that the two products of $Y$-terms
equal $\Pi$.  Moreover,
$\prod_{k\in G}\chi(k)=1$ because the elements of the odd-order group $G$
pair as $k$ and $-k$.  Only $(k,\chi)=(0,1)$ contributes an exceptional
factor on the right.  Hence
\begin{equation}
 \Pi^2=a^2\omega^{-N^2}.
 \label{eq:HI-omega-quadratic-product}
\end{equation}

To compute the trace Gram determinant, use translated and character-weighted
copies of $A$. Put $\mathcal P_G:=G\times\widehat G$, with
group law $(g,\chi)+(h,\eta)=(g+h,\chi\eta)$ and identity $0=(0,1)$.
The canonical identification of $\mathcal P_G$ with its dual
$\widehat G\times G$ sends $(g,\chi)$ to $(\chi,g)$.  It gives the symmetric
nondegenerate pairing
\begin{equation*}
 \langle(g,\chi),(h,\eta)\rangle:=\chi(h)\eta(g).
\end{equation*}
Let $\mathcal F_{\mathcal P_G}:=(\langle x,y\rangle)_{x,y\in\mathcal P_G}$,
using the same ordering on both indices.  Character orthogonality gives
$\mathcal F_{\mathcal P_G}^2=N^2P_{\mathrm{inv}}$, where $P_{\mathrm{inv}}$ is inversion
on $\mathcal P_G$.  This permutation has $(N^2-1)/2$ transpositions, an
even number since $N$ is odd.  Therefore
\begin{equation}
 (\det\mathcal F_{\mathcal P_G})^2=N^{2N^2}.
 \label{eq:HI-omega-Fourier-determinants}
\end{equation}

For $(g,\chi)\in\mathcal P_G$, define
\begin{equation*}
 (M_{g,\chi})_{i,j}:=A_{i+g,j}\chi(j).
\end{equation*}
For $y=(v,\psi)\in\mathcal P_G$, Fourier-transforming this family gives
\begin{equation}
 \widetilde{M}_y
 :=\sum_{g,\chi}\psi(g)\chi(v)M_{g,\chi}
 =\lambda_yE_y,
 \label{eq:HI-omega-basis-transform}
\end{equation}
where
\begin{equation*}
 \lambda_y:=N\omega\psi(-v)Y_v(\psi),
 \qquad
 E_y:=\varphi_\psi\mathbf e_{-v}^{\top},
 \qquad
 \varphi_\psi:=(\psi(-i))_{i\in G}.
\end{equation*}
The sum over $\chi$ forces $j=-v$, after which
\eqref{eq:HI-omega-column-transform} gives the displayed scalar.  For fixed
$v$, the vectors $\varphi_\psi$ form a Fourier basis.  Since $N\neq0$ and
every $Y_v(\psi)$ is nonzero, every $\lambda_y$ is nonzero.
Thus the matrices $E_y$ form a basis of $M_G(\kk)$, and invertibility
of the parameter Fourier transform shows that the matrices
$M_{g,\chi}$ do as well.

Let $\Gamma$ be the trace Gram matrix of this basis,
\begin{equation*}
 \Gamma_{(g,\chi),(h,\eta)}
 :=\operatorname{Tr}(M_{g,\chi}M_{h,\eta}).
\end{equation*}
It is invertible because the trace pairing on $M_G(\kk)$ is nondegenerate.
The rank-one basis has the same Fourier matrix as its trace Gram matrix:
for $y=(v,\psi)$ and $y'=(w,\phi)$,
\begin{equation*}
 \operatorname{Tr}(E_yE_{y'})
 =\psi(w)\phi(v)=\langle y,y'\rangle.
\end{equation*}
With $D_\lambda:=\operatorname{diag}(\lambda_y)$,
the basis transformation \eqref{eq:HI-omega-basis-transform} therefore gives
\begin{equation*}
 \mathcal F_{\mathcal P_G}\Gamma
 \mathcal F_{\mathcal P_G}
 =D_\lambda\mathcal F_{\mathcal P_G}D_\lambda.
\end{equation*}
Each transformed basis element is multiplied by one scalar $\lambda_y$.
The trace pairing has two inputs, so each $\lambda_y$ occurs twice in the
Gram determinant. This gives the square of $\prod_y\lambda_y$ below.
For each $\psi$, pairing $v$ with $-v$ gives $\prod_v\psi(-v)=1$.
Hence $\prod_y\lambda_y=(N\omega)^{N^2}\Pi$.
Taking determinants and cancelling the nonzero determinant of
$\mathcal F_{\mathcal P_G}$ gives
\begin{equation}
 \det\mathcal F_{\mathcal P_G}\,\det\Gamma
 =\left(\prod_y\lambda_y\right)^2
 =(N\omega)^{2N^2}\Pi^2.
 \label{eq:HI-omega-direct-Gram-determinant}
\end{equation}
This expresses the trace Gram determinant in terms of the Fourier product
computed in Part 1.

\smallskip
\noindent\emph{Part 2. The Gram matrix from the cubic identity.}
We now use the cubic identity to write the same Gram matrix as a matrix with
known determinant plus one rank-one term. This relates a determinant ratio to
the Fourier product.
Direct matrix multiplication gives
\begin{equation*}
 \operatorname{Tr}(M_{g,\chi}M_{h,\eta})
 =\sum_{k,l}A_{g+l,k}A_{h+k,l}\chi(k)\eta(l).
\end{equation*}
Multiply \eqref{eq:HI-cubic} by $\chi(k)\eta(l)$ and sum over $k,l$.
The left-hand side becomes
\begin{equation*}
 \omega^{-1}Y_g(\chi)Y_h(\eta)
 X_{g+h}((\chi\eta)^{-1}).
\end{equation*}
Since $Y_{g+h}(\chi\eta)\neq0$, Lemma~\ref{lem:HI-Fourier-E3} replaces
the last factor by $\varepsilon(g+h,\chi\eta)/Y_{g+h}(\chi\eta)$; here
$\varepsilon(g+h,(\chi\eta)^{-1})=\varepsilon(g+h,\chi\eta)$.
The constant term contributes only at $g=h=0$ and $\chi=\eta=1$, so
\begin{align}
 \Gamma_{(g,\chi),(h,\eta)}
 ={} \omega^{-1}\varepsilon(g+h,\chi\eta)
       \frac{Y_g(\chi)Y_h(\eta)}{Y_{g+h}(\chi\eta)} +aN^2\delta_{g,0}\delta_{h,0}
                  \delta_{\chi,1}\delta_{\eta,1}.
 \label{eq:HI-omega-cubic-Gram-kernel}
\end{align}

For $x=(g,\chi)\in\mathcal P_G$, put
\begin{equation*}
 Z_x:=Y_g(\chi),
 \qquad
 f_x:=\frac{\varepsilon(g,\chi)}{Z_x},
 \qquad
 D:=\operatorname{diag}(Z_x),
 \qquad
 H_{x,y}:=f_{x+y}.
\end{equation*}
Then \eqref{eq:HI-omega-cubic-Gram-kernel} becomes
\begin{equation*}
 \Gamma=\Gamma_0+aN^2\mathbf e_0\mathbf e_0^{\top},
 \qquad
 \Gamma_0:=\omega^{-1}DHD.
\end{equation*}
Thus the exceptional cubic term changes only the origin entry of the Gram
matrix.

It remains to compute the determinant of $H$.  Applying
\eqref{eq:HI-omega-inverse-pair} to $(g,\chi^{-1})$ gives
\begin{equation*}
 f_{(g,\chi)}=\omega\chi(g)^{-1}Y_{-g}(\chi^{-1}).
\end{equation*}
For $y=(v,\psi)\in\mathcal P_G$, character orthogonality and
\eqref{eq:HImatrix1} now give
\begin{align*}
 \widehat f(y)
 &:=\sum_{g,\chi}f_{(g,\chi)}\psi(g)\chi(v)=\omega N\sum_gA_{-g,v-g}\psi(g)
 =N Y_v(\psi).
\end{align*}
The last equality uses $\omega A_{-g,v-g}=A_{v,g}$.
In particular, $\lambda_y=\omega\psi(-v)\widehat f(y)$: the basis rescalings
and the Hankel transform involve the same row Fourier coefficients.

The Hankel matrix is diagonalized by congruence. Put $u=x+x'$ in the matrix
product and use character orthogonality to obtain
\begin{align*}
 (\mathcal F_{\mathcal P_G}H\mathcal F_{\mathcal P_G})_{y,y'}
 =\sum_u f_u\langle u,y'\rangle
       \sum_x\langle x,y-y'\rangle
 =N^2\delta_{y,y'}\widehat f(y).
\end{align*}
Taking determinants and using \eqref{eq:HI-omega-Fourier-determinants}
therefore gives
\begin{equation}
 \det H=\prod_y\widehat f(y)=N^{N^2}\Pi.
 \label{eq:HI-omega-Hankel-determinant}
\end{equation}
Since $\det D=\Pi$, equations
\eqref{eq:HI-omega-direct-Gram-determinant} and
\eqref{eq:HI-omega-Hankel-determinant} give
\begin{equation*}
 \det\Gamma_0=\omega^{-N^2}N^{N^2}\Pi^3,
 \qquad
 \frac{\det\Gamma_0}{\det\Gamma}
 =\frac{\det\mathcal F_{\mathcal P_G}}{N^{N^2}}\Pi,
\end{equation*}
where the ratio uses $\omega^{3N^2}=1$.  Squaring and using
\eqref{eq:HI-omega-Fourier-determinants} yields
\begin{equation}
 \left(\frac{\det\Gamma_0}{\det\Gamma}\right)^2
 =\Pi^2.
 \label{eq:HI-omega-Gram-ratio-square}
\end{equation}
Thus the squared determinant ratio equals the squared Fourier product. It
remains to evaluate the ratio from the rank-one correction.

\smallskip
\noindent\emph{Part 3. The rank-one correction and the comparison.}
We compute the determinant ratio from one entry of the inverse Gram matrix,
using a trace-dual element. This gives the second formula for $\Pi^2$. Let $R_0$ have an all-ones zero row and all other rows zero, and set
\begin{equation*}
 (R_0)_{i,j}:=\delta_{i,0},
 \qquad\text{ and }\qquad
 W:=A-\omega R_0.
\end{equation*}
Equations \eqref{eq:HImatrix3} and \eqref{eq:HImatrix2}, respectively, give
\begin{align}
 \operatorname{Tr}(M_{g,\chi}A)
 &=N\delta_{g,0}\delta_{\chi,1}-a,
 \label{eq:HI-omega-trace-A}\\
 \operatorname{Tr}(M_{g,\chi}R_0)
 &=-\omega^{-1}a.
 \label{eq:HI-omega-trace-R-zero}
\end{align}
The term $-\omega R_0$ cancels the constant $-a$ in
\eqref{eq:HI-omega-trace-A}, giving
\begin{equation}
 \operatorname{Tr}(M_{g,\chi}W)
 =N\delta_{g,0}\delta_{\chi,1}.
 \label{eq:HI-omega-dual-vector}
\end{equation}
Thus $W/N$ is trace-dual to
$M_{0,1}=A$, and $(\Gamma^{-1})_{0,0}
 =\operatorname{Tr}(W^2)/N^2$.

Equation~\eqref{eq:HI-omega-dual-vector} at $(g,\chi)=(0,1)$ gives
$\operatorname{Tr}(AW)=N$.  Also,
\eqref{eq:HI-omega-trace-R-zero} and $R_0^2=R_0$ give
$\operatorname{Tr}(R_0W)
 =-\omega^{-1}a-\omega$. 
It follows that
\begin{equation*}
 \operatorname{Tr}(W^2)
 =\operatorname{Tr}\bigl((A-\omega R_0)W\bigr)
 =N+a+\omega^2.
\end{equation*}
The matrix determinant lemma applied to
$\Gamma_0=\Gamma-aN^2\mathbf e_0\mathbf e_0^{\top}$ now gives
\begin{align}
 \frac{\det\Gamma_0}{\det\Gamma} =1-aN^2(\Gamma^{-1})_{0,0} = 1-a(N+a+\omega^2) = -a\omega^2,
 \label{eq:HI-omega-rank-one-ratio}
\end{align}
where the last equality uses $a(N+a)=1$.

Squaring \eqref{eq:HI-omega-rank-one-ratio} and comparing it with
\eqref{eq:HI-omega-Gram-ratio-square} gives
\begin{equation*}
 \Pi^2=a^2\omega.
\end{equation*}
This is the second formula for $\Pi^2$. Together with
\eqref{eq:HI-omega-quadratic-product}, and using $a\neq0$, it implies
\begin{equation}
 \omega^{N^2+1}=1.
 \label{eq:HI-omega-obstruction}
\end{equation}
Every square modulo $3$ is $0$ or $1$, so $3\nmid N^2+1$.  Since
$\omega^3=1$, equation~\eqref{eq:HI-omega-obstruction} forces $\omega=1$.
\end{proof}

\begin{remark}
Over $\mathbb C$, Theorem~\ref{thm:HI-omega-trivial} applies to every finite
abelian group of odd order.  It also applies in characteristic $p$ whenever
$p\nmid N$.  In particular, it applies to every odd-order finite abelian
group in characteristic $2$.  The proof is algebraic: it uses no conjugation,
positivity, Hermitian form, or division by $2$, $3$, or a discriminant.

If $p\mid N$ and $p\neq3$, Fourier inversion is unavailable, so this
argument gives no conclusion.  In characteristic $3$, the identity
$t^3-1=(t-1)^3$ already shows that every cube root of unity is $1$.  The
theorem concerns the full strict scalar equations in this manuscript.  
\end{remark}


\section{Reconstruction over an arbitrary characteristic}
\label{app:reconstruction}\label{app:HI-general-reconstruction}\label{sec:background}
We prove the reconstruction in two stages. First, we show that a Leavitt
realization produces a spherical fusion category. Second, starting from an
HI-matrix, we construct the endomorphisms required for such a realization.
For the explicit endomorphisms, we use the coefficient calculations of
Evans--Gannon and check their validity in arbitrary characteristic.
We retain the normalization and rigidity arguments, the treatment of
coincident scalar roots, and the grading proof of the mixed intertwiner
vanishing, since these require care outside characteristic zero.

\subsection{Additive completion and endomorphisms}\label{subsec:Cauchy}
Let $\C$ be a $\kk$-linear preadditive category. Its \emph{additive
envelope} $\operatorname{Add}(\C)$ has finite formal direct sums of objects
of $\C$ as its objects and matrices of morphisms as its morphisms.  If $\D$
is additive, its \emph{Karoubi envelope} (or \emph{idempotent completion})
$\operatorname{Kar}(\D)$ has pairs $(X,p)$, where $p\in\End_{\D}(X)$ is
idempotent, as its objects.  We write
$\C^{\oplus}:=\operatorname{Add}(\C)$ and
$\overline{\C^{\oplus}}:=\operatorname{Kar}(\C^{\oplus})$; the latter is the
\emph{Cauchy completion} of $\C$.  Thus, when $\C$ is already additive, its
Cauchy completion is its Karoubi envelope.

If $\C$ is monoidal and its tensor product is $\kk$-bilinear, the tensor
product extends to $\overline{\C^{\oplus}}$ by taking finite direct sums and
then restricting along the idempotents. The inclusion
$\C\hookrightarrow\overline{\C^{\oplus}}$ is monoidal. We use the following elementary criterion. It requires only that $\kk$ be a
field.

\begin{lemma}[Finite semisimplicity criterion]
\label{lem:finite-semisimplicity-criterion}
Let $\D$ be an additive $\kk$-linear category.  Suppose that
$X_1,\ldots,X_n$ are objects such that every object of $\D$ is isomorphic to
a finite direct sum of the $X_i$ and
\begin{equation*}
 \Hom_{\D}(X_i,X_j)=\delta_{i,j}\kk.
\end{equation*}
Then $\D$ is a finite semisimple $\kk$-linear abelian category.  Its simple
objects up to isomorphism are exactly $X_1,\ldots,X_n$, and it is idempotent
complete.
\end{lemma}

\subsection{Endomorphism categories}\label{subsec:endoCat}
Let $\Lambda$ be a $\kk$-algebra. Its algebra endomorphisms form a
$\kk$-linear preadditive monoidal category $\End(\Lambda)$ with the following
structure.
\begin{itemize}
 \item The objects are algebra endomorphisms $f:\Lambda\to\Lambda$.
 \item The morphism spaces are
 \begin{equation*}
  \Hom(f,g)=\{a\in\Lambda\mid af(x)=g(x)a
  \text{ for every }x\in\Lambda\}.
 \end{equation*}
 \item Composition is multiplication in $\Lambda$:
 \begin{equation*}
  (g\xrightarrow{b}h)\circ(f\xrightarrow{a}g)
  =(f\xrightarrow{ba}h).
 \end{equation*}
 \item The tensor product of objects is composition, $f\otimes g=f\circ g$.
 \item The tensor product of morphisms is
 \begin{equation*}
 (f\xrightarrow{a}g)\otimes(f'\xrightarrow{b}g')
  =(ff'\xrightarrow{g(b)a=af(b)}gg').
 \end{equation*}
 The equality of the two expressions follows by applying the
 intertwining relation for the first morphism to the second intertwiner.
\end{itemize}
For the reconstruction, we use the subcategory of $\End(\Lambda)$ generated
by endomorphisms representing the basis elements of the fusion ring. Let $\E$
be a collection of endomorphisms that contains the identity and is closed
under composition.  The corresponding
subcategory $\C(\E)$ is again $\kk$-linear, preadditive, and monoidal. In the
application below, it is generated by the chosen invertible endomorphisms and
one distinguished noninvertible endomorphism.

\subsection{Leavitt realizations produce spherical fusion categories}
\label{sec:LeavittRealization}

We now pass from a Leavitt realization to a fusion category. Starting from the
realization, we show that its additive envelope is semisimple with finitely
many simple isomorphism classes. We construct duals from the Leavitt splitting
maps and prove that the resulting pivotal structure is unique and spherical.
The argument works in arbitrary characteristic. The final remark records when
the global dimension vanishes in positive characteristic. The additive
envelope is already idempotent complete, so the Cauchy completion adds no new
objects up to isomorphism.

\begin{definition}
\label{def:Leavitt-realization}
Fix a finite abelian group $G$ of odd order
$N:=|G|=2n+1$.  A \emph{Leavitt realization} of $\mathfrak{HI}_G$ over
$\kk$ consists of the Leavitt
algebra
\begin{equation}
 L:=\cL_{N+1}=\cL_{2n+2},
 \label{eq:display-0022}
\end{equation}
written with generators $s,s',t_g,t_g'$ for $g\in G$ and defining relations
\begin{equation}
 s's=1,
 \qquad s't_g=t_g's=0,
 \qquad t_g't_h=\delta_{g,h},
 \qquad ss'+\sum_{g\in G}t_gt_g'=1,\semanticTag{LA}\label{eq:LeavittRelations}
\end{equation}
together with a unital $\kk$-algebra endomorphism $\rho:L\to L$ and unital
$\kk$-algebra automorphisms $\alpha_g:L\to L$ satisfying
\begin{align}
 \alpha_g\alpha_h&=\alpha_{g+h}, \semanticTag{LR1}\label{eq:LR1}\\
 \alpha_g\rho&=\rho\alpha_{-g}, \semanticTag{LR2}\label{eq:LR2}\\
 \rho^2(x)&=sxs'+\sum_{g\in G}t_g\alpha_g(\rho(x))t_g'
 \qquad(x\in L), \semanticTag{LR3}\label{eq:LR3}
\end{align}
and the simplicity and orthogonality conditions
\begin{align}
 \Hom(\alpha_g\rho,\alpha_h\rho)=\delta_{g,h}\kk,\qquad\qquad
 \semanticTag{LR4}\label{eq:display-0023}\\
 \Hom(\alpha_g,\alpha_h\rho)=0,
 \qquad
 \Hom(\alpha_g\rho,\alpha_h)=0. \semanticTag{LR5}\label{eq:display-0024}
\end{align}

Here
\begin{equation*}
 \Hom(\beta,\gamma)
 :=\{x\in L:x\beta(y)=\gamma(y)x\text{ for every }y\in L\}.
\end{equation*}
\end{definition}

Throughout the paper, the phrase \emph{Leavitt relations} means exactly the
four identities in \eqref{eq:LeavittRelations}.
The word ``endomorphism'' includes preservation of the unit.  In particular,
when endomorphisms are constructed by specifying the images of the Leavitt
generators, the universal property may be invoked only after those images have
been shown to satisfy all relations in \eqref{eq:LeavittRelations}.

\begin{remark}[Other diagonal Hom spaces]
\label{rem:LR-alpha-Hom}
The conditions above imply $\Hom(\alpha_g,\alpha_h)=\delta_{g,h}\kk$. Indeed, if $x\in\Hom(\alpha_g,\alpha_h)$, then, for every $y\in L$, $x\alpha_g\rho(y)=\alpha_h\rho(y)x$. This gives an inclusion: $\Hom(\alpha_g,\alpha_h)\hookrightarrow\Hom(\alpha_g\rho,\alpha_h\rho)$. Equation~\eqref{eq:display-0023} gives zero when $g\neq h$.  When $g=h$, the target is the scalar subspace $\kk 1_L$, and
$\kk 1_L\subseteq\Hom(\alpha_g,\alpha_g)$.  Hence the source is also
$\kk 1_L$.
\end{remark}

We now construct the category. Set
\begin{equation*}
 \E:=\{\alpha_g\rho^m:g\in G,\ m\geq0\},
\end{equation*}
and let $\C(\E)$ be the full $\kk$-linear subcategory of
$\operatorname{End}(L)$ on these objects.  Set
\begin{equation*}
 \D:=\operatorname{Add}(\C(\E)).
\end{equation*}
We use the same symbol $N$ for $N 1_{\kk}\in\kk$. Algebra endomorphisms remain
unital.

We first identify the tensor unit and decompose every object of $\C(\E)$ in
$\operatorname{Add}(\C(\E))$ as a finite direct sum of the objects
$\alpha_h$ and $\alpha_h\rho$.  We then show that these $2N$ objects are
all the simple objects and that the additive
envelope is already the Cauchy completion.

\begin{lemma}[The neutral group element]
\label{lem:LR-unit}
One has $\alpha_0=\operatorname{id}_L$.  In particular, every $\alpha_g$ is
an automorphism with inverse $\alpha_{-g}$.
\end{lemma}

\begin{proof}
The relation $\alpha_0^2=\alpha_0$ and invertibility of $\alpha_0$ give
$\alpha_0=\operatorname{id}_L$.  Equation~\eqref{eq:LR1} then gives
$\alpha_g\alpha_{-g}=\alpha_{-g}\alpha_g=\operatorname{id}_L$.
\end{proof}

It follows from Lemma~\ref{lem:LR-unit} and \eqref{eq:LR1}--\eqref{eq:LR2}
that $\E$ contains the tensor unit and is closed under composition.  Thus
$\C(\E)$ is a monoidal subcategory of $\operatorname{End}(L)$.

\begin{lemma}[The Leavitt splitting]
\label{lem:LR-splitting}
In $\D$ there is an isomorphism
\begin{equation}
 \rho^2\cong \unit\oplus\bigoplus_{g\in G}\alpha_g\rho
 \label{eq:display-0027}
\end{equation}
whose inclusions are $s,t_g$ and whose projections are $s',t_g'$.  Thus,
every $\alpha_g\rho^m$ is a finite direct sum of objects among
$\{\alpha_h,\alpha_h\rho:h\in G\}$.
\end{lemma}

\begin{proof}
Multiplying \eqref{eq:LR3} on the right by $s$ and $t_g$, and on the left by
$s'$ and $t_g'$, gives
\begin{equation*}
 \rho^2(x)s=sx,\qquad \rho^2(x)t_g=t_g(\alpha_g\rho)(x),
 \qquad s'\rho^2(x)=xs',\qquad
 t_g'\rho^2(x)=(\alpha_g\rho)(x)t_g'.
\end{equation*}
Thus these elements have the asserted sources and targets.  The two matrix
composites are the identity by the Leavitt relations: projection after
inclusion is the identity on each summand and zero between distinct summands,
while the sum of inclusion after projection is the completeness relation.

The relation $\alpha_g\rho=\rho\alpha_{-g}$ implies
$\rho^r\alpha_h=\alpha_{(-1)^rh}\rho^r$.  Hence, for $m\geq2$,
\begin{equation}
 \alpha_g\rho^m
 \cong \alpha_g\rho^{m-2}\oplus
 \bigoplus_{h\in G}\alpha_{g+(-1)^{m-2}h}\rho^{m-1}.
 \label{eq:display-0029}
\end{equation}
Strong induction on $m$, starting with $m=0,1$, proves the final assertion.
\end{proof}

\begin{proposition}[Semisimplicity of the additive envelope]
\label{prop:fdHom}
The category $\D$ is semisimple abelian and is already idempotent complete.
Its $2N$ simple objects, up to isomorphism, are exactly
\begin{equation*}
 \{\alpha_g,\alpha_g\rho:g\in G\}.
\end{equation*}
In particular, all Hom spaces are finite dimensional and the unit is simple.
Thus the inclusion $\D\to\overline{\C(\E)^{\oplus}}$ is an equivalence.
\end{proposition}

\begin{proof}
By Lemma~\ref{lem:LR-splitting}, every object is isomorphic to a finite direct
sum of objects $\alpha_g$ and $\alpha_g\rho$.  The assumed Hom-space
dimensions in \eqref{eq:display-0023}--\eqref{eq:display-0024}, together with Remark~\ref{rem:LR-alpha-Hom}, give
\begin{equation}
\Hom(\alpha_g,\alpha_h)=\delta_{g,h}\kk,\quad  \Hom(\alpha_g\rho,\alpha_h\rho)=\delta_{g,h}\kk,\quad\Hom(\alpha_g,\alpha_h\rho)
  =\Hom(\alpha_g\rho,\alpha_h)=0.
 \label{eq:display-0031}
\end{equation}
Lemma~\ref{lem:finite-semisimplicity-criterion}, applied to this family of
$2N$ objects, proves semisimplicity, abelianness, idempotent completeness,
Hom-finiteness, and the asserted list of simple objects.  Finally,
$\alpha_0=\unit$ by Lemma~\ref{lem:LR-unit}, so the unit is simple.
\end{proof}

The tensor products of the simple objects are
\begin{align}
 \alpha_g\alpha_h&=\alpha_{g+h}, \label{eq:display-0060}\\
 \alpha_g(\alpha_h\rho)&=\alpha_{g+h}\rho
 \cong(\alpha_h\rho)\alpha_{-g},\notag
\\
 (\alpha_g\rho)(\alpha_h\rho)
 &\cong\alpha_{g-h}\oplus\bigoplus_{k\in G}\alpha_k\rho.
 \label{eq:display-0062}
\end{align}
The first two formulas follow from \eqref{eq:LR1} and \eqref{eq:LR2}. For the
last one, use
$(\alpha_g\rho)(\alpha_h\rho)=\alpha_{g-h}\rho^2$, apply
Lemma~\ref{lem:LR-splitting}, and reindex the noninvertible summands.

\begin{lemma}[The rigidity scalar]
\label{lem:LR-rigidity-scalar}
Set
\begin{equation*}
 a:=s'\rho(s),\qquad\text{ and }\qquad a':=\rho(s')s.
\end{equation*}
Then $a,a'\in\End(\rho)=\kk$, $a=a'\neq0$, and
$\alpha_g(s)=s$, $\alpha_g(s')=s'$ for every $g\in G$.
\end{lemma}

\begin{proof}
The splitting and the Hom table \eqref{eq:display-0031} show that
$a,a'\in\End(\rho)=\kk$.  Before using the coefficient comparisons, we
record the normalization they require.  Conjugation by $\alpha_h$ sends
$\Hom(\beta,\gamma)$ to
$\Hom(\alpha_h\beta\alpha_{-h},\alpha_h\gamma\alpha_{-h})$.
By \eqref{eq:LR2} twice, $\alpha_h\rho^2\alpha_{-h}=\rho^2$.
The splitting and \eqref{eq:display-0031} therefore give
\begin{equation*}
 \Hom(\unit,\rho^2)=\kk s,\qquad\text{ and }\qquad
 \Hom(\rho^2,\unit)=\kk s'.
\end{equation*}
Thus $\alpha_h(s)=\psi(h)s$ and
$\alpha_h(s')=\psi(h)^{-1}s'$ for a character
$\psi:G\to\kk^*$.  Similarly,
$\alpha_h(t_g)\in\kk t_{g+2h}$ and
$\alpha_h(t_g')\in\kk t_{g+2h}'$.
Since multiplication by $2$ is bijective on $G$, replace the splitting maps
by $\widetilde t_g=\alpha_{g/2}(t_0)$ and
$\widetilde t_g'=\alpha_{g/2}(t_0')$.  Each new pair is obtained from
$t_g,t_g'$ by inverse scalar rescalings, since
$\widetilde t_g'\widetilde t_g=1$.  Thus every product $t_gt_g'$ and
$t_g\alpha_g(\rho(x))t_g'$ is unchanged.  The Leavitt relations and
\eqref{eq:LR3} are preserved, and in this normalization
$\alpha_h(\widetilde t_g)=\widetilde t_{g+2h}$ and likewise for the
primed maps.  We may therefore use the coefficient comparisons in the proof of
\cite[Theorem~1, pp.~15--17, especially equations~(4.19), (4.22),
and~(4.24)]{evans2017non}.
These calculations use only the Leavitt relations, the splitting and the
scalar Hom spaces.  They remain valid over $\kk$: the rescaling of the
$t_g,t_g'$ uses the bijection $g\mapsto2g$ on $G$, and
Lemma~\ref{lem:HI-reduced-words} justifies coefficient comparison in every
characteristic.  Before the final normalization in the cited proof, they give
scalars $u,v\in\kk$ such that
\begin{align*}
 a^2+N a=a'^2+N a'=1,\qquad aa'+N u=aa'+N v=1,
 \qquad\text{ and }\qquad aa'=uv.
\end{align*}
In particular, $a,a'\neq0$.

We check the remaining scalar comparison without assuming characteristic
zero.  If $N=0$ in $\kk$, the displayed identities give $a^2=aa'=1$,
hence $a=a'$.  If $N\neq0$, they give $u=v$; write their common value as
$z$.  Then $aa'=z^2$ and $z^2+N z=1$.
Suppose $a\neq a'$.  Since they are distinct roots of
$X^2+N X-1$, we have $a+a'=-N$ and $aa'=-1$.
Thus $z^2=-1$ and $N z=2$, which imply $N^2+4=0$.  Consequently
\begin{equation*}
 (a-a')^2=(a+a')^2-4aa'=N^2+4=0,
\end{equation*}
contradicting $a\neq a'$ in a field.  Hence $a=a'$ in every characteristic.

Finally, $a\in\kk$ and \eqref{eq:LR2} give
\begin{equation*}
 a=\alpha_h(s'\rho(s))
   =\psi(h)^{-1}s'\rho(\psi(-h)s)
   =\psi(h)^{-2}a.
\end{equation*}
Since $a\ne0$, every value of $\psi$ has square $1$.  The group $G$ has
odd order, so $\psi$ is trivial.  Hence $\alpha_h(s)=s$ and
$\alpha_h(s')=s'$ for every $h\in G$.
\end{proof}

\begin{lemma}[Rigidity]
\label{lem:CErigid}
The monoidal categories $\C(\E)$ and $\D$ are rigid.
\end{lemma}

\begin{proof}
Each $\alpha_g$ has dual $\alpha_{-g}$, with evaluation and coevaluation both
equal to $1$.  Give $\rho$ the self-duality with evaluation $s'$ and
coevaluation $a^{-1}s$.  Lemma~\ref{lem:LR-rigidity-scalar} gives the two
triangle identities
\begin{equation*}
 a^{-1}\rho(s')s=1,
 \qquad a^{-1}s'\rho(s)=1.
\end{equation*}
Every object of $\C(\E)$ is a tensor product of objects of these two types and
is therefore dualizable.  Finite direct sums of dualizable objects are
dualizable componentwise, so $\D$ is rigid as well.
\end{proof}

We use the standard conventions for pivotal and spherical structures
\cite[Definitions~4.7.7 and~4.7.14, pp.~74--75]{EGNO}.

\begin{lemma}[Pivotality and sphericality]
\label{lem:CEspherical}
For the duality in Lemma~\ref{lem:CErigid}, the double-dual is the identity as
a strict monoidal functor.  The identity transformation is the unique pivotal
structure.  It is spherical and, under the equivalence in
Proposition~\ref{prop:fdHom}, gives a spherical structure on
$\overline{\C(\E)^{\oplus}}$.  Moreover,
\begin{equation}
 \dim(\alpha_g)=1,\qquad \dim(\alpha_g\rho)=a^{-1}.
 \label{eq:display-0055}
\end{equation}
Consequently,
\begin{equation}
 a^2+N a-1=0.
 \label{eq:display-0034}
\end{equation}
\end{lemma}

\begin{proof}
For $X_{g,m}:=\alpha_g\rho^m$, use the standard tensor-product duality.  It
gives
\begin{equation}
 X_{g,m}^{\vee}=\alpha_{(-1)^{m+1}g}\rho^m.
 \label{eq:display-0047}
\end{equation}
We choose the usual tensor-product evaluation and coevaluation maps, so
\begin{equation*}
 (XY)^{\vee}=Y^{\vee}X^{\vee},
 \qquad
 \operatorname{ev}_{XY}
 =\operatorname{ev}_Y\,Y^{\vee}(\operatorname{ev}_X),
 \qquad
 \operatorname{coev}_{XY}
 =X(\operatorname{coev}_Y)\operatorname{coev}_X.
\end{equation*}
These are literal equalities.  Hence $(-)^{\vee}$ is strict anti-monoidal,
$(-)^{\vee\vee}$ is strict monoidal, and \eqref{eq:display-0047} shows that
the double-dual fixes every object of $\C(\E)$.  On $\D$, take duals of finite
direct sums componentwise; the double-dual then fixes these objects as well.

For a morphism $f:X\to Y$, write
$f^{\vee}:Y^{\vee}\to X^{\vee}$ for its image under the chosen duality
functor.  The splitting in Lemma~\ref{lem:LR-splitting} and the Hom table
\eqref{eq:display-0031} show that the Hom spaces containing the following
maps are one-dimensional:
\begin{equation*}
 s^{\vee}\in\kk^\times s',\quad (s')^{\vee}\in\kk^\times s,
 \quad t_g^{\vee}\in\kk^\times t_g',\quad\text{ and }\quad
 (t_g')^{\vee}\in\kk^\times t_g.
\end{equation*}
Dualizing $s's=1$ and $t_g't_g=1$ shows that the two scalars in each pair are
reciprocal.  Hence the double-dual fixes $s,s',t_g,t_g'$.

We now describe the decompositions used for the remaining objects.  Start
with the splitting of $\rho^2$.  For $m\geq2$, tensor it on the left by
$\alpha_g\rho^{m-2}$; this is the decomposition in
\eqref{eq:display-0029}.  Repeating this step lowers the power of $\rho$ until
all summands are simple.  Every inclusion and projection in these chosen
decompositions is obtained from $s,s',t_g,t_g'$ by composition and tensoring
with identity maps.  The double-dual fixes all these maps.  Relative to the
chosen decompositions, \eqref{eq:display-0031} shows that every morphism is a
matrix of scalars.  It follows that the double-dual fixes every morphism.

Any pivotal structure is therefore a monoidal natural automorphism $\eta$ of
the identity functor.  Write its components on $\alpha_g$ and
$\alpha_g\rho$ as the scalars $x_g$ and $y_g$.  Naturality for the unit
summand and every noninvertible summand in
$\rho^2\cong\unit\oplus\bigoplus_k\alpha_k\rho$ gives
$y_0^2=x_0=1$ and $y_0^2=y_k$ for every $k$.  Taking $k=0$ and using
$y_0\neq0$ gives $y_0=1$, hence $y_k=1$ for every $k$.  Naturality with the
inclusions in
$(\alpha_g\rho)\rho\cong\alpha_g\oplus\bigoplus_k\alpha_k\rho$
then gives $x_g=1$.
Thus $\eta=\operatorname{id}$, so the identity transformation is the unique
pivotal structure.

The left and right dimensions of $\alpha_g$ are both $1$.  The object
$\alpha_g\rho$ is self-dual.  Because $\alpha_g$ fixes $s$ and $s'$, its
evaluation is $s'$ and its coevaluation is $a^{-1}s$.  Thus its left and right
dimensions are both $a^{-1}$.  In a finite semisimple category, equality of
the left and right dimensions on the simple objects implies equality of the
two traces on every endomorphism.  Hence the pivotal structure is spherical.
Proposition~\ref{prop:fdHom} identifies $\D$ with the Cauchy completion.

Finally, the dimension function is a character of the Grothendieck ring
\cite[Proposition~4.7.12, pp.~74--75]{EGNO}.  Apply it to
\eqref{eq:display-0027}.  Using
\eqref{eq:display-0055} gives
$a^{-2}=1+N a^{-1}$, which is equivalent to
\eqref{eq:display-0034}.
\end{proof}

\begin{theorem}[Leavitt reconstruction]
\label{thm:LRtoHIcat}
Let $(L,\rho,\{\alpha_g\}_{g\in G})$ be a Leavitt realization as in
Definition~\ref{def:Leavitt-realization}.  Then $\overline{\C(\E)^{\oplus}}$ is a spherical fusion category.  Its simple isomorphism classes are exactly
$\{\alpha_g,\alpha_g\rho:g\in G\}$, with fusion rules
\eqref{eq:display-0060}--\eqref{eq:display-0062}.  If
$d:=\dim(\alpha_g\rho)=a^{-1}$, then
\begin{equation}
 d^2=1+N d.
 \label{eq:display-0063}
\end{equation}
The pivotal structure is unique.  Under the strict equality
$(-)^{\vee\vee}=\operatorname{id}$ of Lemma~\ref{lem:CEspherical}, it is the
identity transformation.
In particular, this is a spherical $\mathfrak{HI}_G$-category.
\end{theorem}

\begin{proof}
Proposition~\ref{prop:fdHom} gives a finite semisimple category with simple
unit, Lemma~\ref{lem:CErigid} gives rigidity, and
Lemma~\ref{lem:CEspherical} gives the unique spherical pivotal structure and
the dimensions.  The fusion rules were computed above.  Equation
\eqref{eq:display-0063} is the reciprocal form of
\eqref{eq:display-0034}.
\end{proof}

The theorem does not require a nonzero global dimension.

\begin{remark}[Global dimension in positive characteristic]
\label{rem:LR-global-dimension}
The global dimension with respect to this spherical structure is $\operatorname{Dim}\bigl(\overline{\C(\E)^{\oplus}}\bigr)
 =N(1+d^2)=N(2+N d)$. No nonzero-global-dimension or separability hypothesis was used.  In
characteristic $p>0$ this scalar is zero precisely when either $p\mid N$, or
$p\neq2$ and $p\mid(N^2+4)$.  Thus zero global dimension is compatible with
the conclusion that the category is spherical fusion.
\end{remark}

A Leavitt realization over an algebraically closed field of any
characteristic therefore produces a spherical $\mathfrak{HI}_G$-category with
the stated simple objects, fusion rules, and dimensions. The next section
constructs a realization from the finite HI-matrix equations.

\subsection{HI matrices produce Leavitt realizations}
The preceding subsection proves the first stage of the reconstruction. We now
verify that the endomorphisms attached to an HI-matrix satisfy the hypotheses
of the Leavitt reconstruction theorem.
Retain $d^2=Nd+1$, $a=d^{-1}$, and $\omega^3=1$. Choose
\begin{equation}\label{eq:HI-general-b}
 b^2=\omega^{-1}a.
\end{equation}
These scalars exist because the field is algebraically closed; $a,b,\omega$
are nonzero. We now prove the conditions in
Definition~\ref{def:Leavitt-realization} for the explicit matrix assignments.
\subsection{Reduced words}
\label{subsec:HI-general-reduced-words}

We now turn from reducing the HI equations to constructing endomorphisms of
the Leavitt algebra.  The following normal form makes every later coefficient
comparison characteristic-free.

Recall that $L=\cL_{N+1}$ is the Leavitt algebra defined by
\eqref{eq:display-0022} and \eqref{eq:LeavittRelations}.
Let $I:=\{\ast\}\sqcup G$, and write
\begin{equation*}
 u_\ast=s,\quad u_g=t_g,\qquad
 v_\ast=s',\quad v_g=t_g'.
\end{equation*}

\begin{lemma}[Reduced-word basis]
\label{lem:HI-reduced-words}
Every element of $L$ has a unique expression as a finite $\kk$-linear
combination of words
\begin{equation}
 u_{i_1}\cdots u_{i_r}v_{j_1}\cdots v_{j_s},
 \qquad r,s\geq0,\quad i_1,\ldots,i_r,j_1,\ldots,j_s\in I,
 \label{eq:HI-reduced-word-form}
\end{equation}
where, if $r,s>0$, one does not have $i_r=j_1=\ast$.
Thus coefficients of distinct words in \eqref{eq:HI-reduced-word-form} may be
compared over $\kk$ in every characteristic.  Also, $L$ has the
$\mathbb Z$-grading
\begin{equation}
 \deg u_i=1,\qquad \deg v_i=-1.
 \label{eq:display-0110}
\end{equation}
\end{lemma}

\begin{proof}
Use the monic reductions
\begin{equation*}
 v_i u_j\longmapsto\delta_{i,j},\qquad
 u_\ast v_\ast\longmapsto1-\sum_{g\in G}u_gv_g.
\end{equation*}
These are the reductions in \cite[Lemma~1, p.~11]{evans2017non}.
For shortlex order with $u_\ast>u_g$, every reduction decreases the word.
The only overlaps are $v_i u_\ast v_\ast$ and
$u_\ast v_\ast u_j$; their two reductions agree by the Kronecker-delta
relations.  Thus the irreducible words are precisely
\eqref{eq:HI-reduced-word-form}, with unique coefficients.
All leading coefficients are $1$, so the same argument works over every
field.  Homogeneity of the defining relations gives
\eqref{eq:display-0110}.
\end{proof}

\subsection{The endomorphism formulas and the Leavitt relations}
\label{subsec:HI-general-endomorphism-formulas}

Using the HI-matrix entries as coefficients, we define the proposed images
of the Leavitt generators and then verify that they satisfy the defining
relations. Define the following elements of $L$:
\begin{align}
 \label{eq:rhos-general}
 S&:=\rho(s)
 =as+b\sum_{r\in G}t_rt_r,
 \qquad 
 S':=\rho(s')
 =as'+\omega b\sum_{r\in G}t_r't_r',\\
 \label{eq:rhotg-general}
 T_g&:=\rho(t_g)
 =bs t_{-g}'+\omega t_{-g}ss'
   +\sum_{h,k\in G}A_{h+g,k+g}
      t_ht_{h+k+g}t_k',\\
\label{eq:rhotg-prime-general}
 T_g'&:=\rho(t_g')
 =\omega b t_{-g}s'+\omega^{-1}ss't_{-g}'
   +\sum_{h,k\in G}A_{k+g,h+g}
      t_kt_{g+h+k}'t_h'.
\end{align}
For $u\in G$, define
\begin{equation}
 \label{eq:alphag-general}
 \alpha_u(s)=s,\qquad
 \alpha_u(s')=s',\qquad
 \alpha_u(t_h)=t_{h+2u},\qquad\text{ and }\qquad
 \alpha_u(t_h')=t_{h+2u}'.
\end{equation}
These formulas are the algebraic forms of
\cite[equations~(4.3)--(4.6), p.~14, and Theorem~2(a),
p.~18]{evans2017non}, with 
$\overline{\omega}$ replaced by $\omega^{-1}$.
We use the coefficient verifications from that proof, recording below
why their algebraic steps remain valid over $\kk$.  Until the Leavitt
relations have been verified, $\rho$ denotes only the assignment on
generators.

\begin{lemma}[The permutation automorphisms]
\label{lem:HI-alpha-endomorphisms}
For every $u\in G$, the assignments \eqref{eq:alphag-general} preserve all Leavitt
relations and hence extend uniquely to a unital algebra automorphism of $L$.
Also,
\begin{equation}
 \alpha_u\alpha_v=\alpha_{u+v}.
 \label{eq:display-0122}
\end{equation}
\end{lemma}

\begin{proof}
The assignments fix $s,s'$ and permute the pairs $t_h,t_h'$ by
$h\mapsto h+2u$.  They therefore preserve every Leavitt relation.
The inverse permutation is induced by $\alpha_{-u}$, and composition adds
the translations, proving \eqref{eq:display-0122}.
\end{proof}

\begin{proposition}[The matrix formulas define $\rho$]
\label{prop:HI-rho-endomorphism}
The elements $S,S',T_g,T_g'$ satisfy
\begin{align*}
 S'S=1, \quad S'T_g=0 \quad T_g'S=0, \quad T_g'T_h=\delta_{g,h},\quad\text{ and }\quad SS'+\sum_{g\in G}T_gT_g' =1.
\end{align*}
Thus \eqref{eq:rhos-general}--\eqref{eq:rhotg-prime-general} extend uniquely to a unital
algebra endomorphism $\rho:L\to L$.
\end{proposition}

\begin{proof}
The verification is the coefficient calculation in
\cite[proof of Theorem~2(a), Section~5, pp.~18--19,
equation~(5.1)]{evans2017non}.
With $a=d^{-1}$ and $\omega b^2=a$, its scalar relation is
$a^2+Na=1$.  The products $S'T_g$ and $T_g'S$ vanish by the column and
row sums, while $T_g'T_h$ and the completeness relation reduce to
\eqref{eq:HImatrix3}.  For example, the coefficient sum in $T_g'T_h$ is
\begin{equation*}
 \sum_u A_{g+r,g+u}A_{h+u,g+r}
 =\delta_{g,h}-a\delta_{r,-g};
\end{equation*}
its second term cancels $a t_{-g}t_{-h}'$, leaving
$\delta_{g,h}(ss'+\sum_rt_rt_r')=\delta_{g,h}$.

For completeness, after the row and column sums cancel the three-letter
terms, the residual is
\begin{equation*}
 a\sum_{h,u}t_ht_ht_u't_u'-\sum_{h,v}t_ht_vt_v't_h'+Z,
 \qquad\text{and}\qquad
 Z=\sum_{g,h,k,u}A_{h+g,k+g}A_{k+g,u+g}
 t_ht_{h+k+g}t_{g+u+k}'t_u'.
\end{equation*}
A word $t_ht_vt_w't_u'$ in $Z$ satisfies $r:=v-h=w-u$.
Putting $m=u+g$, its coefficient is
\begin{equation*}
 \sum_m A_{h-u+m,r}A_{r,m}
 =\delta_{h,u}-a\delta_{r,0}
\end{equation*}
by \eqref{eq:HImatrix3}.  These are exactly the coefficients needed to
cancel the other two sums.

The cited calculation is a sequence of monic Leavitt reductions,
reindexings of finite sums, and substitutions from
\eqref{eq:HImatrix1}--\eqref{eq:HImatrix3}.  The coefficient identities
use no division by an integer.  Replacing every occurrence of
$\overline\omega$ by $\omega^{-1}$ therefore gives the same reductions
over $\kk$, with coefficients compared using
Lemma~\ref{lem:HI-reduced-words}.  This establishes the identities in
arbitrary characteristic from their algebraic derivation.
The universal property of $L$ gives the unital endomorphism $\rho$.
\end{proof}

\subsection{The composition identities}
\label{subsec:HI-general-composition-identities}

Having constructed the endomorphisms, we prove
\eqref{eq:LR1}--\eqref{eq:LR3}. Lemma~\ref{lem:HI-alpha-endomorphisms}
proves \eqref{eq:LR1}. The next lemma proves \eqref{eq:LR2}. The rest of
the subsection proves \eqref{eq:LR3}.

\begin{lemma}[Relation \texorpdfstring{\eqref{eq:LR2}}{LR2}]
\label{lem:HI-translation-covariance}
The endomorphisms defined in \eqref{eq:rhos-general}--\eqref{eq:alphag-general} satisfy
\begin{equation*}
 \alpha_u\rho=\rho\alpha_{-u}
 \qquad(u\in G).
\end{equation*}
\end{lemma}

\begin{proof}
The sums defining $\rho(s)$ and $\rho(s')$ are fixed by $\alpha_u$.
In the formula for $\alpha_u(\rho(t_l))$, substitute
$(h,k)=(H-2u,K-2u)$.  The coefficient becomes
$A_{H+l-2u,K+l-2u}$ and the middle generator has index $H+K+l-2u$,
so the result is $T_{l-2u}=\rho(\alpha_{-u}(t_l))$.
The same substitution gives $\alpha_u(T_l')=T_{l-2u}'$.
This proves equality on every generator and hence on $L$.
\end{proof}

To prove \eqref{eq:LR3}, define
\begin{equation*}
 \Gamma(x):=sxs'+\sum_{r\in G}t_r\alpha_r(\rho(x))t_r'.
\end{equation*}
The identities in \eqref{eq:LeavittRelations} show directly that
$\Gamma(1)=1$ and
$\Gamma(xy)=\Gamma(x)\Gamma(y)$.  Indeed, the mixed products vanish by
$s't_r=t_r's=0$, and $t_r't_q=\delta_{r,q}$ reduces the remaining double sum
to a single sum.  Thus $\Gamma$ is a unital algebra endomorphism.
It is therefore enough to prove $\rho^2=\Gamma$ on the Leavitt generators.

Multiplying the formulas in \eqref{eq:rhos-general}--\eqref{eq:rhotg-general} on the left by
$s'$ or $t_r'$ gives
\begin{equation}
 \label{eq:HI-left-corners}
 s'S=a,
 \qquad s'T_g=bt_{-g}',
 \qquad t_r'S=bt_r,
 \qquad
 t_r'T_g=\omega\delta_{r,-g}ss'
 +\sum_kA_{r+g,k+g}t_{r+k+g}t_k'.
\end{equation}

The completeness relation shows that these left products determine an
element.  Indeed, if $s'x=s'y$ and $t_r'x=t_r'y$ for every $r$, then
\begin{equation*}
 x-y=\left(ss'+\sum_rt_rt_r'\right)(x-y)=0.
\end{equation*}

\begin{lemma}[Relation \texorpdfstring{\eqref{eq:LR3}}{LR3} on the unprimed generators]
\label{lem:HI-unprimed-LR3}
For $x\in\{s,t_h:h\in G\}$ and $r\in G$,
\begin{equation*}
 s'\rho^2(x)=xs',
 \qquad\text{and}\qquad
 t_r'\rho^2(x)=\alpha_r(\rho(x))t_r'.
\end{equation*}
\end{lemma}

\begin{proof}
This is the verification of the unprimed-generator part of (4.2) in
\cite[proof of Theorem~2(a), Section~5, p.~19,
equations~(5.3)--(5.7)]{evans2017non}.
Translation covariance reduces it to
\begin{equation*}
 s'\rho^2(s)=ss',\quad t_0'\rho^2(s)=St_0',\quad
 s'\rho^2(t_0)=t_0s',\quad\text{and}\quad t_0'\rho^2(t_q)=T_qt_0'
 \quad(q\in G).
\end{equation*}
Here the reduction to index $0$ uses the bijection $g\mapsto2g$ of the
abstract group, not division by $2$ in $\kk$. Expand with \eqref{eq:HI-left-corners} and compare reduced words.
As in the cited proof, all coefficients vanish by the scalar relation,
the row and column sums, and \eqref{eq:HImatrix3}, except for the
four-matrix coefficients in the last identity, which are exactly
\eqref{eq:HImatrix4}.  The auxiliary sums in (5.4)--(5.7) there are
reindexings and contractions of \eqref{eq:HImatrix1} and
\eqref{eq:HImatrix3}; for example,
\begin{equation*}
 \sum_{l,m}A_{l,m}A_{m,l+h}A_{h,k+m}
 =\delta_{h,0}\sum_m A_{h,k+m}-aA_{h,k}
 =-\omega a\delta_{h,0}-aA_{h,k}.
\end{equation*}
For the remaining identity, put $\Delta_q=t_0'\rho(T_q)-T_qt_0'$.
After the left-corner reductions and the replacement
$ss'=1-\sum_gt_gt_g'$, the only reduced-word types are
$1$, $t_rt_{r+q}'$, $st_{r-q}'t_r'$, $t_rt_{r+q}s'$, and
four-letter words. Define
\begin{align*}
 U_{g,k}&:=\sum_{l,m}A_{l,m}A_{l+g,k}A_{k+m,l}
          +b^2\delta_{k,0}+aA_{g,k},\\
 V_{g,h}&:=\sum_rA_{r+g,h}A_{r,-h}
          -\omega\delta_{h,g}+\omega^{-1}b^2\delta_{h,0},\\
 W_{g,k}&:=\sum_mA_{g,m+g}A_{-g,m+k}
          -\omega^{-1}\delta_{k,0}+b^2\delta_{g,0},\\
 F^{(4)}_{g,h,i,k}
 &:=\sum_{l,m}A_{l,m}A_{l+g,h}A_{h+m,l+i}A_{i,k+m}-A_{h-g,i-g}\delta_{k,g}
 +\omega^{-1}a\delta_{h,0}A_{i,k}
 +\omega aA_{g,h}\delta_{i,0}.
\end{align*}
The contraction identities cited above give $U=V=W=0$, while
$F^{(4)}=0$ is \eqref{eq:HImatrix4}. Writing $[w]\Delta_q$ for the
coefficient of a reduced word $w$, direct collection gives
\begin{align*}
 [t_rt_{r+q}']\Delta_q&=V_{q,-r},&
 [st_{r-q}'t_r']\Delta_q&=b\omega W_{q,r},\\
 [t_rt_{r+q}s']\Delta_q&=b\omega U_{-q,r},&
 [1]\Delta_q&=\delta_{q,0}\omega^2
 \left(a^2+b^2\sum_mA_{0,m}\right)=0.
\end{align*}
For $q=-g$, every four-letter word has the form
$t_ht_{h+i-g}t_{k+i-g}'t_{k-g}'$, with coefficient
\begin{equation*}
 F^{(4)}_{g,h,i,k}-\delta_{h,k}V_{-g,-h}=0.
\end{equation*}
Thus every reduced-word coefficient of $\Delta_q$ vanishes.

These are formal coefficient reductions over $\kk$ after
$\overline\omega$ is replaced by $\omega^{-1}$.  No averaging, integer
inversion, or separation of scalar roots is used, so the argument applies
in every characteristic.
\end{proof}

The identities for the primed generators follow from the following symmetry.

\begin{lemma}[Transposition and the primed generators]
\label{lem:HI-prime-reversal}
Let $\dagger:L\to L$ be the $\kk$-linear anti-automorphism determined by $s^\dagger=s'$, $(s')^\dagger=s$, $t_g^\dagger=t_g'$, and $(t_g')^\dagger=t_g$. For an endomorphism $f$, put
$\widetilde f(x):=f(x^\dagger)^\dagger$.  Then
$\widetilde\alpha_u=\alpha_u$, and $\widetilde\rho$ is given by the same
formulas as $\rho$ with
\begin{equation}
 \widetilde\omega=\omega^{-1},
 \qquad \widetilde b=\omega b,
 \qquad\text{ and }\qquad\widetilde A_{g,h}=A_{h,g}.
 \label{eq:display-0159}
\end{equation}
The transposed data again satisfy
\eqref{eq:HImatrix1}--\eqref{eq:HImatrix4}.
\end{lemma}

\begin{proof}
We first check that $\dagger$ is well defined.  Reversing products and swapping
primed and unprimed generators gives
\begin{align*}
 (s's)^\dagger&=s's, &
 (s't_g)^\dagger&=t_g's, &
 (t_g's)^\dagger&=s't_g,
\\
 (t_g't_h)^\dagger&=t_h't_g, &
 \left(ss'+\sum_gt_gt_g'\right)^\dagger
 &=ss'+\sum_gt_gt_g'.
\end{align*}
Thus every identity in \eqref{eq:LeavittRelations} is preserved.  Hence
$\dagger$ is an involutive anti-automorphism of $L$.

The formula \eqref{eq:alphag-general} immediately gives
$\widetilde\alpha_u=\alpha_u$.  Applying $\dagger$ to the formulas for $\rho$
gives
\begin{align*}
 \widetilde\rho(s)&=as+\widetilde b\sum_rt_rt_r,\qquad\widetilde\rho(s')=as'+\widetilde\omega\widetilde b\sum_rt_r't_r',\\ \widetilde\rho(t_g)&=\widetilde b\,s t_{-g}'
   +\widetilde\omega t_{-g}ss'
   +\sum_{h,k}\widetilde A_{h+g,k+g}
      t_ht_{h+k+g}t_k',
\\
\widetilde\rho(t_g')
&=\widetilde\omega\widetilde b\,t_{-g}s'
   +\widetilde\omega^{-1}ss't_{-g}'+\sum_{h,k}\widetilde A_{k+g,h+g}
      t_kt_{g+h+k}'t_h'.
\end{align*}
These are exactly the formulas in
\eqref{eq:rhos-general}--\eqref{eq:rhotg-prime-general} with the data in
\eqref{eq:display-0159}.  Also
$\widetilde\omega^3=1$ and
$(\widetilde b)^2=\omega^2b^2=(\widetilde\omega d)^{-1}$. It remains to check the matrix equations.  For \eqref{eq:HImatrix1},
\begin{equation*}
 \widetilde A_{g,h}=A_{h,g}
 =\omega^{-1}A_{g-h,-h}
 =\widetilde\omega\widetilde A_{-h,g-h}
 =\omega A_{-g,h-g}
 =\widetilde\omega^{-1}\widetilde A_{h-g,-g}.
\end{equation*}
For \eqref{eq:HImatrix2}, \eqref{eq:HI-row-sum} gives
\begin{equation*}
 \sum_r\widetilde A_{r,0}=\sum_rA_{0,r}
 =-\omega a=-\widetilde\omega^{-1}a.
\end{equation*}

Suppose first that $G$ is nontrivial.  By
Corollary~\ref{cor:HI-cubic-full-equivalence-general}, the original matrix satisfies
the cubic identity.  We show that the transposed matrix does as well.  Put
$n=m-g-h$.  Equation~\eqref{eq:HImatrix1} gives
\begin{equation*}
A_{g+h,m}=\omega^{-1}A_{n,-g-h},\qquad A_{m+k,g}=\omega A_{-g,n+h+k},\qquad\text{ and }A_{m+l,h}=\omega A_{-h,n+g+l}.
\end{equation*}
Therefore the left-hand side of the cubic identity for the transposed data is
\begin{align*}
 &\widetilde\omega^{-1}\sum_m
 \widetilde A_{m,g+h}\widetilde A_{g,m+k}\widetilde A_{h,m+l}=\omega^{-1}\sum_n
 A_{n,-g-h}A_{-g,n+h+k}A_{-h,n+g+l}.
\end{align*}
This is the cubic identity for $A$ with
$(g,h,k,l)=(-g,-h,h+k,g+l)$.  Hence
\begin{equation*}
 \widetilde\omega^{-1}\sum_m
 \widetilde A_{m,g+h}\widetilde A_{g,m+k}\widetilde A_{h,m+l}
 =A_{l,h+k}A_{k,g+l}-a\delta_{g,0}\delta_{h,0}.
\end{equation*}
Since the coefficients commute, the right-hand side is
\begin{equation*}
 \widetilde A_{g+l,k}\widetilde A_{h+k,l}
 -a\delta_{g,0}\delta_{h,0},
\end{equation*}
as required.  Corollary~\ref{cor:HI-cubic-full-equivalence-general} now shows that the
transposed data satisfy \eqref{eq:HImatrix1}--\eqref{eq:HImatrix4}.

If $G=\{0\}$, Remark~\ref{rem:HI-trivial-group} gives $\omega=1$, and the
transposed data are the original data.  This completes the proof.
\end{proof}

The first two composition identities \eqref{eq:LR1}--\eqref{eq:LR2}
follow from Lemmas~\ref{lem:HI-alpha-endomorphisms} and
\ref{lem:HI-translation-covariance}. For \eqref{eq:LR3}, apply
Lemma~\ref{lem:HI-unprimed-LR3} to the transposed data of
Lemma~\ref{lem:HI-prime-reversal}, then apply $\dagger$.
Since $f\mapsto\widetilde f$ preserves composition, this gives, for
$y\in\{s',t_h':h\in G\}$,
\begin{equation*}
 \rho^2(y)s=sy,\qquad
 \rho^2(y)t_r=t_r\alpha_r(\rho(y)).
\end{equation*}
For an unprimed generator, insert the completeness relation
$1=ss'+\sum_rt_rt_r'$ on the left and use
Lemma~\ref{lem:HI-unprimed-LR3}; for a primed generator, insert it on the
right and use the displayed identities. In both cases this yields
$\rho^2(x)=\Gamma(x)$. Both sides are algebra endomorphisms, so the identity
holds on all of $L$. It remains to check simplicity and orthogonality.

\subsection{Simplicity and orthogonality}
\label{subsec:HI-general-simplicity}

It remains to verify the three Hom-space conditions in the definition of a
Leavitt realization. Reduced words determine the diagonal space, and the
grading rules out the mixed spaces.

\begin{lemma}[Centralizers]
\label{lem:HI-centralizers}
Recall from \eqref{eq:rhos-general} that
\begin{equation*}
 S=as+b\sum_gt_gt_g,
 \qquad\text{ and }\qquad
 S'=as'+\omega b\sum_gt_g't_g'.
\end{equation*}
Then $\operatorname{Cent}_L(S)=\kk[S]$. In particular, $\operatorname{Cent}_L(S)\cap\operatorname{Cent}_L(S')=\kk$.
\end{lemma}

\begin{proof}
Let $x$ commute with $S$.  The coefficient of $s^m$ in $S^m$ is
$a^m\neq0$.  A term obtained by choosing any summand other than $as$ has a
$t$-letter, while powers with different exponents have different all-$s$
lengths.  Hence no other power of $S$ contains the reduced word $s^m$.
Starting with the largest $m$, subtract suitable scalar multiples of $S^m$.
This removes all pure powers of $s$ from the reduced expression for $x$.
It is enough to prove that the resulting $x$ is zero.

\smallskip
\noindent\emph{A word not beginning with $s'$.}
Suppose first that some reduced word in $x$ does not begin with $s'$.  Among
those words choose the maximal length $l$ of an initial string of $s$'s, and
collect the nonzero sum $w$ of terms with that initial string.  In $Sx$, the
summand $asw$ is nonzero because left multiplication by $s$ is injective.  It
consists of reduced words beginning with $s^{l+1}$.  No term in
$\bigl(b\sum_gt_gt_g\bigr)x$ begins with $s$, and right multiplication by $S$
does not increase the initial string of $s$'s of any remaining word, since no
such word is a pure power of $s$.  Hence these words cannot occur in $xS$,
contradicting $Sx=xS$.

\smallskip
\noindent\emph{Every word beginning with $s'$.}
It remains that every word in $x$ begins with $s'$.  Such a reduced word has
no unprimed block, so it is entirely primed.  Hence $x=s'x_1$ with
$x_1$ a sum of primed-only words.  Reducing $ss'=1-\sum_gt_gt_g'$ in $Sx$
produces the nonzero family
\begin{equation}
 -a\sum_gt_gt_g'x_1.
 \label{eq:display-0169}
\end{equation}
The family in \eqref{eq:display-0169} is nonzero when $x_1\neq0$, because
$a\neq0$ and its words are distinct reduced words beginning with $t_gt_g'$.
The quadratic
part of $Sx$ begins with two unprimed $t$'s.  In $xS$, reductions can occur
only where the final primed block of a word in $x$ meets the initial
unprimed block of a term in $S$.  These boundary cancellations cannot create
the initial block $t_gt_g'$.  Thus
$x_1=0$.
This proves $\operatorname{Cent}_L(S)=\kk[S]$.

It remains to impose commutation with $S'$.  Suppose that a polynomial
$P(S)$ of degree $m\geq1$ commutes with $S'$.  In $P(S)S'$, the leading
term contributes reduced words of length $2m+2$ by choosing
$b\sum_gt_gt_g$ from every copy of $S$ and
$\omega b\sum_gt_g't_g'$ from $S'$.  Their coefficients are nonzero.  No
word of this length occurs in $S'P(S)$: at the boundary between $S'$ and the
first copy of $S$, a primed letter followed by an unprimed letter contracts
or gives zero.  Lower
powers of $S$ also give shorter words.  This contradiction shows that
$m=0$, and hence the common centralizer is $\kk$.
\end{proof}

\begin{proposition}[Simplicity and orthogonality]
\label{prop:HI-simplicity}
For all $g,h\in G$,
\begin{align*}
\Hom(\alpha_g\rho,\alpha_h\rho)=\delta_{g,h}\kk,\qquad\Hom(\alpha_g\rho,\alpha_h)=0,\qquad\text{ and }\qquad\Hom(\alpha_g,\alpha_h\rho)=0.
\end{align*}
\end{proposition}

\begin{proof}
\smallskip
\noindent\emph{The diagonal space.}
If $x\in\Hom(\alpha_g\rho,\alpha_h\rho)$, evaluation at $s,s'$ shows that
$x$ commutes with $S=\rho(s)$ and $S'=\rho(s')$, because every $\alpha_u$
fixes these two elements.  By
Lemma~\ref{lem:HI-centralizers}, $x=\lambda\in\kk$.  If $\lambda\neq0$, then
evaluation at $t_0$ and Lemma~\ref{lem:HI-translation-covariance} give $\rho(t_{-2g})=\rho(t_{-2h})$. After reducing the $t_{-q}ss'$ term in \eqref{eq:rhotg-general}, the word
$st_{2g}'$ has coefficient $b$ in $\rho(t_{-2g})$ and coefficient zero in
$\rho(t_{-2h})$ unless $g=h$.  Since $b\neq0$, this proves the diagonal
formula.

\smallskip
\noindent\emph{The mixed spaces.}
For the mixed spaces, use the grading of
Lemma~\ref{lem:HI-reduced-words}.  The identities
\begin{equation*}
 t_0't_0'\left(\sum_gt_gt_g\right)=1
\qquad\text{ and }\qquad \left(\sum_gt_g't_g'\right)t_0t_0=1
\end{equation*}
show that right multiplication by $\sum_gt_g't_g'$ and left multiplication
by $\sum_gt_gt_g$ are injective.  Indeed, a zero product on the right may be
multiplied by $t_0t_0$, and a zero product on the left may be multiplied by
$t_0't_0'$, to recover the original element.
Let $0\neq x\in\Hom(\alpha_g\rho,\alpha_h)$ and write
$x=\sum_mx_m$ in homogeneous components.  Evaluation at $s'$ gives
$xS'=s'x$.  If $m_0$ is the least degree occurring in $x$, the degree
$m_0-2$ component of this equation is
$\omega b x_{m_0}\sum_gt_g't_g'=0$.  Since $\omega b\neq0$ and this right
multiplication is injective, we get $x_{m_0}=0$, a contradiction.

Finally, if $0\neq x\in\Hom(\alpha_g,\alpha_h\rho)$, evaluation at $s$ gives
$xs=Sx$.  If $M$ is the greatest degree occurring in $x$, its degree $M+2$
component is $b(\sum_gt_gt_g)x_M=0$.  Since $b\neq0$ and this left
multiplication is injective, we get $x_M=0$, again a contradiction.  This
argument does not require $a\neq1$ and remains valid when the characteristic
divides $N$.
\end{proof}

Remark~\ref{rem:LR-alpha-Hom} also gives
$\Hom(\alpha_g,\alpha_h)=\delta_{g,h}\kk$.

\begin{theorem}[HI-matrix reconstruction]
\label{thm:HImatrix-to-Leavitt-general}
Let $\kk$ be algebraically closed of characteristic $p\geq0$, let $G$ be a
finite abelian group of odd order $N=2n+1$, choose scalars satisfying
\eqref{eq:HI-scalars} and \eqref{eq:HI-general-b}, and let $A$ be a
$(d,\omega)$-HI-matrix.  Then
\eqref{eq:rhos-general}--\eqref{eq:alphag-general} define a Leavitt realization of
$\mathfrak{HI}_G$ in the sense of Section~\ref{sec:LeavittRealization}.
\end{theorem}

\begin{proof}
 Proposition~\ref{prop:HI-rho-endomorphism} and
 Lemma~\ref{lem:HI-alpha-endomorphisms} give the required unital algebra
 endomorphisms.  The preceding composition lemmas prove
\eqref{eq:LR1}--\eqref{eq:LR3}, and
Proposition~\ref{prop:HI-simplicity} proves all simplicity and orthogonality
conditions.
\end{proof}

\begin{remark}[Characteristic and scalar assumptions]
\label{rem:HI-characteristic-audit}
The proof divides only by $d$, $\omega$, and $b$.  The first two are nonzero
by \eqref{eq:HI-scalars}, and $b$ is nonzero by
\eqref{eq:HI-general-b}.  It takes one square root to choose $b$.  It never
divides by $2$, $3$, $N$, or the discriminant $N^2+4$, and it does not
average over $G$.

In characteristic $2$, the integer $N$ equals $1$ in $\kk$, while
multiplication by $2$ is still an automorphism of the abstract odd-order group
$G$; all sign identities in \eqref{eq:HImatrix1}--\eqref{eq:HImatrix4} and
the cubic equation are polynomial
identities in $\kk$.  If
$p\mid N$, then $d^2=1$, and both choices $d=1$ and $d=-1$ (when distinct)
are allowed.  If $p\mid(N^2+4)$, the polynomial for $d$ may have a repeated
root, but no root separation or discriminant inversion is used.  In
characteristic $3$, the only cube root of unity is $\omega=1$.
\end{remark}

Combining Theorem~\ref{thm:HImatrix-to-Leavitt-general} with
Theorem~\ref{thm:LRtoHIcat}, every $(d,\omega)$-HI-matrix therefore gives a
spherical $\mathfrak{HI}_G$-category in every characteristic.  We denote this
category by $\mathcal C(A)$ when its dependence on the matrix matters.


\section{The Q-system theorem}
\label{app:HI-q-system-algebra}\label{app:q-system}

We return to the reconstruction setup of Appendix~
\ref{app:HI-general-reconstruction}, with $\omega=1$ and $b^2=a$.
Thus $\kk$ is algebraically closed and $G$ is a finite abelian group of odd
order. We allow $N$ to be zero in $\kk$. The category used below is obtained
from Theorems~\ref{thm:HImatrix-to-Leavitt-general} and
\ref{thm:LRtoHIcat}.

This appendix identifies the algebraic meaning of the Q-system boundary
condition. The scalar equation shows that either $a=1$ or $a=-1$ forces
$N=0$ in $\kk$. Thus $a\neq\pm1$ whenever
$\operatorname{char}(\kk)$ does not divide $N$.  We include the exceptional
cases here. The theorem concerns associative and separable algebra
structures; it imposes no positivity condition.

An algebra is separable if its multiplication admits a section as a
bimodule map \cite[Definition~7.8.29, p.~146]{EGNO}.

\begin{theorem}[The Q-system condition]
\label{thm:HI-fixed-boundary-algebra}
\label{thm:HI-char-two-Q-system-boundary}
Let $A$ be a $(d,1)$-HI-matrix, let $\mathcal C=\mathcal C(A)$ be the
spherical $\mathfrak{HI}_G$-category reconstructed from $A$ by
Theorems~\ref{thm:HImatrix-to-Leavitt-general} and~\ref{thm:LRtoHIcat}, and put
$Q:=\unit\oplus\rho$.
\begin{enumerate}
\item If $a=1$, the restriction of every unital associative multiplication
  $m:Q\otimes Q\to Q$ to $\rho\otimes\rho$ is zero.  In particular,
  $Q$ has no separable algebra structure.
\item Suppose $a\neq1$.  Then
  the Q-system condition \eqref{eq:display-0271} holds
  if and only if $Q$ admits a unital associative algebra structure whose
  multiplication is nonzero on $\rho\otimes\rho$.  After rescaling the
  summand $\rho\subset Q$, the only component of the multiplication not fixed
  by the unit is its restriction to $\rho\otimes\rho$, namely
  \begin{equation}
   \label{eq:HI-fixed-boundary-multiplication}
   m_{\rho,\rho}=x s'\oplus t_0',
   \qquad x=\frac{b}{1-a}.
  \end{equation}
\item Suppose $a\neq1$.  If also $a\neq-1$, the Q-system
  condition in \eqref{eq:display-0271} is equivalent to the existence of a
  connected separable
  algebra structure on $Q$.  If instead $a=-1$, the algebra in
  part~\textup{(b)} is not separable.
\end{enumerate}
For the normalized algebra in part~\textup{(b)}, with $a\neq\pm1$, an
explicit separability element is
\begin{equation}
 \label{eq:HI-fixed-boundary-separability-element}
 e=r\,\operatorname{id}_{\unit}\oplus z s,
 \qquad
 r=\frac{a}{1+a},
 \qquad
 z=\frac{1-a}{b(1+a)}.
\end{equation}
The associated bimodule section has restrictions
\begin{align}
 \Delta|_{\unit}
 &=r\,\operatorname{id}_{\unit}\oplus z s,
 \label{eq:HI-fixed-boundary-coproduct-unit}\\
 \Delta|_\rho
 &=r\,\operatorname{id}_{\rho\otimes\unit}
   \oplus r\,\operatorname{id}_{\unit\otimes\rho}
   \oplus zb\,t_0.
 \label{eq:HI-fixed-boundary-coproduct-rho}
\end{align}
Thus $\Delta:Q\to Q\otimes Q$ is a $Q$-bimodule map satisfying
$m\circ\Delta=\operatorname{id}_Q$.
\end{theorem}

\begin{proof}[Proof of Theorem~\ref{thm:HI-fixed-boundary-algebra}]
We first write every possible multiplication and reduce associativity to
scalar equations. We then treat the cases $a=1$ and $a\neq1$. Finally, we
solve the equations for a separability element.

\emph{Multiplication and associativity.}
Since $\Hom(\unit,Q)=\kk$, we may normalize the unit to the canonical
inclusion $\unit\to Q$. The unit axioms then force the mixed multiplication
components to be identities. Write the component on $\rho\otimes\rho$ as
\begin{equation*}
 m_{\rho,\rho}=xs'\oplus yt_0'
\end{equation*}
for $x,y\in\kk$. Compare the two associativity composites on $\rho^3$.
Comparing the coefficients of $s'$ and of each $t_h't_h'$ gives
\begin{align}
 (1-a)x&=by^2,  \label{eq:display-0273}\\
 y^2\delta_{h,0}&=bx+y^2A_{0,h}
 \qquad(h\in G). 
 \label{eq:display-0274}
\end{align}
The two sides lie in the span of the displayed reduced words $s'$ and
$t_h't_h'$ (with $h\in G$); these words are linearly independent by
Lemma~\ref{lem:HI-reduced-words}.  Hence this coefficient comparison is
exhaustive for the component with codomain $\rho$.
The comparison uses
\begin{equation*}
 x s'+y^2t_0't_0'
 =x\rho(s')+y^2t_0'\rho(t_0'),
\end{equation*}
together with
\begin{equation*}
 \rho(s')=as'+b\sum_{h\in G}t_h't_h',
 \qquad
 t_0'\rho(t_0')=bs'+\sum_{h\in G}A_{0,h}t_h't_h'.
\end{equation*}
The component with codomain $\unit$ follows from
$s'\rho(t_0')=s't_0'$.

\emph{The case $a=1$.}
If $a=1$, equation~\eqref{eq:display-0273} gives $y=0$, since $b\neq0$.
Equation~\eqref{eq:display-0274} then gives $x=0$.  Thus the only
multiplication has $m_{\rho,\rho}=0$.  To see directly that it is not separable,
write a possible separability element as
$e=r\,\operatorname{id}_{\unit}\oplus zs$.  Comparing the left and right
$Q$-actions on $\rho$ gives $r=0$: the two contributions are in
the distinct summands $\rho\otimes\unit$ and $\unit\otimes\rho$, while the
$zs$ terms vanish because $m_{\rho,\rho}=0$.  On the other hand,
$m\circ e=\operatorname{id}_{\unit}$ gives $r=1$.  This contradiction proves
the claim.

\emph{The case $a\neq1$.}
Now suppose $a\neq1$. If $(x,y)\neq(0,0)$, then $y\neq0$. Rescale the
$\rho$ summand by the automorphism of $Q$ that is the identity on $\unit$ and
$\lambda\operatorname{id}_\rho$ on $\rho$.  Transporting the multiplication
along this automorphism changes the coefficients to
$(\lambda^{-2}x,\lambda^{-1}y)$.  Taking $\lambda=y$ normalizes $y=1$.
Equation~\eqref{eq:display-0273} then gives $x=b/(1-a)$, and
\eqref{eq:display-0274} gives
\begin{equation*}
 A_{0,h}
 =\delta_{h,0}-\frac{b^2}{1-a}
 =\delta_{h,0}+\theta.
\end{equation*}
Equation~\eqref{eq:HImatrix1} gives the same formula for $A_{g,0}$.
Conversely, this boundary and the values in
\eqref{eq:HI-fixed-boundary-multiplication} satisfy the two associativity
equations \eqref{eq:display-0273} and \eqref{eq:display-0274}.  This proves
the boundary equivalence and the formula for $m_{\rho,\rho}$.

\emph{Separability.}
The square-zero argument uses only $m_{\rho,\rho}=0$, so it excludes
separability for such a multiplication for every $a$. A separable
multiplication must therefore be nonzero on $\rho\otimes\rho$.
We now use the normalized multiplication with
$m_{\rho,\rho}\neq0$.

Every possible separability element has the form $e=r\,\operatorname{id}_{\unit}\oplus zs$, 
because the fusion rules give
$\Hom(\unit,Q\otimes Q)=\kk\operatorname{id}_{\unit}\oplus\kk s$.
Define
\begin{equation*}
 \Delta=(m\otimes\operatorname{id}_Q)
 (\operatorname{id}_Q\otimes e).
\end{equation*}
Comparing the left and right $Q$-actions on the $\rho$ summand, using
\begin{equation*}
 s'\rho(s)=\rho(s')s=a,
 \qquad
 t_0'\rho(s)=\rho(t_0')s=bt_0,
\end{equation*}
the two expressions agree exactly when
\begin{equation}
 \label{eq:HI-fixed-boundary-centrality}
 r=zxa.
\end{equation}
Under this relation the restrictions of $\Delta$ are
$\Delta|_{\unit}=r\,\operatorname{id}_{\unit}\oplus zs$ 
and
$\Delta|_\rho =r\,\operatorname{id}_{\rho\otimes\unit}
  \oplus r\,\operatorname{id}_{\unit\otimes\rho} \oplus zb\,t_0$. 
Thus $e$ is central precisely when $\Delta$ is a $Q$--$Q$-bimodule map.

The section equation on the unit summand is
\begin{equation}
 \label{eq:HI-fixed-boundary-section-unit}
 r+zx=1.
\end{equation}
On the $\rho$ summand it is $2r+zb=1$. 
The value of $x$ in \eqref{eq:HI-fixed-boundary-multiplication} and the values
of $r,z$ in \eqref{eq:HI-fixed-boundary-separability-element} satisfy
\begin{equation*}
 r=zxa,\qquad r+zx=1,\qquad 2r+zb=1.
\end{equation*}
For the last identity, the left-hand side is
\begin{equation*}
 \frac{2a}{1+a}+\frac{1-a}{1+a}=1.
\end{equation*}
Hence \eqref{eq:HI-fixed-boundary-separability-element} is a central
section whenever $a\neq-1$.

Conversely, \eqref{eq:HI-fixed-boundary-centrality} and
\eqref{eq:HI-fixed-boundary-section-unit} imply $zx(1+a)=1$. They are inconsistent when $a=-1$, and when $a\neq-1$ they force the values of $r$ and $z$ in \eqref{eq:HI-fixed-boundary-separability-element}. 
The section equation on the $\rho$ summand then follows from
$(1-a)x=b$.  This proves all separability assertions.  The equality
$\Hom(\unit,Q)=\kk$ proves connectedness.
\end{proof}

\begin{corollary}[Special Frobenius structure]
\label{cor:HI-fixed-boundary-special-Frobenius}
\label{cor:HI-char-two-Q-special-Frobenius}
Assume $a\neq\pm1$ and the equivalent conditions of
Theorem~\ref{thm:HI-fixed-boundary-algebra}.  Let
$\pi_{\unit}:Q\to\unit$ be the projection and put
\begin{equation}
 \label{eq:HI-fixed-boundary-counit}
 \varepsilon:=\frac{1+a}{a}\pi_{\unit}.
\end{equation}
Then $(Q,m,u,\Delta,\varepsilon)$ is a connected special Frobenius algebra,
with
\begin{equation*}
 m\Delta=\operatorname{id}_Q,
 \qquad
 \varepsilon u=\frac{1+a}{a}\operatorname{id}_{\unit}.
\end{equation*}
\end{corollary}

\begin{proof}
We verify the counit, Frobenius, coassociativity, and specialness conditions
in that order.

By \eqref{eq:HI-fixed-boundary-coproduct-unit} and
\eqref{eq:HI-fixed-boundary-coproduct-rho}, the only component surviving
either counit composite has coefficient
$\frac{1+a}{a}r=1$. 
Therefore
\begin{equation*}
 (\varepsilon\otimes\operatorname{id}_Q)\Delta
 =\operatorname{id}_Q
 =(\operatorname{id}_Q\otimes\varepsilon)\Delta.
\end{equation*}

Theorem~\ref{thm:HI-fixed-boundary-algebra} shows that $\Delta$ is a
$Q$--$Q$-bimodule map.  Hence
\begin{equation*}
 \Delta m
 =(m\otimes\operatorname{id}_Q)
   (\operatorname{id}_Q\otimes\Delta)
 =(\operatorname{id}_Q\otimes m)
   (\Delta\otimes\operatorname{id}_Q),
\end{equation*}
which is the Frobenius identity.

It remains to check coassociativity.  Centrality is the equality of morphisms
\begin{equation}
 (m\otimes\operatorname{id}_Q)(\operatorname{id}_Q\otimes e)
 =
 (\operatorname{id}_Q\otimes m)(e\otimes\operatorname{id}_Q).
 \label{eq:display-0291}
\end{equation}
For a compact verification, write $e=\sum_i x_i\otimes y_i$.  This is
Sweedler-style notation for composites in the strictified linear monoidal
category; it does not assume that $Q$ has a concrete realization.  In this
notation,
\eqref{eq:display-0291} says
$\sum_i qx_i\otimes y_i=\sum_i x_i\otimes y_iq$.
For each $i$, apply this identity with $q=x_i$, tensor with $y_i$, and sum
over $i$.  Relabeling the two dummy indices, and then attaching the input
$q$ to the first tensor factor, gives
\begin{equation*}
 \sum_{i,j}qx_ix_j\otimes y_j\otimes y_i
 =\sum_{i,j}qx_i\otimes y_ix_j\otimes y_j.
\end{equation*}
These are $(\Delta\otimes\operatorname{id}_Q)\Delta(q)$ and
$(\operatorname{id}_Q\otimes\Delta)\Delta(q)$, respectively.
Thus $\Delta$ is coassociative.  The section equations give
$m\Delta=\operatorname{id}_Q$, and the definition
\eqref{eq:HI-fixed-boundary-counit} gives
$\varepsilon u=(1+a)a^{-1}\operatorname{id}_{\unit}$.
\end{proof}

In characteristic $2$, put $C:=\theta(\theta-1)$. The scalar equation gives
\begin{equation*}
 a=d^2,\qquad b=d,\qquad
 \theta=d,\qquad C=1.
\end{equation*}
In this case, Theorem~\ref{thm:HI-fixed-boundary-algebra} gives
\begin{equation*}
 x=1,\qquad r=d,\qquad z=d^2,\qquad
 \varepsilon=d^2\pi_{\unit},
\end{equation*}
These are the characteristic-$2$ specializations of
\eqref{eq:HI-fixed-boundary-multiplication},
\eqref{eq:HI-fixed-boundary-separability-element}, and
\eqref{eq:HI-fixed-boundary-counit}.

\subsection{Cyclic Q-system normal form}

We return to arbitrary characteristic and now assume that $G$ is cyclic.
The Q-system boundary then allows reconstruction of the matrix from
half of its first nonzero row.

\begin{proposition}[Cyclic Q-system normal form]
\label{prop:HI-cyclic-Q-system-normal-form}
Let $G=\mathbb Z/N\mathbb Z$, with odd $N\geq3$, and let $A$ be a
$(d,1)$-HI-matrix satisfying the Q-system condition, with $d\ne1$.
Use the representatives $0,\ldots,N-1$ for inequalities between indices.
Put $B=(d-1)A$ and $v_g=B_{1,g}$. Then every $v_g$ is nonzero, and
\begin{equation*}
 v_0=v_1=-1,\qquad v_g=v_{1-g}.
\end{equation*}
Thus the independent coordinates are $v_2,\ldots,v_{(N+1)/2}$. 
For nonzero row coordinates, define
\begin{equation*}
 Q_0=Q_1=1,\qquad Q_s=\prod_{j=2}^s v_j.
\end{equation*}
The resulting product formula is
\begin{equation}\label{eq:char2-row-recovery}
 B_{g,h}=\frac{Q_h}{Q_gQ_{h-g}},\qquad
 B_{h,g}=\frac d{B_{g,h}}\qquad(0<g<h<N),
\end{equation}
with $B_{0,0}=d-2$ and all other axis and diagonal entries equal to $-1$. 
A row reconstructed by these formulas must still satisfy the remaining
HI equations.
\end{proposition}

\begin{proof}
Write $\theta=1/(1-d)$. The Q-system condition and
\eqref{eq:HImatrix1} give
$A_{g,0}=A_{0,g}=A_{g,g}=\theta+\delta_{g,0}$.
Taking $g=h=0$ and distinct nonzero $k,l$ in the cubic identity gives
\begin{equation*}
 A_{k,l}A_{l,k}=a+N\theta^3+3\theta^2
 =\frac{d}{(d-1)^2}.
\end{equation*}
Here the three exceptional summands have indices $0,-k,-l$;
the last equality uses $a=d^{-1}$ and $d^2-Nd=1$.
Thus $B_{k,l}B_{l,k}=d\ne0$, proving reciprocity and nonvanishing.

Next take $(g,h,k,l)=(x,0,y,t)$ in \eqref{eq:HI-cubic}, with
$x,y\ne0$. Expanding $A_{0,m+t}=\theta+\delta_{m,-t}$ and using
\eqref{eq:HImatrix3} with $(h,k)=(-y,x)$ to cancel the sum gives
\begin{equation*}
 B_{x+t,y}B_{y,t}=B_{-t,x}B_{x,y-t}.
\end{equation*}
Set $x=1-h$, $y=g$, $t=h$, where $1<g<h<N$.
By \eqref{eq:HImatrix1}, this becomes
\begin{equation*}
 v_gB_{g,h}=v_hB_{g-1,h-1}.
\end{equation*}
Starting with $B_{1,h}=v_h$, induction on $g$ proves
\eqref{eq:char2-row-recovery}. Finally,
\eqref{eq:HImatrix1} gives $v_h=B_{h-1,N-1}$ for $2\leq h<N$, so
\begin{equation*}
 v_h=\frac{Q_{N-1}}{Q_{h-1}Q_{N-h}}=v_{N+1-h}.
\end{equation*}
Together with $v_0=v_1=-1$, this proves the asserted row symmetry.
\end{proof}


\section{Analytic details of the double-sine construction}
\label{app:double-sine}

This appendix proves the two finite summation identities used in
Section~\ref{sec:double-sine}. We retain its notation: $N=2r+1\ge3$,
$d=d_-$, $\varpi_1$, $\varpi_2$, $\Omega$, $F$, and the sequence $s$.
We first collect the analytic facts, then evaluate the four-factor and
six-factor sums by contour integrals. Finally, we convert those
evaluations into the cyclic identities stated in
Lemmas~\ref{ds:lem:four-factor-sum} and~\ref{ds:lem:six-factor-sum}.

\subsection{Analytic facts}\label{ds:sec:hypGamma}
For the contour proof, we use the hyperbolic gamma function and its beta integral.
In the notation of \cite[\S4.1]{vBRS}, the hyperbolic
gamma function $G(z;\varpi_1,\varpi_2)$ is related to our double sine by
\begin{equation}\label{ds:short-analytic:Gamma-definition}
 \Gamma(w):=S(w)^{-1}=G\bigl(i(w-\Omega/2);\varpi_1,\varpi_2\bigr).
\end{equation}
Indeed, the right side satisfies the reciprocal shifts to
\eqref{ds:short-S-convention} and has value $1$ at $\Omega/2$. The quotient of the two sides has periods
$\varpi_1,\varpi_2$. These periods generate a dense subgroup of the real
line, so a meromorphic function with both periods is constant. The value
at the midpoint makes this constant equal to one.

Using the zeros and poles of the hyperbolic gamma function
\cite[Proposition~III.3]{ruijsenaars1997first}, together with
\eqref{ds:short-analytic:Gamma-definition} and $F(z)=S(\varpi_1+z)$,
we obtain, with $\Lambda=\mathbb Z_{\ge0}\varpi_1+\mathbb Z_{\ge0}\varpi_2$,
\[
 \operatorname{Zeros}(F)=-\varpi_1-\Lambda,\qquad
 \operatorname{Poles}(F)=\varpi_2+\Lambda.
\]
The integral representation in \cite[\S4.1, p.~47]{vBRS} also gives
$S(w)>0$ for $0<w<\Omega$.
In particular, $F$ has neither a zero nor a pole at an integer. To see this,
substitute $\varpi_2=\varpi_1+N$ into either set: the coefficient of the
irrational $\varpi_1$ is nonzero. Hence all values used in the sequence
are finite and nonzero.

We use the hyperbolic Saalsch\"utz integral
\cite[Corollary~4.6,(4.13)]{vBRS}, the hyperbolic analogue of the
nonterminating $q$-Saalsch\"utz formula
\cite[(2.10.12)]{GasperRahman2004}, in the form
\begin{equation}\label{ds:short-analytic:beta-formula}
 \int_C\prod_{j=1}^3\Gamma(c_j+z)\prod_{i=1}^3\Gamma(e_i-z)\,dz
 =i\sqrt{\varpi_1\varpi_2}\prod_{i,j=1}^3\Gamma(c_j+e_i),\qquad \sum_jc_j+\sum_ie_i=\Omega.
\end{equation}
The contour is oriented upward, with $-c_j-\Lambda$ to its left and
$e_i+\Lambda$ to its right. A pole from one family must not coincide with
a pole from the other. For the standard separating-contour convention
in Barnes integrals, see \cite[\S4.1]{GasperRahman2004}.
To convert the hyperbolic integral to our notation, use $\omega_1=\varpi_1$, $\omega_2=\varpi_2$, 
$\omega=\Omega/2$, with its variable $v=iz$ and parameters
$u_i=i(\Omega/2-e_i)$, $u_{3+j}=i(\Omega/2-c_j)$. Reflection gives the integrand in \eqref{ds:short-analytic:beta-formula}.
The source's left-to-right $v$-contour becomes a downward $z$-contour.
Reversing this orientation and using $dv=i\,dz$ gives the displayed
$+i$ for the upward contour. The identity extends meromorphically in
parameters satisfying the balancing condition, with separating contours
as in \cite[\S4.1 and~4.4]{vBRS}.

For integer tuples $P,Q$ of the same length with $\sum_iP_i=\sum_jQ_j$ in $\mathbb Z$, put
\begin{equation}\label{ds:short-eq:analytic:balanced-ratio}
 H(z)=\frac{\prod_iF(z+P_i)}{\prod_jF(z+Q_j)},\qquad
 J=\frac{\sum_iP_i^2-\sum_jQ_j^2}{2}.
\end{equation}
The ratio asymptotic \cite[\S4.1,(4.4)]{vBRS}, with
$\zeta=i(z-N/2)$, yields
\begin{equation}\label{ds:short-eq:analytic:balanced-limits}
 H(x+iy)\longrightarrow e^{\pm\pi iJ/(\varpi_1\varpi_2)}\quad(y\to\pm\infty),
\end{equation}
uniformly on bounded real $x$-intervals. Indeed,
\[H(z)=\frac{\prod_jG(\zeta+iQ_j;\varpi_1,\varpi_2)}{\prod_iG(\zeta+iP_i;\varpi_1,\varpi_2)}.\]
Balance also makes $J$ an integer, since $x^2\equiv x\pmod2$ for integers $x$. The balancing condition  
cancels the terms linear in the argument; the
remaining constant exponent at $\Re\zeta\to+\infty$ is $-\pi iJ/(\varpi_1\varpi_2)$. Reflection gives the other limit.

\subsection{Separating contours and the four-factor sum}\label{ds:sec:contours}
Our aim is to express a finite sum of values $H(k)$ as a contour integral.
We use three translates of a contour separating the two families of poles.
In each difference of contours, the poles of $H$ have winding number zero.
Only cotangent poles at integers contribute, and their residues give the
required finite sum.
Assume now $Q_j\ge0$ and $P_i\not\equiv Q_j\pmod N$. Set
\[
 \mathcal L=\bigcup_j(-Q_j-\Lambda),\qquad
 \mathcal R=\bigcup_i(-P_i+\Lambda),\qquad \rho=(\varpi_1-1)/2\in(0,1).
\]
These two locally finite sets are disjoint: an opposing coincidence would equate a nonzero integer $P_i-Q_j$ to an element of $\Lambda$, impossible by irrationality. The line $\Re z=\rho$ avoids these arrays and their translates by $\varpi_1$ and $N$: their points lie in $\mathbb Z+\varpi_1\mathbb Z$, whereas $\rho=(\varpi_1-1)/2$ does not, by irrationality. We orient the line upward and describe the contour $C$ as a sum of paths: $\rho+i\mathbb R$ minus one small positively oriented circle about each \emph{distinct} point of $\mathcal R$ to the left of that line. There are finitely many such points. Circles can be joined by retraced arcs, giving the usual Barnes detours. Choose the discs disjoint, with each containing no other point of the locally finite union of the polar sets below, their relevant translates, and $\mathbb Z$. Their boundaries then avoid all these points. For any point $w$ off these boundaries, let $\ell_C(w)$ denote its left-side index: $1$ to the left of the line, less the indicators of these discs. Thus $\ell_C=1$ on $\mathcal L$ and $0$ on $\mathcal R$.

Choose constants $\xi_i$ with $\sum_i \xi_i=2$ and put
\begin{align}
 D_q(z)&=\sum_i \xi_i\cot\frac{\pi(z+P_i)}q\quad(q=\varpi_1,\varpi_2),\notag\\
 E_H(z)&=2\cot(\pi z)\bigl(H(z+N)-H(z)\bigr)\notag\\
 &\quad-D_{\varpi_2}(z)\bigl(H(z+N)-H(z-\varpi_1)\bigr)
 +D_{\varpi_1}(z)\bigl(H(z)-H(z-\varpi_1)\bigr).\label{ds:short-eq:analytic:folding-remainder}
\end{align}
For the four-factor sum, $E_H$ will vanish. For the six-factor sum, we choose the coefficients so that $E_H$ becomes one hyperbolic beta integrand.
\begin{lemma}[A finite sum from a separating contour]\label{ds:short-lem:analytic:folding}
If $E_H$ is absolutely integrable on $C$, then
\begin{equation}\label{ds:short-eq:analytic:folding}
 \int_C E_H(z)\,dz
 =4i\sum_{k\in\mathbb Z}\bigl(\ell_C(k-N)-\ell_C(k)\bigr)H(k).
\end{equation}
The sum is finite. No restriction on same-side pole multiplicities is required.
\end{lemma}
\begin{proof}
At height $T$, change variables by $N$ and $-\varpi_1$ in \eqref{ds:short-eq:analytic:folding-remainder}. With all vertical paths upward, the resulting combination is
\begin{equation}\label{ds:short-eq:analytic:folding-chain}
 2\int_{C+N-C}\!\cot(\pi z)H(z)\,dz
 -\int_{C+N-(C-\varpi_1)}\!D_{\varpi_2}(z+\varpi_1)H(z)\,dz
 +\int_{C-(C-\varpi_1)}\!D_{\varpi_1}(z)H(z)\,dz.
\end{equation}
Close each contour difference by horizontal segments. Before applying the
coefficients in \eqref{ds:short-eq:analytic:folding-chain}, the combined
contributions from the top and bottom segments tend respectively to $iN(h_++h_-)$, 
$2i\varpi_2(h_++h_-)$, $2i\varpi_1(h_++h_-)$, where $h_\pm$ are \eqref{ds:short-eq:analytic:balanced-limits}. 
With the coefficients $2,-1,1$ in \eqref{ds:short-eq:analytic:folding-chain}, their sum is 
$2i(N-\varpi_2+\varpi_1)(h_++h_-)=0$. Thus the residue theorem applies to \eqref{ds:short-eq:analytic:folding-chain} after taking $T\to\infty$.

Every pole of $H$ is in $-P_i+\varpi_2+\Lambda$ or $-Q_j-\varpi_1-\Lambda$. 
In the first case all of $w,w+\varpi_1,w-N$ belong to $\mathcal R$; in the second 
they all belong to $\mathcal L$. Its winding coefficient is therefore zero 
in \emph{each} of the three contour differences. This argument applies to the whole meromorphic integrand in each term of
\eqref{ds:short-eq:analytic:folding-chain}. It therefore also covers
higher-order poles and points where a cotangent pole coincides with a
pole of $H$.

The additional poles of $D_{\varpi_2}(z+\varpi_1)$ are $w=-P_i-\varpi_1+k\varpi_2$. For $k\ge1$, both $w+\varpi_1,w-N$ belong to $\mathcal R$, so their winding difference is zero. For $k\le0$, the numerator contains the zero $S(k\varpi_2)$. No other numerator can have a pole there: its $\varpi_1$-coefficient would require $k=m+n+2>0$. No denominator can have a zero there: comparison of $\varpi_1$-coefficients would force $Q_j-P_i\in N\mathbb Z$. Hence the zero cancels the simple cotangent pole.

Similarly, at an additional $D_{\varpi_1}$-pole $w=-P_i+k\varpi_1$, the points $w,w+\varpi_1$ are in $\mathcal R$ for $k\ge0$. For $k\le-1$ there is a zero $S((k+1)\varpi_1)$; another numerator pole would require $k+1=m+n+2>0$, and a denominator zero would again force a forbidden congruence. This exhausts the cotangent poles.

Only the integer poles of $\cot(\pi z)$ remain. $H$ is regular there, so $\cot(\pi z)H(z)$ has residue $H(k)/\pi$, and the first contour difference has index $\ell_C(k-N)-\ell_C(k)$. The coefficient $2$ in \eqref{ds:short-eq:analytic:folding-chain} gives \eqref{ds:short-eq:analytic:folding}. Local finiteness and the fixed horizontal widths justify the finite residue sums throughout.
\end{proof}

\emph{The four-factor sum.}
For $1\le p\le h\le r$, take $P=(-p,h)$, $Q=(0,h-p)$, and denote $H$ by $R$. The tuples have no entry in common modulo $N$, since $p,h,p+h$ are nonzero modulo $N$. Choose $\xi_1=\xi_2=1$. Then $E_R=0$.
Here is an algebraic check which will also be useful for six factors. Put
$x=e^{2\pi iz/\varpi_1}$, $y=e^{2\pi iz/\varpi_2}$ and $t_i=e^{2\pi iP_i/\varpi_1}$.
If $\mathcal P=\prod_i(1-t_ix)$, $\mathcal B=\prod_i(y-t_i)$ and
$X_q=\prod_j\sin(\pi(z+Q_j)/q)/\prod_i\sin(\pi(z+P_i)/q)$, the shifts give
$H=X_{\varpi_2}H(z-\varpi_1)$ and $H(z+N)=X_{\varpi_1}H(z-\varpi_1)$. In the two-factor case,
$X_{\varpi_1}-1=\alpha x/\mathcal P$, $X_{\varpi_2}-1=\alpha y/\mathcal B$ for some $\alpha$, by exact balance. If $t=t_1,u=t_2$, then
\[
 \frac{\cot\pi z}{i}=\frac{xy+1}{xy-1},\quad
 \frac{D_{\varpi_1}}{2i}=\frac{tu x^2-1}{(1-tx)(1-ux)},\quad
 \frac{D_{\varpi_2}}{2i}=\frac{y^2-tu}{(y-t)(y-u)}.
\]
Substitution in \eqref{ds:short-eq:analytic:folding-remainder} gives zero: the numerator cancels using
\[
 x\mathcal B-y\mathcal P=(xy-1)(y-tu x),\qquad
 x(y^2-tu)-y(tu x^2-1)=(xy+1)(y-tu x).
\]

At integers $\ell_C(k)=\mathbf1_{\{k\le0\}}-\mathbf1_{\{k=-h\}}$: the only integer point moved to the right is $-h$. Lemma~\ref{ds:short-lem:analytic:folding} therefore gives
$\sum_{j=1}^{N}R(j)=R(N-h)-R(-h)$.
Using \eqref{ds:short-eq:analytic:F-properties} at the nonzero integers $p,h,p+h$ gives
\begin{equation}\label{ds:short-analytic:two-index-moment}
 \sum_{j=1}^{N}\frac{F(j-p)F(j+h)}{F(j)F(j+h-p)}
 =\varpi_2\frac{F(p)F(h)}{F(0)F(p+h)}=\varpi_2 R(N-h),
 \qquad R(-h)=-d_+ R(N-h).
\end{equation}
This includes $p=h$; no simple-pole specialization was used.

\subsection{The cubic sum from one beta integral}\label{ds:short-sec:cubic}
Call $p,q\in G^3$ \emph{balanced} if $\sum_i p_i=\sum_j q_j$ in $G$, that is, modulo $N$. First suppose the triples are balanced and have no entry in common, so $p_i\ne q_j$ for every $i,j$. Translating both triples by the same element preserves both conditions. Translate and permute so $p_1=0$; then $q_j\ne0$ for every $j$. Choose $Q_j\in\{1,\ldots,N-1\}$, and set
\begin{equation}\label{ds:short-cubic:lifts}
 \sigma=\sum_jQ_j,\quad U\equiv r_0(\mathrm{mod}N),\quad U>\sigma,\quad V=\sigma-U,\quad
 P=(0,V,U),\qquad 0\le r_0<N,
\end{equation}
where $r_0$ represents $p_3$; for instance $U=r_0+3N$ works since $\sigma\le3(N-1)$. Then $V<0<U$, the $P_i$ are distinct integers, and the two lifted sums agree. Use $H$ from \eqref{ds:short-eq:analytic:balanced-ratio}, and put
$\Pi=\prod_{i,j}F(P_i-Q_j)$.

Retain $x,y,t_i,\mathcal P,\mathcal B$ above, set
$u_j=e^{2\pi iQ_j/\varpi_1}$ and $T=\prod_it_i=\prod_ju_j$, and define
\begin{equation}\label{ds:short-cubic:alpha-beta}
 \prod_j(1-u_jx)-\mathcal P=x(\alpha+\beta x),\qquad
 A_i=\alpha t_i+\beta=t_i^{-1}\prod_j(t_i-u_j).
\end{equation}
All $A_i$ and all $t_i-t_j$ for $i\ne j$ are nonzero: the corresponding integer differences cannot be multiples of the irrational $\varpi_1$.
Linear interpolation of the affine expression $\alpha t+\beta$ determines coefficients for which $E_H$ is a constant multiple of the integrand $I_3$ in the hyperbolic beta formula. Choose
\begin{equation}\label{ds:short-eq:cubic:adapted-weights}
 w_1=\frac{(t_1-t_3)A_2}{(t_1-t_2)A_3},\quad
 w_2=\frac{(t_3-t_2)A_1}{(t_1-t_2)A_3},\qquad
 (\xi_1,\xi_2,\xi_3)=(w_2,w_1,1).
\end{equation}
Since $A_i$ is affine in $t_i$, $w_1+w_2=1$, as required in Lemma~\ref{ds:short-lem:analytic:folding}.
For these coefficients, $E_H$ is a constant multiple of that integrand:
\begin{equation}\label{ds:short-eq:cubic:single-beta}
 E_H=B I_3,\qquad
 I_3=\Gamma(\Omega-z-P_3)\Gamma(-z-P_1)\Gamma(-z-P_2)\prod_j\Gamma(z+Q_j),\quad
 B=4i\frac{A_1A_2}{A_3}(-1)^{\sigma-P_3}\frac{t_3}{T}.
\end{equation}
We verify this identity explicitly. Write
$W_{ij}=(t_i+t_j)(xy+1)-2t_it_jx-2y$ and $k_i=(1-t_ix)(y-t_i)$.
Let $E_{ij}$ denote the expression obtained by assigning coefficient one
to the pair $i,j$ and zero to the remaining index. Substitution of
$X_{\varpi_1}-1=x(\alpha+\beta x)/\mathcal P$ and
$X_{\varpi_2}-1=y(\alpha y+\beta)/\mathcal B$ in \eqref{ds:short-eq:analytic:folding-remainder} gives
\begin{equation}\label{ds:short-eq:cubic:pair-certificate}
 \frac{E_{ij}}{H(z-\varpi_1)}=-\frac{2iA_kxyW_{ij}}{\mathcal P\mathcal B}\;\;(\{i,j,k\}=\{1,2,3\}),\quad
 \frac{(t_1-t_3)W_{23}+(t_3-t_2)W_{13}}{t_1-t_2}=-2k_3.
\end{equation}
The second identity follows by comparing the coefficients of $xy,x,y,1$. Hence
\[
 E_H=w_1E_{23}+w_2E_{13}=4i\frac{A_1A_2}{A_3}\frac{xyk_3}{\mathcal P\mathcal B}H(z-\varpi_1).
\]
Finally, reflection and the $\Omega$-shift of $\Gamma$ give
\begin{equation}\label{ds:short-eq:cubic:single-beta-ratio}
 \frac{I_3}{H(z-\varpi_1)}=(-1)^{\sigma-P_3}\frac{T\,xyk_3}{t_3\mathcal P\mathcal B}.
\end{equation}
To check the sign, write each removed factor as
$-4\sin(\pi(z+P_i)/\varpi_1)\sin(\pi(z+P_i)/\varpi_2)
=(-1)^{P_i+1}e^{-\pi iz}k_i/t_i$; two reciprocal factors give \eqref{ds:short-eq:cubic:single-beta-ratio}. This proves \eqref{ds:short-eq:cubic:single-beta}. Equations \eqref{ds:short-eq:analytic:balanced-limits} and \eqref{ds:short-eq:cubic:single-beta-ratio} also give $I_3(z)=O(e^{-2\pi|\Im z|})$ uniformly for $\Re z$ in bounded intervals as $|\Im z|\to\infty$, so the integral of $E_H$ converges absolutely.

We must also check that the beta integral can be evaluated on the same
contour, including when poles within one family coincide. The contour $C$
of Lemma~\ref{ds:short-lem:analytic:folding} separates the beta poles: its left pole families are $-Q_j-\Lambda$ and its right pole families are subsets of $\mathcal R$, with parameters $e=(-P_1,-P_2,\Omega-P_3)$. Their balance is $\Omega$. A circle about a point of $\mathcal R$ that is not a beta pole encloses a regular point of $I_3$ and contributes zero. 

Put $\Delta_{ij}=Q_j-P_i$. The arguments $\Delta_{ij}$ and $\Omega+\Delta_{ij}$ are neither zeros nor poles of $\Gamma$. For example, $\Delta_{ij}=-m\varpi_1-n\varpi_2$ forces $m=n=0$ and then $\Delta_{ij}=0$; the other three divisor equalities are excluded by the same comparison of $\varpi_1$-coefficients. Thus no left pole coincides with a right pole. Repeated $Q_j$ can merge
left poles, and repeated residue classes of the $P_i$ can merge right poles. 
To apply \eqref{ds:short-analytic:beta-formula} at these parameters, keep
each cluster of coincident poles inside a fixed disc disjoint from the
contour and the opposite family as the parameters vary. 

To control the two unbounded ends of the contour, use the full asymptotic
formula for $G$ in \cite[Proposition~III.4]{ruijsenaars1997first}, with $a_+=\varpi_1$, $a_-=\varpi_2$. Its error is uniform for the imaginary argument in compact intervals. Under \eqref{ds:short-analytic:Gamma-definition}, bounded parameter perturbations stay in such intervals; balance makes the combined exponential $e^{2\pi iz}$ at the upper end and $e^{-2\pi iz}$ at the lower end, times a locally bounded constant. Thus the tails are uniformly $O(e^{-2\pi|\Im z|})$. On the compact contour portion convergence is uniform. We may therefore pass to the limiting parameters in the integral. This
justifies the formula without computing residues at multiple poles.

With $C_0=\prod_{i,j}\Gamma(\Delta_{ij})$ and $C_3=\prod_{i,j}\Gamma(Q_j+e_i)$ the gamma-product in \eqref{ds:short-analytic:beta-formula}, the shifts give
\begin{equation}\label{ds:short-eq:cubic:single-beta-constants}
 \frac{C_3}{C_0}=(-1)^{\sigma-P_3+1}\frac{A_3^2}{Tt_3},\qquad
 \Pi=i\frac{A_1A_2A_3}{T^2}C_0,\qquad
 \int_C E_H=-4i\sqrt{\varpi_1\varpi_2}\,\Pi.
\end{equation}
Indeed, for an integer $\Delta=Q_j-P_i$,
$\Gamma(\Omega+\Delta)/\Gamma(\Delta)
=(-1)^{\Delta+1}(t_i-u_j)^2/(t_iu_j)$ and
$F(-\Delta)=2\sin(\pi\Delta/\varpi_1)\Gamma(\Delta)$.
The nine latter factors have total exponent $\sum_{i,j}\Delta_{ij}=0$, leaving $i^9=i$; these formulas prove all constants in \eqref{ds:short-eq:cubic:single-beta-constants}.

At integers the present contour has
$\ell_C(k)=\mathbf1_{\{k<0\}}-\mathbf1_{\{k=-U\}}$; the points $0,-U$ are moved to the right, while $-V$ is a positive integer to the right of $\rho$. Therefore \eqref{ds:short-eq:analytic:folding} and \eqref{ds:short-eq:cubic:single-beta-constants} imply
\begin{equation}\label{ds:short-cubic:finite-sum}
 \sum_{m=0}^{N-1}H(m)+\varpi_1 H(-U)=-\sqrt{\varpi_1\varpi_2}\,\Pi,
 \qquad
 \sum_{m=0}^{N-1}H(m)-\varpi_2 H(m_*)=-\sqrt{\varpi_1\varpi_2}\,\Pi,
\end{equation}
where $m_*$ represents $-r_0$. To obtain the first equality, use
$H(-U+N)=d H(-U)$ in the finite sum. To obtain the second, follow
$-U,-U+N,\ldots,m_*$. Every step starts at a negative integer. Only the $U$-factor starts a step at zero; no $Q_j$-factor can do so because the triples have no entry in common modulo $N$, while the $V$-origin $-V$ is positive and the $P_1$-origin is zero. Generic signs cancel by balance, while the initial exceptional ratio is $d$ times its generic sign. Thus $H(m_*)=d H(-U)$ and $\varpi_1/d=-\varpi_2$.

\subsection{Conversion to cyclic sums}\label{ds:app:cyclic-sums}
We now prove Lemmas~\ref{ds:lem:four-factor-sum}
and~\ref{ds:lem:six-factor-sum}, using the sums $K$, $M$, and $T$
defined in Section~\ref{ds:short-sec:matrix}.

\begin{proof}[Proof of Lemma~\ref{ds:lem:four-factor-sum}]
We use the ratio $R$ from Appendix~\ref{ds:sec:contours}.
For $1\le p\le h\le r$, use $m=j-p$, $1\le j\le N$, in $K(p,h)$. The four integer indices used in the formula for $\bar s$ are $j-p,j+h,-j,p-h-j$; only $j+h$ can be a positive multiple of $N$, when $j+h=N$. The product identity for values at opposite indices and the phase difference $ph$ therefore give
\[
 K(p,h)=\kappa^2(-1)^{ph}\left(\sum_{j=1}^{N}R(j)+(d^{-1}-1)R(N-h)\right)=0
\]
by \eqref{ds:short-analytic:two-index-moment}. The sum is symmetric in $x,y$, and substituting $m=-u-y$ in $K(-x,y)$ shows evenness in $x$, hence also in $y$. Thus this covers every nonzero $x,y$ at odd order. On the axes, expansion of $R_0(m)R_0(m+y)$ and \eqref{ds:matrix-scalar-identities} give \eqref{ds:short-eq:matrix:K-all}.
\end{proof}

\begin{proof}[Proof of Lemma~\ref{ds:lem:six-factor-sum}]
For triples with no entry in common, use the integer representatives \eqref{ds:short-cubic:lifts}. The product formed using $\bar s$ is
$\kappa^3(-1)^JH(m)$, since all $m+Q_j$ are positive and nonzero. Exactly one factor, $m+U$, meets a positive multiple of $N$, at $m_*$. Thus
\[
 M(p,q)=\kappa^3(-1)^J\left(\sum_{m=0}^{N-1}H(m)-\varpi_2 H(m_*)\right)
 =-\kappa^3(-1)^J\sqrt{\varpi_1\varpi_2}\,\Pi.
\]
Every difference $P_i-Q_j$ in the product $T(p,q)$ avoids multiples of $N$, so
$T(p,q)=\Omega^{-9/2}(-1)^{J_\partial}\Pi$, where
$J_\partial=\frac12\sum_{i,j}(P_i-Q_j)^2$.
Both exponents are integers and
$J_\partial-J=\sum_iP_i^2+2\sum_jQ_j^2-\sigma^2$ is even. Since $\kappa=-1/\Omega$ and $\Omega=\varpi_1\varpi_2$, we obtain \eqref{ds:short-eq:matrix:six-factor}.
Translations and separate permutations preserve both sides, so the chosen lifts impose no restriction.
\end{proof}


\section{Detailed proofs for connected separable algebras}
\label{app:separable}

This appendix gives the algebraic details behind Section~\ref{sec:separable}.
Throughout, $G=\bZ/N\bZ$ has odd order $N\geq3$, the pointed
associator is trivial, and $\C$ is the fixed complex unitary cyclic
HI category.  The object $Q=\one\oplus\rho$ is equipped with
the connected separable algebra structure assumed in the main text.  We use
$\operatorname{FPdim}(\rho)=d$ and the normalized Hermitian HI matrix $A$
of the fixed category.  The statements below are formulated in the strict
pointed normalization; the negative scalar branch in the HI equations is
allowed in the intermediate Fourier argument and carries no unitarity claim.

\begin{remark}[Normalized presentation]
\label{sap:normalized-presentation}
The reconstruction and normalization results used here identify the fixed
category with the endomorphism category of its normalized Hermitian HI
matrix.  We use the positive Q-system boundary from the main text and the
matrix reconstruction of Evans--Gannon, including their equivalence
criterion under automorphisms of $G$; see
\cite[Theorems~1 and~2(a)--(c), and the paragraph after equation~(4.11)]{evans2017non}.
The only phase parametrization used later is the one recorded in
\cite[Section~7.2]{evans2017non}, where the earlier source is attributed as
\cite{IzumiLR2}.
\end{remark}

For a Hall subgroup $H$, the subgroup transform $\mathscr P_HA$
is defined by \eqref{sap:eq:partial-fourier}, with $B=A$.
The Q-system boundary is \eqref{eq:display-0271}.

The main text states the subgroup-dual description in
Section~\ref{sa:subgroup-dual}, the description of the $Q$-dual in
Section~\ref{sa:Q-dual}, and the classification in
Theorem~\ref{sa:classification}.  The proofs below follow that order.  The
partial Fourier theorem is stated here
separately because it is stronger: it applies to every full untwisted HI
matrix, including the negative scalar branch, before the categorical
identification is made.

\begin{proposition}[Uniqueness of algebra structures]
\label{sap:prop:short-uniqueness}
Every connected separable algebra structure on $\one\oplus\alpha_g\rho$
gives an algebra isomorphic to a conjugate of the fixed algebra
$Q=\one\oplus\rho$ by an invertible object.
\end{proposition}

\begin{proof}
Let $X=\alpha_g\rho$.  The HI fusion rules give
$\dim\Hom(X\otimes X,X)=1$, and $X$ is simple.  Every separable algebra has
a semisimple module category by \cite[Proposition~7.8.30]{EGNO}, so
\cite[Lemma~8]{OstrikModules} gives at most one algebra-isomorphism class of
such structures on $\one\oplus X$.  The conjugate algebra
$\alpha_{g/2}Q\alpha_{-g/2}$ has underlying object
$\one\oplus\alpha_g\rho$, and hence realizes this unique class.
The invertible bimodule $\alpha_{g/2}Q$ gives its Morita equivalence with $Q$.
\end{proof}

\subsection{Separable algebras}
\label{sap:sec:separable-algebras}
\label{sap:sec:standard-algebras}

For $m,n$ in a finite semisimple $\C$-module category $\M$, the internal
Hom $\underline{\Hom}(m,n)$ represents
$X\mapsto\Hom_{\M}(X\otimes m,n)$.  Its internal End
$\underline{\End}(m)$ is an algebra.  If $\M$ is indecomposable and $m$ is
simple, this algebra is connected and separable, and
$\M\simeq\operatorname{Mod}_{\C}(\underline{\End}(m))$
\cite[Section~3.2 and Theorem~1]{OstrikModules}; see also
\cite[Sections~7.9--7.10]{EGNO}.

The next three lemmas let us normalize multiplication on the pointed
summands, compare Frobenius--Perron dimensions, and pass an algebra to a
dual category.  These lemmas
apply to fusion categories without a unitarity assumption.

\begin{lemma}
\label{sap:lem:pointed-truncation}
Let $\mathcal E$ be a fusion category and let $A$ be a connected separable
algebra in $\mathcal E$.
\begin{enumerate}[label=(\roman*)]
 \item The regular right module $A_A$ is simple and generates
 $\operatorname{Mod}_{\mathcal E}(A)$ as an $\mathcal E$-module category.
 In particular, this module category is indecomposable.
 \item The invertible simple summands of $A$ form a subgroup $K$,
 each occurring once.  Then $A_{\mathrm{pt}}:=\bigoplus_{u\in K}u$ is a
 subalgebra with nonzero multiplication coefficients $\psi(u,v)$.
 If its pointed associator is $\omega$, then
 $\delta\psi=\omega^{-1}$, where
 \begin{equation*}
  (\delta\psi)(u,v,w)
  =\frac{\psi(v,w)\psi(u,vw)}{\psi(uv,w)\psi(u,v)}.
 \end{equation*}
 Thus $\psi$ is a $2$-cochain whose coboundary trivializes $\omega$; in
 particular, $[\omega]=0$.
 \item Under the hypotheses of (ii), if $\omega=1$ and
 $H^2(K,\bC^{\times})=0$, then $A$ is
 algebra-isomorphic to an algebra on the same object whose pointed
 subalgebra has the standard group-algebra multiplication.
\end{enumerate}
\end{lemma}

\begin{proof}
Part (i) follows from connectedness, free-module adjunction, and the
semisimplicity argument in
\cite[Lemma~7.8.12 and Proposition~7.8.30, including its proof]{EGNO}.

For (ii), free-module adjunction gives $\Hom_{\mathcal E}(u,A)\cong\Hom_A(u\otimes A,A)$ for every invertible simple $u$.  Both right modules are simple by (i),
so this space has dimension zero or one.  It is nonzero precisely when
$u\otimes A\cong A$.  Thus the invertible constituents form the stabilizer
subgroup $K$ of $A_A$, and each occurs once.
For $u\in K$, multiplication gives a nonzero right-module map
$u\otimes A\to A$, hence an isomorphism.
Thus its coefficients on the invertible summands are nonzero, and
$A_{\mathrm{pt}}$ is a twisted group algebra as in
\cite[Example~7.8.3(4)]{EGNO}.  To fix the cocycle convention, associativity
reads $\psi(u,v)\psi(uv,w)=\omega(u,v,w)\psi(v,w)\psi(u,vw)$, which gives
$\delta\psi=\omega^{-1}$.

For (iii), a coboundary is removed by rescaling the homogeneous summands
\cite[Section~2.6, equations~(2.32)--(2.33)]{EGNO}.  Extend this normalized
rescaling by the identity on the noninvertible summands of $A$ and transport
the multiplication.
\end{proof}

\begin{lemma}[Dimensions under relative tensor product]
\label{sap:lem:relative-dimension}
Let $R$ be a connected separable algebra in a fusion category $\mathcal A$.
For $R$-bimodules $X,Y$,
\begin{equation*}
 \FPdim_{\mathcal A}(Y\otimes_R X)
 =\frac{\FPdim_{\mathcal A}(Y)\FPdim_{\mathcal A}(X)}
 {\FPdim_{\mathcal A}(R)},
 \qquad
 \FPdim_{{}_R\mathcal A_R}(X)
 =\frac{\FPdim_{\mathcal A}(X)}{\FPdim_{\mathcal A}(R)}.
\end{equation*}
The first formula also holds when $Y$ is only a right $R$-module.
There is an analogous formula for a bimodule tensored with a left module.
\end{lemma}

\begin{proof}
The second formula is \cite[Exercise~7.16.9(i)]{EGNO}; multiplicativity in
${}_R\mathcal A_R$ gives the first.  For a right $R$-module $Y$, apply the
bimodule formula to $R\otimes Y$ and $X$.  The isomorphism
$(R\otimes Y)\otimes_R X\cong R\otimes(Y\otimes_R X)$ lets us cancel
$\FPdim_{\mathcal A}(R)$ and obtain the claimed formula.  The left-module
case is analogous, using $Y\otimes R$.
\end{proof}

For a connected separable algebra $R$ in a fusion category $\mathcal A$,
write $\mathfrak T_R$ for the Morita $2$-equivalence
from $\mathcal A$-module categories to ${}_R\mathcal A_R$-module categories,
given by
\begin{equation*}
 \mathfrak T_R(\mathcal M)=\operatorname{Fun}_{\mathcal A}
 (\operatorname{Mod}_{\mathcal A}(R),\mathcal M).
\end{equation*}
The module category of the connected separable algebra is nonzero and
semisimple, hence faithful and exact; its simple regular module makes it
indecomposable.

\begin{lemma}[Algebra extensions in the dual category]
\label{sap:lem:algebra-extension-transport}
Let $R\hookrightarrow A$ be a unital embedding of connected separable
algebras in a fusion category $\mathcal A$.  Then ${}_RA_R$ is a connected
separable algebra in ${}_R\mathcal A_R$, with multiplication induced by
$A\otimes_R A\to A$.  Moreover,
\begin{equation}
 \mathfrak T_R\bigl(\operatorname{Mod}_{\mathcal A}(A)\bigr)
 \simeq
 \operatorname{Mod}_{{}_R\mathcal A_R}({}_RA_R).
 \label{sap:eq:algebra-extension-module-transport}
\end{equation}
The subalgebra ${}_RR_R$ becomes the tensor unit of ${}_R\mathcal A_R$.
\end{lemma}

\begin{proof}
Let $s:A\to A\otimes A$ be an $A$-$A$-bimodule splitting of the
multiplication $\mu$, and let $q:A\otimes A\to A\otimes_R A$ be the
quotient map.  If $\mu_R$ is the relative multiplication, then $q$ is
an $A$-$A$-bimodule map and
\begin{equation*}
 \mu_R(qs)=\mu s=\operatorname{id}_A.
\end{equation*}
Thus $qs$ proves separability in ${}_R\mathcal A_R$.  By
\cite[Proposition~7.11.1, p.~154]{EGNO}, the left-hand category in
\eqref{sap:eq:algebra-extension-module-transport} is the category of $R$-$A$
bimodules via $-\otimes_R M$.  The right-hand category also consists of these bimodules: an
action $M\otimes_R A\to M$ in ${}_R\mathcal A_R$ extends the right
$R$-action to an $A$-action.  These identifications preserve the left
${}_R\mathcal A_R$-action.  Precomposition with $\eta_R:\one\to R$ gives an
injection $\Hom_{R\text{-}R}(R,A)\hookrightarrow\Hom_{\mathcal A}(\one,A)$;
the latter is one-dimensional, and $R\hookrightarrow A$ gives a nonzero
element of the former, so ${}_RA_R$ is connected.  Finally,
${}_RR_R$ is the tensor unit by definition of relative tensor product.
\end{proof}

The next observation relates restriction to a fusion subcategory and passage
to the dual category.  Let $\mathcal B\subseteq\mathcal A$ be a
fusion subcategory, and let $R$ be a
connected separable algebra in $\mathcal B$.  Put
$\mathcal P={}_R\mathcal B_R\subseteq{}_R\mathcal A_R$.

\begin{lemma}[Restriction and Morita equivalence]
\label{sap:lem:restriction-morita-transport}
For a finite semisimple $\mathcal A$-module category $\mathcal M$, restriction
of functors gives an equivalence of $\mathcal P$-module categories
\begin{equation*}
 \operatorname{Res}_{\mathcal P}\mathfrak T_R(\mathcal M)
 \simeq
 \operatorname{Fun}_{\mathcal B}
 \bigl(\operatorname{Mod}_{\mathcal B}(R),
       \operatorname{Res}_{\mathcal B}\mathcal M\bigr).
\end{equation*}
\end{lemma}

\begin{proof}
By the module-functor description
\cite[Proposition~7.11.1 and Theorem~7.12.16]{EGNO}, evaluation at $R$
identifies both sides with the category of left $R$-module objects $Z$ in
$\mathcal M$.  Its inverse sends $Z$ to $Y\mapsto Y\otimes_R Z$, using the
action of $\mathcal A$ or its restriction to $\mathcal B$, respectively.
For $X\in\mathcal P$, precomposition acts on either evaluation by
$Z\mapsto X\otimes_R Z$.  The associativity constraints on both sides are
those of relative tensor product.  Thus restriction preserves the
$\mathcal P$-module structure, not just the underlying category.
\end{proof}

\subsection{Pointed module categories}

Put $\mathcal P=\operatorname{Vec}_G$ and, for $J\leq G$,
write $\mathcal N_J=\operatorname{Mod}_{\mathcal P}(R_J)$.
For a cyclic subgroup $H$, the standard calculation
\cite[Example~2.6.4]{EGNO} gives
\begin{equation}
 H^2(H,\bC^\times)=0.
 \label{sap:eq:H2-cyclic}
\end{equation}
The classification of pointed module categories therefore says that every
indecomposable $\mathcal P$-module category is equivalent to a unique
$\mathcal N_J$: its simples form the transitive $G$-set $G/J$, with no
additional cocycle choice
\cite[Example~2.1]{OstrikDouble}
\cite[Theorem~1.1 and the following paragraph]{NataleModuleEquivalence}.

\begin{lemma}[Pointed module-functor rank]
\label{sap:lem:pointed-module-rank}
For subgroups $H,K\leq G$,
\begin{equation*}
 r(H,K):=\rank\operatorname{Fun}_{\mathcal P}(\mathcal N_H,\mathcal N_K)
 =[G:H+K]\,|H\cap K|.
\end{equation*}
\end{lemma}

\begin{proof}
This is the untwisted abelian specialization of the mixed-bimodule formula
\cite[Proposition~3.1, p.~5]{OstrikDouble}; see also
\cite[Lemma~3.3 and Remark~3.4]{NataleModuleEquivalence}.
Indeed, simple $R_H$-$R_K$ bimodules are indexed by a double coset and an
irreducible representation of its stabilizer.  The double cosets are the
$[G:H+K]$ cosets of $H+K$.  Each stabilizer is $H\cap K$, whose irreducible
representations are its $|H\cap K|$ characters.  All cocycles in this
calculation are trivial.
\end{proof}

\subsection{Subgroup algebras and their dual categories}
\label{sap:sec:subgroup-algebras}

We describe $R_H$, its modules and its bimodules.  We then determine the
pointed associator and, for Hall subgroups, the HI matrix of the dual.

\subsubsection{The subgroup algebra and its modules}

\begin{proposition}
\label{sap:prop:subgroup-algebras}
For every $H\leq G$, the object $R_H$ has a connected separable algebra
structure.  Every connected separable algebra structure on the same
underlying object is algebra-isomorphic to the standard one.
\end{proposition}

\begin{proof}
Section~\ref{sec:separable} gives the subgroup-algebra construction
\cite[Example~7.8.3(3)]{EGNO} and proves connectedness and uniqueness,
using \cite[Remark~14(iii)]{OstrikModules} and \eqref{sap:eq:H2-cyclic}.
For later use, we record the bimodule splitting on the $\alpha_g$-summand:
\begin{equation}
 \alpha_g\longmapsto
 \frac1{|H|}\sum_{h\in H}\alpha_h\otimes\alpha_{g-h}.
 \label{sap:eq:RH-splitting}
\end{equation}
\end{proof}

\begin{remark}
The proof of Theorem~\ref{sap:prop:subgroup-algebras} uses only that the pointed
subcategory on $H$ is strict and that $H$ is cyclic.  It does not use the
algebra $Q=\one\oplus\rho$ or any special form of the HI matrix.
\end{remark}

\begin{proposition}
\label{sap:prop:RH-modules}
The simple right $R_H$-modules are
\begin{equation}
 U_{\bar g}:=\alpha_g\otimes R_H,
 \qquad
 V_{\bar g}:=\alpha_g\rho\otimes R_H,
 \qquad \bar g\in G/H.
 \label{sap:eq:RH-simple-modules}
\end{equation}
Consequently, $\rank\operatorname{Mod}_{\C}(R_H)=2[G:H]$.  Its restriction
to $\operatorname{Vec}_G$ is the direct sum of two copies of the untwisted transitive
module category on $G/H$, with $G$ acting by translation on the simple
objects indexed by $G/H$.
\end{proposition}

\begin{proof}
The free-module adjunction calculation, using
\cite[Lemma~7.8.12]{EGNO}, is
\begin{equation*}
 \dim\Hom_{R_H}(U_{\bar g},U_{\bar h})
 =\dim\Hom_{R_H}(V_{\bar g},V_{\bar h})
 =\delta_{\bar g,\bar h}.
\end{equation*}
Together with the mixed Hom vanishing, this proves simplicity and pairwise
nonisomorphism of the modules in
\eqref{sap:eq:RH-simple-modules}.  The exhaustiveness of the list is the
argument in Section~\ref{sec:separable}, using separability and the free
modules \cite[Proposition~7.8.30 and its proof]{EGNO}.  The restriction to
the pointed subcategory is two copies of the untwisted transitive module
category on $G/H$, with the invertibles acting by translation; the absence
of a cocycle follows from \eqref{sap:eq:H2-cyclic}, and the module-category
description is given in \cite[Example~7.4.10]{EGNO}.
\end{proof}

\subsubsection{The subgroup Morita dual}
\label{sap:sec:subgroup-dual}

We now compute the dual category for $R_H$.  We use its fusion rules in the
orbit reduction.  Its pointed associator determines when the dual admits
a strict-pointed HI matrix.

Fix $H\leq G$ and choose a section
$s:G/H\to G$ with $s(0)=0$.  We use its carry
$c(P,Q)=s(P)+s(Q)-s(P+Q)\in H$ as in the introduction.
For $P\in G/H$ and $u\in\widehat H$, define objects
\begin{equation}
 U_{P,u}=\bigoplus_{x\in H}\alpha_{s(P)+x},
 \qquad
 V_{P,u}=\bigoplus_{x\in H}\alpha_{s(P)+x}\rho.
 \label{sap:eq:subgroup-bimodule-objects}
\end{equation}
The underlying objects depend only on the coset label; the character $u$
records the right $H$-action.  On the summand indexed by $x$, the left action
of $\alpha_h$ has coefficient $1$ and sends $x$ to $x+h$.  The right actions
are
\begin{align}
 (U_{P,u})_x\otimes\alpha_h&\longrightarrow(U_{P,u})_{x+h},
 &\text{coefficient }u(h),
 \label{sap:eq:U-right-action}\\
 (V_{P,u})_x\otimes\alpha_h&\longrightarrow(V_{P,u})_{x-h},
 &\text{coefficient }u(h).
 \label{sap:eq:V-right-action}
\end{align}
Here the arrows are the canonical nonzero maps in the normalized
endomorphism presentation.  The strict pointed action and
$\rho\alpha_h=\alpha_{-h}\rho$ show directly that the two $R_H$-actions
commute.

\begin{proposition}[Simple subgroup bimodules]
\label{sap:prop:subgroup-bimodules}
Let
\begin{equation*}
 \D_H={}_{R_H}\C_{R_H},
 \qquad
 L_H=(G/H)\times\widehat H.
\end{equation*}
Then the following statements hold.
\begin{enumerate}[label=(\roman*)]
 \item The objects $U_{P,u}$ and $V_{P,u}$ in
 \eqref{sap:eq:subgroup-bimodule-objects} are all the simple objects of $\D_H$.
 They are pairwise nonisomorphic and satisfy
 \begin{equation*}
  \FPdim_{\D_H}(U_{P,u})=1,
  \qquad
  \FPdim_{\D_H}(V_{P,u})=d.
 \end{equation*}
 \item The invertible simples are the $U_{P,u}$ and form the group $L_H$.
 The tensorators in \eqref{sap:eq:pointed-bimodule-tensorator} identify both
 $(U_{P,u}\otimes_{R_H}U_{Q,v})\otimes_{R_H}U_{R,w}$ and
 $U_{P,u}\otimes_{R_H}(U_{Q,v}\otimes_{R_H}U_{R,w})$ with
 $U_{P+Q+R,uvw}$.  Under these identifications, the associator acts by the
 scalar $\omega_H$ in \eqref{sa:carry-cocycle}.
 \item Put $\sigma=V_{0,1}$, where $1$ denotes the trivial character, and
 write
 \begin{equation}
  W_\ell=U_\ell\otimes_{R_H}\sigma
  \qquad(\ell\in L_H).
  \label{sap:eq:W-labels}
 \end{equation}
 Then the fusion rules are
 \begin{align}
  U_\ell U_r&=U_{\ell+r},
  &
  U_\ell W_r&=W_{\ell+r},
  \notag\\
  W_\ell U_r&=W_{\ell-r},
  &
  W_\ell W_r&=U_{\ell-r}\oplus\bigoplus_{t\in L_H}W_t.
  \label{sap:eq:dual-fusion-two}
 \end{align}
 Thus $\D_H$ has HI fusion rules with pointed group $L_H$.
\end{enumerate}
\end{proposition}

\begin{proof}
After forgetting the left action, $U_{P,u}$ and $V_{P,u}$ are the two simple
right-module types over the coset $P$ in Theorem~\ref{sap:prop:RH-modules}; rescaling the
$H$-summands removes the character twist.  Hence they are simple.  A bimodule
map is determined by one scalar, and the action formulas force the character
labels to agree.  Thus the displayed bimodules are pairwise nonisomorphic.

For exhaustion, the left $H$-action preserves each right-module isotypic
component, because it fixes the coset label.  A simple bimodule must therefore
be isotypic as a right module.  The explicit actions with trivial character
give coherent identifications on that right-module type; relative to these,
the left action is an ordinary representation of $H$ on the multiplicity
space.  Simplicity forces this representation to be a character, giving
precisely the displayed bimodules.  Their dimensions follow from
Theorem~\ref{sap:lem:relative-dimension}.

We next compute the pointed tensor product.  Denote the summands of
$U_{P,u}$ by $e_x$, $x\in H$.  The balanced map
\begin{equation}
 e_x\otimes e_y
 \longmapsto
 u(y)e_{x+y+c(P,Q)}
 \label{sap:eq:pointed-bimodule-tensorator}
\end{equation}
identifies
$U_{P,u}\otimes_{R_H}U_{Q,v}$ with $U_{P+Q,uv}$.  The factor $u(y)$ balances the right action on the first factor with the
left action on the second.  The map also preserves the remaining actions.
It is nonzero between objects of dimension one, so it is an isomorphism.

On three pointed bimodules, the left-parenthesized composite in
\eqref{sap:eq:pointed-bimodule-tensorator} has coefficient
$u(y+z)v(z)$.  The right-parenthesized composite has the same coefficient
times $u(c(Q,R))$.  This proves
\eqref{sa:carry-cocycle} with the convention in the statement.

Finally, we determine the products of noninvertible simples.  Directly from
\eqref{sap:eq:U-right-action} and \eqref{sap:eq:V-right-action},
\begin{equation}
 \sigma U_\ell\cong U_{-\ell}\sigma.
 \label{sap:eq:sigma-inversion}
\end{equation}
This proves the first three fusion formulas once the $W_\ell$ are defined by
\eqref{sap:eq:W-labels}.  Dualizing the explicit actions shows that
$\sigma^\vee\cong\sigma$.  Frobenius reciprocity and
\eqref{sap:eq:sigma-inversion} therefore give
\begin{equation*}
 [\sigma^2:U_\ell]
 =[\sigma U_\ell:\sigma]
 =\delta_{\ell,0}.
\end{equation*}
Write
\begin{equation}
 \sigma^2=U_0\oplus\bigoplus_{\ell\in L_H}n_\ell W_\ell.
 \label{sap:eq:sigma-square-unknown}
\end{equation}
Equation \eqref{sap:eq:sigma-inversion} implies
$\sigma^2U_r\cong U_r\sigma^2$.  Comparing the coefficients of $W_s$ in
this identity gives
\begin{equation}
 n_{s+r}=n_{s-r}
 \qquad(r,s\in L_H).
 \label{sap:eq:n-translation}
\end{equation}
The group $L_H$ has odd order $N$, so multiplication by $2$ is bijective and
\eqref{sap:eq:n-translation} makes all $n_\ell$ equal.  Frobenius--Perron
dimensions in \eqref{sap:eq:sigma-square-unknown}, together with $d^2=1+Nd$,
give common value $1$.  Commuting the invertible factor reduces every
product of noninvertible simples to this square.  Multiplying on the left by $U_{\ell-r}$
gives the last rule in
\eqref{sap:eq:dual-fusion-two}.
\end{proof}

The pointed bimodules are exactly those whose underlying objects lie in
$\mathcal P=\operatorname{Vec}_G$.  Thus
$(\D_H)_{\mathrm{pt}}={}_{R_H}\mathcal P_{R_H}$, the pointed dual that
appears in Theorem~\ref{sap:lem:restriction-morita-transport}.

\subsubsection{The carry class}
\label{sap:sec:carry}

We determine when the pointed associator of $\D_H$ is cohomologically
trivial.  This is exactly the condition that $H$ be Hall.

Put $m=|H|$ and $q=[G:H]$.  The pointed group of $\D_H$ is
$L_H=(G/H)\times\widehat H\cong C_q\times C_m$.
Here $C_a$ denotes a cyclic group of order $a$.
Choose representatives $i,j,k\in\{0,\ldots,q-1\}$ for the $C_q$
coordinates.  In the standard bimodule basis, its pointed associator is
represented by
\begin{equation*}
 \omega_H((i,r),(j,s),(k,t))
 =\zeta_m^{\,r\kappa(j,k)},
 \qquad
 \kappa(j,k)=\left\lfloor\frac{j+k}{q}\right\rfloor.
\end{equation*}
Replacing $s$ by
$s+\beta$, with $\beta:G/H\to H$ and $\beta(0)=0$, changes the carry to
$c'(P,Q)=c(P,Q)+\beta(P)+\beta(Q)-\beta(P+Q)$.
For the cochain $\psi((P,u),(Q,v))=u(\beta(Q))$, we have
\begin{equation*}
 (\delta\psi)((P,u),(Q,v),(R,w))
 =u\bigl(\beta(Q+R)-\beta(Q)-\beta(R)\bigr).
\end{equation*}
Here $\delta$ is the group-cohomology coboundary defined in
Theorem~\ref{sap:lem:pointed-truncation}.
Thus the new associator is $\omega_H\delta\psi^{-1}$.

The carry criterion and the order-nine example are stated and proved in
Section~\ref{sa:subgroup-dual}; the cochain calculation above records the
section-independence needed for that statement.

\subsubsection{Subgroup Fourier--Morita equivalence}
\label{sap:sec:subgroup-fourier-morita}

Recall that $d=\FPdim(\rho)=(N+\sqrt{N^2+4})/2$.  A \emph{full untwisted
cyclic HI matrix} is a matrix $B=(B_{x,y})_{x,y\in G}$ that satisfies all
four Evans--Gannon HI equations in the strict-pointed presentation, with the
Evans--Gannon scalar $\omega=1$.  We allow
the scalar parameter $d_B\in\{d,-d^{-1}\}$.  Explicitly, $B$ satisfies
\cite[equations~(4.7)--(4.10)]{evans2017non} with $\omega=1$ and
$\delta_\pm=d_B$.  We refer to these as the linear symmetry, row-sum,
quadratic, and quartic equations, respectively.
The reconstruction theorem \cite[Theorem~2(a), p.~18]{evans2017non}
gives a fusion category $\C(B)$.  No unitarity or algebra boundary condition
is imposed on $\C(B)$.  Its noninvertible simples have Frobenius--Perron
dimension $d$, regardless of the sign of $d_B$.  For the fixed matrix $A$, we
have $d_A=d$.

Let $H$ be Hall, so $G=H\oplus K$.
The standard subgroup algebra $R_H=\bigoplus_{h\in H}\alpha_h$ is separable,
and its dual has pointed group $L_H=K\times\widehat H\cong C_N$.

We first prove that the partial Fourier transform preserves the HI equations,
then identify the reconstructed category with the bimodule dual.

\begin{theorem}[Partial Fourier transform of HI matrices]
\label{sap:thm:partial-fourier-HI}
Define
\begin{equation}
 (\mathscr P_HB)_{(P,u),(Q,v)}
 =\frac1{|H|}\sum_{x,y\in H}
 B_{P+x,Q+y}\overline{u(y/2)}v(x/2),
\label{sap:eq:partial-fourier}
\end{equation}
where $P,Q\in K$ and $u,v\in\widehat H$.
Then $\mathscr P_HB$ is a full untwisted cyclic HI matrix with scalar
parameter $d_B$.
\end{theorem}

\begin{proof}
Put $C=\mathscr P_HB$ and $a=d_B^{-1}$.
Since $|K\times\widehat H|=N\geq3$,
Theorem~\ref{thm:cubic-criterion} reduces the proof to
\eqref{eq:HImatrix1} and \eqref{eq:HI-cubic}.
The symmetry \eqref{eq:HImatrix1} follows from
\eqref{sap:eq:partial-fourier} by the change of variables
$(x',y')=(-y,x-y)$.

For a matrix $T$ indexed by a finite abelian group $J$, define
its cubic residual by
\begin{equation*}
 \mathcal R_T(g,h,k,l)
 :=\sum_{n\in J}T_{n,g+h}T_{g,n+k}T_{h,n+l}
   -T_{g+l,k}T_{h+k,l}+a\delta_{g,0}\delta_{h,0}.
\end{equation*}
Since $B$ is an HI-matrix,
Theorem~\ref{thm:cubic-criterion} gives $\mathcal R_B=0$.

For $P,Q,R,S\in K$ and $\chi,\psi,\lambda,\mu\in\widehat H$, we have
\begin{align}
 &\mathcal R_C\bigl((P,\chi),(Q,\psi),(R,\lambda),(S,\mu)\bigr)
 \notag\\
 &\quad=\frac1{|H|^2}\sum_{x,y,r,s\in H}
 \overline{\chi(r/2)}\,\overline{\psi(s/2)}\lambda(x/2)\mu(y/2)\cdot\mathcal R_B(P+x,Q+y,R+r,S+s).
 \label{sap:eq:partial-fourier-cubic-residual}
\end{align}
To verify the cubic terms, expand the three factors with intermediate label
$(M,\eta)\in K\times\widehat H$ using $H$-index pairs $(t,z)$, $(x,r_0)$,
and $(y,s_0)$.
The sum over $\eta$ is
$\sum_\eta\eta((x+y-z)/2)=|H|\delta_{z,x+y}$.
Set $r_0=t+r$, $s_0=t+s$, and $n=M+t$; the remaining sum over $n\in G$ is the
cubic term on the right of \eqref{sap:eq:partial-fourier-cubic-residual}.
For the quadratic terms, expand
$C_{(P+S,\chi\mu),(R,\lambda)}C_{(Q+R,\psi\lambda),(S,\mu)}$
using shifts $(x+s,r)$ and $(y+r,s)$.  Their character factor is
\begin{align*}
 &\overline{(\chi\mu)(r/2)}\lambda((x+s)/2)
   \overline{(\psi\lambda)(s/2)}\mu((y+r)/2)=\overline{\chi(r/2)}\overline{\psi(s/2)}
          \lambda(x/2)\mu(y/2).
\end{align*}
The two matrix entries are
$B_{P+S+x+s,R+r}B_{Q+R+y+r,S+s}$, which gives the quadratic term in
\eqref{sap:eq:partial-fourier-cubic-residual}.
Finally, $P+x=Q+y=0$ forces $P=Q=x=y=0$ since $G=H\oplus K$; summing over
$r,s$ then gives $|H|^2\delta_{\chi,1}\delta_{\psi,1}$, as required for the
delta terms. Since $\mathcal R_B=0$,
\eqref{sap:eq:partial-fourier-cubic-residual} gives
$\mathcal R_C=0$. Thus $C$ satisfies
\eqref{eq:HImatrix1} and \eqref{eq:HI-cubic}, and
Theorem~\ref{thm:cubic-criterion} completes the proof.
\end{proof}

For the rest of this subsection, put $m=|H|$ and $q=|K|$.
The bimodule calculation in Theorem~\ref{sap:prop:subgroup-bimodules}(i) uses only the
fusion rules and the strict pointed action.  Applied to $\C(B)$, it gives
simple bimodules $U_{P,u}$ and $V_{P,u}$ of dimensions $1$ and $d$.

\begin{theorem}[Subgroup Fourier--Morita equivalence]
\label{sap:thm:subgroup-fourier-morita}
There is a tensor equivalence
\begin{equation}
 {}_{R_H}\C(B)_{R_H}
 \simeq
 \C(\mathscr P_HB).
 \label{sap:eq:subgroup-fourier-morita}
\end{equation}
The equivalence from right to left sends the standard invertible and
noninvertible labels $(P,u)$ to $U_{P,u^{-1}}$ and $V_{P,u}$, respectively.
\end{theorem}

\begin{proof}
We realize the partial Fourier matrix in a crossed product, then construct
one tensor functor from its standard reconstruction to $R_H$-bimodules.
The simple-object correspondence will prove that this functor is an equivalence.

\smallskip
\noindent\emph{The crossed-product HI system.}
Put $a=d_B^{-1}$ and choose $b\in\bC$ with $b^2=a$.  Let $\mathbf{L}$ be the
Leavitt algebra used to reconstruct $\C(B)$, with generators
$s,s',t_g,t_g'$ for $g\in G$.  The primed generators are algebraic companions;
they are not assumed to be adjoints.  No $*$-structure is used.
In the untwisted convention,
\begin{align}
 \alpha_z(s)&=s,
 &\alpha_z(s')&=s',
 \notag\\
 \alpha_z(t_g)&=t_{g+2z},
 &\alpha_z(t_g')&=t_{g+2z}',\notag
\\
 \rho(s)&=as+b\sum_gt_gt_g,
 &\rho(s')&=as'+b\sum_gt_g't_g',
 \label{sap:eq:partial-algebraic-rho-s}\\
 \rho(t_g)
 &=bs t_{-g}'+t_{-g}ss'
   +\sum_{h,k}B_{h+g,k+g}t_ht_{h+k+g}t_k',
 \label{sap:eq:partial-algebraic-rho-t}\\
 \rho(t_g')
 &=bt_{-g}s'+ss't_{-g}'
   +\sum_{h,k}B_{k+g,h+g}t_kt_{g+h+k}'t_h'.
 \label{sap:eq:partial-algebraic-rho-t-prime}
\end{align}
These are the formulas in
\cite[equations~(4.3)--(4.6), with $\omega=1$]{evans2017non}.

Form the algebraic crossed product
\begin{equation*}
 \widetilde{\mathbf{L}}=\mathbf{L}\rtimes_{\alpha|_H}H,
 \qquad
 (cz_x)(ez_y)
 =c\alpha_x(e)z_{x+y}.
\end{equation*}
and denote its canonical invertible elements by $z_x$, $x\in H$.
Thus $z_xc=\alpha_x(c)z_x$ for $c\in\mathbf{L}$.  For
$P\in K$ and $u\in\widehat H$, define endomorphisms of $\widetilde{\mathbf{L}}$ by
\begin{align*}
 \widehat\alpha_{P,u}(c)&=\alpha_P(c),
 &
 \widehat\alpha_{P,u}(z_x)&=u(x)z_x,
\\
 \widehat\rho(c)&=\rho(c),
 &
 \widehat\rho(z_x)&=z_{-x}
\end{align*}
for $c\in\mathbf{L}$.  The equality $G=H\oplus K$ and the relation
$\rho\alpha_x=\alpha_{-x}\rho$ show that these assignments preserve the
crossed-product relations.  They satisfy
\begin{equation*}
 \widehat\alpha_{P,u}\widehat\alpha_{Q,v}
 =\widehat\alpha_{P+Q,uv},
 \qquad
 \widehat\alpha_{P,u}\widehat\rho
 =\widehat\rho\widehat\alpha_{-P,u^{-1}}.
\end{equation*}

Define the character basis
\begin{align}
 \widehat s&=ms,
 &\widehat{\,s'\,}&=m^{-1}s',
 \label{sap:eq:partial-character-s}\\
 \widehat t_{P,u}
 &=\sum_{x\in H}u(x/2)t_{P+x}z_x,
 &
 \widehat{\,t'_{P,u}\,}
 &=\frac1m\sum_{x\in H}\overline{u(x/2)}
   z_{-x}t_{P+x}'.
 \label{sap:eq:partial-character-t}
\end{align}
Thus $\widehat{\,s'\,}$ and $\widehat{\,t'_{P,u}\,}$ denote the transformed
primed generators, not the primes of $\widehat s$ and
$\widehat t_{P,u}$.
\begin{lemma}[Character-basis relations]
\label{sap:lem:partial-character-basis}
The elements in \eqref{sap:eq:partial-character-s} and
\eqref{sap:eq:partial-character-t} form a faithful Leavitt family and split
$\widehat\rho^2$ with labels in $L_H=K\times\widehat H$.  They are intertwiners for this splitting, satisfy the reconstruction formulas for
$\widehat\rho(\widehat s)$ and $\widehat\rho(\widehat{\,s'\,})$, and obey
\begin{equation*}
 \widehat\alpha_{Q,v}(\widehat t_{P,u})
 =\widehat t_{P+2Q,uv^2},
 \qquad
 \widehat\alpha_{Q,v}(\widehat{\,t'_{P,u}\,})
 =\widehat{\,t'_{P+2Q,uv^2}\,}.
\end{equation*}
\end{lemma}

\begin{proof}
Character orthogonality gives the contractions
$\widehat{\,s'\,}\widehat s=1$ and
$\widehat{\,t'_{P,u}\,}\widehat t_{Q,v}=\delta_{P,Q}\delta_{u,v}$,
the vanishing mixed contractions, and the completeness relation.  It also
transforms the two quadratic sums in
\eqref{sap:eq:partial-algebraic-rho-s} into the corresponding sums in the new
basis.  Checking the intertwiner equations on
$\mathbf{L}$ and the elements $z_x$ gives the stated types.
The same formulas give the pointed action.  The resulting homomorphism from the one-vertex Leavitt algebra is
injective by
\cite[Theorem~3.11 and Corollary~3.13(iii), pp.~329--333]{abrams2005leavitt}.
\end{proof}

It remains to identify the transformed coefficient matrix.
Let $C$ be the matrix obtained by expanding
$\widehat\rho(\widehat t_{0,1})$ and
$\widehat\rho(\widehat{\,t'_{0,1}\,})$ in this basis.  We compare coefficients using the reduced-word basis of the Leavitt algebra
and the unique crossed-product expansion
$\widetilde{\mathbf{L}}=\bigoplus_{x\in H}\mathbf{L}z_x$.  In the unprimed calculation, the source cubic words are
$t_ht_{h+k+g}t_k'z_{-g}$.  The term at shifts $x,y,z\in H$ in
$\widehat t_{P,u}\widehat t_{P+Q,uv}\widehat{\,t'_{Q,v}\,}$ is
\begin{equation*}
 \frac1m u((x+y)/2)v((y-z)/2)
 t_{P+x}t_{P+Q+2x+y}t'_{Q+2x+2y-z}z_{x+y-z}.
\end{equation*}
Matching words gives
$g=z-x-y$, $h=P+x$, and $k=Q+2x+2y-z$.
Set $X=x+g=z-y$ and $Y=x+y$.  The source coefficient becomes
$B_{P+X,Q+Y}$, and the character factor gives
\begin{equation}
 B_{P+X,Q+Y}
 =\frac1m\sum_{u,v\in\widehat H}
 C_{(P,u),(Q,v)}u(Y/2)\overline{v(X/2)}.
 \label{sap:eq:partial-fourier-inverse}
\end{equation}
For the primed calculation, the source term is
$m^{-1}B_{k+g,h+g}t_{k+2g}t_{h+k+3g}'t_{h+2g}'z_g$.
The term at shifts $x,y,z$ in
$\widehat t_{Q,v}\widehat{\,t'_{P+Q,uv}\,}
\widehat{\,t'_{P,u}\,}$ is
\begin{equation*}
 \frac1{m^2}v((x-y)/2)\overline{u((y+z)/2)}
 t_{Q+x}t'_{P+Q+2x-y}t'_{P+2x-2y-z}z_{x-y-z}.
\end{equation*}
Now $g=x-y-z$, $k=Q-x+2y+2z$, and $h=P+z$.
Putting $X=y+z$ and $Y=x-y$ gives the source entry $B_{Q+X,P+Y}$.
After cancelling the source factor $m^{-1}$, coefficient comparison
gives \eqref{sap:eq:partial-fourier-inverse} with the two labels interchanged.
This agrees with the reversed coefficient indices in
\eqref{sap:eq:partial-algebraic-rho-t-prime}, so the primed array is also $C$.
Fourier inversion on $H\times H$ therefore gives
$C=\mathscr P_HB$.  The noncubic terms in
\eqref{sap:eq:partial-algebraic-rho-t} and
\eqref{sap:eq:partial-algebraic-rho-t-prime} transform directly, while
Theorem~\ref{sap:lem:partial-character-basis} handles the quadratic terms.
Doubling and squaring are bijective on the two label groups.
Thus $\widehat\rho$ and $\widehat\alpha_{P,u}$ satisfy the reconstruction
formulas for $\mathscr P_HB$ at every label $(P,u)$.

\smallskip
\noindent\emph{The tensor functor to bimodules.}
Put $C=\mathscr P_HB$.  By Theorem~\ref{sap:thm:partial-fourier-HI} and
\cite[Theorem~2(a)]{evans2017non}, $\C(C)$ is a fusion category with the
standard $2N$ simple labels.  We construct a tensor functor to the bimodule
category and check that it identifies these two lists of simples.

The character basis in Theorem~\ref{sap:lem:partial-character-basis} gives a unital
embedding $j:\mathbf L_C\hookrightarrow\widetilde{\mathbf L}$ with
\begin{equation}
 j\alpha^C_{P,u}=\widehat\alpha_{P,u}j,
 \qquad\text{ and }\qquad j\rho^C=\widehat\rho j.
 \label{sap:eq:partial-reconstruction-equivariance}
\end{equation}
Here word multiplication means $\theta\phi=\theta\circ\phi$.
For a source word $\theta$, let $\widehat\theta$ be the corresponding
hatted word and put $\sigma_\theta=\widehat\theta|_{\mathbf L}$.
Equation~\eqref{sap:eq:partial-reconstruction-equivariance} then gives
$j\theta=\widehat\theta j$ for every word.  Source morphisms
$f:\theta\to\phi$ are Leavitt intertwiners, with
$f\theta(a)=\phi(a)f$ for every $a\in\mathbf L_C$.
We index the hatted word's character by its source word.
Thus $\lambda_\theta\in\widehat H$ and the sign
$\epsilon_\theta\in\{1,-1\}$ are defined by
\begin{equation*}
 \widehat\theta(z_x)
 =\lambda_\theta(x)z_{\epsilon_\theta x}.
\end{equation*}
For the generators they are $(u,1)$ and $(1,-1)$, respectively, and
\begin{equation}
 \lambda_{\theta\phi}(x)
 =\lambda_\phi(x)\lambda_\theta(\epsilon_\phi x),
 \qquad\text{and }\qquad \epsilon_{\theta\phi}=\epsilon_\theta\epsilon_\phi.
 \label{sap:eq:partial-lambda-composition}
\end{equation}

We first check that $j$ sends source morphisms to intertwiners on the
whole crossed product.  Repeated fusion splittings give maps
$i^\theta_{\tau,r}:\tau\to\theta$ and
$p^\theta_{\tau,r}:\theta\to\tau$ through the standard simple labels.
They are composites and tensor products of the Leavitt splitting maps and
their translates.  By Theorem~\ref{sap:lem:partial-character-basis} and
\eqref{sap:eq:partial-reconstruction-equivariance}, their images intertwine on
$\mathbf L$ and every $z_x$.  Since the standard labels are
pairwise nonisomorphic simples, every source morphism has an expansion
\begin{equation*}
 f:\theta\to\phi,
 \qquad
 f=\sum_{\tau,r,t}i^\phi_{\tau,r}c_{\tau,r,t}p^\theta_{\tau,t},
 \qquad c_{\tau,r,t}\in\bC.
\end{equation*}
Intertwiner identities are preserved by composition and tensor product,
so $j(f)\widehat\theta(b)=\widehat\phi(b)j(f)$ for every
$b\in\widetilde{\mathbf L}$. Define
$\Phi_H(\theta)=\bigoplus_{x\in H}\alpha_x\sigma_\theta$.
The left $R_H$-action translates the index with coefficient $1$; the right
$\alpha_h$-action is
\begin{equation*}
 \alpha_x\sigma_\theta\otimes\alpha_h
 \longrightarrow\alpha_{x+\epsilon_\theta h}\sigma_\theta,
 \qquad\text{ with coefficient }\lambda_\theta(h)^{-1}.
\end{equation*}
These actions commute, since
$\sigma_\theta\alpha_h=\alpha_{\epsilon_\theta h}\sigma_\theta$.
For $f:\theta\to\phi$, write
\begin{equation*}
 j(f)=\sum_{r\in H}t_rz_r,
 \qquad\text{ and }\qquad \Phi_H(f)_{x,x+r}=\alpha_x(t_r).
\end{equation*}
Rows index target summands.  The intertwiner equation on $\mathbf L$
gives $t_r\in\Hom(\alpha_r\sigma_\theta,\sigma_\phi)$, and that on
$z_{-h}$ gives
\begin{equation}
 \lambda_\theta(h)^{-1}t_{r+\epsilon_\theta h}
 =\lambda_\phi(h)^{-1}
   \alpha_{-\epsilon_\phi h}(t_{r+\epsilon_\phi h}).
 \label{sap:eq:direct-right-linearity}
\end{equation}
These are exactly the block typing and right-linearity conditions;
left linearity follows from translation.  Crossed-product multiplication
and block composition both use the convolution
$\sum_{q+r=v}s_q\alpha_q(t_r)$, so the construction preserves composition
and identities.

The tensorator sends the $(x,y)$-component to the
$(x+\epsilon_\theta y)$-summand with coefficient
$\lambda_\theta(y)^{-1}$:
\begin{equation*}
 J_{\theta,\phi}:\Phi_H(\theta)\otimes_{R_H}\Phi_H(\phi)
 \longrightarrow\Phi_H(\theta\phi).
\end{equation*}
The character law balances the two middle actions, and
\eqref{sap:eq:partial-lambda-composition} preserves the remaining right action.
The inverse on the $v$-summand is the separability average
\begin{equation*}
 \frac1m\sum_{y\in H}\lambda_\theta(y)
       (v-\epsilon_\theta y,y).
\end{equation*}
The two composites are the identity and the relative-tensor projector
from \eqref{sap:eq:RH-splitting}.  Naturality in the first variable is
\eqref{sap:eq:direct-right-linearity} with $h=y$, after reindexing.
For second-variable naturality, write $j(g)=\sum_qs_qz_q$.
On the $(x,y+q)$-component the two coefficients of
$\alpha_{x+\epsilon_\theta y}\sigma_\theta(s_q)$ are
$\lambda_\theta(y)^{-1}$ and
$\lambda_\theta(y+q)^{-1}\lambda_\theta(q)$, which agree.
Here all-word equivariance gives
$j(f\otimes g)=j(f)\widehat\theta(j(g))$.
The pentagon is
\begin{equation*}
 \lambda_\theta(y)^{-1}\lambda_{\theta\phi}(z)^{-1}
 =\lambda_\phi(z)^{-1}
  \lambda_\theta(y+\epsilon_\phi z)^{-1},
\end{equation*}
which follows from \eqref{sap:eq:partial-lambda-composition}.
For the identity word the bimodule is $R_H$, and the tensorators are the
module-action maps, proving the unit diagrams.  Thus $\Phi_H$ extends
through additive and idempotent completion to a tensor functor on $\C(C)$.

Its values on standard simples are
\begin{equation*}
 \Phi_H(\alpha^C_{P,u})\cong U_{P,u^{-1}},
 \qquad \Phi_H(\alpha^C_{P,u}\rho^C)\cong V_{P,u}.
\end{equation*}
The second character is $u$, since
$\lambda_{\alpha^C_{P,u}\rho^C}(h)=u(-h)$.
By Theorem~\ref{sap:prop:subgroup-bimodules} applied to $\C(B)$, these are pairwise nonisomorphic simples
and exhaust the target.  On Hom spaces between simples, $\Phi_H$
is therefore either $0\to0$ or the unital linear map $\bC\to\bC$.
Semisimplicity proves full faithfulness on all objects, and the same list
proves essential surjectivity.  This proves
\eqref{sap:eq:subgroup-fourier-morita}.  Together with \eqref{sap:eq:W-labels},
the correspondence also gives $W_{(P,u)}\cong V_{P,u^{-1}}$.
\end{proof}

\subsubsection{Consequences of subgroup Fourier duality}
\label{sap:sec:fourier-consequences}

We retain the setup of Subsubsection~\ref{sap:sec:subgroup-fourier-morita}.

\begin{corollary}[Hermiticity under the partial Fourier transform]
\label{sap:cor:partial-fourier-hermiticity}
For $B$ as in Theorem~\ref{sap:thm:partial-fourier-HI}, define
$(B^*)_{g,h}=\overline{B_{h,g}}$.  Then
\begin{equation}
 (\mathscr P_HB)^*=\mathscr P_H(B^*).
 \label{sap:eq:partial-fourier-star}
\end{equation}
Consequently, $\mathscr P_HB$ is Hermitian if and only if $B$ is Hermitian.
\end{corollary}
The case $d_B=d>0$ is called the positive branch.  In this case, the
transformed matrix has the same scalar parameter.  Thus
the partial Fourier transform along a Hall subgroup cannot produce a
positive-branch non-Hermitian matrix from a Hermitian one.

\begin{proof}
Conjugating \eqref{sap:eq:partial-fourier} and exchanging $x$ and $y$ proves
\eqref{sap:eq:partial-fourier-star}.  Injectivity from
\eqref{sap:eq:partial-fourier-inverse}
gives the Hermiticity equivalence.
\end{proof}

The full Fourier transform is the case $H=G$.

\begin{corollary}[Full Fourier--Morita equivalence]
\label{sap:thm:algebraic-full-fourier-morita}
\label{sap:cor:algebraic-full-fourier-centers}
For $u,v\in\widehat G$, put
\begin{equation*}
 (\mathscr F_GB)_{u,v}
 =\frac1N\sum_{x,y\in G}
 B_{x,y}\overline{u(y/2)}v(x/2).
\end{equation*}
Then $\mathscr F_GB$ is a full untwisted cyclic HI matrix with scalar $d_B$, and ${}_{R_G}\C(B)_{R_G}\simeq\C(\mathscr F_GB)$ as tensor categories, with the simple-object correspondence in
Theorem~\ref{sap:thm:subgroup-fourier-morita}.  Moreover,
under $G\cong\widehat{\widehat{G\,}}\!\!$ one has $\mathscr F_{\widehat G}\mathscr F_GB=B$. The fusion categories
$\C(B)$ and $\C(\mathscr F_GB)$ are
categorically Morita equivalent, with
\begin{equation}
 \mathcal Z(\C(B))
 \simeq_{\mathrm{br}}
 \mathcal Z(\C(\mathscr F_GB)).
 \label{sap:eq:algebraic-full-fourier-centers}
\end{equation}
\end{corollary}

\begin{proof}
Apply Theorems~\ref{sap:thm:partial-fourier-HI} and \ref{sap:thm:subgroup-fourier-morita} with
$H=G$ and $K=\{0\}$.
Applying character orthogonality twice recovers the two original matrix
indices and proves involutivity.  The bimodule equivalence gives categorical Morita
equivalence \cite[Proposition~7.11.1 and Remark~7.12.5]{EGNO}, and invariance
of the centre under Morita equivalence \cite[Proposition~8.5.3]{EGNO},
with our reverse-composition convention, gives
\eqref{sap:eq:algebraic-full-fourier-centers}.
\end{proof}

Apply these results to $B=A$, with $d_A=d$.
Then \eqref{sap:eq:partial-fourier} is the formula in the subgroup-dual
description of Section~\ref{sa:subgroup-dual}.

We also compute the composition of the partial and full transforms.
Retain the Hall decomposition $G=H\oplus K$ and identify
$\widehat{L_H}=\widehat K\times H$ by the pairing
$\langle(\phi,x),(P,u)\rangle=\phi(P)u(x)$.  Let
$\Theta_0:\widehat K\times H\to H\times\widehat K=L_K$ be
$\Theta_0(\phi,x)=(x,\phi)$.

\begin{lemma}[Subgroup-Fourier composition]
\label{sap:lem:composition}
The full symplectic Fourier transform on $L_H$ satisfies
\begin{equation*}
 \mathscr F_{L_H}\mathscr P_HA
 =\Theta_0^*\mathscr P_KA.
\end{equation*}
\end{lemma}

\begin{proof}
For $\phi,\psi\in\widehat K$ and $x,y\in H$, substitute
\eqref{sap:eq:partial-fourier} into the full transform.  The sums over $v$ and
$u$ force the first and second $H$-indices to be $x$ and $y$, respectively.
Thus
\begin{align*}
 (\mathscr F_{L_H}\mathscr P_HA)_{(\phi,x),(\psi,y)}
 &=\frac1q\sum_{P,Q\in K}
 A_{P+x,Q+y}\overline{\phi(Q/2)}\psi(P/2)\\
 &=(\mathscr P_KA)_{\Theta_0(\phi,x),\Theta_0(\psi,y)}.\qedhere
\end{align*}
\end{proof}

\begin{remark}[Pivotal and ribbon structures]
The braided equivalence in \eqref{sap:eq:algebraic-full-fourier-centers} is not
asserted to preserve a chosen pivotal or ribbon normalization.
\end{remark}

\begin{proof}[Proof of the subgroup-dual description]
The dual is the bimodule category
$\D_H$.  Its fusion rules and carry associator are given by
Theorem~\ref{sap:prop:subgroup-bimodules}; Theorem~\ref{sa:carry-criterion} makes this associator trivial
exactly for Hall $H$.  In that case,
Theorems~\ref{sap:thm:partial-fourier-HI} and \ref{sap:thm:subgroup-fourier-morita} give
\eqref{sap:eq:partial-fourier}, including
the boundary cases $H=\{0\}$ and $H=G$.  The scalar parameter is unchanged,
and Hermiticity follows from Theorem~\ref{sap:cor:partial-fourier-hermiticity} because
$A$ is Hermitian.
\end{proof}

\subsection{The algebra \texorpdfstring{$Q$}{Q} and its dual category}
\label{sap:sec:Q-algebra}

We describe the right modules over $Q=\one\oplus\rho$ and identify its dual
category.  We first recall the corresponding matrix boundary condition.

By Theorem~\ref{thm:algebraic-q-system}, the assumed separable algebra on $Q$
gives the Q-system boundary condition
\begin{equation*}
 A_{g,0}=\delta_{g,0}-\frac1{d-1}
 \qquad(g\in G).
\end{equation*}
Conversely, this boundary gives a Q-system on $\one\oplus\rho$
by the criterion stated in \cite[Section~6.1, equation~(6.1)]{izumi2018classification}.
Hermiticity and cyclic symmetry also give
$A_{0,g}=A_{g,g}=-(d-1)^{-1}$ for every $g\neq0$. The free-module argument used for $R_H$ also gives the $Q$-module rank in
\eqref{sa:classification-ranks}.

\begin{proposition}[The category of $Q$-modules]
\label{sap:prop:Q-modules}
The category $\operatorname{Mod}_{\C}(Q)$ has $N+1$ simple objects.  Its
restriction to the pointed subcategory has one regular orbit of size $N$ and
one fixed orbit.
\end{proposition}

\begin{proof}
Let $F(X)=X\otimes Q$ be the free right $Q$-module, and put
$U_g=F(\alpha_g)$ and $W_g=F(\alpha_g\rho)$.
By \cite[Lemma~7.8.12]{EGNO}, free-module adjunction and the fusion rules give
\begin{align*}
 \dim\Hom_Q(U_g,U_h)=\delta_{g,h},\,\,\, \dim\Hom_Q(U_k,W_g)=\delta_{k,g},\text{ and }\dim\Hom_Q(W_g,W_h)=1+\delta_{g,h}.
\end{align*}
Since $Q$ is separable, its module category is semisimple.
Thus the $U_g$ are pairwise nonisomorphic simples, and
$W_g\cong U_g\oplus X_g$ for actual objects $X_g$.
The second equality gives $\Hom_Q(U_k,X_g)=0$; semisimplicity
gives $\Hom_Q(X_g,U_k)=0$ as well.  Subtracting these summands gives
\begin{equation*}
 \Hom_Q(U_k,X_g)=0,
 \qquad\text{ and }\qquad\dim\Hom_Q(X_g,X_h)=1\quad(g,h\in G).
\end{equation*}
The diagonal equality makes every $X_g$ simple; the other equalities
identify them all with one simple $X$.
Every simple right $Q$-module is a summand of a free one by
\cite[Proposition~7.8.30 and its proof]{EGNO}, so the simples are
$\{U_g:g\in G\}\sqcup\{X\}$.  The invertible object $\alpha_k$ sends $U_g$
to $U_{k+g}$.  Applying $\alpha_k$ to $W_g\cong U_g\oplus X$ shows that it
fixes $X$.  Thus the pointed orbit sizes are $N$ and $1$.  The noninvertible
action is $\rho U_g\cong U_{-g}\oplus X$ and $\rho X\cong\bigoplus_{h\in G}U_h\oplus(N-1)X$. The first formula is $\rho U_g\cong W_{-g}$.  For the second, apply $\rho$ to $W_0\cong U_0\oplus X$ and use
$\rho W_0=U_0\oplus\bigoplus_h W_h$.
\end{proof}

We now identify the dual category.  Tensor generation here allows taking
direct summands of tensor powers.
We identify the tensor category using the subfactor construction and
Theorem~\ref{sap:prop:short-uniqueness}, then compute its fusion matrix.

\begin{proof}[Proof of the $Q$-dual description]
It remains to identify this bimodule category and compute its fusion matrix.
Evans--Gannon identify the normalized data of $\C$ with Izumi's normalization
\cite[p.~15, paragraph following equation~(4.11)]{evans2017non}.  By
\eqref{eq:display-0271}, these data satisfy Izumi's equations and the
algebra boundary equation~(6.1).  The construction in
\cite[Theorems~5.7 and~7.3, and the construction following
Theorem~7.3]{izumi2018classification} gives a $3^G$ subfactor $P\subset M$ whose $M$-$M$ even
part has the same data and whose canonical algebra has underlying object
$\one\oplus\rho$.  The converse reconstruction
\cite[Theorem~2(a)]{evans2017non} identifies this $M$-$M$ even part with the
fixed category $\C=\C(A)$ by a label-preserving tensor equivalence.
Under this equivalence, the canonical subfactor algebra becomes an algebra on
$\one\oplus\rho$; by
Theorem~\ref{sap:prop:short-uniqueness}, it is algebra-isomorphic to $Q$.  Hence the
$P$-$P$ even part is ${}_Q\C_Q$.

Let $\mathcal U:{}_Q\C_Q\to\operatorname{Mod}_{\C}(Q)$ forget the left
$Q$-action, with left adjoint $\mathcal L(Y)=Q\otimes Y$.
Let $\iota:P\hookrightarrow M$ denote the inclusion.
Under the subfactor identification, the modules $U_g$ of
Theorem~\ref{sap:prop:Q-modules} correspond to $\alpha_g\iota$, and $X$ corresponds
to the remaining odd simple $\kappa$.
Izumi's calculation
\cite[Section~7, especially Proposition~7.4 and Theorem~7.5]{izumi2018classification}
gives the complete simple list and its restrictions.  Here odd order
removes nonzero elements of order two, and Theorem~7.5 gives the
multiplicity-one conclusion.  Relabelling Izumi's
$\widehat\rho,\sigma_i,\pi_{g_i}$ as $\eta,\mu_i,\nu_i$, we obtain
\begin{align}
 \mathcal U(\one)&=U_0,&
 \mathcal U(\eta)&=U_0\oplus X,&
 \mathcal U(\mu_i)&=X,
 \label{sap:eq:Q-dual-induction-two}\\
 \mathcal U(\nu_i)&=U_{g_i}\oplus U_{-g_i}\oplus X.
 \label{sap:eq:Q-dual-induction-three}
\end{align}
Let $R$ be the matrix of $\mathcal U$, with rows indexed by the right-module
simples and columns ordered as $\one,\eta,\mu_1,\ldots,\mu_r,\nu_1,\ldots,\nu_r$.
Writing $e_g,e_X$ for the corresponding coordinate vectors, its columns are
\begin{equation}
 e_0,\qquad e_0+e_X,\qquad
 e_X\ (1\leq i\leq r),\qquad
 e_{g_i}+e_{-g_i}+e_X\ (1\leq i\leq r).
 \label{sap:eq:Q-branching-columns}
\end{equation}
Adjunction gives $R^{\top}$ for the matrix of $\mathcal L$; equivalently,
\begin{equation*}
 \mathcal L(U_0)=\one\oplus\eta,\qquad
 \mathcal L(X)=\eta\oplus\bigoplus_i\mu_i\oplus\bigoplus_i\nu_i,
 \qquad \mathcal L(U_{\pm g_i})=\nu_i.
\end{equation*}
Since $U_0=Q_Q$, the first equality identifies $Q\otimes Q$ with
$\one\oplus\eta$ as a bimodule.
For any bimodule $Z$, we have
$\mathcal L\mathcal U(Z)=Q\otimes Z
\cong(Q\otimes Q)\otimes_Q Z$.
Hence the matrix of left multiplication by $\eta$ is
\begin{equation*}
 N_\eta=R^{\top}R-I_{N+1}.
\end{equation*}
The pairwise inner products of the columns in
\eqref{sap:eq:Q-branching-columns} give the fusion matrix \eqref{sa:Q-dual-matrix}.
The underlying $\C$-dimensions of $U_g$ and $X$ are $d+1$ and
$(d-1)(d+1)$, respectively, by their free-module decompositions.
Dividing the dimensions in
\eqref{sap:eq:Q-dual-induction-two}--\eqref{sap:eq:Q-dual-induction-three}
by $\FPdim(Q)=d+1$, as in Theorem~\ref{sap:lem:relative-dimension}, gives
the dimensions displayed in Section~\ref{sa:Q-dual}.
The coefficient of the unit in $\eta^2$ is one, so $\eta$ is self-dual.
The same column contains every simple, so $\eta$ tensor-generates.
Since $d>2$, no nonunit simple is invertible.
Finally, for $N=3$ the identification with $\mathscr H_1$ and the fusion rules
in these labels are \cite[Table~2 and Lemma~3.18]{GrossmanSnyder2012}.
\end{proof}

\subsection{Classification of connected separable algebras}
\label{sap:sec:classification-analysis}

We now prove that the algebras already described exhaust the Morita classes.
We choose the pointed orbit with smaller diagonal coefficient and show that
this coefficient is unchanged on passing to the subgroup dual.  A bound on
endomorphism spaces then leaves two possible internal End algebras.
The subgroup rank formula and the Hall boundary obstruction complete the
classification.

\subsubsection{Two pointed orbits}
\label{sap:sec:two-orbits}

We first count the pointed orbits using only the fusion rules.  Let $L$ be any finite
abelian group of odd order $n\geq3$, and let $\mathcal E$ be a fusion
category whose simple objects are $\beta_\ell$ and $\beta_\ell\sigma$, for
$\ell\in L$, with $\sigma\cong\sigma^\vee$ and fusion rules
\begin{equation*}
 \beta_\ell\beta_r=\beta_{\ell+r},\qquad
 \beta_\ell\sigma=\sigma\beta_{-\ell},\qquad\text{ and }\qquad
 \sigma^2=\one\oplus\bigoplus_{\ell\in L}\beta_\ell\sigma.
\end{equation*}
No assumption is made on the associator of its pointed subcategory.

Let $\M$ be an indecomposable finite semisimple $\mathcal E$-module category.  The
invertible objects permute its simple objects.  Let
$\mathcal O_1,\ldots,\mathcal O_t$ be the resulting $L$-orbits.  For
$y\in\mathcal O_i$, define
\begin{equation*}
 M_{ij}:=\sum_{v\in\mathcal O_j}[\sigma y:v].
\end{equation*}
The relation $\sigma\beta_\ell=\beta_{-\ell}\sigma$ makes this independent of
the choice of $y$.

\begin{proposition}[Two-orbit theorem]
\label{sap:prop:two-orbits}
Every indecomposable semisimple $\mathcal E$-module category has exactly two
pointed orbits.
\end{proposition}

\begin{proof}
On the orbit quotient, every invertible acts as the identity.  The fusion rule for $\sigma^2$ is
\begin{equation*}
 M^2=I_t+nM.
\end{equation*}
Self-duality of $\sigma$ makes its full fusion matrix symmetric, so
$M_{ij}=0$ if and only if $M_{ji}=0$.
Thus reducibility of $M$ would split $\M$ into two nonzero stable
unions of pointed orbits, contradicting indecomposability.  The two
roots of $x^2-nx-1$ are
\begin{equation*}
 d_n=\frac{n+\sqrt{n^2+4}}2,
 \qquad
 -d_n^{-1}.
\end{equation*}
The polynomial is irreducible over $\mathbb Q$ because
$n^2<n^2+4<(n+1)^2$.  Integrality of $M$ gives equal multiplicities
for the conjugate roots, while Perron--Frobenius gives multiplicity one for
$d_n$.  Hence $t=2$.
\end{proof}

For $n=N$, Theorem~\ref{sap:prop:two-orbits} applies to $\C$ and every subgroup dual
$\D_H$.  For $\C$, take $L=G$ and
$\sigma=\rho$; for $\D_H$, use the pointed group and the distinguished
noninvertible object from Theorem~\ref{sap:prop:subgroup-bimodules}.  The two eigenvalues
have sum $N$ and product $-1$, so the orbit matrix is
\begin{equation}
 M=\begin{pmatrix}a&b\\ c&N-a\end{pmatrix}.
 \label{sap:eq:two-orbit-matrix}
\end{equation}
Then
\begin{equation}
 bc=a(N-a)+1.
 \label{sap:eq:bc-relation}
\end{equation}
If the orbit sizes are $s_1,s_2$, symmetry of the fusion matrix of
the chosen self-dual noninvertible is
\begin{equation}
 s_1b=s_2c.
 \label{sap:eq:orbit-balance}
\end{equation}
Both sides count the entries joining the two orbits in the symmetric
fusion matrix.

We now return to the fixed category $\C$.  For an indecomposable
$\C$-module category $\M$, the pointed classification in
Section~\ref{sec:separable} identifies its two orbits as
\begin{equation}
 \operatorname{Res}_{\mathcal P}\M
 \simeq\mathcal N_H\oplus\mathcal N_K,
 \qquad \mathcal P=\operatorname{Vec}_G.
 \label{sap:eq:original-pointed-orbits}
\end{equation}
Here $H$ and $K$ are the stabilizers of simples in the two original orbits.
Order them so that the first has the smaller diagonal coefficient $a$ in
\eqref{sap:eq:two-orbit-matrix}.  Since the two coefficients sum to the odd
integer $N$, we have $a\leq(N-1)/2$.  Choose a simple $x$ in this first orbit,
so $\operatorname{Stab}_G(x)=H$.
We use the notation $\mathfrak T_H:=\mathfrak T_{R_H}$ from the preliminaries.
In particular,
\begin{equation}
 \mathfrak T_H\bigl(\operatorname{Mod}_{\C}(R_H)\bigr)\simeq\D_H
 \label{sap:eq:regular-module-transport}
\end{equation}
as regular $\D_H$-module categories.

\subsubsection{Stabilizer reduction}
\label{sap:sec:core-reduction}
\label{sap:sec:core-list}

\begin{proposition}[Stabilizer reduction]
\label{sap:prop:core-list}
\label{sap:prop:stabilizer-dichotomy}
With $H,K$ and the smaller coefficient $a$ chosen above,
$\mathfrak T_H(\M)$ has a regular pointed orbit with diagonal
coefficient $a$.  Exactly one of the following holds:
\begin{enumerate}[label=(\roman*)]
 \item $a=0$, $H=K$, and $\M\simeq\operatorname{Mod}_{\C}(R_H)$.
 The transported orbit sizes are $N,N$, and a simple object has unit
 internal End.
 \item $a=1$ and $G=H\oplus K$.  The transported orbit sizes are $N,1$,
 and a simple object has connected separable internal End with underlying
 object $\one\oplus\sigma$, where $\sigma=V_{0,1}$ in $\D_H$.
\end{enumerate}
In particular, every class not represented by a subgroup algebra has
complementary Hall stabilizers.
\end{proposition}

\begin{proof}
By Theorem~\ref{sap:lem:restriction-morita-transport}, restriction to
$(\D_H)_{\mathrm{pt}}={}_{R_H}\mathcal P_{R_H}$ gives
\begin{equation}
 \operatorname{Fun}_{\mathcal P}(\mathcal N_H,\mathcal N_H)
 \oplus\operatorname{Fun}_{\mathcal P}(\mathcal N_H,\mathcal N_K).
 \label{sap:eq:transported-pointed-orbits}
\end{equation}
The first summand is regular.  The second is indecomposable by pointed
Morita equivalence \cite[Theorem~7.12.16]{EGNO}, hence is one pointed orbit.
Their sizes are
\begin{equation}
 s_1=N,\qquad\text{ and }\qquad s_2=r(H,K)=[G:H+K]|H\cap K|.
 \label{sap:eq:transported-orbit-ranks}
\end{equation}

Choose a simple $x$ in the first original orbit and put
$A_x=\underline{\End}_{\C}(x)$.
Its pointed support is $H$ with multiplicity one.  Internal Hom gives
$[A_x:\alpha_g\rho]=[\rho x:\alpha_{-g}x]$.
As $g$ varies, each simple in the orbit of $x$ occurs $|H|$ times, so the
total noninvertible multiplicity is $|H|a$.
By Theorem~\ref{sap:lem:pointed-truncation}(iii), normalize its pointed subalgebra
to $R_H$.  Then Theorem~\ref{sap:lem:algebra-extension-transport} gives the connected
separable algebra $\overline A_x={}_{R_H}(A_x){}_{R_H}$ representing
$\mathfrak T_H(\M)$, and Theorem~\ref{sap:lem:relative-dimension} gives
\begin{equation}
 \FPdim_{\C}(A_x)=|H|(1+ad),
 \qquad\text{ and }\qquad \FPdim_{\D_H}(\overline A_x)=1+ad.
 \label{sap:eq:diagonal-transport}
\end{equation}
Only the summand $R_H$ becomes pointed, by
Theorem~\ref{sap:prop:subgroup-bimodules}(i), and it becomes the tensor unit.
Thus the simple regular $\overline A_x$-module $\widetilde x$ has trivial
pointed stabilizer and internal End $\overline A_x$.
Under evaluation in Theorem~\ref{sap:lem:restriction-morita-transport}, it comes from
$x$ with its $R_H$-action, so lies in the first summand of
\eqref{sap:eq:transported-pointed-orbits}.
If its orbit has diagonal coefficient $a'$, regularity and internal Hom
give $\FPdim_{\D_H}(\overline A_x)=1+a'd$.
Comparison with \eqref{sap:eq:diagonal-transport} proves $a'=a$.

Write $\sigma\widetilde x=\bigoplus_z n_z z$.
Self-duality and the square fusion rule give
\begin{equation*}
 \sum_z n_z^2
 =\dim\Hom(\widetilde x,\sigma^2\widetilde x)
 =1+\sum_{\ell\in L_H}[\sigma\widetilde x:U_{-\ell}\widetilde x]
 =1+a,
\end{equation*}
since the $U_{-\ell}\widetilde x$ run through the first orbit once.
Consequently,
\begin{equation*}
 a+b=\sum_z n_z\leq\sum_z n_z^2=1+a.
\end{equation*}
Indecomposability gives $b>0$, so $b=1$, and equality forces every $n_z$
to be $0$ or $1$.  Equations \eqref{sap:eq:bc-relation} and
\eqref{sap:eq:orbit-balance} give
\begin{equation}
 c=a(N-a)+1,\qquad\text{ and }\qquad N=s_2c.
 \label{sap:eq:c-bound}
\end{equation}
Thus $c\leq N$.  If $a\geq2$, then
$c-N=(a-1)(N-a-1)>0$ from $a\leq(N-1)/2$.
Hence $a=0,1$.

The pointed part of $\underline{\End}(\widetilde x)$ is the unit, and
$[\underline{\End}(\widetilde x):W_\ell]
=[\sigma\widetilde x:U_{-\ell}\widetilde x]$.
Its noninvertible part is therefore multiplicity-free of length $a$.
If $a=0$, the internal End is the unit and \eqref{sap:eq:c-bound} gives
$s_2=N$.  The rank formula then says $|H+K|=|H\cap K|$, hence $H=K$.
The transported category is regular.  Since $\mathfrak T_H$ is a
2-equivalence, \eqref{sap:eq:regular-module-transport} identifies
$\M\simeq\operatorname{Mod}_{\C}(R_H)$, proving (i).
If $a=1$, then $s_2=1$, so both factors in
\eqref{sap:eq:transported-orbit-ranks} are one.
Thus $G=H\oplus K$, which makes $H,K$ Hall since $G$ is cyclic.
The internal End is $\one\oplus W_\ell$ for some $\ell$.
Translating $\widetilde x$ by $U_r$ conjugates its internal End and replaces
$\ell$ by $\ell+2r$.  Choosing $r=-\ell/2$ proves (ii).
\end{proof}

\subsubsection{The Hall-subgroup obstruction}
\label{sap:sec:hall-short}
\label{sap:sec:hall-analysis}

By Theorem~\ref{sap:prop:stabilizer-dichotomy}, any class not already represented by a
subgroup algebra has complementary Hall stabilizers.
Write $G=H\oplus K$, $m=|H|$, and $D=\mathscr P_HA$.
We compare the algebra $Q$ in $\C$ with a hypothetical connected
separable algebra on $\one\oplus\sigma$ in $\D_H$.  The hypothetical algebra imposes a subgroup-sum identity on $A$.
The cyclic phase parametrization supplied by the given algebra turns this identity into
a convolution equation.

By Theorems~\ref{sap:thm:partial-fourier-HI} and \ref{sap:thm:subgroup-fourier-morita} and
Theorem~\ref{sap:cor:partial-fourier-hermiticity}, $D$ is
a normalized Hermitian HI matrix for $\D_H$.  A connected separable algebra
on $\one\oplus\sigma$ therefore imposes \eqref{eq:display-0271}
on $D$.  For $y\in H\setminus\{0\}$, specialize
\eqref{sap:eq:partial-fourier-inverse} at $P=Q=0$ and sum over $x\in H$ to obtain
\begin{equation}
 \sum_{x\in H}A_{x,y}
 =\sum_{u\in\widehat H}D_{(0,u),(0,1)}u(y/2)
 =1.
 \label{sap:eq:partial-boundary-key}
\end{equation}
The sum over $x$ removes every character $v\neq1$.  The last equality uses
$D_{(0,u),(0,1)}=\delta_{u,1}-(d-1)^{-1}$ and
$\sum_u u(y/2)=0$ for $y\neq0$.

The next lemma extends the cyclic phase parametrization to the boundary
entries of $A$.  We use it to test whether a subgroup dual admits an
algebra structure on the unit plus one noninvertible simple.

\begin{lemma}[Reciprocal phase normalization]
\label{sap:lem:reciprocal-phases}
There exist $z_g \in \mathbb{C}$, for $g\in G\setminus\{0\}$, such that
\begin{equation}
 |z_g|=1,
 \qquad
 z_{-g}=\overline{z_g}.
 \label{sap:eq:z-reciprocity}
\end{equation}
and after setting $z_0=-d^{-1/2}$,
\begin{equation}
 A_{x,y}
 =\gamma z_xz_{y-x}\overline{z_y},
 \qquad
 \gamma=\frac{\sqrt d}{d-1}. \qquad (x,y\in G, \ y\neq 0)
 \label{sap:eq:G-factorization}
\end{equation}
\end{lemma}

\begin{proof}
Use representatives $0,1,\ldots,N-1$.  Evans--Gannon give real phase parameters such that, for
$0<x<y<N$,
\begin{equation}
 A_{x,y}=\gamma\exp\bigl(\mathrm i(j_y-j_x-j_{y-x})\bigr)
 \label{sap:eq:EG-ordered-phases}
\end{equation}
and extend by Hermiticity.  Both this phase formula and the recurrence
below are stated in \cite[Section~7.2]{evans2017non}, which attributes
them to Izumi \cite{IzumiLR2}.
Put $z_g^{(0)}=e^{-ij_g}$.  The recurrence
\begin{equation*}
 j_{n+1+i}=j_{n+1}+j_n-j_{n-i},
 \qquad N=2n+1,\quad 1\leq i<n,
\end{equation*}
shows that $z_g^{(0)}z_{N-g}^{(0)}$ is constant, say $c$, for every $g\neq0$.
For $1\leq g<n$, take $i=n-g$; the pair $g=n$ and the remaining pairs follow
by symmetry between $g$ and $N-g$.  Choose
$\xi\in\mathbb T$ with $\xi^N=c^{-1}$ and set
$z_g=\xi^gz_g^{(0)}$ for $1\leq g<N$.  Since
$\xi^x\xi^{y-x}\overline{\xi^y}$ is $1$ whenever $0<x<y<N$,
\eqref{sap:eq:EG-ordered-phases} is unchanged, while $z_gz_{N-g}=1$ proves
\eqref{sap:eq:z-reciprocity}.

Hermiticity gives \eqref{sap:eq:G-factorization} for $x>y$.  The cases $x=0$ and
$x=y$ follow from the algebra boundary values
$A_{0,y}=A_{y,y}=-(d-1)^{-1}$ and the choice $z_0=-d^{-1/2}$.
\end{proof}

\begin{theorem}[Nonexistence on $\one\oplus W_\ell$]
\label{sap:thm:hall-short}
If $H\neq\{0\}$ is Hall, including $H=G$, then no object
$\one\oplus W_\ell$, $\ell\in L_H$, admits a connected separable algebra
structure in $\D_H$.
\end{theorem}

\begin{proof}
Suppose such an algebra exists, and put $m=|H|>1$.
We derive a quadratic matrix identity whose trace gives a contradiction.
Give $R_H$ its standard special $C^*$-Frobenius structure and realize
$\D_H$ by the equivalent category of unitary $R_H$-bimodules.
This model is a multitensor $C^*$-category.  Transfer the hypothetical
separable algebra to this model and apply \cite[Theorem~4.13]{GiorgettiYuanZhao2024};
it is algebra-isomorphic there to a special $C^*$-Frobenius algebra.
The equivalence preserves algebra structures and algebra isomorphisms.
Since doubling is
bijective on the odd-order group $L_H$, conjugation moves its noninvertible
summand to the distinguished one while preserving connectedness and
separability.  Hence \eqref{sap:eq:partial-boundary-key} holds.

Take the phases from Theorem~\ref{sap:lem:reciprocal-phases} and define the matrix
indexed by $H$ by $T_{x,x}=0$ and $T_{x,y}=z_{x-y}$ for $x\neq y$.
Reciprocity makes $T$ Hermitian, with trace zero.  For $r\neq0$,
\eqref{sap:eq:partial-boundary-key} and \eqref{sap:eq:G-factorization} give
\begin{equation*}
 \sum_{u\in H}z_u z_{r-u}=\frac{d-1}{\sqrt d}\,z_r.
\end{equation*}
Removing $u=0,r$, where $z_0=-d^{-1/2}$, gives the off-diagonal entries of
\begin{equation}
 T^2=\lambda T+(m-1)I_m,
 \qquad \lambda=\frac{d+1}{\sqrt d}.
 \label{sap:eq:hall-trace-quadratic}
\end{equation}
The diagonal entries are $m-1$, since $|z_u|=1$ for $u\neq0$.

By \eqref{sap:eq:hall-trace-quadratic}, the eigenvalues lie among the roots
$t_+,t_-$ of $t^2-\lambda t-(m-1)$, which have opposite signs since $m>1$.  Both occur as eigenvalues of $T$, because its trace is zero.
If their multiplicities are $k,m-k$, then $0<k<m$, $kt_++(m-k)t_-=0$, and $t_+t_-=-(m-1)$. Eliminating the roots and using their sum $\lambda$ gives
\begin{equation*}
 \lambda^2=(m-1)\frac{(m-2k)^2}{k(m-k)}\in\mathbb Q.
\end{equation*}
But $\lambda^2=d+d^{-1}+2=\sqrt{N^2+4}+2$ is irrational, since
$N^2<N^2+4<(N+1)^2$.  This contradiction proves the theorem.
\end{proof}

\begin{remark}[Small orders]
Theorem~\ref{sa:classification} gives three indecomposable classes for $N=3,5,7$
and four for $N=9$.  For $N=3$ and $N=5$, this recovers the Grossman--Snyder
classifications discussed in Section~\ref{sec:separable}.  At $N=9$, the representatives
are $\one$, $R_{C_3}$, $R_{C_9}$, and $Q$.
\end{remark}

\begin{proof}[Proof of Theorem~\ref{sa:classification}]
It is enough to classify indecomposable finite semisimple $\C$-module
categories, by Theorem~\ref{sap:lem:pointed-truncation}(i) and the internal End
description recalled in Section~\ref{sec:separable}.

Let $\M$ be such a module category.  Order its pointed orbits by their
diagonal coefficients as in \eqref{sap:eq:original-pointed-orbits}, and let $H,K$
be their stabilizers.  By
Theorem~\ref{sap:prop:core-list}, either $\M$ is represented
by $R_H$, or $G=H\oplus K$ and $\mathfrak T_H(\M)$ has a connected
internal End with underlying object $\one\oplus\sigma$.

The latter case is excluded by Theorem~\ref{sap:thm:hall-short} unless $H=\{0\}$.
Then $\mathfrak T_H$ is the identity, so $\M$ itself has such an internal
End.  By Theorem~\ref{sap:prop:short-uniqueness}, its algebra structure is an invertible
conjugate of $Q$, hence Morita equivalent to $Q$.  Thus every module category is represented by an algebra on the list.

To distinguish these classes, Theorems~\ref{sap:prop:RH-modules} and \ref{sap:prop:Q-modules} give ranks
$2N/|H|$ and $N+1$.  Different subgroups of a cyclic group have different
orders.  Moreover, $2N>N+1$, while $2N/|H|\leq N$ for nontrivial $H$.
Thus all listed classes are distinct.
Since a cyclic group has one subgroup for every divisor of its order, there
are $\tau(N)$ subgroup-algebra classes and the class of $Q$, giving
$\tau(N)+1$ classes in total.
\end{proof}


\section{Finite cyclic Q-system data}
\label{app:finite-cyclic-data}
\label{sec:tables}

We collect the characteristic-zero cyclic examples and their modular
data, using the notation of Section~\ref{sec:centers}.
\subsection{Compact modular-data notation}
\label{ce:compact-MD}

We give the compact notation used for the cyclic rows below (it generalizes easily to any odd abelian group).  For $G=\mathbb Z/N$, let
$H$ be a metric group of order $\mu=N^2+4$.  When a center calculation
identifies its remaining family with $(H,q_H)$, we write the resulting
arrays as $\mathcal{MD}_{\sigma}(G;H,q_H)$.  Write
\begin{equation*}
 r_G=\frac{N^2-1}{2},\qquad m_H=\frac{\mu-1}{2},\qquad
 y=\frac{N}{\sqrt{\mu}},
\end{equation*}
and let $\Phi_G$ and $\Phi_H=(H\setminus\{0\})/\{\pm1\}$ be the sign-orbit
sets.  The remaining primaries are written
\begin{equation*}
  0_+,\ 0_-,\qquad
  \mathfrak a_{[a,b]}\quad([a,b]\in\Phi_G),\qquad
  \mathfrak d_{[h]}\quad([h]\in\Phi_H).
\end{equation*}
In particular, $\mathfrak d_{[h]}=\mathfrak d_{[-h]}$; the subscript is an
orbit label, not a multiplicity label.

For a quadratic form $q_H:H\to\mathbb Q/\mathbb Z$, put
\begin{equation*}
 B_H(h,h')=q_H(h+h')-q_H(h)-q_H(h').
\end{equation*}
The untwisted compact notation
$\mathcal{MD}_{\sigma}(G;H,q_H)$ uses
\begin{equation*}
 T_{\mathfrak a_{[a,b]}}=\exp(2\pi iab/N),\qquad
 T_{\mathfrak d_{[h]}}=\exp\left(2\pi i q_H(h)\right),
\end{equation*}
where $a,b\in\mathbb Z/N$ are the group and character coordinates.   
We use the notation only when the center
identification also gives the required normalization, including the
Gauss-sum sign $\sum_{h\in H}\exp(2\pi i q_H(h))=-r_\sigma$.

Let $\mathbf1_{r\times s}$ denote the $r\times s$ all-one matrix.  In the
ordered basis
\begin{equation*}
 (0_+,0_-),\quad
 (\mathfrak a_{[a,b]})_{[a,b]\in\Phi_G},\quad
 (\mathfrak d_{[h]})_{[h]\in\Phi_H},
\end{equation*}
the normalized $S$-matrix is
\begin{equation}\label{ce:MD-definition}
 S=\frac1N
 \begin{pmatrix}
  A_\sigma&
  \mathbf1_{2\times r_G}&
  B_\sigma\\
  \mathbf1_{r_G\times2}&
  E&
  \mathbf0_{r_G\times m_H}\\
  B_\sigma^{\top}&
  \mathbf0_{m_H\times r_G}&
  F_\sigma
 \end{pmatrix},
\end{equation}
where
\begin{align*}
 A_\sigma
 &=\frac12
 \begin{pmatrix}
  1-\sigma y&1+\sigma y\\
  1+\sigma y&1-\sigma y
 \end{pmatrix},&
 B_\sigma
 &=\sigma y
 \begin{pmatrix}
  1&\cdots&1\\
  -1&\cdots&-1
 \end{pmatrix},\\
 (E)_{[a,b],[a',b']}
 &=2\cos\left(2\pi\,
   \frac{ab'+a'b}{N}\right),&\text{ and }\qquad
 (F_\sigma)_{[h],[h']}
 &=-2\sigma y\cos\left(2\pi B_H(h,h')\right).
\end{align*}
In this block decomposition, $A_\sigma,B_\sigma,E,F_\sigma$ denote
blocks of the modular matrix. The notation $A_{g,h}$ continues to denote
entries of the HI-matrix. In particular,
\begin{equation*}
 S_{\mathfrak d_{[h]},\mathfrak d_{[h']}}
 =\frac{(F_\sigma)_{[h],[h']}}N
 =-\frac{2\sigma}{\sqrt{\mu}}\cos\left(2\pi B_H(h,h')\right)
 =-\frac1{r_\sigma}
   \left(e^{2\pi iB_H(h,h')}+e^{-2\pi iB_H(h,h')}\right).
\end{equation*}
For a cyclic metric group $H=\mathbb Z/\mu$, write
\begin{equation*}
 q_H(h)=\frac{s h^2}{\mu}\pmod{\mathbb Z},
 \qquad B_H(h,h')=\frac{2s hh'}{\mu}.
\end{equation*}
The label $\mu^\sigma_s$ records this cyclic metric group, its spherical
branch, and its coefficient $s$.  Replacing $s$ by $sv^2$ for a unit $v$
modulo $\mu$ gives an equivalent metric form.  A product such as
$65^+_2\mathbin{\times}13^+_2$ denotes a direct product of metric groups,
with the sum of the displayed quadratic forms.  A phrase such as
$r$ rows: $\mu^\sigma_s$ records row multiplicity and is not a metric-group
product.

\subsection{Arithmetic recognition of the metric group}
\label{ce:recognition-proof}

We give the arithmetic proof of Proposition~\ref{ce:metric-recognition}.
Write $q=q_H$, $M=|H|$, and, for $p^a\Vert M$, write $H_p$ for the
$p$-primary part of $H$. We recover its primary ranks and determinant
signs from the Gauss sums in \eqref{ce:twist-trace}. For odd $p$, an odd-primary nondegenerate homogeneous quadratic module has
an orthogonal decomposition
\begin{equation*}
 (H_p,q|_{H_p})\cong
 \bigoplus_{s=1}^{a}\ \bigoplus_{\nu=1}^{n_{p,s}}
 \langle c_{p,s,\nu}\rangle_{p^s},
 \qquad\text{ with}\qquad
 q(x)=\frac{c_{p,s,\nu}x^2}{p^s},
 \quad c_{p,s,\nu}\in(\mathbb Z/p^s)^\times .
\end{equation*}
The notation $\langle c\rangle_{p^s}$ denotes the cyclic group of order
$p^s$ with the displayed quadratic form.  Set
\begin{equation*}
 \Delta_{p,s}=\left(\frac{\prod_\nu c_{p,s,\nu}}{p}\right),
\end{equation*}
with the empty product equal to $1$.  The integers $n_{p,s}$ and signs
$\Delta_{p,s}$ determine the isometry class.  To see the decomposition,
put $p^e=\exp(H_p)$ and choose elements $x,y$ whose polar pairing has order
$p^e$. Nondegeneracy gives such a pair. At least one of
$q(x),q(y),q(x+y)$ has denominator $p^e$;
otherwise their polarization would have smaller order.  Choose such an
element $z$.  Its cyclic summand is nondegenerate because $2q(z)$ has the
same order, and hence it splits off orthogonally.  Iterating gives the
displayed decomposition.  At a fixed
scale, a unit modulo $p^s$
is a square times either $1$ or a fixed nonsquare $\nu$.  The two nonsquare
summands can be replaced by two square summands: the sets of squares and of
$\nu$ minus squares in $\mathbb F_p$ intersect, so there are units $u,v$
with $u^2+v^2\equiv\nu\pmod p$.  Hensel's lemma lifts this relation to
$\mathbb Z/p^s$ (both coordinates are nonzero), and the change of basis with
columns $(u,v)$ and $(-v,u)$ gives the required equivalence.

The required Gauss sums can be read from the ordinary prime powers.  For
any integer $k$, nondegeneracy and oddness give
\begin{equation*}
 \left|\mathcal G_A(k)\right|^2=M|H[k]|.
\end{equation*}
To see this, write one variable as $y+z$ in the product of $\mathcal G_A(k)$
with its complex conjugate.  The sum over $y$ vanishes unless $kz=0$ by
nondegeneracy of the polar form.  In that case $k b(z,z)=0$, so
$2kq(z)=0$.  The order of $q(z)$ is odd, since the order of $z$ is odd and
$q$ is homogeneous; hence $kq(z)=0$.  If
$f_{p,t}=\sum_s\min(t,s)n_{p,s}$, then
\begin{equation*}
 \left|\mathcal G_A(p^t)\right|^2=M p^{f_{p,t}},
 \qquad f_{p,0}=0,\quad f_{p,a}=f_{p,a+1}=a.
\end{equation*}
Thus $n_{p,s}=2f_{p,s}-f_{p,s-1}-f_{p,s+1}$ for $1\leq s\leq a$ and it remains to recover the signs. Put
\begin{equation*}
 h_p=\sum_{u\bmod p}\exp\!\left(2\pi i\frac{u^2}{p}\right).
\end{equation*}
The quadratic Gauss-sum evaluation
\cite[Theorem~1.1 and Lemma~4.1]{murtyPathakGauss} gives, for a unit $c$,
\begin{equation*}
 \sum_{x\bmod p^j}\exp\!\left(2\pi i\frac{cx^2}{p^j}\right)
 =\begin{cases}
    p^{j/2},&j\text{ even},\\
    \left(\dfrac cp\right)p^{(j-1)/2}h_p,&j\text{ odd}.
   \end{cases}
\end{equation*}
For completeness, the recurrence behind this formula is
$g_j(c)=p g_{j-2}(c)$ for $j\geq2$: writing
$x=u+p^{j-1}v$ and summing over $v$ forces $p\mid u$.
For $0\leq t\leq a$, put
\begin{equation*}
 Q_{p,t}=\sum_{x\in H_p}\exp\!\left(2\pi i p^tq(x)\right).
\end{equation*}
Consequently, with
\begin{equation*}
 b_{p,t}=\sum_{\substack{s>t\\s-t\ {\rm odd}}}n_{p,s},
 \qquad
 E_{p,t}=\frac{a+f_{p,t}-b_{p,t}}2,
\end{equation*}
the normalized local sums
\begin{equation*}
 u_{p,t}=\frac{Q_{p,t}}{p^{E_{p,t}}h_p^{b_{p,t}}}
\end{equation*}
recover the determinant signs through
\begin{equation*}
 \Delta_{p,s}=u_{p,s-1}u_{p,s+1},
 \qquad u_{p,t}=1\quad(t\geq a).
\end{equation*}
The local sum $Q_{p,t}$ is obtained from the global trace without using CRT
exponents.  Define
\begin{equation*}
 \mathcal G_{A,p,t}=\mathcal G_A(p^t),\qquad
 \varepsilon_p=
 \prod_{\substack{p'\mid M\\p'\ne p}}
 \left(\frac p{p'}\right)^{\sum_{s\ {\rm odd}}n_{p',s}}.
\end{equation*}
Since multiplication by $p^t$ changes each other-primary odd block by the
factor recorded in $\varepsilon_p^t$, while the $p$-primary factor at
$t=a$ is $p^a$, we have
\begin{equation*}
 Q_{p,t}=p^a\varepsilon_p^{a-t}
       \frac{\mathcal G_{A,p,t}}{\mathcal G_{A,p,a}}
       \qquad(0\leq t<a).
\end{equation*}
The denominators are nonzero by the norm identity.  Since every prime
dividing $M=N^2+4$ is coprime to $N$, the trace formula simplifies here to
$\mathcal G_{A,p,t}=2\operatorname{Tr}(D_A^{p^t})-N$ for $t\geq1$.
The value at $t=0$ is the known number $\mathcal G_A(1)=-r_\sigma$.
Therefore at most $\sum_{p\mid M}v_p(M)$ traces of powers of $D_A$ determine
all $n_{p,s}$ and $\Delta_{p,s}$, and hence determine $(H,q)$.

\subsection{Computing the invariants from the HI matrix}
\label{ce:return-recurrence}

Here we compute the return coefficients in \eqref{ce:return-identity} without
forming powers of $D_A$. Put $a=d^{-1}$ and define $G\times G$ arrays
\begin{equation*}
 V_1(g,h)=A_{g,2h},\qquad
 V_{j+1}(g,h)=\sum_{x\in G}A_{g+x,2h}V_j(h,x).
\end{equation*}
Set $\beta_0=1$, $\beta_1=-a$, and, for $j\geq1$, put
\begin{equation*}
 \beta_{2j}=\sum_{g,h}V_j(g,h)V_j(h,g),\qquad
 \beta_{2j+1}=\sum_{g,h}V_j(g,h)V_{j+1}(h,g).
\end{equation*}
These contractions are bilinear; there is no complex conjugation.
The desired coefficients are given by $u_0=1$, $u_1=0$, and
\begin{equation}\label{ce:return-recursion}
 u_n=a u_{n-2}+a\sum_{t=0}^{n-3}\beta_t u_{n-3-t}\qquad(n\geq2),
\end{equation}
where an empty sum is zero. To prove \eqref{ce:return-recursion}, define
\begin{equation*}
 (LU)(g,h)=\sum_x A_{g+x,2h}U(h,x),\qquad
 v(g,h)=\delta_{h,0},\qquad\text{ and }\qquad w(U)=\sum_hU(0,h).
\end{equation*}
Writing $D_A^kC_\rho=u_kC_\rho+f_kF_\rho+\sum U_k(g,h)X_{g,h}$,
the operator formula gives
\begin{equation*}
 u_{k+1}=a^2f_k+a w(U_k),\qquad
 f_{k+1}=d u_k,\qquad U_{k+1}=a f_kv+LU_k.
\end{equation*}
The initial values are $u_0=1$, $f_0=0$ and $U_0=0$.
Eliminating the last two formal generating series gives
\begin{equation*}
 \sum_{k\geq0}u_kz^k
 =\frac1{1-az^2-az^3\sum_{t\geq0}w(L^tv)z^t}.
\end{equation*}
For the bilinear form $\langle U,V\rangle=\sum_{g,h}U(g,h)V(h,g)$,
exchanging $g$ and $x$ gives
\begin{equation*}
 \langle LU,V\rangle
 =\sum_{g,h,x}A_{g+x,2h}U(h,x)V(h,g)
 =\langle U,LV\rangle.
\end{equation*}
Since $w(U)=\langle v,U\rangle$ and $L^jv=V_j$, it follows that
$w(L^{2j}v)=\langle V_j,V_j\rangle$ and
$w(L^{2j+1}v)=\langle V_j,V_{j+1}\rangle$.
Also $w(v)=1$ and
$w(Lv)=\sum_hA_{0,2h}=\sum_hA_{h,0}=-a$ by (E1)--(E2).
Thus $w(L^tv)=\beta_t$, proving the recurrence.
Only four array updates after $V_1$ are needed for $\Theta_{13}$;
$\Theta_5$ is already the quadratic contraction \eqref{ce:fifth-invariant}.

The exponent of $H$ also has a direct spectral interpretation: it is the
least common multiple of the remaining twist orders. Indeed, if $h$ has
odd order $t$, homogeneity and polarization give $t^2q_H(h)=2tq_H(h)=0$,
so $tq_H(h)=0$. Conversely, an integer annihilating all values of $q_H$
annihilates its polarization, and hence $H$ by nondegeneracy. For $N=11,29,39$, one additional invariant and the spherical sign
identify the metric form. The possibilities are
\begin{center}
\begin{tabular}{c|c|c|c}
 $N$ & invariant & value & $H$\\
 \hline
 $11$ & $\Theta_5(A)$ & $\sigma\sqrt5$ & $C_{125}$\\
      &                 & $\pm5$ & $C_{25}\times C_5$\\
      &                 & $5\sigma\sqrt5$ & $C_5^3$\\
 $29$ & $\Theta_{13}(A)$ & $\pm\sqrt{13}$ & $C_5\times C_{169}$\\
      &                  & $\pm13$ & $C_5\times C_{13}^2$\\
 $39$ & $\Theta_5(A)$ & $\pm\sqrt5$ & $C_{25}\times C_{61}$\\
      &                 & $\pm5$ & $C_5^2\times C_{61}$
\end{tabular}
\end{center}
The signed invariant and $\mathcal G_A(1)=-r_\sigma$ determine all
local determinant classes. These are abstract metric possibilities;
we do not assert that each is realized by an HI matrix.

\emph{Exact specialization.}\label{ce:finite-field-recognition}
We record the condition under which a finite-field comparison proves
a characteristic-zero spectral identity.  Remove the known factors from
$\det(XI-D_A)$ and denote the resulting remaining-twist polynomial by
$R_A(X)$.  Let $R_{H,q}(X)$ be the polynomial obtained from the metric-group
twists.  Suppose that their coefficients have a common good specialization
at a prime $\ell\nmid\mu$ and that the two reduced polynomials agree.  Both
polynomials split over characteristic zero into $\mu$th roots of unity.
Since $X^\mu-1$ is separable in characteristic $\ell$, distinct such roots
remain distinct after reduction.  The reduced polynomials therefore have the
same multiplicity at every root, and the polynomials agree in characteristic
zero.  A congruence between one trace and one candidate value is sufficient
only when the candidate fingerprints remain pairwise distinct under the same
compatible specialization.  This is the check required in the finite-field
identifications at $N=39,41,43$; the embeddings used for the displayed rows are described below.

\subsection{Finite cyclic examples}

Throughout this appendix the matrices are over $\mathbb C$. Put $N=2n+1$
and $G=\mathbb Z/N$. We use the compact modular-data
notation of Section~\ref{ce:compact-MD}: a label such as $85^+_{42}$ means
$\mu^+_{42}$ for the cyclic metric group of order $\mu=N^2+4=85$,
with the sign recording the spherical branch.  Products of labels denote
orthogonal sums of metric groups. Tables~\ref{tab:cyclic-census} and~\ref{tab:cyclic-negative} list our
positive- and negative-branch examples through $N=43$, respectively. For every positive row, we use the reconstruction coordinates of
\cite[Section~7.2]{evans2017non}.  Given
$j=(j_2,\ldots,j_{n+1})$, put $j_1=0$ and extend by
\begin{equation*}
 j_{n+1+i}=j_{n+1}+j_n-j_{n-i}\qquad(1\leq i<n).
\end{equation*}
With $d=d_+$, the matrix is recovered by
\begin{equation*}
 A_{g,h}=\frac{\sqrt d}{d-1}
 \exp\!\bigl(\mathrm i(j_h-j_g-j_{h-g})\bigr)\quad(0<g<h<N),
 \qquad A_{h,g}=\overline{A_{g,h}},
\end{equation*}
with $A_{g,0}=A_{0,g}=A_{g,g}=-(d-1)^{-1}$ for $g\ne0$ and
$A_{0,0}=1-(d-1)^{-1}$.

\emph{From decimal phases to exact matrices.}
For a positive row, put
$z_g=\exp\!\bigl(\mathrm i(j_g-j_{g-1})\bigr)$
and replace complex conjugates by inverses on the unit circle.
After adjoining a square root of $d$ and clearing denominators,
the HI equations become polynomial equations in these coordinates.
We refine the numerical solution and apply the algebraic recognition
and exact verification procedure of
Section~\ref{sec:char2-algebraic-recovery}.
Root isolation specifies the intended embedding.
The rounded decimals alone do not determine an algebraic number uniquely.

For each odd $N\leq43$ and each scalar fibre, the displayed rows
are complete up to tensor equivalence by
Corollary~\ref{cor:HI-categorical-completeness}, as explained in
Section~\ref{sec:char2-algebraic-recovery}. 
For example, at $N=15$ the positive locus consists of $24$ marked
Hermitian matrices in three $\operatorname{Aut}(\mathbb Z/15)$-orbits,
represented by the three order-fifteen rows of
Table~\ref{tab:cyclic-census}.

When all displayed representatives belong to one algebraic family over
$\mathbb Q(d)$, we write $j^{(1)},j^{(2)},\ldots$ and
$\alpha^{(1)},\alpha^{(2)},\ldots$.  When several algebraic families occur,
the family is recorded by a subscript, as in $j_A^{(1)},j_B^{(1)}$ and
$\alpha_A^{(1)},\alpha_B^{(1)}$. At $N=11$ the families $A$ and $B$ have
degrees $10$ and $2$; at $N=29$ they have degrees $56$ and $4$.
Representative numbers are chosen separately for the two spherical signs
and do not specify a pairing between signs.

\emph{Rows at orders thirty-three through forty-three.}
For the positive rows at these orders, the exact computations first give
the principal arguments of the entries $B_{1,g}$.
Table~\ref{tab:cyclic-census} records the same
cumulative coordinates as the earlier positive rows:
\begin{equation*}
 j_g\equiv\sum_{r=2}^g\arg B_{1,r}\pmod{2\pi},\qquad
 -\pi<j_g\leq\pi.
\end{equation*}
Thus $B_{1,g}=\sqrt d\exp\!\bigl(\mathrm i(j_g-j_{g-1})\bigr)$.

The exact algebraic families contain respectively
$100,72,108,120,80,84$ marked matrices for each scalar sign at
$N=33,35,37,39,41,43$. The displayed representatives give respectively
$5,3,3,5,2,2$ pairwise tensor-inequivalent classes under unit relabeling.
The exhaustive characteristic-two censuses have the same respective
numbers of $\operatorname{Aut}(\mathbb Z/N)$-orbits.
At $N=43$, the two displayed representatives are inequivalent because
their modular data differ.

The entries are rounded evaluations of algebraic presentations checked
against the full HI equations. At $N=39$, the four representatives with
subscript $A$ belong to a degree-$96$ family over $\mathbb Q(d)$, and
the representative with subscript $B$ belongs to a degree-$24$ family.

For $N=33,35,37$ the modular labels follow from
Corollary~\ref{ce:prime-discriminant}. At $N=39,41,43$, exact
comparisons of the full remaining twist spectra at the split primes
$2379001$, $2901571$ and $31074811$ cover all marked reductions of the
specified families. Compatible matrix and cyclotomic specializations
identify the characteristic-zero modular data by
Appendix~\ref{ce:finite-field-recognition}; the remaining $S$-blocks
were also checked at these primes. Quadratic Gauss sums fix the radical
embeddings. Matching to the displayed decimal representatives uses
high-precision phase evaluations at $N=39,41$, and exact factorization
with root isolation at $N=43$.

\begingroup
\footnotesize
\begin{longtable}{@{}>{\raggedright\arraybackslash}p{.12\linewidth}
 >{\raggedright\arraybackslash}p{.13\linewidth}
 >{\fontsize{7}{8.8}\selectfont\raggedright\arraybackslash}p{\dimexpr.75\linewidth-4\tabcolsep\relax}@{}}
\caption{Unitary solutions.}\label{tab:cyclic-census}\\
\toprule
row & modular label & \footnotesize phase vector $j$\\
\midrule
\endfirsthead
\multicolumn{3}{c}{\tablename\ \thetable\ -- continued from previous page}\\
\toprule
row & modular label & \footnotesize phase vector $j$\\
\midrule
\endhead
\midrule
\multicolumn{3}{r@{}}{continued on next page}\\
\endfoot
\bottomrule
\endlastfoot
$j(3)$ & $13^+_6$ & $\bigl(1.2920763\bigr)$ \\
$j(5)$ & $29^+_{14}$ & $\bigl(0.1846862,\allowbreak 1.5984701\bigr)$ \\
$j(7)$ & $53^+_{26}$ & $\bigl(2.4712280,\allowbreak 0.5168555,\allowbreak 0.2137724\bigr)$ \\
$j^{(1)}(9)$ & $85^+_{42}$ & $\bigl(2.3969766,\allowbreak 2.0792511,\allowbreak -0.2079168,\allowbreak -2.5086739\bigr)$ \\
$j^{(2)}(9)$ & $85^+_{6}$ & $\bigl(-2.3647370,\allowbreak 1.0310571,\allowbreak 1.5696921,\allowbreak 0.3383837\bigr)$ \\
$j_A^{(1)}(11)$ & $125^+_{62}$ & $\bigl(0.9996506,\allowbreak 2.7258434,\allowbreak -0.5714202,\allowbreak -1.7797340,\allowbreak 1.2675984\bigr)$ \\
$j_B^{(1)}(11)$ & $25^+_{12} \times 5^+_2$ & $\bigl(-2.6444397,\allowbreak -1.7629598,\allowbreak -2.6444397,\allowbreak 2.7572656,\allowbreak 0.1128259\bigr)$ \\
$j(13)$ & $173^+_{86}$ & $\bigl(-3.1050384,\allowbreak 0.5993398,\allowbreak -0.1117084,\allowbreak -0.9697663,\allowbreak 1.3368480,\allowbreak 1.0048312\bigr)$ \\
$j^{(1)}(15)$ & $229^+_{114}$ & $\bigl(-1.0777622,\allowbreak -0.7748018,\allowbreak -2.1718632,\allowbreak -1.6068402,\allowbreak -0.2575080,\allowbreak 2.0925023,\allowbreak 0.7228956\bigr)$ \\
$j^{(2)}(15)$ & $229^+_{114}$ & $\bigl(-0.8200962,\allowbreak 2.2979417,\allowbreak -2.2018727,\allowbreak -1.3053464,\allowbreak -1.5531161,\allowbreak -0.5755386,\allowbreak -2.0576419\bigr)$ \\
$j^{(3)}(15)$ & $229^+_{114}$ & $\bigl(-1.7195632,\allowbreak -1.3052084,\allowbreak 2.8412918,\allowbreak -0.1489797,\allowbreak -0.5712485,\allowbreak 0.0066450,\allowbreak 0.2236904\bigr)$ \\
$j(17)$ & $293^+_{146}$ & $\bigl(-1.4660739,\allowbreak 0.2914894,\allowbreak 3.1307349,\allowbreak -2.6931849,\allowbreak 1.3981531,\allowbreak -0.6119383,\allowbreak -1.6670776,\allowbreak -1.7548211\bigr)$ \\
$j^{(1)}(19)$ & $365^+_{182}$ & $\bigl(-2.6774651,\allowbreak 1.0889721,\allowbreak -0.8994418,\allowbreak 0.0154481,\allowbreak -1.2409281,\allowbreak -0.4933943,\allowbreak 1.8398793,\allowbreak -1.5258845,\allowbreak -2.0843734\bigr)$ \\
$j^{(2)}(19)$ & $365^+_{11}$ & $\bigl(-1.1232827,\allowbreak -2.7155433,\allowbreak -2.5538699,\allowbreak -0.7474485,\allowbreak 1.7489569,\allowbreak 1.5103331,\allowbreak -2.6124761,\allowbreak -2.7460298,\allowbreak -2.6924799\bigr)$ \\
$j^{(1)}(21)$ & $445^+_{222}$ & $\bigl(-2.2370593,\allowbreak -1.5238885,\allowbreak -1.0139143,\allowbreak -0.8419553,\allowbreak 1.4781116,\allowbreak -0.9083293,\allowbreak -0.3741674,\allowbreak 2.3389385,\allowbreak -1.7444562,\allowbreak 1.8977849\bigr)$ \\
$j^{(2)}(21)$ & $445^+_6$ & $\bigl(0.2558899,\allowbreak 2.9194734,\allowbreak 2.9503132,\allowbreak 1.3415835,\allowbreak -0.0065763,\allowbreak 1.0901391,\allowbreak 0.9954091,\allowbreak -1.4362241,\allowbreak 2.4093776,\allowbreak 0.0671580\bigr)$ \\
$j^{(3)}(21)$ & $445^+_{222}$ & $\bigl(-2.2025440,\allowbreak -1.9404632,\allowbreak -3.0108077,\allowbreak -2.6513734,\allowbreak -2.6276981,\allowbreak -0.3349944,\allowbreak -2.1862386,\allowbreak -1.9231830,\allowbreak -0.1691744,\allowbreak 1.9753120\bigr)$ \\
$j^{(4)}(21)$ & $445^+_6$ & $\bigl(2.6908357,\allowbreak 1.2426190,\allowbreak -2.4360935,\allowbreak 1.1685504,\allowbreak 1.9651482,\allowbreak -2.1124178,\allowbreak 0.0707480,\allowbreak -1.7294321,\allowbreak -2.1420528,\allowbreak -1.7089847\bigr)$ \\
$j^{(1)}(23)$ & $533^+_{266}$ & $\bigl(0.7193532,\allowbreak 1.1507367,\allowbreak -0.3661121,\allowbreak 1.2000074,\allowbreak -2.5246756,\allowbreak 0.0378541,\allowbreak -1.6153618,\allowbreak 1.4169612,\allowbreak -1.6456978,\allowbreak -1.5483920,\allowbreak 0.6236422\bigr)$ \\
$j^{(2)}(23)$ & $533^+_3$ & $\bigl(0.5637638,\allowbreak 0.5944634,\allowbreak -1.1281791,\allowbreak -0.8519442,\allowbreak 1.5891687,\allowbreak 1.7194777,\allowbreak 2.4667932,\allowbreak 1.8808503,\allowbreak -1.5202376,\allowbreak 1.4817129,\allowbreak -0.2559956\bigr)$ \\
$j^{(1)}(25)$ & $629^+_{314}$ & $\bigl(0.3339126,\allowbreak 1.7042370,\allowbreak 2.9944255,\allowbreak 2.4721207,\allowbreak 2.4430927,\allowbreak -0.9063114,\allowbreak -2.5268216,\allowbreak 1.0196997,\allowbreak 1.2080308,\allowbreak 2.8066300,\allowbreak -0.1828353,\allowbreak 1.1979048\bigr)$ \\
$j^{(2)}(25)$ & $629^+_3$ & $\bigl(-0.5127148,\allowbreak -0.6413885,\allowbreak -0.5896815,\allowbreak 0.4889445,\allowbreak -2.5418207,\allowbreak 0.9660650,\allowbreak 1.4519598,\allowbreak -1.3553586,\allowbreak 0.6226600,\allowbreak -0.0764020,\allowbreak 1.5578307,\allowbreak 2.5802782\bigr)$ \\
$j^{(1)}(27)$ & $733^+_{366}$ & $\bigl(-2.4092469,\allowbreak -0.8563151,\allowbreak 3.0329948,\allowbreak 3.1371830,\allowbreak -0.9926756,\allowbreak -2.7341860,\allowbreak 0.9349933,\allowbreak 2.9664625,\allowbreak -1.2284420,\allowbreak -2.0725752,\allowbreak -2.8810383,\allowbreak 2.3033243,\allowbreak 0.3403274\bigr)$ \\
$j^{(2)}(27)$ & $733^+_{366}$ & $\bigl(-3.0415938,\allowbreak -1.7376577,\allowbreak 2.4244864,\allowbreak -0.1942298,\allowbreak 2.7311386,\allowbreak -2.7990193,\allowbreak -2.7674486,\allowbreak 2.7795863,\allowbreak -1.3563228,\allowbreak -2.3226722,\allowbreak -0.2362278,\allowbreak 1.8383025,\allowbreak 2.9236301\bigr)$ \\
$j^{(3)}(27)$ & $733^+_{366}$ & $\bigl(1.4237361,\allowbreak 2.4690096,\allowbreak -0.8026222,\allowbreak 1.5221708,\allowbreak -1.7539809,\allowbreak 2.2230386,\allowbreak 2.0496486,\allowbreak 1.5390928,\allowbreak -1.6108452,\allowbreak -0.3152309,\allowbreak -0.9287271,\allowbreak 0.6804064,\allowbreak -0.6576277\bigr)$ \\
$j_A^{(1)}(29)$ & $845^+_{422}$ & $\bigl(1.3127646,\allowbreak 0.0414993,\allowbreak 0.5026149,\allowbreak -1.3399270,\allowbreak 1.5639559,\allowbreak -0.9521122,\allowbreak -2.6826562,\allowbreak 3.0077611,\allowbreak -1.3009749,\allowbreak 1.6389150,\allowbreak 2.0377123,\allowbreak -1.8638674,\allowbreak 0.8641946,\allowbreak -1.5906218\bigr)$ \\
$j_A^{(2)}(29)$ & $845^+_3$ & $\bigl(-1.1805882,\allowbreak -2.0828341,\allowbreak 1.3424986,\allowbreak 1.6095980,\allowbreak 0.5660763,\allowbreak -1.7827282,\allowbreak -0.1211648,\allowbreak 2.1097822,\allowbreak 2.8898475,\allowbreak -2.5651291,\allowbreak -2.8248067,\allowbreak 1.7160818,\allowbreak 2.9340175,\allowbreak 1.3691798\bigr)$ \\
$j_B^{(1)}(29)$ & $65^+_2 \times 13^+_2$ & $\bigl(-1.64251434,\allowbreak -1.29303720,\allowbreak 1.74738570,\allowbreak 2.09686284,\allowbreak 2.44633999,\allowbreak 2.09686284,\allowbreak -1.145899545,\allowbreak -2.788413893,\allowbreak -0.4469452614,\allowbreak 1.502779925,\allowbreak 0.2520090213,\allowbreak -1.390505327,\allowbreak 2.551211350,\allowbreak 1.300440446\bigr)$ \\
$j^{(1)}(31)$ & $965^+_{482}$ & $\bigl(-2.617177804,\allowbreak -0.464461128,\allowbreak 2.39062333,\allowbreak -2.693178171,\allowbreak 0.955318177,\allowbreak -1.6648769535,\allowbreak -2.29945502,\allowbreak 1.47842445,\allowbreak 1.008586631,\allowbreak 1.43438315,\allowbreak 2.960718542,\allowbreak 0.947576986,\allowbreak 0.330689829,\allowbreak 2.9697571515,\allowbreak -1.507647469\bigr)$ \\
$j^{(2)}(31)$ & $965^+_{11}$ & $\bigl(2.474615516,\allowbreak -3.01932787,\allowbreak -0.20656944,\allowbreak -0.551196761,\allowbreak -2.362054488,\allowbreak 1.7829773505,\allowbreak -1.08525091,\allowbreak -2.44205921,\allowbreak -2.214844455,\allowbreak 2.980046108,\allowbreak -0.91082405,\allowbreak 0.063255467,\allowbreak -1.363391313,\allowbreak -2.7187506376,\allowbreak -2.333442255\bigr)$ \\
$j^{(1)}(33)$ & $1093^{+}_{546}$ & $\bigl(-2.90935208,\allowbreak -3.03906161,\allowbreak 1.15689750,\allowbreak 1.55673557,\allowbreak -1.07152462,\allowbreak 1.68013575,\allowbreak 1.10197818,\allowbreak 0.24069563,\allowbreak -2.48926255,\allowbreak -1.62869169,\allowbreak -0.82298684,\allowbreak 1.10485327,\allowbreak -2.49018273,\allowbreak 0.92744248,\allowbreak -1.51091796,\allowbreak 0.45817844\bigr)$ \\
$j^{(2)}(33)$ & $1093^{+}_{546}$ & $\bigl(-2.88458991,\allowbreak -1.69698063,\allowbreak 2.15822291,\allowbreak 0.77881548,\allowbreak 0.73592017,\allowbreak 1.73394591,\allowbreak -2.27018420,\allowbreak 1.26928611,\allowbreak -3.06274598,\allowbreak 1.17751451,\allowbreak 0.25153269,\allowbreak 2.59415668,\allowbreak 1.92057295,\allowbreak 0.46730227,\allowbreak -0.36908265,\allowbreak -2.60479949\bigr)$ \\
$j^{(3)}(33)$ & $1093^{+}_{546}$ & $\bigl(-2.80128812,\allowbreak 1.29812898,\allowbreak -1.75431895,\allowbreak -0.33478543,\allowbreak 1.24419542,\allowbreak 0.18784159,\allowbreak 0.87624330,\allowbreak 1.28207988,\allowbreak -1.55665312,\allowbreak -0.60139609,\allowbreak 1.84099168,\allowbreak -1.55627592,\allowbreak -1.17019042,\allowbreak -1.77735527,\allowbreak -0.50116163,\allowbreak -2.72305181\bigr)$ \\
$j^{(4)}(33)$ & $1093^{+}_{546}$ & $\bigl(-2.68781642,\allowbreak 2.05810584,\allowbreak -0.44318017,\allowbreak 1.08671268,\allowbreak -3.07241308,\allowbreak 2.46114122,\allowbreak -0.33462735,\allowbreak -1.04518877,\allowbreak -2.78125046,\allowbreak 0.91057337,\allowbreak 1.92695133,\allowbreak 1.27118647,\allowbreak 2.60323603,\allowbreak -1.34499043,\allowbreak 1.69417288,\allowbreak 2.62765178\bigr)$ \\
$j^{(5)}(33)$ & $1093^{+}_{546}$ & $\bigl(-2.64053685,\allowbreak -1.33484794,\allowbreak 0.56132453,\allowbreak -0.66610337,\allowbreak 3.14009321,\allowbreak -2.70254458,\allowbreak 1.35967071,\allowbreak 0.91646863,\allowbreak 0.42599182,\allowbreak 2.12190292,\allowbreak 2.77552698,\allowbreak 2.27631917,\allowbreak 2.22546349,\allowbreak 2.66944962,\allowbreak 0.20190707,\allowbreak 2.75918605\bigr)$ \\
$j^{(1)}(35)$ & $1229^{+}_{614}$ & $\bigl(-2.81089063,\allowbreak -1.92652714,\allowbreak -3.13707655,\allowbreak 2.66476167,\allowbreak 2.68730099,\allowbreak -2.37240357,\allowbreak 2.59293204,\allowbreak -2.38970735,\allowbreak -1.80658102,\allowbreak 0.75659409,\allowbreak 0.26627533,\allowbreak -0.86658396,\allowbreak 1.93362594,\allowbreak -1.61688000,\allowbreak -0.02991702,\allowbreak 1.71526796,\allowbreak 0.56698174\bigr)$ \\
$j^{(2)}(35)$ & $1229^{+}_{614}$ & $\bigl(-2.58633223,\allowbreak -2.48673065,\allowbreak 1.78743663,\allowbreak 2.19770567,\allowbreak 2.47535759,\allowbreak -1.81437718,\allowbreak 2.01406312,\allowbreak -1.48532922,\allowbreak 2.10662455,\allowbreak -2.45710300,\allowbreak -2.02248981,\allowbreak 2.07753485,\allowbreak -1.79504230,\allowbreak -0.38935444,\allowbreak -0.75657117,\allowbreak 0.53705343,\allowbreak 1.61992748\bigr)$ \\
$j^{(3)}(35)$ & $1229^{+}_{614}$ & $\bigl(-1.93071574,\allowbreak 2.98219380,\allowbreak -2.64760422,\allowbreak -0.25862687,\allowbreak 1.58734692,\allowbreak 0.18899346,\allowbreak 2.26186841,\allowbreak -0.34353253,\allowbreak -0.32821702,\allowbreak 1.11143573,\allowbreak -1.97137515,\allowbreak 0.53137850,\allowbreak 3.09505134,\allowbreak 2.61066748,\allowbreak -2.39823186,\allowbreak -2.28931228,\allowbreak -1.43279508\bigr)$ \\
$j^{(1)}(37)$ & $1373^{+}_{686}$ & $\bigl(-2.93211149,\allowbreak -0.70392961,\allowbreak 3.04557250,\allowbreak 1.55147116,\allowbreak -2.46650241,\allowbreak 1.92702138,\allowbreak 2.00677927,\allowbreak 2.00148491,\allowbreak -0.30717680,\allowbreak -2.78086400,\allowbreak 2.90109018,\allowbreak -0.73623583,\allowbreak -1.89159622,\allowbreak 0.45096040,\allowbreak 2.52171688,\allowbreak -2.08529667,\allowbreak -2.62088685,\allowbreak -2.98612965\bigr)$ \\
$j^{(2)}(37)$ & $1373^{+}_{686}$ & $\bigl(-2.86183690,\allowbreak -0.23300076,\allowbreak 0.91485787,\allowbreak 1.22232078,\allowbreak 2.07492203,\allowbreak 1.11375472,\allowbreak -1.72514627,\allowbreak -3.10134637,\allowbreak -2.80244353,\allowbreak 1.10588280,\allowbreak 0.56375681,\allowbreak -1.95988914,\allowbreak -2.99499232,\allowbreak -0.23460333,\allowbreak 2.34740839,\allowbreak 1.38694876,\allowbreak 0.30196591,\allowbreak 3.09657873\bigr)$ \\
$j^{(3)}(37)$ & $1373^{+}_{686}$ & $\bigl(-2.55850855,\allowbreak -0.25148601,\allowbreak -0.27194303,\allowbreak 1.42309483,\allowbreak -0.38417040,\allowbreak 2.59311271,\allowbreak -0.53107919,\allowbreak 0.75414553,\allowbreak 3.04043797,\allowbreak 0.54187279,\allowbreak 0.93137805,\allowbreak 1.39860249,\allowbreak 0.01771039,\allowbreak 0.65560457,\allowbreak -2.20898615,\allowbreak 1.91344766,\allowbreak -0.17366787,\allowbreak -0.63791792\bigr)$ \\
$j_A^{(1)}(39)$ & $1525^+_{762}$ & $\bigl(-3.13369268,\allowbreak 2.95818037,\allowbreak -0.98823140,\allowbreak -1.40041149,\allowbreak -0.72042208,\allowbreak 0.59431207,\allowbreak 1.40554587,\allowbreak 0.48348010,\allowbreak 1.28356200,\allowbreak -0.83600716,\allowbreak 1.01826302,\allowbreak 2.74779526,\allowbreak 1.23563362,\allowbreak -0.84117736,\allowbreak -0.20423459,\allowbreak -0.41653466,\allowbreak -1.23532759,\allowbreak 2.78463832,\allowbreak -0.07374024\bigr)$ \\
$j_A^{(2)}(39)$ & $1525^+_{762}$ & $\bigl(-3.12164776,\allowbreak -2.68160281,\allowbreak -0.33164731,\allowbreak -0.02197197,\allowbreak 2.06781805,\allowbreak -0.48986462,\allowbreak -2.45904707,\allowbreak -2.99029930,\allowbreak 2.35984068,\allowbreak 2.84001351,\allowbreak 0.77420286,\allowbreak 1.37857867,\allowbreak 2.24328091,\allowbreak 1.38786498,\allowbreak 0.51012702,\allowbreak 1.13575418,\allowbreak -2.21631103,\allowbreak 1.98120639,\allowbreak 1.50523810\bigr)$ \\
$j_A^{(3)}(39)$ & $1525^+_{6}$ & $\bigl(-3.03803470,\allowbreak 0.92190281,\allowbreak 2.96162525,\allowbreak -2.31806250,\allowbreak 1.06657924,\allowbreak -1.47716199,\allowbreak 0.57609544,\allowbreak -0.86056489,\allowbreak -0.68333111,\allowbreak -3.12754059,\allowbreak -2.92125976,\allowbreak 1.85335242,\allowbreak -0.42893833,\allowbreak -1.03445083,\allowbreak 1.69642075,\allowbreak -1.88800191,\allowbreak -0.42808229,\allowbreak -0.02132942,\allowbreak -0.94808374\bigr)$ \\
$j_A^{(4)}(39)$ & $1525^+_{6}$ & $\bigl(-3.03365890,\allowbreak -2.38050810,\allowbreak 3.01742052,\allowbreak 1.51578516,\allowbreak -0.37153659,\allowbreak 2.01208215,\allowbreak -0.74781721,\allowbreak -2.53821679,\allowbreak 2.94677725,\allowbreak 0.53296501,\allowbreak -2.78076847,\allowbreak 2.43999499,\allowbreak -2.62185550,\allowbreak -1.25089710,\allowbreak -1.70783724,\allowbreak 0.29744333,\allowbreak 0.76047085,\allowbreak -1.63915357,\allowbreak -1.92999404\bigr)$ \\
$j_B^{(1)}(39)$ & $5^+_{1}\times\allowbreak 5^+_{1}\times\allowbreak 61^+_{2}$ & $\bigl(-2.83797502,\allowbreak 0.83610832,\allowbreak 2.64880121,\allowbreak -2.04518786,\allowbreak -2.64821957,\allowbreak -0.46792812,\allowbreak 1.41855078,\allowbreak 2.40382317,\allowbreak -0.09687180,\allowbreak 0.66356249,\allowbreak -1.63817423,\allowbreak 0.70967933,\allowbreak 1.27333779,\allowbreak 0.15781794,\allowbreak -0.96297334,\allowbreak 2.98358380,\allowbreak 2.15727672,\allowbreak -1.53795992,\allowbreak 0.85779195\bigr)$ \\
$j^{(1)}(41)$ & $1685^+_{842}$ & $\bigl(-2.93155045,\allowbreak -1.90223096,\allowbreak -1.78842531,\allowbreak -2.75857905,\allowbreak 0.57535106,\allowbreak 1.85901855,\allowbreak -0.66086818,\allowbreak -3.08947449,\allowbreak 1.01132521,\allowbreak -2.65234394,\allowbreak -1.41511324,\allowbreak -0.50010945,\allowbreak 2.20509583,\allowbreak -2.86332858,\allowbreak 2.31018496,\allowbreak 3.12285157,\allowbreak -1.84568475,\allowbreak -2.24845059,\allowbreak -0.34331650,\allowbreak -1.67396055\bigr)$ \\
$j^{(2)}(41)$ & $1685^+_{11}$ & $\bigl(-2.91011077,\allowbreak -0.57969312,\allowbreak 2.16367483,\allowbreak 0.13013752,\allowbreak -1.06188213,\allowbreak -1.61255072,\allowbreak -0.85268651,\allowbreak -2.11089060,\allowbreak 1.26349217,\allowbreak 0.52252630,\allowbreak -1.11531668,\allowbreak 1.80425216,\allowbreak 2.80726699,\allowbreak 1.87294535,\allowbreak 2.51598754,\allowbreak -2.73327444,\allowbreak 1.34768116,\allowbreak 1.05692826,\allowbreak -1.58770514,\allowbreak -0.78725361\bigr)$ \\
$j^{(1)}(43)$ & $1853^+_{3}$ & $\bigl(-3.03031904,\allowbreak -2.85070822,\allowbreak 1.41722344,\allowbreak 1.81570063,\allowbreak 0.82808079,\allowbreak -1.47090652,\allowbreak 2.12199138,\allowbreak -1.23767246,\allowbreak -2.65713723,\allowbreak 1.30971258,\allowbreak -2.29167336,\allowbreak -2.97522241,\allowbreak 2.48673959,\allowbreak -0.88976244,\allowbreak 0.36694230,\allowbreak 1.74276444,\allowbreak 2.49202310,\allowbreak 1.25550479,\allowbreak 0.52959473,\allowbreak 3.00471824,\allowbreak 2.77627799\bigr)$ \\
$j^{(2)}(43)$ & $1853^+_{926}$ & $\bigl(-2.79377627,\allowbreak -0.23506613,\allowbreak -1.37464406,\allowbreak 0.07411665,\allowbreak -0.46849875,\allowbreak -0.50063710,\allowbreak -2.57982020,\allowbreak 3.06545300,\allowbreak -0.81994313,\allowbreak 0.76459725,\allowbreak 1.30342936,\allowbreak 1.68478365,\allowbreak -0.92621382,\allowbreak 0.94445134,\allowbreak -1.56023914,\allowbreak 2.44009164,\allowbreak -2.70387910,\allowbreak -0.44445444,\allowbreak 2.96180300,\allowbreak -1.37842932,\allowbreak -0.79414721\bigr)$ \\
\end{longtable}
\endgroup

For a negative row $\alpha=(\alpha_2,\ldots,\alpha_{n+1})$, the coordinates
are $\alpha_g=(d_--1)A_{1,g}$ for $2\leq g\leq n+1$.
The full matrix is recovered by
Proposition~\ref{prop:HI-cyclic-Q-system-normal-form}, with $d=d_-$
and $v_g=\alpha_g$ for these indices.
Every row in Table~\ref{tab:cyclic-negative} uses this convention.
The displayed vectors are rounded decimal approximations;
the matrices reconstructed from them do not, by themselves,
verify the HI equations.

\begingroup
\footnotesize
\begin{longtable}{@{}>{\raggedright\arraybackslash}p{.12\linewidth}
 >{\raggedright\arraybackslash}p{.13\linewidth}
 >{\fontsize{7}{8.8}\selectfont\raggedright\arraybackslash}p{\dimexpr.75\linewidth-4\tabcolsep\relax}@{}}
\caption{Non-unitary solutions.}\label{tab:cyclic-negative}\\
\toprule
row & modular label & \footnotesize scaled first-row vector $\alpha$\\
\midrule
\endfirsthead
\multicolumn{3}{c}{\tablename\ \thetable\ -- continued from previous page}\\
\toprule
row & modular label & \footnotesize scaled first-row vector $\alpha$\\
\midrule
\endhead
\midrule
\multicolumn{3}{r@{}}{continued on next page}\\
\endfoot
\bottomrule
\endlastfoot
$\alpha(3)$ & $13^-_1$ & $\bigl(1.2434888\bigr)$ \\
$\alpha(5)$ & $29^-_1$ & $\bigl(1.5147145,\allowbreak -1.6938509\bigr)$ \\
$\alpha(7)$ & $53^-_1$ & $\bigl(1.6526511,\allowbreak -2.0677335,\allowbreak 2.21047598\bigr)$ \\
$\alpha^{(1)}(9)$ & $85^-_1$ & $\bigl(1.7329877,\allowbreak -2.289969,\allowbreak 2.6393036,\allowbreak -2.758412\bigr)$ \\
$\alpha^{(2)}(9)$ & $85^-_3$ & $\bigl(-0.0833485,\allowbreak 0.1369450,\allowbreak 0.3271260,\allowbreak -0.1461992\bigr)$ \\
$\alpha_A^{(1)}(11)$ & $125^-_1$ & $\bigl(1.7845947,\allowbreak -2.4333014,\allowbreak 2.9201099,\allowbreak -3.2221873,\allowbreak 3.3245994\bigr)$ \\
$\alpha_B^{(1)}(11)$ & $25^-_1 \times 5^-_1$ & $\bigl(0.9491070,\allowbreak 0.2046143,\allowbreak -0.4406823,\allowbreak -0.4406823,\allowbreak 0.9491070\bigr)$ \\
$\alpha(13)$ & $173^-_1$ & $\bigl(1.8201765,\allowbreak -2.5318290,\allowbreak 3.1141143,\allowbreak -3.5465260,\allowbreak 3.8128364,\allowbreak -3.9027751\bigr)$ \\
$\alpha^{(1)}(15)$ & $229^-_1$ & $\bigl(1.8460330,\allowbreak -2.6029967,\allowbreak 3.2542442,\allowbreak -3.7820521,\allowbreak 4.1709631,\allowbreak -4.4092041,\allowbreak 4.4894462\bigr)$ \\
$\alpha^{(2)}(15)$ & $229^-_1$ & $\bigl(-0.0398642,\allowbreak 0.0354818,\allowbreak -0.0430297,\allowbreak 0.0883484,\allowbreak 0.1957864,\allowbreak -0.0626497,\allowbreak 0.0502065\bigr)$ \\
$\alpha^{(3)}(15)$ & $229^-_1$ & $\bigl(0.3612877,\allowbreak -0.4571683,\allowbreak 0.0780723,\allowbreak 1.5080356,\allowbreak -0.0482876,\allowbreak -1.2299295,\allowbreak -0.1833390\bigr)$ \\
$\alpha(17)$ & $293^-_1$ & $\bigl(1.8655938,\allowbreak -2.6564473,\allowbreak 3.3592131,\allowbreak -3.9587578,\allowbreak 4.4410548,\allowbreak -4.7943831,\allowbreak 5.0099846,\allowbreak -5.0824596\bigr)$ \\
$\alpha^{(1)}(19)$ & $365^-_1$ & $\bigl(1.8808675,\allowbreak -2.6978671,\allowbreak 3.4402322,\allowbreak -4.0950639,\allowbreak 4.6498819,\allowbreak -5.0936512,\allowbreak 5.4173427,\allowbreak -5.6142856,\allowbreak 5.6803920\bigr)$ \\
$\alpha^{(2)}(19)$ & $365^-_7$ & $\bigl(1.2900891,\allowbreak -0.6416790,\allowbreak -0.3807808,\allowbreak 0.6623690,\allowbreak -0.1380162,\allowbreak -0.6607167,\allowbreak 0.5616232,\allowbreak 0.2270854,\allowbreak -0.6071557\bigr)$ \\
$\alpha^{(1)}(21)$ & $445^-_1$ & $\bigl(1.8931012,\allowbreak -2.7307935,\allowbreak 3.5043385,\allowbreak -4.2027135,\allowbreak 4.8148956,\allowbreak -5.3307561,\allowbreak 5.7415407,\allowbreak -6.0401740,\allowbreak 6.2214666,\allowbreak -6.2822522\bigr)$ \\
$\alpha^{(2)}(21)$ & $445^-_3$ & $\bigl(-0.1358047,\allowbreak 0.0388500,\allowbreak 0.0572406,\allowbreak -0.3621546,\allowbreak 0.6415308,\allowbreak 0.0339559,\allowbreak -0.0964870,\allowbreak 0.0327963,\allowbreak 0.1177727,\allowbreak 1.2434487\bigr)$ \\
$\alpha^{(3)}(21)$ & $445^-_1$ & $\bigl(0.5605789,\allowbreak 0.9357197,\allowbreak -0.9074126,\allowbreak -0.2320658,\allowbreak 0.0951729,\allowbreak 1.0803263,\allowbreak -0.5913956,\allowbreak -0.4152355,\allowbreak -0.2049533,\allowbreak 0.8227112\bigr)$ \\
$\alpha^{(4)}(21)$ & $445^-_3$ & $\bigl(-0.0354184,\allowbreak -0.0420227,\allowbreak 1.6141476,\allowbreak 0.0214035,\allowbreak -0.0456781,\allowbreak -0.0993081,\allowbreak -0.6556051,\allowbreak 0.0155921,\allowbreak -0.0961727,\allowbreak 0.3427684\bigr)$ \\
$\alpha^{(1)}(23)$ & $533^-_1$ & $\bigl(1.9031068,\allowbreak -2.7575276,\allowbreak 3.5561312,\allowbreak -4.2894565,\allowbreak 4.9477837,\allowbreak -5.5219252,\allowbreak 6.0036414,\allowbreak -6.3859039,\allowbreak 6.6630778,\allowbreak -6.8310509,\allowbreak 6.8873218\bigr)$ \\
$\alpha^{(2)}(23)$ & $533^-_6$ & $\bigl(1.3945246,\allowbreak -0.8570807,\allowbreak -0.3578319,\allowbreak 1.1931777,\allowbreak -1.2079603,\allowbreak 0.4574721,\allowbreak 0.4912947,\allowbreak -0.7403223,\allowbreak 0.1471603,\allowbreak 1.0477742,\allowbreak -1.5530340\bigr)$ \\
$\alpha^{(1)}(25)$ & $629^-_1$ & $\bigl(1.9114337,\allowbreak -2.7796229,\allowbreak 3.5987220,\allowbreak -4.3605692,\allowbreak 5.0565828,\allowbreak -5.6784710,\allowbreak 6.2185965,\allowbreak -6.6702048,\allowbreak 7.0275830,\allowbreak -7.2861762,\allowbreak 7.4426728,\allowbreak -7.4950629\bigr)$ \\
$\alpha^{(2)}(25)$ & $629^-_6$ & $\bigl(1.6849559,\allowbreak -1.8522369,\allowbreak 1.4162542,\allowbreak -0.4815149,\allowbreak -0.7432366,\allowbreak 1.6655411,\allowbreak -2.0433085,\allowbreak 1.7718447,\allowbreak -0.9419326,\allowbreak -0.2029681,\allowbreak 0.9753562,\allowbreak -1.2627561\bigr)$ \\
$\alpha^{(1)}(27)$ & $733^-_1$ & $\bigl(1.8554544,\allowbreak -2.5267070,\allowbreak 2.9454784,\allowbreak -3.0667676,\allowbreak 2.8749734,\allowbreak -2.3857297,\allowbreak 1.6436935,\allowbreak -0.7152035,\allowbreak -0.3381156,\allowbreak 1.2211322,\allowbreak -1.9178588,\allowbreak 2.3652139,\allowbreak -2.5194527\bigr)$ \\
$\alpha^{(2)}(27)$ & $733^-_1$ & $\bigl(0.6229592,\allowbreak 0.3307942,\allowbreak -0.0569321,\allowbreak -0.3418595,\allowbreak -0.2107442,\allowbreak 0.0713592,\allowbreak 0.2689532,\allowbreak 0.1564335,\allowbreak -0.0856702,\allowbreak -0.2404725,\allowbreak -0.1243237,\allowbreak 0.1023910,\allowbreak 0.2325228\bigr)$ \\
$\alpha^{(3)}(27)$ & $733^-_1$ & $\bigl(1.9184663,\allowbreak -2.7981623,\allowbreak 3.63428084,\allowbreak -4.4197441,\allowbreak 5.1469554,\allowbreak -5.8084422,\allowbreak 6.3971810,\allowbreak -6.90679639,\allowbreak 7.3316995,\allowbreak -7.6671916,\allowbreak 7.9095418,\allowbreak -8.0560451,\allowbreak 8.1050625\bigr)$ \\
$\alpha_A^{(1)}(29)$ & $845^-_1$ & $\bigl(1.9244810,\allowbreak -2.8139216,\allowbreak 3.6643605,\allowbreak -4.4696296,\allowbreak 5.2229811,\allowbreak -5.9176781,\allowbreak 6.5472851,\allowbreak -7.1058449,\allowbreak 7.5879998,\allowbreak -7.9890824,\allowbreak 8.3051856,\allowbreak -8.5332166,\allowbreak 8.6709382,\allowbreak -8.7169964\bigr)$ \\
$\alpha_A^{(2)}(29)$ & $845^-_6$ & $\bigl(0.5270628,\allowbreak 0.1860281,\allowbreak -0.0759645,\allowbreak -0.5134186,\allowbreak -0.0590021,\allowbreak 0.6006571,\allowbreak 0.0763627,\allowbreak -1.7787721,\allowbreak 0.0159781,\allowbreak 2.0677539,\allowbreak 0.0996648,\allowbreak -0.7073088,\allowbreak -0.0398396,\allowbreak 0.9646156\bigr)$ \\
$\alpha_B^{(1)}(29)$ & $65^-_1 \times 13^-_1$ & $\bigl(-0.5218822,\allowbreak 0.2057083,\allowbreak 0.09962041,\allowbreak 0.2057083,\allowbreak 0.2057083,\allowbreak -0.1674305,\allowbreak 0.09962041,\allowbreak -0.5218822,\allowbreak -0.08108325,\allowbreak -1.077645,\allowbreak -0.03926699,\allowbreak -0.5218822,\allowbreak 0.4247714,\allowbreak -0.03926699\bigr)$ \\
$\alpha^{(1)}(31)$ & $965^-_1$ & $\bigl(1.9296816,\allowbreak -2.8274695,\allowbreak 3.6900981,\allowbreak -4.5121662,\allowbreak 5.2876590,\allowbreak -6.0104925,\allowbreak 6.6747782,\allowbreak -7.2749805,\allowbreak 7.8060240,\allowbreak -8.2633733,\allowbreak 8.6430954,\allowbreak -8.9419089,\allowbreak 9.1572226,\allowbreak -9.2871651,\allowbreak 9.3306054\bigr)$ \\
$\alpha^{(2)}(31)$ & $965^-_{13}$ & $\bigl(-2.7399559,\allowbreak -0.8393053,\allowbreak -0.1182483,\allowbreak -1.4435280,\allowbreak -0.4174896,\allowbreak -0.2941592,\allowbreak -2.2494138,\allowbreak -1.3632528,\allowbreak 0.13112346,\allowbreak -1.2077058,\allowbreak -0.904472714,\allowbreak 0.099198200,\allowbreak -2.12552714,\allowbreak -2.221277522,\allowbreak -0.01804310\bigr)$ \\
$\alpha^{(1)}(33)$ & $1093^{-}_{1}$ & $\bigl(1.94328725,\allowbreak 0.01863406,\allowbreak -1.87022383,\allowbreak 0.04005396,\allowbreak 0.79983049,\allowbreak -0.26267870,\allowbreak 0.16242696,\allowbreak 0.35658909,\allowbreak 0.34846383,\allowbreak -0.13648156,\allowbreak -0.73772057,\allowbreak -0.08941442,\allowbreak 0.51963575,\allowbreak -0.81099518,\allowbreak -0.05356399,\allowbreak 1.33423479\bigr)$ \\
$\alpha^{(2)}(33)$ & $1093^{-}_{1}$ & $\bigl(1.93422118,\allowbreak -2.83923189,\allowbreak 3.71234296,\allowbreak -4.54880412,\allowbreak 5.34323384,\allowbreak -6.09012600,\allowbreak 6.78409211,\allowbreak -7.42000311,\allowbreak 7.99308494,\allowbreak -8.49898925,\allowbreak 8.93384863,\allowbreak -9.29432091,\allowbreak 9.57762492,\allowbreak -9.78156881,\allowbreak 9.90457174,\allowbreak -9.94567912\bigr)$ \\
$\alpha^{(3)}(33)$ & $1093^{-}_{1}$ & $\bigl(1.89092342,\allowbreak -2.65059354,\allowbreak 3.22692730,\allowbreak -3.57861404,\allowbreak 3.67938591,\allowbreak -3.52028520,\allowbreak 3.11037540,\allowbreak -2.47592164,\allowbreak 1.65788352,\allowbreak -0.70668553,\allowbreak -0.33755155,\allowbreak 1.25020393,\allowbreak -2.02909031,\allowbreak 2.62650126,\allowbreak -3.00227588,\allowbreak 3.13050550\bigr)$ \\
$\alpha^{(4)}(33)$ & $1093^{-}_{1}$ & $\bigl(1.81224929,\allowbreak -2.80075776,\allowbreak -0.01339349,\allowbreak 1.31456736,\allowbreak -4.61182118,\allowbreak 0.01656695,\allowbreak 0.63106959,\allowbreak -5.53101224,\allowbreak -0.06729999,\allowbreak 0.13713607,\allowbreak -5.13182837,\allowbreak -0.65931785,\allowbreak -0.01746619,\allowbreak -3.69262239,\allowbreak -1.96215143,\allowbreak 0.01067282\bigr)$ \\
$\alpha^{(5)}(33)$ & $1093^{-}_{1}$ & $\bigl(1.69814576,\allowbreak -1.87541310,\allowbreak 1.46362734,\allowbreak -0.60649274,\allowbreak -0.40993946,\allowbreak 1.09428198,\allowbreak -1.26183616,\allowbreak 0.88673582,\allowbreak -0.13953696,\allowbreak -0.78465180,\allowbreak 1.28999380,\allowbreak -1.22336116,\allowbreak 0.63336447,\allowbreak 0.24957299,\allowbreak -0.92934429,\allowbreak 1.18937567\bigr)$ \\
$\alpha^{(1)}(35)$ & $1229^{-}_{1}$ & $\bigl(1.93821699,\allowbreak -2.84953342,\allowbreak 3.73174096,\allowbreak -4.58064500,\allowbreak 5.39141344,\allowbreak -6.15905059,\allowbreak 6.87862078,\allowbreak -7.54537730,\allowbreak 8.15484819,\allowbreak -8.70289914,\allowbreak 9.18578258,\allowbreak -9.60017738,\allowbreak 9.94322132,\allowbreak -10.21253759,\allowbreak 10.40625598,\allowbreak -10.52302916,\allowbreak 10.56204414\bigr)$ \\
$\alpha^{(2)}(35)$ & $1229^{-}_{1}$ & $\bigl(1.86991231,\allowbreak 0.01718196,\allowbreak -1.70855729,\allowbreak 0.02912538,\allowbreak 0.78374100,\allowbreak 0.33602380,\allowbreak -0.06459944,\allowbreak -0.91682491,\allowbreak -0.02711630,\allowbreak 1.07513290,\allowbreak -0.03437599,\allowbreak -0.63552642,\allowbreak -0.18372300,\allowbreak 0.11028420,\allowbreak 0.69098091,\allowbreak 0.03260751,\allowbreak -0.96962096\bigr)$ \\
$\alpha^{(3)}(35)$ & $1229^{-}_{1}$ & $\bigl(1.69174126,\allowbreak -1.83664059,\allowbreak 1.34713054,\allowbreak -0.35060972,\allowbreak -0.91870928,\allowbreak 1.85715709,\allowbreak -2.17560733,\allowbreak 1.76711558,\allowbreak -0.75675610,\allowbreak -0.56098402,\allowbreak 1.59800433,\allowbreak -2.07319944,\allowbreak 1.85086670,\allowbreak -1.02033873,\allowbreak -0.14461563,\allowbreak 0.93974880,\allowbreak -1.23888710\bigr)$ \\
$\alpha^{(1)}(37)$ & $1373^{-}_{1}$ & $\bigl(1.94176039,\allowbreak -2.85862557,\allowbreak 3.74879109,\allowbreak -4.60853940,\allowbreak 5.43351670,\allowbreak -6.21917846,\allowbreak 6.96099717,\allowbreak -7.65458055,\allowbreak 8.29574913,\allowbreak -8.88059264,\allowbreak 9.40551376,\allowbreak -9.86726367,\allowbreak 10.26297136,\allowbreak -10.59016795,\allowbreak 10.84680680,\allowbreak -11.03127959,\allowbreak 11.14242875,\allowbreak -11.17955623\bigr)$ \\
$\alpha^{(2)}(37)$ & $1373^{-}_{1}$ & $\bigl(2.11677531,\allowbreak 0.01491531,\allowbreak -2.21513281,\allowbreak 0.02632909,\allowbreak 1.09632757,\allowbreak -0.35677131,\allowbreak 0.09712764,\allowbreak 0.61294931,\allowbreak 0.11516128,\allowbreak -0.36782643,\allowbreak -0.47430974,\allowbreak -0.10092149,\allowbreak 0.44352014,\allowbreak -0.23995221,\allowbreak -0.13014358,\allowbreak 0.99172085,\allowbreak -0.03132459,\allowbreak -1.38253405\bigr)$ \\
$\alpha^{(3)}(37)$ & $1373^{-}_{1}$ & $\bigl(1.73852139,\allowbreak -2.01370677,\allowbreak 1.72564987,\allowbreak -0.93852225,\allowbreak -0.14639143,\allowbreak 0.95568517,\allowbreak -1.37211809,\allowbreak 1.26540690,\allowbreak -0.65770371,\allowbreak -0.28967220,\allowbreak 1.15126092,\allowbreak -1.70769781,\allowbreak 1.78687884,\allowbreak -1.36127169,\allowbreak 0.55842734,\allowbreak 0.39275842,\allowbreak -1.08658635,\allowbreak 1.34334670\bigr)$ \\
$\alpha_A^{(1)}(39)$ & $1525^-_{3}$ & $\bigl(1.96770666,\allowbreak -0.00547312,\allowbreak -9.93471099,\allowbreak -0.01286042,\allowbreak -0.91900026,\allowbreak -3.24339452,\allowbreak 0.00482034,\allowbreak -9.58073700,\allowbreak 0.01198489,\allowbreak 0.82005346,\allowbreak -3.74267160,\allowbreak -0.00738332,\allowbreak 2.90291926,\allowbreak 0.00471194,\allowbreak -7.43587194,\allowbreak 0.14164621,\allowbreak -0.07169065,\allowbreak -6.70971847,\allowbreak -0.00359744\bigr)$ \\
$\alpha_A^{(2)}(39)$ & $1525^-_{1}$ & $\bigl(1.94492347,\allowbreak -2.86670597,\allowbreak 3.76388434,\allowbreak -4.63315284,\allowbreak 5.47057573,\allowbreak -6.27200781,\allowbreak 7.03328869,\allowbreak -7.75035185,\allowbreak 8.41929530,\allowbreak -9.03643249,\allowbreak 9.59833154,\allowbreak -10.10184707,\allowbreak 10.54414659,\allowbreak -10.92273270,\allowbreak 11.23546170,\allowbreak -11.48055902,\allowbreak 11.65663170,\allowbreak -11.76267792,\allowbreak 11.79809391\bigr)$ \\
$\alpha_A^{(3)}(39)$ & $1525^-_{3}$ & $\bigl(1.91336749,\allowbreak -2.72825665,\allowbreak 3.40456879,\allowbreak -3.90746278,\allowbreak 4.21033149,\allowbreak -4.29656414,\allowbreak 4.16055088,\allowbreak -3.80800372,\allowbreak 3.25559619,\allowbreak -2.52988877,\allowbreak 1.66537624,\allowbreak -0.70071276,\allowbreak -0.33708195,\allowbreak 1.26830886,\allowbreak -2.09840666,\allowbreak 2.79103771,\allowbreak -3.31238774,\allowbreak 3.63621670,\allowbreak -3.74603225\bigr)$ \\
$\alpha_A^{(4)}(39)$ & $1525^-_{1}$ & $\bigl(1.64351683,\allowbreak -1.65090298,\allowbreak 0.98338433,\allowbreak 0.11480992,\allowbreak -0.87291100,\allowbreak 1.10050567,\allowbreak -0.68114539,\allowbreak -0.22505237,\allowbreak 1.06299007,\allowbreak -1.50219773,\allowbreak 1.33704854,\allowbreak -0.63542799,\allowbreak -0.29681604,\allowbreak 0.92092023,\allowbreak -0.99172426,\allowbreak 0.47409123,\allowbreak 0.41889737,\allowbreak -1.17884533,\allowbreak 1.47759886\bigr)$ \\
$\alpha_B^{(1)}(39)$ & $5^-_{1}\times\allowbreak 5^-_{1}\times\allowbreak 61^-_{1}$ & $\bigl(1.85351625,\allowbreak -2.47325587,\allowbreak 2.77173185,\allowbreak -2.70514019,\allowbreak 2.28076022,\allowbreak -1.55590031,\allowbreak 0.62795848,\allowbreak 0.39485378,\allowbreak -1.25428087,\allowbreak 1.87338267,\allowbreak -2.17374219,\allowbreak 2.11936626,\allowbreak -1.72449025,\allowbreak 1.05131678,\allowbreak -0.19747623,\allowbreak -0.79169361,\allowbreak 1.60340098,\allowbreak -2.13890310,\allowbreak 2.32596659\bigr)$ \\
$\alpha^{(1)}(41)$ & $1685^-_{1}$ & $\bigl(1.94776386,\allowbreak -2.87393194,\allowbreak 3.77733103,\allowbreak -4.65501272,\allowbreak 5.50340810,\allowbreak -6.31872629,\allowbreak 7.09713688,\allowbreak -7.83487122,\allowbreak 8.52828732,\allowbreak -9.17391600,\allowbreak 9.76849629,\allowbreak -10.30900402,\allowbreak 10.79267561,\allowbreak -11.21702821,\allowbreak 11.57987678,\allowbreak -11.87934858,\allowbreak 12.11389519,\allowbreak -12.28230229,\allowbreak 12.38369713,\allowbreak -12.41755395\bigr)$ \\
$\alpha^{(2)}(41)$ & $1685^-_{17}$ & $\bigl(1.86260109,\allowbreak -2.50599124,\allowbreak 2.84477353,\allowbreak -2.83083429,\allowbreak 2.46059842,\allowbreak -1.77545406,\allowbreak 0.85497687,\allowbreak 0.20293039,\allowbreak -1.09204466,\allowbreak 1.78507157,\allowbreak -2.19636593,\allowbreak 2.27179712,\allowbreak -1.99968499,\allowbreak 1.41280631,\allowbreak -0.58247374,\allowbreak -0.40647358,\allowbreak 1.31847768,\allowbreak -2.06813502,\allowbreak 2.56081407,\allowbreak -2.73252593\bigr)$ \\
$\alpha^{(1)}(43)$ & $1853^-_{1}$ & $\bigl(1.95032819,\allowbreak -2.88043017,\allowbreak 3.78938030,\allowbreak -4.67454172,\allowbreak 5.53266872,\allowbreak -6.36028554,\allowbreak 7.15385853,\allowbreak -7.90989073,\allowbreak 8.62498182,\allowbreak -9.29587021,\allowbreak 9.91946525,\allowbreak -10.49287291,\allowbreak 11.01341733,\allowbreak -11.47865903,\allowbreak 11.88641051,\allowbreak -12.23474971,\allowbreak 12.52203135,\allowbreak -12.74689647,\allowbreak 12.90828017,\allowbreak -13.00541753,\allowbreak 13.03784793\bigr)$ \\
$\alpha^{(2)}(43)$ & $1853^-_{6}$ & $\bigl(-1.76848200,\allowbreak -0.75377868,\allowbreak 0.74626278,\allowbreak 1.01666968,\allowbreak 0.02920041,\allowbreak 1.55477723,\allowbreak 1.83623830,\allowbreak 0.48815344,\allowbreak -0.36025546,\allowbreak 0.52939498,\allowbreak -0.17388609,\allowbreak -1.55690534,\allowbreak -1.20872065,\allowbreak -0.08976415,\allowbreak -1.00865037,\allowbreak -0.89195242,\allowbreak 0.26489605,\allowbreak 0.88926734,\allowbreak -0.14505852,\allowbreak 0.70697952,\allowbreak 1.41173802\bigr)$ \\
\end{longtable}
\endgroup

\end{document}